\documentclass[A4, 10pt, reqno]{amsart}

\usepackage[utf8]{inputenc}
\usepackage[top=1.0in, bottom=0.9in, left=1.0in, right=1.0in]{geometry}
\usepackage[all]{xy}
\usepackage{amsfonts, amsmath, amssymb, amsthm}
\usepackage{enumerate, enumitem}
\usepackage{xparse, etoolbox}
\usepackage{mathtools, nccmath, textcomp, relsize}
\usepackage[new]{old-arrows}
\usepackage{stmaryrd}
\usepackage{ifthen}
\usepackage{tikz, tikz-cd}
\usepackage[sf, bf]{titlesec}	%sf for sans serif section headings
\usepackage[titles]{tocloft}
\usepackage{verbatim}
\usepackage{xcolor}

\usepackage{mathrsfs}
\usepackage[mathcal]{euscript}
\usepackage[bbgreekl]{mathbbol}

\usepackage[backend=biber,style=numeric, sorting=nty, doi=false, url=false, maxbibnames=99, maxalphanames=99]{biblatex}
\usepackage[T1]{fontenc}
\usepackage{lmodern}
\usepackage{sansmath}

\usepackage{hyperref}
\hypersetup{
	colorlinks=true,
	linkcolor=blue,
	citecolor=blue,
	urlcolor=black,
}

\newcommand{\addperiod}[1]{#1.}
\titleformat{\section}[block]{\scshape\Large\filcenter}{\thesection.}{1em}{}
\titleformat{\subsection}[runin]{\normalfont\large\bfseries}{\thesubsection.}{1em}{\addperiod}
\titleformat{\subsubsection}[runin]{\normalfont\bfseries}{\thesubsubsection.}{1em}{\addperiod}

\numberwithin{equation}{section}

\makeatletter
\def\keywords{\xdef\@thefnmark{}\@footnotetext}
\makeatother

\newcommand{\showdetailed}{0} %set to 1 to show details
\theoremstyle{plain}
\newtheorem{thm}{Theorem}[section]
\newtheorem{cor}[thm]{Corollary}
\newtheorem{lem}[thm]{Lemma}
\newtheorem{prop}[thm]{Proposition}

\theoremstyle{definition}
\newtheorem{defi}[thm]{Definition}

\newtheorem{exam}[thm]{Example}
\newtheorem{rem}[thm]{Remark}
\newtheorem{ques}[thm]{Question}
\AtEndEnvironment{exam}{\qed}
\AtEndEnvironment{const}{\qed}

\theoremstyle{remark}

\newtheorem*{nota}{Notation}

\newcommand{\isomorphic}{\xrightarrow{\hspace{0.5mm} \sim \hspace{0.5mm}}}
\newcommand{\lisomorphic}{\xleftarrow{\hspace{0.5mm} \sim \hspace{0.5mm}}}

\newcommand{\dlog}{d\textup{\hspace{0.5mm}log\hspace{0.3mm}}}
\newcommand{\defeq}{\vcentcolon=}

\DeclareMathAlphabet{\pazocal}{OMS}{zplm}{m}{n}

\newcommand{\calL}{\mathcal{L}}
\newcommand{\calR}{\mathcal{R}}
\newcommand{\calS}{\mathcal{S}}

\newcommand{\AR}{A_{\calR}}
\newcommand{\AS}{A_{\calS}}
\newcommand{\AL}{A_{\calL}}
\newcommand{\BR}{B_{\calR}}
\newcommand{\BS}{B_{\calS}}

\newcommand{\DR}{D_{\calR}}
\newcommand{\DS}{D_{\calS}}
\newcommand{\DL}{D_{\calL}}
\newcommand{\DRn}{D_{\calR,n}}
\newcommand{\DLn}{D_{\calL,n}}
\newcommand{\ER}{E_{\calR}}

\newcommand{\NR}{N_{\calR}}
\newcommand{\NS}{N_{\calS}}
\newcommand{\NL}{N_{\calL}}
\newcommand{\NRn}{N_{\calR,n}}
\newcommand{\NLn}{N_{\calL,n}}
\newcommand{\TR}{T_{\calR}}
\newcommand{\TS}{T_{\calS}}

\newcommand{\GR}{G_{\calR}}
\newcommand{\GS}{G_{\calS}}

\newcommand{\HR}{H_{\calR}}
\newcommand{\HS}{H_{\calS}}

\newcommand{\GammaR}{\Gamma_{\calR}}
\newcommand{\GammaS}{\Gamma_{\calS}}
\newcommand{\GammaL}{\Gamma_{\calL}}
\newcommand{\ALplusp}{A_{\calL}^+(\mathfrak{p})}
\newcommand{\NLp}{N_{\calL}(\mathfrak{p})}

\newcommand{\Acrys}{A_{\textup{cris}}}
\newcommand{\OAcrys}{\pazocal{O} A_{\textup{cris}}}
\newcommand{\Alogcrys}{A_{\textup{log-cris}}}
\newcommand{\OAlogcrys}{\pazocal{O} A_{\textup{log-cris}}}
\newcommand{\Bcrys}{B_{\textup{cris}}}
\newcommand{\OBcrys}{\pazocal{O} B_{\textup{cris}}}
\newcommand{\Blogcrys}{B_{\textup{log-cris}}}
\newcommand{\OBlogcrys}{\pazocal{O} B_{\textup{log-cris}}}
\newcommand{\ODlogcrys}{\pazocal{O} D_{\textup{log-cris},R}}
\newcommand{\ODcrysS}{\pazocal{O} D_{\textup{cris,S}}}
\newcommand{\BdR}{B_{\textup{dR}}}
\newcommand{\OBdR}{\pazocal{O} B_{\textup{dR}}}
\newcommand{\BlogdR}{B_{\textup{log-dR}}}
\newcommand{\OBlogdR}{\pazocal{O} B_{\textup{log-dR}}}
\newcommand{\ODlogdR}{\pazocal{O} D_{\textup{log-dR},R}}
\newcommand{\ODdRS}{\pazocal{O} D_{\textup{dR},S}}
\newcommand{\ODdRL}{\pazocal{O} D_{\textup{dR},L}}
\newcommand{\OBHT}{\pazocal{O} B_{\textup{HT}}}
\newcommand{\OBlogHT}{\pazocal{O} B_{\textup{log-HT}}}
\newcommand{\ODlogHT}{\pazocal{O} D_{\textup{log-HT},R}}
\newcommand{\ODHTS}{\pazocal{O} D_{\textup{HT,S}}}
\newcommand{\Spec}{\textup{Spec}}
\newcommand{\colim}{\mathrm{colim}}

\newcommand{\CRp}{\widehat{\overline{R}}}
\newcommand{\CSp}{\widehat{\overline{S}}}
\newcommand{\Amax}{ A_{\textup{max}}}
\newcommand{\OAmax}{\pazocal{O} A_{\textup{max}}}
\newcommand{\Alogmax}{A_{\textup{log-max}}}
\newcommand{\OAlogmax}{\pazocal{O} A_{\textup{log-max}}}
\newcommand{\Bmax}{B_{\textup{max}}}
\newcommand{\OBmax}{\pazocal{O} B_{\textup{max}}}
\newcommand{\Blogmax}{B_{\textup{log-max}}}
\newcommand{\OBlogmax}{\pazocal{O} B_{\textup{log-max}}}
\newcommand{\OAlog}{\pazocal{O} A_{\textup{log}}}

\newcommand{\OBlog}{\pazocal{O} B_{\textup{log}}}

\newcommand{\ODlog}{\pazocal{O} D_{\mathrm{log}}}

\newcommand{\etale}{\textup{\'et}}
\newcommand{\Fil}{\mathrm{Fil}}
\newcommand{\logPD}{\textup{logPD}}
\newcommand{\PD}{\textup{PD}}

\newcommand{\ARpi}{A_{\calR,\varpi}}
\newcommand{\OARpi}{\pazocal{O} A_{\calR,\varpi}}
\newcommand{\ASpi}{A_{\calS,\varpi}}
\newcommand{\OASpi}{\pazocal{O} A_{\calS,\varpi}}
\newcommand{\ODR}{\pazocal{O} D_{\calR}}

\newcommand{\Fr}{\mathrm{Fr}}

\newcommand{\Gal}{\mathrm{Gal}}
\usepackage{graphicx}

\newcommand{\GRp}{G_R(\mathfrak{p})}
\newcommand{\GRhatp}{\widehat{G}_R(\mathfrak{p})}

\newcommand{\Rbarp}{(\overline{R})_{\mathfrak{p}}}
\newcommand{\Cplusp}{\mathbb{C}^+(\mathfrak{p})}
\newcommand{\Cpplus}{\mathbb{C}^+_{\mathfrak{p}}}
\newcommand{\Cp}{\mathbb{C}_{\mathfrak{p}}}
\newcommand{\Cpplusflat}{\mathbb{C}^{+, \flat}_{\mathfrak{p}}}
\newcommand{\Lbar}{\overline{L}}

\newcommand{\Lbarp}{\overline{L}(\mathfrak{p})}
\newcommand{\OLbarp}{O_{\overline{L}(\mathfrak{p})}}
\newcommand{\pins}{\mathfrak{p} \in \mathscr{P}(\overline{R})}

\newcommand{\Addresses}{
    {\footnotesize
    \rule{2cm}{0.4pt}\vspace{2mm}
    
    \begin{tabular}{ l l l l l }
        \textsc{Abhinandan} &&&& \textsc{Denis Benois}\par\nopagebreak\\
        \textsc{IAZD, Leibniz Universit\"at Hannover} &&&& \textsc{IMB - UMR 5251, Universit\'e de Bordeaux}\par\nopagebreak\\
        \textsc{Welfengarten 1} &&&& \textsc{351, cours de la Lib\'eration}\par\nopagebreak\\
        \textsc{30167 Hannover, Germany} &&&& \textsc{33405 Talence, France}\par\nopagebreak\\
        \href{mailto:abhinandan@math.uni-hannover.de}{abhinandan@math.uni-hannover.de} &&&& \href{mailto:denis.benois@math.u-bordeaux.fr}{denis.benois@math.u-bordeaux.fr}
    \end{tabular}
    }
}

\title{Log-crystalline representations and $(\varphi, \Gamma)$-modules}
\author{Abhinandan, Denis Benois}
\date{\today}

\begin{document}


\maketitle{
	\textsc{Abstract.} Let $F$ be an absolutely unramified mixed characteristic local field with rings of integers $O_F$.
	For a small affine algebra $R$ over $O_F$, we consider the logarithmic \'etale fundamental group $G_R$ of its generic fibre equipped with a ``horizontal'' log-structure, in the sense of Fujiwara--Kato.
	We study $p\textrm{-adic}$ representations of $G_R$ and obtain a classification of these representations in terms of \'etale $(\varphi, \Gamma_R)\textrm{-modules}$.
	Moreover, we define and study the notion of log-crystalline representations of $G_R$, a generalisation of crystalline representations from the non-logarithmic/smooth case.
	Furthermore, in the logarithmic setting, we show that log-crystalline representations of $G_R$ are equivalent to Wach modules for $R$, extending our previous results from the (non-logarithmic) crystalline case.
}

\section{Introduction}

\subsection{Crystalline representations and \texorpdfstring{$(\varphi, \Gamma)\textrm{-modules}$}{-}}

The notion of $p\textrm{-adic}$ \textit{crystalline representations} was defined by Fontaine in \cite{fontaine-annals}, which positively answered the \textit{mysterious functor} question of Grothendieck from \cite{grothendieck-barsotti-tate}, and gave a precise meaning to the notion of \textit{good reduction} for $p\textrm{-adic}$ Galois representations.
On the other hand, in his seminal paper \cite{fontaine-phigamma}, Fontaine established the theory of \textit{\'etale $(\varphi, \Gamma)\textrm{-modules}$} classifying all $p\textrm{-adic}$ Galois representations of a local field, and conjectured a precise relationship between absolutely crystalline representations and $(\varphi, \Gamma)\textrm{-modules}$.
A description of absolutely crystalline representations in terms of $(\varphi, \Gamma)\textrm{-modules}$, known as \textit{Wach modules}, was obtained in \cite{wach-free, colmez-hauteur, berger-limites}, thus resolving Fontaine's conjecture.
In the relative case, i.e.\ for $p\textrm{-adic}$ representations of the \'etale fundamental group of certain affinoid algebras, the theory of crystalline representations was developed in \cite{brinon-relatif}, the theory of $(\varphi, \Gamma)\textrm{-modules}$ in \cite{andreatta-phigamma}, and the relation between the two was obtained in \cite{abhinandan-relative-wach-i, abhinandan-relative-wach-ii}.

In terms of applications, classical Wach modules have been useful in understanding Iwasawa theory of $p\textrm{-adic}$ representations and the $p\textrm{-adic}$ Langlands program (for example, see \cite{benois-iwasawa, benois-berger, berger-exponential, berger-breuil-cristallines, lei-loeffler-zerbes-wach, lei-loeffler-zerbes-coleman}).
Additionally, one may also relate Fontaine-Herr complexes of $(\varphi,\Gamma_F)\textrm{-modules}$, computing Galois cohomology of $p\textrm{-adic}$ representations, to complexes of syntomic type with coefficients in Wach modules (see \cite{abhinandan-crystalline-galcoh}).

Relative Wach modules have been used to construct syntomic cohomology with coefficients of crystalline nature in \cite{abhinandan-syntomic}, inspired by the work of Colmez--Nizio{\l} in \cite{colmez-niziol} which treated the case of constant coefficients.  

\subsubsection{The objective of this paper} 

In this paper, we extend some of the results mentioned above to the case of small affine algebras equipped with a ``horizontal'' logarithmic structure associated to a normal crossing divisor.
More concretely, we first obtain a complete classification of $p\textrm{-adic}$ representations of the logarithmic fundamental group of such algebras in terms of $(\varphi, \Gamma)\textrm{-modules}$ (see Theorem \ref{intro_thm:GR_rep_GR0m}), thus generalising the main categorical equivalence of \cite{andreatta-phigamma} to the logarithmic setting.
Then, we construct new rings of $p\textrm{-adic}$ periods $\OBlogdR(-)$ and $\OBlogcrys(-)$, and define the notion of \textit{log-de Rham} and \textit{log-crystalline} representations.
Our constructions are closer in spirit to some constructions of Andreatta--Iovita in \cite{andreatta-iovita-semistable}, however, the category of $p\textrm{-adic}$ log-crystalline representations we consider is \textit{different} from the semistable representations considered in op.\ cit.
Lastly, we show that the category of log-crystalline representations is \textit{equivalent} to the category of Wach modules (appropriately extended to the logarithmic setting, see Theorem \ref{intro_thm:logcrys_wach_equiv}).
This generalisation is \textit{completely novel} and constitutes the \textit{main result} of this paper. 

This work is motivated by our aim to extend the approach to syntomic cohomology with coefficients developed in \cite{abhinandan-syntomic}, to open varieties.
We shall address this question in a forthcoming work. 

In order to motivate the precise results obtained in this paper, let us first look at the case of local fields and the relative case, in a bit more detail.

\subsubsection{\texorpdfstring{$p\textrm{-adic}$}{-} representations of local fields}

Let $\kappa$ be a perfect field of characteristic $p$, set $O_F$ to be the ring of $p\textrm{-typical}$ Witt vectors $W(\kappa)$ and $F$ its field of fractions.
Let $\varphi$ denote the absolute Frobenius on $F$.
Fix an algebraic closure $\overline{F}$ of $F$ and set $G_F = \Gal(\overline{F}/F)$.
Let $F_{\infty}$ denote the cyclotomic $p\textrm{-extension}$ of $F$.
Set $\Gamma_F = \Gal(F_{\infty}/F)$ and denote by $\chi \colon \Gamma_F \isomorphic \mathbb{Z}_p^{\times}$ the $p\textrm{-adic}$ cyclotomic character. 

Let $A_F^+ = O_F\llbracket q-1 \rrbracket$, and set $A_F = A_F^+[1/(q-1)]^{\wedge}$ as the $p\textrm{-adic}$ completion.
These rings are equipped with an $O_F\textrm{-linear}$ action of $\Gamma_F$ with $g(q)=q^{\chi(g)}$ for $g$ in $\Gamma_F$, and a $\Gamma_F\textrm{-equivariant}$ semilinear Frobenius such that $\varphi(q) = q^p$.
A $(\varphi,\Gamma_F)\textrm{-module}$ over $A_F$ is a finitely generated module equipped with $A_F\textrm{-semilinear}$ and continuous actions of $\Gamma_F$ and $\varphi$ commuting with each other.
Fontaine's fundamental result from \cite{fontaine-phigamma} establishes an equivalence
\begin{equation}\label{eq:reps_phigamma_classical}
	D_F \colon \mathrm{Rep}_{\mathbb{Z}_p}(G_F) \isomorphic (\varphi,\Gamma_F)\textup{-Mod}_{A_F}^{\mathrm{\'et}},
\end{equation}
between the category of $\mathbb{Z}_p\textrm{-representations}$ of $G_F$ and the category of \'etale $(\varphi,\Gamma)\textrm{-modules}$ over $A_F$.

We turn our attention to $(\varphi,\Gamma_F)\textrm{-modules}$ associated to crystalline representations of $G_F$.
A \textit{Wach module} over $A_F^+$ is a finite free module $N$ equipped with a semilinear action of $\Gamma_F$ such that its induced action is trivial on $N/\mu N$, and a $\Gamma_F\textrm{-equivariant}$ and $A_F^+\textrm{-linear}$ isomorphism $\varphi_N \colon (\varphi^*N) [1/[p]_q] \isomorphic N [1/[p]_q]$.
Here $\varphi^*N = A_F^+ \otimes_{\varphi, A_F^+} N$, $\mu = q-1$ and $[p]_q = \tfrac{q^p-1}{q-1}$.
For a crystalline representation $V$, there exists a direct relationship between its associated (rational) Wach module and the filtered Dieudonn\'e modules $D_{\mathrm{cris}}(V)$.

\begin{thm}[{\cite[Wach]{wach-free}, \cite[Berger]{berger-limites}}] 
	Let $T$ be a finite free $\mathbb{Z}_p\textrm{-representation}$ of $G_F$ such that $V = T[1/p]$ is crystalline.
	Then,
	\begin{enumerate}
		\item[\textup{(1)}] There exists a unique Wach module $N_F(T) \subset D_F(T)$ such that the natural map $A_F \otimes_{A_F^+} N_F(T) \rightarrow D_F(T)$ is an isomorphism.

		\item[\textup{(2)}] There exists a natural isomorphism $(N_F(T)/\mu N_F(T))[1/p] \isomorphic D_{\mathrm{cris}}(V)$ of filtered Dieudonn\'e modules, where the former is equipped with an induced Frobenius and a filtration induced from the Nygaard filtration on $N_F(T)$.
	\end{enumerate}

	More generally, the following functor induces a natural equivalence of categories:
	\begin{equation*}
		\begin{aligned}
			\textup{Rep}_{\mathbb{Z}_p}^{\textup{cris}}(G_F) &\isomorphic \{\textrm{Wach modules over } A_F^+\}\\
			T &\longmapsto N_F(T).
		\end{aligned}
	\end{equation*}
\end{thm} 
 
\subsubsection{Relative \texorpdfstring{$p\textrm{-adic}$}{-} Hodge theory}\label{subsec:relative_padicHT}

Let us fix an integer $d \geqslant 0$ and let $S$ be a $p\textrm{-adically}$ completed \'etale algebra over $O_F\langle x_1^{\pm 1}, \ldots, x_d^{\pm 1}\rangle$ with connected special fibre.
Fix a compatible system $(\zeta_{p^n})_{n\geqslant 0}$ of primitive $p^n\textrm{-th}$ roots of unity, and for each $1 \leqslant i \leqslant d$, fix a compatible system $(x_i^{1/p^n})_{n\geqslant 0}$ of $p^n\textrm{-th}$ roots of $x_i$.
Consider the tower 
\begin{equation*}
	S_n \defeq S[\zeta_{p^n}, x_1^{1/p^n}, \ldots, x_d^{1/p^n}], \quad n\geqslant 1.
\end{equation*}
Set $S_{\infty} = \cup_{n\geqslant 1} S_n$, $\Gamma_S = \Gal(S_{\infty}[1/p]/S[1/p])$ and denote by $\chi \colon \Gamma_S \twoheadrightarrow \Gamma_F \isomorphic \mathbb{Z}_p^{\times}$ the $p\textrm{-adic}$ cyclotomic character.
Each automorphism $g \in \Gamma_S$ is completely determined by its action on $\zeta_{p^n}$ and $x_i^{1/p^n}$ for $1 \leqslant i\leqslant d$ and $n \geqslant 0$.
Therefore, there exist unique maps $\psi_i \colon \Gamma_S \rightarrow \mathbb{Z}_p$ such that 
\begin{equation}\label{eq:action_GammaS}
	g(x_i^{1/p^n}) = x_i^{1/p^n}\zeta_{p^n}^{\psi_i(g)}, \quad \textrm{for all } g \in \Gamma_S \textrm{ and } n \geqslant 0.
\end{equation}

Let $A_S^+ = S\llbracket q-1 \rrbracket$ and let $A_S$ denote the $p\textrm{-adic}$ completion of $A_S^+[1/\mu]$.
We equip $A_S$ with an $O_F\textrm{-linear}$ action of $\Gamma_S$ setting $g(x_i) = q^{\psi_i(g)}x_i$ and $g(q) = q^{\chi(g)}$ for all $g \in \Gamma_S$, and extend the action of $\varphi$ from $A_F$ to $A_S$ by setting  $\varphi(x_i) = x_i^p$ for $1 \leqslant i \leqslant d$. 
A $(\varphi, \Gamma_S)\textrm{-module}$ over $A_S$ is a finitely generated module equipped with continuous and semilinear actions of $\Gamma_S$ and $\varphi$ commuting with each other. 
In \cite{andreatta-phigamma}, Andreatta constructed an equivalence
\begin{equation}\label{eq:reps_phigamma_relative}
	D_S \colon \mathrm{Rep}_{\mathbb{Z}_p}(G_S) \isomorphic (\varphi,\Gamma_S) \textup{-Mod}_{A_S}^{\mathrm{ét}},
\end{equation}
between the category of $\mathrm{Z}_p\textrm{-representations}$ of the \'etale fundamental group $G_S$ of $S[1/p]$ and the category of \'etale $(\varphi,\Gamma_S)\textrm{-modules}$ over $A_S$, generalising \eqref{eq:reps_phigamma_classical}.\footnote{Note that if $d=0$, i.e.\ $S[1/p] = F$, then \eqref{eq:reps_phigamma_relative} agrees with \eqref{eq:reps_phigamma_classical}.}

Independently, and around the same time, Brinon \cite{brinon-relatif} constructed the relative analogues of Fontaine's de Rham and crystalline $p\textrm{-adic}$ period rings, and used them to define the notion of de Rham and crystalline representations of $G_S$. 
The theory of Wach modules for crystalline representations of $G_S$ was developed by the first named author in \cite{abhinandan-relative-wach-i, abhinandan-relative-wach-ii}, and subsequently used to construct syntomic complexes with coefficients \cite{abhinandan-syntomic}.

\subsection{Main results}

Let us now describe the main results obtained in this paper.
We begin by fixing some notation.

\subsubsection{\texorpdfstring{$p\textrm{-adic}$}{-} Galois representations}

Let us fix integers $0 \leqslant a \leqslant d$.
Let us set $R$ to be a \textit{$p\textrm{-adically}$ completed \'etale} algebra over $O_F\langle x_1^{\pm 1}, \ldots, x_a^{\pm 1}, x_{a+1}, \ldots, x_d\rangle$ with connected special fibre, and set $S \defeq R[1/(x_{a+1}\cdots x_d)]^{\wedge}_p$.
Note that $S$ is a $p\textrm{-adically}$ completed \'etale algebra over $O_F\langle x_1^{\pm 1}, \ldots, x_a^{\pm 1}, x_{a+1}^{\pm}, \ldots, x_d^{\pm}\rangle$, and this notation is consistent with that of Section \ref{subsec:relative_padicHT}.
Similar to $S$ above, we may extend the absolute Frobenius on $O_F$ to $R$ by setting $\varphi(x_i) = x_i^p$ for $1 \leqslant i \leqslant d$. 

Let us define $G_R$ to be the logarithmic fundamental group of $R[1/p]$ with respect to the log-structure defined by the divisor $\{x_{a+1} \cdots x_d = 0\}$ (see \cite[Section 10.3]{kato-fontaine-illusie-ii}), which coincides with the tame fundamental group of $R[1/p, 1/(x_{a+1} \cdots x_d)]$ (see \cite{grothendieck-murre}).
By choosing compatible geometric points corresponding to fixed algebraically closed fields $\overline{\mathrm{Fr}(R)} \subset \overline{\mathrm{Fr}(S)}$ we obtain a continuous homomorphism of Galois groups $G_S \rightarrow G_R$; note that this homomorphism is \textit{not} surjective (see Example \ref{eg:gabber_example}).

As in Section \ref{subsec:relative_padicHT}, we consider generalised cyclotomic towers over $R$ and $S$, namely
\begin{equation*}
	R_{n} \defeq R\big[\zeta_{p^n}, x_1^{1/p^n}, \ldots, x_d^{1/p^n}\big], \qquad S_{n} \defeq S\big[\zeta_{p^n}, x_1^{1/p^n}, \ldots, x_d^{1/p^n}\big] = R_n \otimes_R S, \quad n\geqslant 0.
\end{equation*}
More generally, for any $m \geqslant 1$ coprime to $p$, we set
\begin{equation*}
	R_{n,m} \defeq R_n\big[\zeta_{m}, x_1^{1/mp^n}, \ldots, x_d^{1/mp^n}\big], \qquad S_{n,m} \defeq S_n\big[\zeta_{m}, x_1^{1/mp^n}, \ldots, x_d^{1/mp^n}\big] = R_{n,m} \otimes_R S.
\end{equation*}
These extensions  will be used  to kill  ramification along the divisor $\{x_{a+1} \cdots x_d = 0\}$. 

Let us set $R_{\infty} \defeq \cup_{n \geqslant 0} R_n$, $R_{\infty,m} \defeq \cup_{n \geqslant 0} R_{n,m}$, and define the Galois groups $\Gamma_R \defeq \mathrm{Gal}(R_{\infty}[1/p])/R[1/p])$ and $H_R \defeq \textup{Ker}(G_R \twoheadrightarrow \Gamma_R)$.
Analogously, set $S_{\infty} \defeq \cup_{n \geqslant 0} S_n$, $S_{\infty,m} \defeq \cup_{n \geqslant 0} S_{n,m}$, and define the Galois groups $\Gamma_S \defeq \mathrm{Gal}(S_{\infty}[1/p])/S[1/p])$ and $H_R \defeq \textup{Ker}(G_R \twoheadrightarrow \Gamma_R)$.
Note that $\Gamma_R \isomorphic \Gamma_S$.

Let $R^{\mathrm{ét}}_{\infty,m}$ denote the union of finite normal $R_{\infty,m}\textrm{-subalgebras}$ $A \subset \overline{R}$ such that $A[1/p]$ is \'etale over $R_{\infty,m}[1/p]$.
Our first result follows from an application of Abhyankar's Lemma (see \cite[XIII, Proposition 5.2]{sga1}) and purity of the ramification locus (see \cite[\href{https://stacks.math.columbia.edu/tag/0EA4}{Tag 0EA4}]{stacks-project}).
\begin{thm}[{Theorem \ref{thm:GR_rep_GR0m}}]\label{intro_thm:GR_rep_GR0m}
	Let $T$ be a finitely generated $\mathbb{Z}_p\textrm{-representation}$ of $G_R$.
	Then there exists a positive integer $m$ coprime to $p$, depending on $T/pT$, and such that the action of $G_R$ on $T$ factors through $\Gal(R^{\textup{\'et}}_{\infty,m}[1/p]/R[1/p]).$
\end{thm}

The preceding theorem reduces us from the study of $p\textrm{-adic}$ representations of $G_R$ to the study of $p\textrm{-adic}$ representations of the group $\Gal(R^{\textrm{\'et}}_{\infty,m}[1/p]/R[1/p])$ for some $m$.
For this reason, till we end of this introduction section, we fix some $m \geqslant 1$ coprime to $p$ and set $\calR \defeq R_{0,m}$, $\calR_{\infty} \defeq R_{\infty,m}$, $\calS \defeq S_{0,m}$ and $\calS_{\infty} \defeq S_{\infty,m}$.
We define the Galois groups
\begin{equation*}
	\begin{array}{ l l l }
		G_{\calR/R}^{\etale} \defeq \mathrm{Gal} (\calR_{\infty}^{\mathrm{ét}}[1/p]/R[1/p]), & & \Gamma_{\calR} \defeq \Gal(\calR_{\infty}[1/p]/\calR[1/p]),\\
		H_{\calR}^{\etale} \defeq \Gal(\calR_{\infty}^{\etale}[1/p]/\calR_{\infty}[1/p]), & & \Gamma_{\calS} \defeq \Gal(\calS_{\infty}[1/p]/\calS[1/p]).
	\end{array}
\end{equation*}
We remark that the groups $\Gamma_{\calR}$, $\Gamma_{\calS}$, $\Gamma_{R}$ and $\Gamma_{S}$ are canonically isomorphic.
Moreover we have natural isomorphisms of Galois groups
\begin{equation*}
	\begin{array}{ l }
		\Delta \defeq \mathrm{Gal}(\calR[1/p]/R[1/p]) \isomorphic \Gal(\calR_{\infty}[1/p]/R_{\infty}[1/p]), \\
		G_{\calR/R}^{\etale} /H_{\calR}^{\etale} \isomorphic \Gal(\calR_{\infty}[1/p]/R[1/p]) \isomorphic \Gamma_{\calR} \times \Delta.
	\end{array}
\end{equation*}

\subsubsection{\'Etale \texorpdfstring{$(\varphi, \Gamma_{\calR}\times \Delta)\textrm{-modules}$}{-}}

Recall that we set $\calR = R_{0,m}$ for some fixed $m$ coprime to $p$.
Set $A_{\calR}^+ = \calR \llbracket q-1 \rrbracket$ and $A_{\calR}= A_{\calR}^+[1/(q-1)]^{\wedge}$ as the $p\textrm{-adic}$ completion.
The rings $A_{\calR}^+$ and $A_{\calR}$ are equipped with a natural action of $\varphi$ such that $\varphi(q) = q^{p}$ and $\varphi(x_i) = x_i^p$ for $1 \leqslant i \leqslant d$.
We may now equip these rings with an action of $\Gamma_{\calR} \times \Delta$.
Note that $\Delta$ acts naturally on $\calR$, and we extend this action to $A_{\calR}^+$ and $A_{\calR}$ by setting $d(q) = q$ for each $d \in \Delta$.
The action of $\Gamma_{\calR}$ is defined by the formulas
\begin{itemize}
	\item $g(x_i^{1/m})=  q^{\psi_i(g)/m} x_i^{1/m}$ for $g \in \Gamma_{\calR}$, and where $\psi_i$ was defined in \eqref{eq:action_GammaS}.

	\item $g(q) = q^{\chi(g)}$ for $g \in \Gamma_{\calR}$, and where $\chi$ denote the $p\textrm{-adic}$ cyclotomic character. 
\end{itemize}
Let $\calR_{\infty}^{\flat}$ denote the tilt of $\calR_{\infty}$, and let $\varepsilon = (1, \zeta_p, \zeta_{p^2}, \ldots) \in \calR_{\infty}^{\flat}$.
We remark that in the main body of this paper, we identify $A_{\calR}^+$ and $A_{\calR}$ with some subrings of the period ring $W(\overline{\calR}_{\infty}^{\flat}[1/(\varepsilon-1)])$.

Combining the tilting correspondence of Scholze \cite{scholze-perfectoid} with Katz's theorem \cite[Proposition~4.1.1]{katz73} and the decompletion results of Kedlaya--Liu \cite{kedlaya-liu2}, we obtain the following result\footnote{This result might be known to experts (see the unpublished work of Kedlaya--Liu \cite{kedlaya-liu2}), but to the best of our knowledge, it is not documented in the literature, and we record the proof with all details to fill this gap.}:

\begin{thm}[{Theorem \ref{thm:padic_classify_decomp}}]\label{intro_thm:padic_classify_decomp}
	There exists a natural categorical equivalence between $\mathbb{Z}_p\textrm{-representations}$ of $G_{\calR/R}^{\etale}$ and \'etale $(\varphi, \GammaR \times \Delta)\textrm{-modules}$ over $\AR$:
	\begin{equation*}
		\DR \colon \textup{Rep}_{\mathbb{Z}_p}(G_{\calR/R}^{\etale}) \isomorphic (\varphi, \GammaR \times \Delta)\textup{-Mod}_{\AR}^{\mathrm{ét}}.
	\end{equation*}
	The preceding equivalence restricts to a natural equivalence of categories:
	\begin{equation*}
		\DR \colon \textup{Rep}_{\mathbb{Z}_p}^{\textup{free}}(G_{\calR/R}^{\etale}) \isomorphic (\varphi, \GammaR \times \Delta)\textup{-Mod}_{\AR}^{\etale,\textup{fproj}},
	\end{equation*}
	where the target category consists of finite projective \'etale $(\varphi, \GammaR \times \Delta)\textrm{-modules}$ over $\AR$.
\end{thm}

\subsubsection{Wach modules}

In this paper, we adopt the definition of Wach modules introduced in \cite[Definition 3.8]{abhinandan-relative-wach-ii}:

\begin{defi}\label{intro_defi:wach_mods}
	A \textit{Wach module} over $\AR^+$ (resp.\ $\AS^+$) is a finitely generated $\AR^+\textrm{-module}$ (resp.\ $\AS^+\textrm{-module}$) $N$ satisfying the following assumptions:
	\begin{enumerate}
		\item[(1)] The sequences $\{p, \mu\}$ and $\{\mu, p\}$ are regular on $N$.
	
		\item[(2)] The module $N$ is equipped with a semilinear action of $\GammaR$ (resp.\ $\GammaS$) such that the induced action of $\GammaR$ (resp.\ $\GammaS$) on $N/\mu N$ is trivial.
	
		\item[(3)] The module $N$ is equipped with an $\AR^+\textrm{-linear}$ and $\GammaR\textrm{-equivariant}$ (resp.\ $\AS^+\textrm{-linear}$ and $\GammaS\textrm{-equivariant}$) Frobenius-structure, i.e.\ an isomorphism $\varphi_N \colon (\varphi^*N)[1/[p]_q] \isomorphic N[1/[p]_q]$.
	\end{enumerate}
	
	The module $N$ is said to be \textit{effective} if we have $\varphi_N(\varphi^*N) \subset N$.
	Denote by $(\varphi, \GammaR)\textup{-Mod}_{\AR^+}^{[p]_q}$, the category of Wach modules over $\AR^+$ with morphisms between objects being $\AR^+\textrm{-linear}$ and $(\varphi, \GammaR)\textrm{-equivariant}$.
\end{defi}

The following definition takes into account the additional structure provided by an action of $\Delta$:  

\begin{defi}\label{intro_defi:wach_mods_with_delta}
	A \textit{Wach module with $\Delta\textrm{-action}$} over $\AR^+$ (resp.\ $\AS^+$) is a Wach module $N$ over $\AR^+$ (resp.\ $\AS^+$) equipped with an additional semilinear action of $\Delta$ commuting with the $(\varphi, \GammaR)\textrm{-action}$ (resp.\ $(\varphi, \GammaS)\textrm{-action}$) on $N$.
	Denote by $(\varphi, \GammaR \times \Delta)\textup{-Mod}_{\AR^+}^{[p]_q}$, the category of Wach modules with $\Delta\textrm{-action}$ over $\AR^+$ with morphisms between objects being $\AR^+\textrm{-linear}$ and $(\varphi, \GammaR \times \Delta)\textrm{-equivariant}$.
\end{defi}

To any Wach module with $\Delta\textrm{-action}$ over $\AR^+$, one may functorially associate a $\mathbb{Z}_p\textrm{-representation}$ of $G_R$.
More precisely, we have that the following natural functors are fully faithful:
\begin{equation*}
	T_{\calR} \colon (\varphi, \GammaR \times \Delta)\textup{-Mod}_{\AR^+}^{[p]_q} \longrightarrow (\varphi, \GammaR \times \Delta)\textup{-Mod}_{\AR}^{\textup{\'et}} \isomorphic \textup{Rep}_{\mathbb{Z}_p}(G_{\calR/R}^{\etale}),
\end{equation*}
where the left arrow is given as $N \mapsto \AR \otimes_{\AR^+} N$ and the right arrow is a quasi-inverse to the functor $D_{\calR}$. 
Set $A_{\inf}(\overline{R})\defeq W(\overline{R}^{\flat})$ and $\tilde{A}(\overline{R})\defeq W(\overline{R}[1/p]^{\flat})$\footnote{Here $\overline{R}$ denotes the integral closer of $R$ in $\overline{\textup{Fr}(R)}$; note that $\widehat{\overline{R}}$ is perfectoid (see Lemma \ref{lem:Rbarhat_perfectoid}).}.
Then, similar to \cite[Proposition 3.14]{abhinandan-relative-wach-ii}, we have the following comparison result.

\begin{prop}[{Proposition \ref{prop:wachmod_comp_relative}}]\label{intro_prop:wachmod_comp_relative}
	Let $N$ be a Wach module with $\Delta\textrm{-action}$ over $\AR^+$, and let $T \coloneq \TR(N)$ be the associated finite free $\mathbb{Z}_p\textrm{-representation}$ of $G_R$.
	Then, we have a natural $G_R\textrm{-equivariant}$ comparison isomorphism:
	\begin{equation}\label{intro_eq:wachmod_comp_relative_ainf}
		A_{\inf}(\overline{R})[1/\mu] \otimes_{\AR^+} N \isomorphic A_{\inf}(\overline{R})[1/\mu] \otimes_{\mathbb{Z}_p} T.
	\end{equation}
	Additionally, \eqref{intro_eq:wachmod_comp_relative_ainf} is compatible with the action of $(\varphi, G_R)$ after base change along $A_{\inf}(\overline{R})[1/\mu] \rightarrow \tilde{A}(\overline{R})$.
\end{prop}

Recall that in \cite[Section 4.1]{abhinandan-relative-wach-ii}, a Wach module over $\AS^+$ was constructed by ``glueing'' together an \'etale $(\varphi, \GammaS)\textrm{-module}$ over $\AS$ with a Wach module in the imperfect residue field case.
Analogously, we show that one may construct a Wach module with $\Delta\textrm{-action}$ over $\AR^+$ by ``glueing'' together an \'etale $(\varphi, \GammaR \times \Delta)\textrm{-module}$ over $\AR$ with a Wach module with $\Delta\textrm{-action}$ over $\AS^+$.
More precisely, we show the following.

\begin{thm}[{Theorem \ref{thm:wachmod_existence}}]\label{intro_thm:wachmod_existence}
	Let $\DR$ be an \'etale $(\varphi, \GammaR \times \Delta)\textrm{-module}$ over $\AR$ and set $\DS \coloneq \AS \otimes_{\AR} \DR$ as an \'etale $(\varphi, \GammaS \times \Delta)\textrm{-module}$ over $\AS$.
	Let $\NS \subset \DS$ be a Wach module with $\Delta\textrm{-action}$ over $\AS^+$.
	Then, $\NR \coloneq \DR \cap \NS \subset \DS$ is a Wach module with $\Delta\textrm{-action}$ over $\AR^+$, i.e.\ it satisfies all the axioms of Definition \ref{intro_defi:wach_mods_with_delta}.
\end{thm}

\subsubsection{Log-crystalline representations}

The discussion below applies to the rings $R$ and $S$.
Let $\omega_R^1$ denote the $p\textrm{-adic}$ completion of the module of logarithmic differentials of $R$ relative to $\mathbb{Z}$ (see \cite[Section 1.7]{kato-fontaine-illusie-i}), namely $\omega_R^1 = \oplus_{i=1}^d R \dlog x_i$, where $\dlog x_i = \tfrac{dx_i}{x_i}$.
Using the logarithmic structure on $R$, we define $p\textrm{-adic}$ period rings
\begin{equation*}
	\OBlogcrys(\overline{R}) \subset \OBlogdR(\overline{R})
\end{equation*}
satisfying the following properties.
Both these rings are $R[1/p]\textrm{-algebras}$ equipped with a natural continuous action of $G_R$.
In addition, $\OBlogdR(\overline{R})$ is equipped with a decreasing filtration and a $G_R\textrm{-equivariant}$ integrable logarithmic connection $\partial \colon \OBlogdR(\overline{R}) \rightarrow \OBlogdR(\overline{R})\otimes_{R[1/p]} \omega_R^1$ satisfying Griffiths' transversality.
We equip the ring $\OBlogcrys(\overline{R})$ with an induced filtration, integrable logarithmic connection, and additionally, a Frobenius endomorphism $\varphi$. 
\begin{prop}[{Proposition \ref{prop:GR_regularity}}]\label{intro_prop:GR_regularity}
	The rings $\OBlogdR(\overline{R})$ and  $\OBlogcrys(\overline{R})$ are $G_R\textrm{-regular}$ in the sense of \cite[Section 8.4]{brinon-relatif}.
\end{prop}

The preceding result allows us to apply Fontaine's formalism of $B\textrm{-admissible}$ representations for $B = \OBlogcrys(\overline{R})$ and $\OBlogdR(\overline{R})$.
For any $p\textrm{-adic}$ Galois representation $V$ of $G_R$, we define  $R[1/p]\textrm{-modules}$ 
\begin{equation*}
	\begin{aligned}
		\ODlogdR (V) &\coloneq \big(\OBlogdR (\overline{R}) \otimes_{\mathbb{Q}_p} V\big)^{G_R},\\
		\ODlogcrys(V) &\coloneq \big(\OBlogcrys (\overline{R}) \otimes_{\mathbb{Q}_p} V\big)^{G_R}.
	\end{aligned}
\end{equation*}
Note that $\ODlogdR(V)$ is equipped with a filtration and an integrable logarithmic connection, and $\ODlogcrys(V)$ is equipped with a Frobenius operator $\varphi$, a filtration and an integrable logarithmic connection.

\begin{defi}\label{intro_defi:logcrys_rep}
	A $p\textrm{-adic}$ representation $V$ of $G_R$ is said to be \textit{log-de Rham} (resp.\ \textit{log-crystalline}) if the natural homomorphism
	\begin{equation*}
		\begin{aligned}
			\alpha_{\textrm{log-dR},R}(V) &\colon \OBlogdR (\overline{R}) \otimes_{R[1/p]} \ODlogdR (V) \longrightarrow \OBlogdR (\overline{R}) \otimes_{\mathbb{Q}_p} V\\
			(\textrm{resp. } \alpha_{\textrm{log-crys},R}(V) &\colon \OBlogcrys (\overline{R}) \otimes_{R[1/p]} \ODlogcrys (V) \longrightarrow \OBlogcrys(\overline{R}) \otimes_{\mathbb{Q}_p} V),
		\end{aligned}
	\end{equation*}
	is an isomorphism.
\end{defi}

We remark that for any log-de Rham representation $V$, the $R[1/p]$-module $\ODlogdR(V)$ is finite projective.
Additionally, if $V$ is log-crystalline then it is log-de Rham, we have $\ODlogcrys(V) \isomorphic \ODlogdR(V)$, the natural Frobenius map $\varphi^*\ODlogcrys(V) \rightarrow \ODlogcrys(V)$ is bijective, and the isomorphism $\alpha_{\textrm{log-cris},R}(V)$ is compatible with the respective Frobenii, filtrations and $G_R\textrm{-actions}$.

\begin{prop}[{Corollary \ref{cor:crys_base_change}}]\label{intro_prop:crys_base_change}
	Let $V$ be a log-crystalline (resp.\ de Rham) representation of $G_R$.
	Then, $V$ is a crystalline (resp.\ de Rham) representation of $G_S$, and we have the following natural isomorphism of $S[1/p]\textrm{-modules}$
	\begin{equation*}
		\begin{aligned}
			S[1/p] \otimes_{R[1/p]} \ODlogcrys(V) &\isomorphic \ODcrysS(V) \\
			(\textrm{resp. } S[1/p] \otimes_{R[1/p]} \ODlogdR(V) &\isomorphic \ODdRS(V)),
		\end{aligned}
	\end{equation*}
	compatible with the aforementioned supplementary structures.
\end{prop}

In the main text, we construct and study an additional subring $\OBlog(\overline{R})$ of $\OBlogcrys (\overline{R})$, and define the notion of a strictly log-crystalline representation using this ring.
We expect to investigate these representations in more detail in another work. 

\subsubsection{Log-crystalline representations and Wach modules}

In \cite[Theorem 1.5]{abhinandan-relative-wach-ii}, Wach modules over $A_S^+$ were shown to be naturally equaivalent to $\mathbb{Z}_p\textrm{-lattices}$ inside crystalline representations of $G_S$.
Over $R$, we obtain the following result:

\begin{thm}[{Theorem \ref{thm:Delta_descent_wach_mods}}]\label{intro_thm:Delta_descent_wach_mods}
	Let $T$ be a finite free $\mathbb{Z}_p\textrm{-representation}$ of $G_R$ such that $T[1/p]$ is log-crystalline.
	Then, there exists a Wach module $N_R(T)$ over $A_R^+$, functorially associated to $T$.
\end{thm}

To construct the Wach module $N_R(T)$ over $A_R^+$, we first use Theorem \ref{intro_thm:GR_rep_GR0m} to find some positive integer $m$ coprime to $p$ such that the action of $G_R$ on $T$ factors through $G_{\calR/R}^{\etale}$ for $\calR = R_{0,m}$.
Then, Theorem \ref{intro_thm:padic_classify_decomp} supplies us with an \'etale $(\varphi, \GammaR \times \Delta)\textrm{-module}$ $\DR(T)$ over $\AR$.
Additionally, from Proposition \ref{intro_prop:crys_base_change} we know that $T[1/p]$ is crystalline as a representation of $G_S$, so \cite[Theorem 1.5]{abhinandan-relative-wach-ii} implies the existence of the Wach module $N_S(T)$ over $A_S^+$ thus yielding $\NS(T) \defeq \AS^+ \otimes_{A_S^+} N_S(T)$ as the Wach module with $\Delta\textrm{-action}$ over $\AS^+$, and we show that it is compatible with $\DR(T)$ over $\AR$ (see Section \ref{subsubsec:OA_comp_iso}).
This allows us to ``glue'' these objects and obtain a Wach module with $\Delta\textrm{-action}$ $\NR(T)$ over $\AR^+$ by Theorem \ref{intro_thm:wachmod_existence}.
Finally, using properties of Wach modules we obtain comparison isomorphisms (see Proposition \ref{prop:OA_comp_iso} and Corollaries \ref{cor:NR_modmu} and \ref{cor:OA_comp_iso_Delta}), which allow us to show that $\NR(T)$ descends with respect to the ring extension $ A_R^+\rightarrow \AR^+$, i.e.\ for $N_R(T) \defeq \NR(T)^{\Delta}$ the natural map $\AR^+ \otimes_{A_R^+} N_R(T) \rightarrow \NR(T)$ is an isomorphism.
Hence, $N_R(T)$ is the desired Wach module over $A_R^+$ associated to $T$.

Conversely, by employing some techniques similar to \cite[Section 3.6]{abhinandan-relative-wach-ii}, we show that for a Wach module $N$ over $A_R^+$ and the associated representation $T_R(N)$ of $G_R$, we have that $T_R(N)[1/p]$ is log-crystalline (see Proposition \ref{prop:wach_is_logcrys}).
Hence, we obtain the following:
\begin{thm}[{Theorem \ref{thm:logcrys_wach_equiv}}]\label{intro_thm:logcrys_wach_equiv}
	The correspondence $T \mapsto N_R(T)$ induces a natural equivalence of categories:
	\begin{equation*}
		\textup{Rep}_{\mathbb{Z}_p}^{\textup{log-cris}}(G_R) \isomorphic (\varphi, \Gamma_R)\textup{-Mod}_{A_R^+}^{[p]_q}.
	\end{equation*}
\end{thm}

\subsection{Relation to previous works}

As mentioned previously, the categorical equivalence in Theorem \ref{intro_thm:padic_classify_decomp} generalises the main result of \cite{andreatta-phigamma}, and the categorical equivalence in Theorem \ref{intro_thm:logcrys_wach_equiv} generalises the main result of \cite{abhinandan-relative-wach-ii}.
Our notion of log-de Rham representations is similar to \cite[Definition 3.20]{andreatta-iovita-semistable} (also see \cite[Section 3]{dllz}), however, the definition of log-crystalline representations is different from the notion of semistable representations in \cite[Section 3.6]{andreatta-iovita-semistable}.

In a somewhat related but different direction, \cite[Definition 3.26]{dlms2} defines semistable $\mathbb{Z}_p\textrm{-local}$ systems over the smooth generic fibre of a semistable $p\textrm{-adic}$ formal scheme using (log-) prismatic techniques developed in \cite{koshikawa-logprism, koshikawa-yao}.
The definition of \cite{dlms2} is different from ours because they work in a different setting, namely, they consider formal schemes equipped with a ``vertical'' log-structure, i.e.\ coming from the special fibre, and whose (adic) generic fibre is a smooth adic space.
Furthermore, following the definition of crystalline local systems from \cite{guo-reinecke, dlms1} and generalising \cite[Definition 3.26]{dlms2} (also see \cite[Corollary 3.46]{dlms2}), note that \cite[Definition 4.18]{inoue} provides a definition of semistable local systems.
By adapting the definition of loc.\ cit.\ to the current setting, one should be able to recover log-crystalline representations of Definition \ref{intro_defi:logcrys_rep}.
However, note that the relationship between semistable representations of \cite{inoue} and the natural objects in the prismatic world, i.e.\ analytic prismatic $F\textrm{-crystals}$, has not been explored in that paper (see \cite[Remark 1.6]{inoue}).

Our current paper \textit{exactly} addresses the question mentioned in the preceding paragraph, for the affinoid algebra $R[1/p]$ equipped with a ``horizontal'' log-structure, from the perspective of $(\varphi, \Gamma)\textrm{-modules}$.
We again emphasise that the categorical equivalence in Theorem \ref{intro_thm:logcrys_wach_equiv} is a \textit{completely novel} result in this setting, in particular, the Breuil--Kisin and prismatic variations are not yet known.
In view of the relationship between Wach modules and analytic prismatic $F\textrm{-crystals}$ in the trivial log-structure case (see \cite{abhinandan-prismatic-wach}), it is natural to ask if the results of loc.\ cit.\ extend to our current setting.
This line of thought will be explored in a subsequent work.

\vspace{2mm}
\noindent \textbf{Outline of the paper.} This paper is divided into six sections and an appendix section.
Section \ref{sec:prelims} collects all the fundamental definitions and results we need to establish, in order to carry out the constructions of this paper.
In Section \ref{sec:phiGamma_modules}, we construct \'etale $(\varphi, \Gamma)\textrm{-modules}$ associated to $p\textrm{-adic}$ representations of the logarithmic fundamental group.
Therein, we first define perfect period rings and then apply decompletion techniques of Kedlaya--Liu \cite{kedlaya-liu2} to construct imperfect period rings (see Section \ref{subsec:period_rings}).
Subsequently, we construct the promised categorical equivalence of Theorem \ref{intro_thm:padic_classify_decomp}, and obtain the consequence for $p\textrm{-adic}$ representations mentioned in Theorem \ref{intro_thm:GR_rep_GR0m} (see Section \ref{subsec:padic_reps_phiGamma}).

In Section \ref{sec:wachmods}, we define and study some properties of Wach modules in order to state and prove Proposition \ref{intro_prop:wachmod_comp_relative}, and carry out the glueing construction of Theorem \ref{intro_thm:wachmod_existence}.
Section \ref{subsec:crystalline_rings} defines several period rings of de Rham/crystalline nature, in order to define and study the properties of log-de Rham and log-crystalline reprentations.
This helps us in establishing Propositions \ref{intro_prop:GR_regularity} and \ref{intro_prop:crys_base_change}.
Finally, in Section \ref{sec:logcris_wachmod} we prove our  main result Theorem \ref{intro_thm:Delta_descent_wach_mods}, i.e.\ construct a Wach module associated to a log-crystalline representation, and establish the categorical equivalence of Theorem \ref{intro_thm:logcrys_wach_equiv}.

In the Appendix \ref{sec:appendix}, Section \ref{subsec:commutative_algebra} collects some facts on commutative algebra and Galois descent, and Section \ref{subsec:Kedlaya_Liu_ii} recalls the decompletion results of Kedlaya--Liu, applies it to our setting and thus obtains some consequences for Galois groups.

\vspace{2mm}
\noindent \textbf{Acknowledgements.} Work on this collaboration started when the two authors visited CIRM, Luminy in October 2024 for the annual meeting of the Agence National de Recherche (ANR) project GALF: ANR-18-CE40-0029.
We would like to thank CIRM for their wonderful hospitality and great working environment, and ANR for the support.
We would also like to thank Qing Liu and Lorenzo Ramero for relevant discussions, and many thanks to Ofer Gabber for communicating Example \ref{eg:gabber_example} which clarified some technical issues in a previous version of the paper.

The work of the first named author was supported by a Simons Collaboration grant on Perfection in Algebra, Geometry and Topology.
The work of the second named author (D.B.) was supported by the ANR grant ANR-25-CE40-4664 in the framework of the project ``Practical $p\textrm{-adic}$ Langlands''.

\section{Preliminaries}\label{sec:prelims}

\subsection{Affinoid algebras and relative cyclotomic towers}\label{subsec:setup_notation}

In this section, we will introduce our base rings and several extensions of those rings, relative to which we shall study various period rings in later sections.

\subsubsection{Fundamental groups}\label{subsubsec:fundamental_groups}

Let $\kappa$ be a perfect field of characteristic $p$.
Set $O_F \defeq W(\kappa)$ to be the ring of $p\textrm{-typical}$ Witt vectors with coefficients in $\kappa$.
Then, $O_F$ is a complete discrete valuation ring with uniformiser $p$ and we set $F \defeq O_F[1/p]$ to be the field of fractions of $O_F$.

\begin{nota}
	Let $\Lambda$ be an $I\textrm{-adically}$ complete algebra for a finitely generated ideal $I \subset \Lambda$, and let $x_1,\ldots,x_d$ denote some indeterminate.
	For any $\mathbf{k} = (k_1,\ldots,k_d)\in \mathbb{N}^d$, set $x^k \defeq x_1^{k_1}\ldots x_d^{k_d}.$
	Then, we let $\Lambda\langle x_1,\ldots, x_d\rangle$ denote the Tate algebra over $\Lambda$, namely
	\begin{equation*}
		\Lambda\langle x_1,\ldots,x_d\rangle \defeq \big\{\textstyle\sum_{\mathbf{k} \in \mathbb{N}} a_{\mathbf{k}}x^{\mathbf{k}}, \textrm{ where } a_{\mathbf{k}} \in \Lambda \textrm{ and } I\textrm{-adically } a_{k} \rightarrow 0 \textrm{ as } |\mathbf{k}| \rightarrow +\infty\big\}.
	\end{equation*}
\end{nota}

Let us fix integers $d \geqslant a \geqslant 0$, and we define the following $p\textrm{-adically}$ complete $O_F\textrm{-algebras}$:
\begin{equation*}
	\begin{aligned}
		R_{\square} &\defeq O_F\langle x_1^{\pm 1}, \ldots, x_a^{\pm 1}, x_{a+1}, \ldots, x_d\rangle,\\
		S_{\square} &\defeq O_F\langle x_1^{\pm 1}, \ldots, x_a^{\pm 1}, x_{a+1}^{\pm 1}, \ldots, x_d^{\pm 1}\rangle.
	\end{aligned}
\end{equation*}
In particular, $S_{\square}$ coincides with the $p\textrm{-adic}$ completion of $R_{\square}[1/(x_{a+1} \cdots x_d)]$.

Let us set $R$ to be a \textit{$p\textrm{-adically}$ completed \'etale} algebra over $R_{\square}$ with connected special fibre, and set $S \defeq S_{\square} \widehat{\otimes}_{R_{\square}} R$.
In particular, $S$ is the $p\textrm{-adic}$ completion of the $R\textrm{-algebra}$ $R[1/(x_{a+1} \cdots x_d)]$.
By Proposition \ref{prop:Rnm_Snm_normal} below, the rings $R$ and $S$ are noetherian regular domains, in particular, these are normal domains.
Furthermore, we shall say that $R$ is $p\textrm{-adically}$ completed \textit{standard \'etale} over $R_{\square}$ if it admits a presentation
\begin{equation*}
	R = R_{\square}\{z_1,\ldots, z_b\}/(Q_1,\ldots, Q_b),
\end{equation*}
for some $Q_1, \ldots, Q_b$ in $R_{\square}\langle z_1,\ldots, z_b\rangle$ such that $\det \big(\frac{\partial Q_i}{\partial z_j}\big)$ is a unit in $R$.

Let $\overline{\mathrm{Fr}(S)}$ denote a fixed algebraic closure of the field of fractions of $S$, and write $\overline{F}$ and $\overline{\mathrm{Fr}(R)}$ for the respective algebraic closures of $F$ and ${\mathrm{Fr}(R)}$ in $\overline{\mathrm{Fr}(S)}$. 

\begin{nota}\phantomsection\label{nota:etale_condition}
	For a morphism of $S\textrm{-algebras}$ $A \rightarrow B$, we shall say that condition {\bf{(\'et)}} is satisfied if the following holds:
	\begin{itemize}[leftmargin=2cm]
		\item[\textbf{(\'et)}] $A$ and $B$ are normal domains, $B$ is finite and integral over $A$, and $B[1/p]$ is étale over $A[1/p]$.
	\end{itemize}
\end{nota}

Let $\overline{S}$ denote the union of $S\textrm{-subalgebras}$ $S' \subset \overline{\mathrm{Fr}(S)}$ satisfying \hyperref[nota:etale_condition]{\textbf{(\'et)}}.
Denote by $G_S \defeq \mathrm{Gal}(\overline{S}[1/p]/S[1/p])$ the \'etale fundamental group of $S[1/p]$, and recall that the Galois theory of schemes says that (see \cite[Chapters I and V]{sga1}):

\begin{prop}\label{prop: galois theory for S}
	The following claims are true: 
	\begin{enumerate}
		\item[\textup{(1)}] The group $G_S$ is naturally isomorphic to the Galois group of the field extension $\mathrm{Fr}(\overline{S})/\mathrm{Fr}(S)$.
		
		\item[\textup{(2)}] The functor $S' \mapsto S'[1/p]$ establishes a natural equivalence between the category of $S\textrm{-subalgebras}$ of $\overline{S}$ satisfying condition \hyperref[nota:etale_condition]{\textup{\textbf{(\'et)}}} and the category of finite \'etale extensions of $S[1/p]$ contained in $\overline{S}[1/p]$.
			A quasi-inverse functor is given via normalisation of $S$.

		\item[\textup{(3)}] The functor sending a finite \'etale extension $S'[1/p]$ of $S[1/p]$ contained in $\overline{S}[1/p]$ to the open subgroup $H = \Gal(\overline{S}[1/p]/S'[1/p]) \subset G_S$ establishes an equivalence between finite \'etale extensions of $S[1/p]$ contained in $\overline{S}[1/p]$ and open subgroups of $G_S$.
			A quasi-inverse functor is given by sending an open subgroup $H \subset G_S$ to the finite \'etale extension $S'[1/p] = (\overline{S}[1/p])^H$.
	\end{enumerate}
\end{prop}

We recall the logarithmic variant of the group $G_S$ for the ring $R$, i.e.\ the logarithmic fundamental group of $R[1/p]$ from \cite{kato-fontaine-illusie-ii, illusie-logetale, grothendieck-murre}.

\begin{nota}\phantomsection\label{nota:logetale_condition}
	For a morphism of $R\textrm{-algebras}$ $A \rightarrow B$, we shall say that condition \textbf{(log-\'et)} is satisfied if the following holds:
	\begin{itemize}[leftmargin=2cm]
		\item[\textbf{(log-\'et)}] $A$ and $B$ are normal domains, $B$ is finite and integral over $A$, and $B[1/p]$ is flat over $A[1/p]$ and unramified outside the divisor $x_{a+1} \cdots x_d = 0$.
	\end{itemize}
\end{nota}

Let $\overline{R}$ denote the union of $R\textrm{-subalgebras}$ $R' \subset \overline{\mathrm{Fr}(R)}$ satisfying \hyperref[nota:logetale_condition]{\textbf{(log-\'et)}}.
For an $R\textrm{-subalgebra}$ $R' \subset \overline{R}$ satisfying \hyperref[nota:logetale_condition]{\textbf{(log-\'et)}}, we shall say that the natural homomorphism $R[1/p] \rightarrow R'[1/p]$ is \textit{log-\'etale Galois} if the induced homomorphism $R[1/p, 1/(x_{a+1} \cdots x_d)] \rightarrow R'[1/p, 1/(x_{a+1} \cdots x_d)]$ is \'etale and Galois in the usual sense.
Let us set
\begin{equation*}
	G_R \defeq \mathrm{Gal}(\overline{R}[1/p])/R[1/p]),
\end{equation*}
to be the logarithmic fundamental group of $R[1/p]$ from \cite[Section 10.3]{kato-fontaine-illusie-ii}, which coincides with the tame fundamental group of $R[1/p, 1/(x_{a+1} \cdots x_d)]$ from \cite{grothendieck-murre}.
In this case, we have a logarithmic variation of Proposition \ref{prop: galois theory for S} (see \cite[Sections 1 and 2]{grothendieck-murre} for details).

\begin{prop}\label{prop: galois theory for R}
	The following claims are true:
	\begin{enumerate}
		\item[\textup{(1)}] The group $G_R$ is naturally isomorphic to the Galois group of the field extension $\mathrm{Fr}(\overline{R})/\mathrm{Fr}(R)$.
	   
		\item[\textup{(2)}] The functor $R' \mapsto R'[1/p]$ establishes a natural equivalence between the category of $R\textrm{-subalgebras}$ of $\overline{R}$ satisfying \hyperref[nota:logetale_condition]{\textup{\textbf{(log-\'et)}}} and the category of finite flat extensions of $R[1/p]$ unramified outside the divisor $x_{a+1} \cdots x_d = 0$.
			A quasi-inverse functor is given via normalisation of $R$.

		\item[\textup{(3)}] The functor sending a finite log-\'etale extension $R'[1/p]$ of $R[1/p]$ contained in $\overline{R}[1/p]$ to the open subgroup $H = \Gal(\overline{R}[1/p]/R'[1/p]) \subset G_R$ establishes an equivalence between finite log-\'etale extensions of $R[1/p]$ contained in $\overline{R}[1/p]$ and open subgroups of $G_R$.
			A quasi-inverse functor is given by sending an open subgroup $H \subset G_R$ to the finite log-\'etale extension $R'[1/p] = (\overline{R}[1/p])^H$.
	\end{enumerate}
\end{prop}
\begin{proof}
	From the general formalism of Grothendieck's Galois theory, we have a canonical surjective homomorphism from the Galois group of the maximal separable extension $\overline{\Fr(R)}/\Fr(R)$ onto $G_R$; the kernel of this homomorphism corresponds to the compositum of all finite extensions of $\Fr(R)$ which are unramified outside the divisor $x_{a+1} \cdots x_d = 0$ (see \cite[Paragraph 2.4.6]{grothendieck-murre}), which proves (1).
	
	The claim in (2) follows directly from the definitions and basic properties of normalisation.
	
	The final part (3) easily follows from the \'etale case.
	Indeed, let $\mathcal{S}$ denote the multiplicatively closed set generated by the set of elements $\{p, x_{a+1}, \ldots, x_d\}$ of $R$.
	Denote by $A = {\mathcal{S}}^{-1}R$ and $A' = {\mathcal{S}}^{-1}R'$ the respective localisation of $R$ and $R'$.
	Then, it is clear that $A'$ is the integral closure of $A$ in $\mathrm{Fr}(R')$.
	Since $A'$ is \'etale over $A$, from \cite[Expos\'e V, Corollaire 3.2]{sga1} we get that ${A'}^H$ is \'etale over $A$ and $A'$ is \'etale over ${A'}^H$.
	The claim now follows by observing that $\mathcal{S}^{-1}(R') = {A'}^H$.
\end{proof}

The injective homomorphism of fields $\overline{\mathrm{Fr}(R)} \hookrightarrow \overline{\mathrm{Fr}(S)}$ induces a natural injective homomorphism of rings $\overline{R} \hookrightarrow \overline{S}$, which further induces a continuous homomorphism of Galois groups $G_S \rightarrow G_R$.
Consequently, the homomorphism $\overline{R} \hookrightarrow \overline{S}$ is $G_S\textrm{-equivariant}$.
Let us note that the homomorphism of groups $G_S \rightarrow G_R$ is \textit{not} surjective.
To illustrate this, we provide the following example, and we would like to thank O.\ Gabber for communicating it to us.

\begin{exam}\label{eg:gabber_example}
	Let $R = \mathbb{Z}_p\langle x \rangle$ and $S = \mathbb{Z}_p\langle x^{\pm 1} \rangle$.
	Set $f(Y) = (Y^{p^2}-x^{p^2}-x)(Y-x)-p$ in $R[Y]$, and define $R' \defeq R[Y]/(f(Y))$.
	By reducing $f(Y)$ modulo $p$, it is easy to see that $f(Y)$ is an irreducible element of $R[Y]$.
	Since $R$ is a unique factorisation domain, therefore, so is $R[Y]$, and it follows that $R'$ is a noetherian domain of dimension two.

	We will first show that $R'$ is normal.
	By Serre's criterion for normality (see \cite[\href{https://stacks.math.columbia.edu/tag/031S}{Tag 031S}]{stacks-project}), recall that it is enough to show that $R'$ satisfies conditions $(R_1)$ and $(S_2)$.
	Let us first note that $R$ is Cohen--Macaulay, and therefore $R'$ is Cohen--Macaulay thanks to \cite[Proposition 18.13]{eisenbud}, thus implying that $R'$ satisfies property $(S_2)$.
	It remains to show that $R'$ satisfies $(R_1)$, i.e.\ for any height one prime ideal $\mathfrak{p} \subset R'$, we have that $R'_{\mathfrak{p}}$ is regular.

	Let $\mathfrak{p} \subset R'$ be a prime ideal such that $p$ is in $\mathfrak{p}$.
	As we have that $(\overline{Y}-\overline{x})(\overline{Y}^{p^2}-\overline{x}^{p^2}-\overline{x}) = p$ in $R'$ (we write $\overline{Y}$ and $\overline{x}$ for the respective images of $Y$ and $x$ inside $R'$), then we either have that $\overline{Y}-\overline{x}$ is in $\mathfrak{p}$ or $\overline{Y}^{p^2}-\overline{x}^{p^2}-\overline{x}$ is in $\mathfrak{p}$.
	In the first case, observe that $R'/(\overline{Y}-\overline{x}) = \mathbb{F}_p[x][Y]/(Y-x) \isomorphic \mathbb{F}_p[x]$, and so it follows that $\mathfrak{p} = (\overline{Y}-\overline{x})$.
	In particular, $\mathfrak{p}$ is principal, and we get that $R'_{\mathfrak{p}}$ is a local noetherian domain with a principal maximal ideal generated by a non-nilpotent element, hence a PID.
	This in turn implies that $R'_{\mathfrak{p}}$ is a discrete valuation ring, and therefore, regular.
	Similarly, in the second case, observe that $R'/(\overline{Y}^{p^2}-\overline{x}^{p^2}-\overline{x}) = \mathbb{F}_p[x,Y]/(Y^{p^2}-x^{p^2}-x)$, and $Y^{p^2}-x^{p^2}-x$ is an irreducible polynomial in $\mathbb{F}_p[x,Y]$, which implies that $R'/(\overline{Y}^{p^2}-\overline{x}^{p^2}-\overline{x})$ is a domain and $\mathfrak{p} = (\overline{Y}^{p^2}-\overline{x}^{p^2}-\overline{x})$.
	Then, similar to the first case, we see that $\mathfrak{p}$ is principal and $R'_{\mathfrak{p}}$ is a discrete valuation ring, in particular, it is regular.

	Now, let us show that $R'[1/p]$ is also regular.
	Indeed, observe that in $R[Y]$ we may write
	\begin{equation*}
		\begin{aligned}
			f(Y) - \tfrac{\partial f}{\partial Y}(Y-x) = -p - p^2Y^{p^2-1}(Y-x)^2 = -p(1+pY^{p^2-1}(Y-x)^2),
		\end{aligned}
	\end{equation*}
	which implies that $\tfrac{\partial f}{\partial Y}$ is a unit in $R'[1/p]$.
	So, $R'[1/p]$ is \'etale over $R[1/p]$, and therefore, it is regular.
	This implies that for any height one prime ideal $\mathfrak{p} \subset R'$ with $p \not \in \mathfrak{p}$, we have that $R'_{\mathfrak{p}}$ is regular.

	Consequently, we obtain that $R'$ satisfies $(R_1)$, and hence by Serre's criterion for normality, we conclude that $R'$ is normal.

	Next, let us note that in $S[Y]$ we have that $\tfrac{\partial f}{\partial Y}|_{Y=x} = -x$ and $f(x) = -p$.
	Therefore, by Hensel's Lemma it follows that $x \textrm{ mod } p$ may be lifted to a root of $f(Y)$ in $S$.
	Now, consider the natural map $R' \rightarrow S$.
	Let $a$ in $S$ denote the lift of $x \textrm{ mod } p$.
	Then, we see that $a-x$ is in $pS$, and it is easy to check that $\overline{Y}-\overline{x}$ cannot be in $pR'$.
	Hence, it follows that the natural map $R'/pR' \rightarrow S/pS$ is \textit{not} injective.
	The above also shows that the natural composition $G_S \rightarrow G_R \twoheadrightarrow \Gal(R'[1/p]/R[1/p])$ sends all $g$ in $G_S$ to the identity map.
	Thus, we conclude that the natural homomorphism $G_S \rightarrow G_R$ is \textit{not} surjective.
\end{exam}

We end this section by fixing some more notation.
Let $\Omega_R$ (resp.\ $\Omega_S$) denote the $p\textrm{-adic}$ completion of the module of K\"ahler differentials of $R$ (resp.\ $S$) relative to $\mathbb{Z}$.
Additionally, let $\omega_R$ denote the $p\textrm{-adic}$ completion of the module of logarithmic differentials of $R$ relative to $\mathbb{Z}$ (see \cite[Section 1.7]{kato-fontaine-illusie-i}).
Then, we have that $\Omega_R = \oplus_{i=1}^d R \hspace{1mm} dx_i$, $\omega_R = \oplus_{i=1}^d R \dlog x_i$ and $\Omega_S = \oplus_{i=1}^d S \dlog x_i$, where $\dlog x_i = \tfrac{dx_i}{x_i}$.

Let $\varphi$ denote an endomorphism of $R_{\square}$ which extends the natural Frobenius on $O_F$ by setting $\varphi(x_i) = x_i^p$, for $1 \leqslant i \leqslant d$, and is thus a lift of the absolute Frobenius modulo $p$.
By $p\textrm{-complete}$ \'etaleness of the morphism $R_{\square} \rightarrow R$, the lift of Frobenius on $R_{\square}$ naturally extends to a unique endomorphism $\varphi \colon R \rightarrow R$ lifting the absolute Frobenius modulo $p$.
By definition, it is clear that the endomorphism $\varphi$ of $R$ is faithfully flat and finite of degree $p^d$.
Furthermore, it is easy to see from the definition that the lift of Frobenius $\varphi$ on $R$ naturally extends to a lift of Frobenius $\varphi \colon S \rightarrow S$, which is again faithfully flat and finite of degree $p^d$.

\subsubsection{Relative cyclotomic towers and Galois groups}\label{subsubsec:galois_groups}

Let us fix a compatible system of primitive $p^n\textrm{-th}$ roots of unity $\{\zeta_{p^n}\}_{n\geqslant 0}$, i.e.\ $\zeta_{p^{n+1}}^p = \zeta_{p^n}$.
For any integer $n \geqslant 0$, set  
\begin{equation*}
   R_{n} \defeq R\big[\zeta_{p^n}, x_1^{1/p^n}, \ldots, x_d^{1/p^n}\big],\qquad 
   S_{n} \defeq S\big[\zeta_{p^n}, x_1^{1/p^n}, \ldots, x_d^{1/p^n}\big] = R_n \otimes_R S,
\end{equation*}
and for any $m \geqslant 1$ coprime to $p$, set  
\begin{equation*}
	R_{n,m} \defeq R_n\big[\zeta_{m}, x_1^{1/mp^n}, \ldots, x_d^{1/mp^n}\big],\qquad 
   S_{n,m} \defeq S_n\big[\zeta_{m}, x_1^{1/mp^n}, \ldots, x_d^{1/mp^n}\big] = R_{n,m} \otimes_R S.
\end{equation*}
Let $\pi_n = \zeta_{p^n}-1$ and set $\mathfrak{p}_{n,m} \defeq \pi_n R_{n,m}$ and $\mathfrak{P}_{n,m} \defeq \pi_n S_{n,m}$.
The following proposition is well known (compare with \cite[Proposition~2.3]{andreatta-phigamma}), but we include it for the convenience of the reader.

\begin{prop}\label{prop:Rnm_Snm_normal}
	The following properties hold:
	\begin{enumerate}
		\item[\textup{(1)}] The respective ideals $\mathfrak{p}_{n,m} \subset R_{n,m}$ and $\mathfrak{P}_{n,m} \subset S_{n,m}$ are the unique prime ideals lying over $pR$.
		
		\item[\textup{(2)}] The rings $R_{n,m}/\mathfrak{p}_{n,m}$ and $S_{n,m}/\mathfrak{P}_{n,m}$ are noetherian regular domains.
			In particular, they are normal.

		\item[\textup{(3)}] The rings $R_{n,m}$ and $S_{n,m}$ are noetherian regular domains.
			In particular, they are normal. 

		\item[\textup{(4)}] The natural homomorphism $R_{n,m}/p^rR_{n,m} \rightarrow S_{n,m}/p^rS_{n,m}$ is injective.  
	\end{enumerate}
\end{prop}
\begin{proof}
	We shall only prove the claims for $R_{n,m}$ as the proof for $S_{n,m}$ follows by a similar argument. 

	For (1), note that we may write $p = \pi_n^{p^{n-1}(p-1)}u$, for some unit $u$ in $O_F[\pi_n]$ and $n \geqslant 1$.
	So, $p$ is contained in $\mathfrak{p}_{n,m}$ and it is easy to observe that,
	\begin{equation*}
		R_{n,m}/\mathfrak{p}_{n,m} \isomorphic (R/pR) [\zeta_{m}, x_1^{\pm 1/mp^n}, \ldots, x_a^{\pm 1/mp^n}, x_{a+1}^{1/mp^n}, \ldots, x_d^{1/mp^n}],
	\end{equation*}
	where the right hand term is a domain because $R/pR$ is a domain.
	In particular, $\mathfrak{p}_{n,m}$ is a prime ideal of $R_{n,m}$ lying over $pR$ and its uniqueness follows from the expression of $p$ in terms of $\pi_n$ from above. 

	For (2), note that we have,
	\begin{equation*}
		\begin{aligned}
			R_{\square,n,m}/\pi_n R_{\square, n,m} &\isomorphic \kappa[\zeta_{m}, x_1^{\pm 1/mp^n}, \ldots, x_a^{\pm 1/mp^n}, x_{a+1}^{1/mp^n}, \ldots, x_d^{1/mp^n}]\\
				&\xleftarrow{\hspace{1mm}\sim\hspace{1mm}} \kappa[\zeta_{m}, y_1^{\pm 1}, \ldots, y_a^{\pm 1}, y_{a+1}, \ldots, y_d],
		\end{aligned}
	\end{equation*}
	where the first isomorphism is clear, the second isomorphism is $\kappa[\zeta_m]\textrm{-linear}$ and induced by sending $y_i \mapsto x_i^{1/mp^n}$.
	As the $\kappa[\zeta_{m}]\textrm{-algebra}$ $\kappa[\zeta_{m}, y_1^{\pm 1}, \ldots, y_a^{\pm 1}, y_{a+1}, \ldots, y_d]$ is smooth, therefore, from the preceding diagram it follows that $R_{\square,n,m}/\pi_n R_{\square, n,m}$ is smooth over $\kappa[\zeta_{m}]$.
	Moreover, recall that $R$ is the $p\textrm{-adic}$ completion of an \'etale $R_{\square}\textrm{-algebra}$, in particular, $R/pR$ is \'etale over $R_{\square}/pR_{\square}$, and so we see that
	\begin{equation*}
		R_{n,m}/\mathfrak{p}_{n,m} = R/pR \otimes_{R_{\square}/pR_{\square}} (R_{\square,n,m}/\pi_n R_{\square,n,m}),
	\end{equation*}
	is \'etale over $R_{\square,n,m}/\pi_n R_{\square, n,m}$.
	Then, from \cite[Chapitre 0, Proposition (19.3.5)]{ega4_1}, we get that the localisation $(R_{n,m}/\mathfrak{p}_{n,m})_{\mathfrak{q}}$ at every maximal ideal $\mathfrak{q} \subset R_{n,m}/\mathfrak{p}_{n,m}$ is formally smooth over $\kappa[\zeta_{m}]$, therefore, geometrically regular by \cite[Chapitre 0, Proposition (22.5.8)]{ega4_1}, and hence $R_{n,m}/\mathfrak{p}_{n,m}$ is geometrically regular by \cite[Chapitre 0, Proposition (22.6.7)]{ega4_1}.
	In particular, we  conclude that $R_{n,m}/\mathfrak{p}_{n,m}$ is a noetherian regular domain, and hence normal.

	For (3) note that $R_{n,m} \subset \overline{R}$ is finite over $R$, therefore, it is a noetherian $p\textrm{-adically}$ complete domain.
	In particular, we see that the element $p$ and thus the principal ideal $\mathfrak{p}_{n,m}$, is contained in every maximal ideal of $R_{n,m}$.
	Then, from (2) and \cite[\href{https://stacks.math.columbia.edu/tag/00NU}{Tag 00NU}]{stacks-project}, it follows that the localisation of $R_{n,m}$ at every maximal ideal is regular, hence, $R_{n,m}$ is regular and thus normal.

	For (4), observe that the natural homomorphism $R_{\square,n,m}/pR_{\square,n,m} \rightarrow S_{\square,n,m}/pS_{\square,n,m}$ is clearly injective.
	As $R$ is flat over $R_{\square}$, extending the preceding injective map along $R_{\square} \rightarrow R$, we obtain 
	that the natural homomorphism $R_{n,m}/pR_{n,m} \rightarrow S_{n,m}/pS_{n,m}$ is injective.
	Moreover, as $R_{n,m}$ and $S_{n,m}$ are $p\textrm{-torsion}$ free, an easy induction on $r \geqslant 0$ shows that the homomorphism $R_{n,m}/p^rR_{n,m} \rightarrow S_{n,m}/p^rS_{n,m}$ is injective. 
\end{proof}

Let us set $R_{\infty} \defeq \cup_{n \geqslant 0} R_n$, $R_{\infty,m} \defeq \cup_{n \geqslant 0} R_{n,m}$, $R_{\infty,\infty} \defeq \cup_{m \geqslant 0} R_{\infty,m}$, and define the following Galois groups:
\begin{equation*}
	\begin{array}{ l l l }
		H_R \defeq \mathrm{Gal}(\overline{R}[1/p])/R_{\infty}[1/p]), & & \Gamma_R \defeq \mathrm{Gal}(R_{\infty}[1/p])/R[1/p]), \\ 
		H_{R,m} \defeq \mathrm{Gal}(\overline{R}[1/p])/R_{\infty,m}[1/p]), & & \Gamma_{R,m} \defeq \mathrm{Gal}(R_{\infty,m}[1/p])/R[1/p]), \\
		H_{R,\infty} \defeq \mathrm{Gal}(\overline{R}[1/p])/R_{\infty,\infty}[1/p]), & & \Gamma_{R,\infty} \defeq \mathrm{Gal}(R_{\infty,\infty}[1/p])/R[1/p]), \\
		\Delta_{R,\infty} \defeq \mathrm{Gal}(R_{\infty,\infty}[1/p]/R_{\infty}[1/p]), & & \Delta_{R,m} \defeq \mathrm{Gal}(R_{0,m}[1/p]/R[1/p]).
	\end{array}
\end{equation*}

Then, it is easy to see that $\Delta_{R,m} \isomorphic \mathrm{Gal}(R_{\infty,m}[1/p]/R_{\infty}[1/p])$,
and we have canonical group isomorphisms
\begin{equation*}
	\Gamma_{R,m} \isomorphic \Delta_{R,m} \times \Gamma_R, \qquad \Gamma_{R,\infty} \isomorphic \Delta_{R,\infty} \times \Gamma_R.
\end{equation*}

Let us now recall the structure of the group $\Gamma_R$.
Set $\Gamma'_R \defeq \textrm{Gal}(R_{\infty}[1/p]/R(\zeta_{p^{\infty}})[1/p])$, and observe that we have an exact sequence 
\begin{equation}\label{eqn: exact sequence of Gamma groups}
	1 \longrightarrow \Gamma'_R \longrightarrow \Gamma_R \longrightarrow \Gamma_F \longrightarrow 1,
\end{equation}
where $\Gamma_F \defeq \mathrm{Gal} (F(\zeta_{p^\infty})/F)$ is canonically isomorphic to $\mathbb{Z}_p^{\times}$ via the $p\textrm{-adic}$ cyclotomic character $\chi \colon \Gamma_F \rightarrow \mathbb{Z}_p^{\times}$.
For $1 \leqslant i \leqslant d$, we fix compatible systems $\{x_i^{1/p^n}\}_{n \geqslant 0}$ of $p^n\textrm{-th}$ roots of $x_i$ and fix elements $\{\gamma_1, \ldots, \gamma_d\}$ in $\Gamma'_R$ such that, for all $n \geqslant 1$, we have
\begin{equation*}
	\gamma_i(x_j^{1/p^n}) = 
	\left\{
		\begin{array}{ll}
			\zeta_{p^n}x_i^{1/p^n} & \textrm{ if } i = j,\\
			x_j^{1/p^n} & \textrm{ if } i \neq j.
		\end{array}
	\right.
\end{equation*}
Then $\{\gamma_1, \ldots, \gamma_d\}$ is  a system of topological generators of the free  abelian pro-$p$-group $\Gamma'_R.$ 
Fix a topological generator of $\mathrm{Gal}(F(\zeta_{p^\infty})/F)$ and denote by $\gamma_0$ in $\Gamma_R$ its unique lift such that 
\begin{equation*}
	\gamma_0(x_i^{1/p^n}) = x_i^{1/p^n} \textrm{ for all } n \geqslant 1 \textrm{ and } 1 \leqslant i \leqslant d.
\end{equation*}
Note that we have $\gamma_0 \gamma_i = \gamma_i^{\chi(\gamma_0)} \gamma_0$ for $1 \leqslant i \leqslant d$.
Therefore, the  exact sequence \eqref{eqn: exact sequence of Gamma groups} splits, and one has a natural isomorphism of groups
\begin{equation*}
	\Gamma_R \isomorphic \Gamma'_R \rtimes \Gamma_F,
\end{equation*}
where we have an isomorphism of groups $\Gamma'_R \isomorphic \mathbb{Z}_p(1)^d$ compatible with the action of $\Gamma_F \isomorphic \mathbb{Z}_p^{\times}$.
Note that the aforementioned isomorphisms depend on the choice of the roots $x_i^{1/p^n}$.

Let us set $S_{\infty} \defeq \cup_{n \geqslant 0} S_n$, $S_{\infty,m} \defeq \cup_{n \geqslant 0} S_{n,m}$, $S_{\infty,\infty} \defeq \cup_{m \geqslant 0} \cup_{n \geqslant 0} S_{n,m}$ and $G_{S,m} = \Gal(\overline{S}[1/p]/S_{0,m}[1/p])$, and analogous to above we define
\begin{equation*}
	\begin{array}{ l l l }
		H_S \defeq \mathrm{Gal}(\overline{S}[1/p]/S_{\infty}[1/p]), & & \Gamma_S \defeq \mathrm{Gal}(S_{\infty}[1/p]/S[1/p]),\\
		H_{S,m} \defeq \mathrm{Gal}(\overline{S}[1/p])/S_{\infty,m}[1/p]), & & \Gamma_{S,m} \defeq \mathrm{Gal}(S_{\infty,m}[1/p])/S[1/p]),\\
		H_{S,\infty} \defeq \mathrm{Gal}(\overline{S}[1/p])/S_{\infty,\infty}[1/p]), & & \Gamma_{S,\infty} \defeq \mathrm{Gal}(S_{\infty,\infty}[1/p])/S[1/p]).
	\end{array}
\end{equation*}
The continuous homomorphism of Galois groups $G_S \rightarrow G_R$ induces isomorphisms $\Gamma_S \isomorphic \Gamma_R$ and $\Gamma_{S,\infty} \isomorphic \Gamma_{R,\infty}$.
Using the preceding isomorphism, we shall identify  $\{\gamma_0, \gamma_1, \ldots, \gamma_d\}$ with a system of topological generators of $\Gamma_S$.

\begin{nota}
	For any $R\textrm{-subalgebra}$ $R' \subset \overline{R}$ satisfying \hyperref[nota:logetale_condition]{\textbf{(log-\'et)}}, denote by $R_n'$ (resp.\ $R_{n,m}'$) the normalisation of $R' \otimes_{R_{\square}} R_{\square,n}$ (resp.\ $R' \otimes_{R_{\square}} R_{\square,n,m}$) in its field of fractions, and set $R_{\infty}' \defeq \cup_{n \geqslant 0} R_n'$ (resp.\ $R_{\infty,m}' \defeq \cup_{n \geqslant 0} R_{n,m}'$).
	Similarly, for any $S\textrm{-subalgebra}$ $S' \subset \overline{S}$ satisfying \hyperref[nota:etale_condition]{\textbf{(\'et)}}, denote by $S_n'$ (resp.\ $S_{n,m}'$) the normalisation of $S' \otimes_{S_{\square}} S_{\square,n}$ (resp.\ $S' \otimes_{S_{\square}} S_{\square,n,m}$) in its field of fractions, and set $S_{\infty}' \defeq \cup_{n \geqslant 0} S_n'$ (resp.\ $S_{\infty,m}' \defeq \cup_{n \geqslant 0} S_{n,m}'$).
	If $R'=R$ (resp.\ $S'=S$) this agrees with our previous notation. 
\end{nota}

\begin{nota}
   For a commutative ring $A$, we shall let $\widehat{A}$ denote the $p\textrm{-adic}$ completion of $A$, namely $\widehat{A} = \lim_r A/p^rA$.
\end{nota}

\begin{lem}\label{lem:Rinftyhat_normal}
	The rings $\widehat{R}_{\infty,m}$, $\widehat{R}_{\infty,\infty}$, $\widehat{S}_{\infty,m}$ and $\widehat{S}_{\infty,\infty}$ are normal domains.
\end{lem}
\begin{proof}
	It is easy to check that the towers of domains $\{R_{n,m}\}_{n \geqslant 0}$, $\{R_{n,p^n+1}\}_{n \geqslant 0}$, $\{S_{n,m}\}_{n \geqslant 0}$ and $\{S_{n,p^n+1}\}_{n \geqslant 0}$ satisfy \cite[A.1, Conditions (i)-(iv)]{andreatta-phigamma}.
	Therefore, from \cite[Lemma A.3 and Proposition A.6]{andreatta-phigamma}, it follows that the respective $p\textrm{-adic}$ completion of the colimit of each tower, i.e.\ $\widehat{R}_{\infty,m}$, $\widehat{R}_{\infty,\infty}$, $\widehat{S}_{\infty,m}$ and $\widehat{S}_{\infty,\infty}$ are normal domains.
\end{proof}

\begin{prop}\label{prop:Rinftym_in_Sinftym}\label{lem:Rinftym_in_Sinftym}
	For each $n \geqslant 0$ and $m \geqslant 1$ coprime to $p$, the following hold true:
	\begin{enumerate}
		\item[\textup{(1)}] For each $r \geqslant 0$, we have that $p^r\widehat{{R}}_{\infty,m} \cap R_{n,m} = p^r R_{n,m}$ (resp.\ $p^r\widehat{{S}}_{\infty,m} \cap S_{n,m} = p^r S_{n,m}$).
		
		\item[\textup{(2)}] The natural homomorphisms $\widehat{R}_{\infty,m} \rightarrow \widehat{S}_{\infty,m}$ and $\widehat{R}_{\infty,\infty} \rightarrow \widehat{S}_{\infty,\infty}$ are injective.

		\item[\textup{(3)}] We have that $R_{n,m} = \widehat{R}_{\infty,m} \cap S_{n,m} \subset \widehat{S}_{\infty,m}$ and $\widehat{R}_{\infty,m} = \widehat{R}_{\infty,\infty} \cap \widehat{S}_{\infty,m} \subset \widehat{S}_{\infty,\infty}$.
	\end{enumerate}
\end{prop}

We postpone the proof of Proposition \ref{prop:Rinftym_in_Sinftym} to the next section (see after Proposition \ref{prop:properties_Rbarhat}).

\subsection{Some properties of $\overline{R}$ and $\overline{S}$}

In this section, we record some technical statements for the rings $\overline{R}$ and $\overline{S}$, and their respective $p\textrm{-adic}$ completions $\CRp$ and $\CSp$.

\subsubsection{Almost \'etale descent}

We begin by showing a standard statement on almost \'etale descent in our setting.

\begin{nota}
	For any topological group $G$ and a topological $G$\textrm{-module} $M$, we shall let $H^i(G,M)$ denote the continuous cohomology of $G$ with coefficients in $M$.
	As usual, we shall consider $\overline{R}$ (resp.\ $\overline{S}$) as a discrete $G_R\textrm{-module}$ (resp.\ $G_S\textrm{module}$), and $\CRp$ (resp.\ $\CSp$) as a continuous (for the $p\textrm{-adic}$ topology) $G_R\textrm{-module}$ (resp.\ $G_S\textrm{-module}$).
\end{nota}

\begin{nota}
	For any ring $A$, the ring $A^{\flat} \coloneq \lim_{\varphi}A/pA$ shall be referred to as the \textit{tilt} of $A$.
\end{nota}

For each $j \geqslant 1$, set $\pi_j \defeq \zeta_{p^j}-1$.
For any $R_{\infty}\textrm{-algebra}$ or $S_{\infty}\textrm{-algebra}$ $B$, let $\mathfrak{m}_B$ denote the ideal of $B$ generated by $\{\pi_j\}_{j \geqslant 1}$, and note that we have $\mathfrak{m}_B = \mathfrak{m}_B^2$.

\begin{lem}\label{lem:Ax-Sen_R}
	Let $A$ denote either of the rings $R$ or $S$, and correspondingly let $A_{\infty,\infty}$ denote $R_{\infty,\infty}$ or $S_{\infty,\infty}$.
	Assume that $B \subset \overline{A}$ is a finite normal $A_{\infty,\infty}\textrm{-subalgebra}$.
	Let $G_B$ denote the Galois group $\Gal(\overline{A}[1/p]/B[1/p])$.
	Then, the following claims are true:
	\begin{enumerate}
		\item[\textup{(1)}] For each $i \geqslant 1$, the cohomology group $H^i(G_B, \widehat{\overline{A}})$ is an almost zero $B\textrm{-module}$, i.e.\ it is annihilated by the ideal $\mathfrak{m}_{B}$.
			In particular, $H^i(G_{B}, \widehat{\overline{A}}[1/p])=0$ for $i\geqslant 1$.
		
		\item[\textup{(2)}] The cokernel of the following natural injective homomorphisms are almost zero, i.e.\ killed by $\mathfrak{m}_{B}$:
			\begin{equation*}
				B/p^nB = \overline{A}^{G_B}/p^n\overline{A}^{G_B} \longhookrightarrow \left(\overline{A}/p^n\overline{A}\right)^{G_B} \quad  \textit{and} \qquad \widehat{B} \longhookrightarrow \bigl.\widehat{\overline{A}}\,\bigr.^{G_{B}}.
			\end{equation*}
			In particular, we get that $\widehat{\overline{A}}[1/p]^{G_{B}} = \widehat{B}[1/p]$.
		
		\item[\textup{(3)}] We have that $\bigl.\widehat{\overline{A}}\,\bigr.^{G_B} = \widehat{B}$.
			In particular, $\bigl.\widehat{\overline{A}}\,\bigr.^{H_{A,\infty}} = \widehat{A}_{\infty,\infty}$.
	\end{enumerate}
\end{lem}
\begin{proof}
	For $A = S$ the claims follow from \cite[Propositions 3.1.1 and 3.1.8]{brinon-relatif}.
	The proof in the case $A = R$ follows by analogous arguments, and we reproduce it for the reader's convenience.
	Also see \cite[Corollary~3.43]{andreatta-iovita-semistable} where the claims are formulated slightly differently.
	
	For (1), note that we may write $B = P_{\infty,\infty}$ for some finite normal $A\textrm{-algebra}$ $P \subset \overline{R}$ satisfying \hyperref[nota:logetale_condition]{\textbf{(log-\'et)}} and such that $G_B = H_P$.
	then, for a fixed $n \geqslant 1$, the topological $H_P\textrm{-module}$ $\overline{R}/p^n\overline{R}$ is discrete, and therefore we may write:
	\begin{equation*}
		H^i(H_P, \overline{R}/p^n\overline{R}) = \colim_{Q/P} H^i(H_P/H_Q, Q_{\infty,\infty}/p^n Q_{\infty,\infty}),
	\end{equation*}
	where $Q \subset \overline{R}$ runs over all finite extensions of $P$ such that $P[1/p] \rightarrow Q[1/p]$ is \'etale and Galois.
	Therefore, any cohomology class in $H^i(H_P, \overline{R}/p^n\overline{R})$ may be represented by an $i\textrm{-cocycle}$ $f$ in $C^i(H_P/H_Q, Q_{\infty,\infty}/p^n Q_{\infty,\infty})$, for some finite \'etale and Galois extension $P[1/p] \rightarrow Q[1/p]$.
	Now, the almost-purity theorem (see \cite[Section 2b]{faltings-almost}, \cite{scholze-perfectoid}, \cite[Section 9.6]{gabber-ramero-almostv5}) shows that there exists some $b_j$ in $Q_{\infty,\infty}$ such that 
	\begin{equation*}
		\mathrm{tr}_{Q_{\infty,\infty}[1/p]/P_{\infty,\infty}[1/p]}(b_j) = \pi_j \in P_{\infty,\infty},
	\end{equation*}
	where $\mathrm{tr}_{Q_{\infty,\infty}[1/p]/P_{\infty,\infty}[1/p]} \colon Q_{\infty,\infty}[1/p] \rightarrow P_{\infty,\infty}[1/p]$ is the usual trace map.
	For $i \geqslant 1$, a classical computation of Tate (see \cite[Section 3.2]{tate-p-divisible}) shows that we have the $(i-1)\textrm{-cochain}$ 
	\begin{equation*}
		a(s_1,\ldots, s_{i-1}) = (-1)^i \textstyle\sum_{\tau\in H_P/H_Q} s_1 \cdots s_{i-1} \tau (b) f(s_1,\ldots, s_{i-1},\tau),
	\end{equation*}
	such that $\partial(a) = \pi_j f$.
	Therefore, we see that $\pi_j f$ is a coboundary, and thus $\pi_j$ annihilates the class of $f$ in $H^i(H_P/H_Q, Q_{\infty,\infty}/p^n Q_{\infty,\infty})$.
	As the previous argument may be repeated for all $j \geqslant 1$ and the chosen cohomology class is arbitrary, therefore, we conclude that $\pi_j$ kills $H^i(H_P/H_Q, Q_{\infty,\infty}/p^n Q_{\infty,\infty})$ for all $j \geqslant 1$, and thus $\pi_j$ kills $H^i(H_P, \overline{R}/p^n \overline{R})$ for all $j \geqslant 1$.
	
	Analogously, since we consider $\overline{R}$ as a discrete $G_R\textrm{-module}$, therefore, we may write 
	\begin{equation*}
		H^i(H_P, \overline{R}) = \colim_{Q/P} H^i(H_P/H_Q, Q_{\infty,\infty}).
	\end{equation*}
	Then, an argument similar to the preceding argument shows that $\pi_j$ kills $H^i(H_P, \overline{R})$ for all $j \geqslant 1$.
	Therefore, from the discussion above we conclude that $\mathfrak{m}_B$ annihilates $H^i(G_{B}, \overline{R}/p^n\overline{R})$ for all $n\geqslant 1$, and $H^i(G_{B}, \overline{R})$, thus proving the claim.
	
	To show (2), note that for each $n \geqslant 1$, the exact sequence
	\begin{equation*}
		0 \longrightarrow \overline{R} \xrightarrow{p^n} \overline{R} \longrightarrow \overline{R}/p^n\overline{R} \longrightarrow 0,
	\end{equation*}
	yields the following exact sequence:
	\begin{equation*}
		0 \longrightarrow \overline{R}^{G_{B}}/p^n \overline{R}^{G_{B}} \longrightarrow (\overline{R}/p^n\overline{R})^{G_{B}} \longrightarrow H^1(G_{B},\overline{R})[p^n] \longrightarrow 0,
	\end{equation*}
	where ``$[p^n]$'' in the last term denotes $p^n\textrm{-torsion}$ elements.
	From part (1), recall that $H^1(G_{B}, \overline{R})$ is killed by $\mathfrak{m}_B$, therefore, we get that the cokernel of the first non-trivial homomorphism above is almost zero. 
	Passing to the limit over $n$, and taking into account that $\overline{R}^{G_{B}} = B$ and $\lim_n (\overline{R}/p^n\overline{R})^{G_{B}} = \widehat{\overline{R}}\,^{G_B}$, where the last equality follows because taking $G_B\textrm{-invariants}$ commutes with limit, we obtain the following exact sequence:
	\begin{equation*}
		0 \longrightarrow \widehat{B} \longrightarrow \widehat{\overline{R}}\,^{G_{B}} \longrightarrow \textstyle\lim_n H^1(G_{B},\overline{R})[p^n].
	\end{equation*}
	Hence, we conclude that the cokernel of the first non-trivial homomorphism above is almost zero, i.e.\ killed by $\mathfrak{m}_B$, thus implying part (2). 
	
	For part (3), note that using part (2) we have 
	\begin{equation*}
		\bigl.\widehat{\overline{A}}\,\bigr.^{G_B} = \widehat{\overline{A}}[1/p]^{G_B} \cap \widehat{\overline{A}} = \widehat{B}[1/p] \cap \widehat{\overline{A}} \subset \widehat{\overline{A}}[1/p].
	\end{equation*}
	Moreover, we have that $\widehat{B}/p^n\widehat{B} = B/p^n B \hookrightarrow \overline{A}/p^n\overline{A} = \widehat{\overline{A}}/p^n\widehat{\overline{A}}$, where the equalities follow from \cite[\href{https://stacks.math.columbia.edu/tag/05GG}{Tag 05GG}]{stacks-project}, and the injective homomorphism is from part (2).
	Therefore, we conclude that $\widehat{B}  \cap p^n\widehat{\overline{A}} = p^n \widehat{B}$, or equivalently, that $\widehat{B}[1/p] \cap \widehat{\overline{A}} = \widehat{B}$, in particular, $\bigl.\widehat{\overline{A}}\,\bigr.^{G_B} = \widehat{B}$.
\end{proof}

\subsubsection{Completions and Galois invariants}

We begin with the following observation:
\begin{prop}\label{prop:properties_Rbarhat}
	The following claims are true:
	\begin{enumerate}
		\item[\textup{(1)}] The ring $\widehat{\overline{R}}$ (resp.\ $\widehat{\overline{S}}$) is $p\textrm{-torsion free}$ and reduced, the homomorphism $\overline{R} \rightarrow \widehat{\overline{R}}$ (resp.\ $\overline{S} \rightarrow \widehat{\overline{S}}$) is injective, and $p^r\widehat{\overline{R}} \cap \overline{R} = p^r\overline{R}$ (resp.\ $p^r\widehat{\overline{S}} \cap \overline{S} = p^r\overline{S}$) for each $r \geqslant 1$.
		
		\item[\textup{(2)}] The ring homomorphism $R[1/p] \rightarrow \widehat{\overline{R}}[1/p]$ (resp.\ $S[1/p] \rightarrow \widehat{\overline{S}}[1/p]$) is faithfully flat. 
	\end{enumerate}
\end{prop}
\begin{proof}
	Both claims are shown in \cite[Proposition 3.5]{andreatta-iovita-semistable} (resp.\ \cite[Proposition 2.0.3]{brinon-relatif}).
\end{proof}

Using Proposition \ref{prop:properties_Rbarhat}, we may now show Proposition \ref{prop:Rinftym_in_Sinftym} as follows:
\begin{proof}[Proof of Proposition \ref{prop:Rinftym_in_Sinftym}]
	In part (1), we only show the claim for $R_{n,m}$, and the claim for $S_{n,m}$ follows by a similar argument.
	Let us note that from Proposition \ref{prop:properties_Rbarhat}, we have that $p^r\widehat{{R}}_{\infty,m} \cap R_{n,m}\subset p^r\overline{R}$.
	As we have that $\overline{R} \cap R_{n,m}[1/p] = R_{n,m}$ because $R_{n,m}$ is normal by Proposition \ref{prop:Rnm_Snm_normal} (3), therefore, we obtain that $p^r\overline{R} \cap R_{n,m} = p^r R_{n,m}$, thus implying that $p^r\widehat{{R}}_{\infty,m} \cap R_{n,m} \subset p^r R_{n,m}$, and hence $p^r\widehat{{R}}_{\infty,m} \cap R_{n,m} = p^r R_{n,m}$.
	
	To show part (2), let us first note that from Proposition \ref{prop:Rnm_Snm_normal} (4), the natural homomorphism $R_{n,m}/p^r R_{n,m} \rightarrow S_{n,m}/p^r S_{n,m}$ is injective.
	Then, by taking the (filtered) colimit over $n$, we get that the natural homomorphism $R_{\infty,m}/p^r R_{\infty,m} \rightarrow S_{\infty,m}/p^r S_{\infty,m}$ is injective.
	Similarly, taking the (filtered) colimit over both $m$ and $n$, we get that the natural homomorphism $R_{\infty,\infty}/p^r R_{\infty,\infty} \rightarrow S_{\infty,\infty}/p^r S_{\infty,\infty}$ is injective.
	Passing to the limit over $r$ in both the homomorphisms above, we obtain that the natural homomorphisms $\widehat{R}_{\infty,m}\rightarrow \widehat{S}_{\infty,m}$ and $\widehat{R}_{\infty,\infty}\rightarrow \widehat{S}_{\infty,\infty}$ are injective.

	Let us next show the first part of (3), i.e.\ $R_{n,m}  = \widehat{R}_{\infty,m} \cap S_{n,m} \subset \widehat{S}_{\infty,m}$.
	Observe that we have $R_{n,m} \cap p^r(\widehat{R}_{\infty,m} \cap S_{n,m}) = p^r R_{n,m}$ by part (1), which implies that the following natural homomorphism is injective:
	\begin{equation}\label{eq:Rnm_modpr}
		R_{n,m}/p^r \longhookrightarrow (\widehat{R}_{\infty,m} \cap S_{n,m})/p^r.
	\end{equation}
	Moreover, we have that the following equalities:
	\begin{equation*}
		\begin{aligned}
			p^r\widehat{R}_{\infty,m} \cap (\widehat{R}_{\infty,m} \cap S_{n,m}) = p^r\widehat{R}_{\infty,m} \cap S_{n,m} = p^r\widehat{R}_{\infty,m} \cap p^r\widehat{S}_{\infty,m} \cap S_{n,m} = p^r(\widehat{R}_{\infty,m} \cap S_{n,m}),\\
			p^r S_{n,m} \cap (\widehat{R}_{\infty,m} \cap S_{n,m}) = \widehat{R}_{\infty,m} \cap p^r S_{n,m} = \widehat{R}_{\infty,m} \cap p^r\widehat{S}_{\infty,m} \cap p^r S_{n,m} = p^r(\widehat{R}_{\infty,m} \cap S_{n,m}),
		\end{aligned}
	\end{equation*}
	where the third equality in the first line follows from part (1), and the third equality in the second line follows from the injectivity of the homomorphism $R_{\infty,m}/p^r R_{\infty,m} \rightarrow S_{\infty,m}/p^r S_{\infty,m}$ shown above.
	In particular, we get that the following natural homomorphism is injective:
	\begin{equation}\label{eq:intersect_modpr}
		(\widehat{R}_{\infty,m} \cap S_{n,m})/p^r \longhookrightarrow (\widehat{R}_{\infty,m}/p^r) \cap (S_{n,m}/p^r) = (R_{\infty,m}/p^r) \cap (S_{n,m}/p^r) \subset \widehat{S}_{\infty,m}/p^r.
	\end{equation}

	We claim that $R_{n,m}/p^r = (R_{\infty,m}/p^r) \cap (S_{n,m}/p^r)$.
	Since $S_{n,m} = R_{n,m} \otimes_R S$, the preceding claim is equivalent to showing that $R_{n,m}/p^r = (R_{\infty,m}/p^r) \cap (R_{n,m}/p^r)[1/x]$, where $x = x_{a+1} \cdots x_d$.
	Observe that the sequence $\{p, x\}$ is regular on $R_{n,m}$ (resp.\ $R_{\infty,m}$).
	Indeed, we have that $R_{n,m}$ (resp.\ $R_{\infty,m}$) is $p\textrm{-torsion}$ free, and from the injective homomorphism $R_{n,m}/p \hookrightarrow S_{n,m}/p$ (resp.\ $R_{\infty,m}/p \hookrightarrow S_{\infty,m}/p$), it follows that the source is $x\textrm{-torsion}$ free.
	This reduces the claim into showing that the natural homomorphism $R_{n,m}/(p^r, x) \rightarrow R_{\infty,m}/(p^r, x)$ is injective.
	Next, observe that the sequence $\{x, p\}$ is regular on $R_{n,m}$ (resp.\ $R_{\infty,m}$).
	Indeed, we have that $R_{n,m}$ (resp.\ $R_{\infty,m}$) is $x\textrm{-torsion}$ free, and since $(R_{n,m}/x)_{p\textrm{-tors}} = (R_{n,m}/p)_{x\textrm{-tors}} = 0$ (resp.\ $(R_{\infty,m}/x)_{p\textrm{-tors}} = (R_{\infty,m}/p)_{x\textrm{-tors}} = 0$), it follows that $\{x, p\}$ is regular on $R_{n,m}$ (resp.\ $R_{\infty,m}$).
	Then, an easy induction on $r \geqslant 1$ reduces us to showing that the natural homomorphism $R_{n,m}/(p, x) \rightarrow R_{\infty,m}/(p, x)$ is injective.
	Let $\pi_n = \zeta_{p^n}-1 = p\cdot u$ for some unit $u$ in $O_F[\zeta_{p^n}]$, and note that the reduced claim is equivalent to showing that the natural homomorphism $R_{n,m}/(\pi_n^{p^{n-1}(p-1)}, x) \rightarrow R_{\infty,m}/(\pi_n^{p^{n-1}(p-1)}, x)$ is injective.
	Again an easy induction on powers of $\pi_n$ reduces us to showing that the natural homomorphism $R_{n,m}/(\pi_n, x) \rightarrow R_{\infty,m}/(\pi_n, x)$ is injective, which is equivalent to showing that $R_{n,m}/\pi_n = (R_{\infty,m}/\pi_n) \cap (R_{n,m}/\pi_n)[1/x] \subset S_{\infty,m}/\pi_n$.
	As any element of $R_{\infty,m}/\pi_n$ is integral over $R_{n,m}/\pi_n$, and the latter is normal by Proposition \ref{prop:Rnm_Snm_normal} (2), therefore, it follows that we have $R_{n,m}/\pi_n = (R_{\infty,m}/\pi_n) \cap (R_{n,m}/\pi_n)[1/x]$, and hence $R_{n,m}/p^r = (R_{\infty,m}/p^r) \cap (S_{n,m}/p^r)$.

	Using the preceding claim and equations \eqref{eq:Rnm_modpr} and \eqref{eq:intersect_modpr}, we obtain the following injective homomorphisms:
	\begin{equation*}
		R_{n,m}/p^r \longhookrightarrow (\widehat{R}_{\infty,m} \cap S_{n,m})/p^r \longhookrightarrow (R_{\infty,m}/p^r) \cap (S_{n,m}/p^r) = R_{n,m}/p^r,
	\end{equation*}
	which implies that $R_{0,m}/p^r \isomorphic (\widehat{R}_{\infty,m} \cap S_{n,m})/p^r \isomorphic (R_{\infty,m}/p^r) \cap (S_{n,m}/p^r)$.
	Passing to the limit over $r$, we obtain the following isomorphism:
	\begin{equation*}
		\begin{aligned}
			R_{n,m} = \textstyle\lim_r R_{n,m}/p^r &\isomorphic \lim_r (\widehat{R}_{\infty,m} \cap S_{n,m})/p^r\\
				&\isomorphic (\lim_r R_{\infty,m}/p^r) \cap (\lim_r S_{n,m}/p^r) = \widehat{R}_{\infty,m} \cap S_{n,m} \longhookrightarrow \widehat{S}_{\infty,m},
		\end{aligned}
	\end{equation*}
	where the second equality follows because $\widehat{R}_{\infty,m}$ and $S_{n,m}$ are both $p\textrm{-adically}$ complete.

	Finally, let us show the second part of (3).
	It is clear that we have $\widehat{R}_{\infty,m} \subset \widehat{R}_{\infty,\infty} \cap \widehat{S}_{\infty,m}$, and to show the converse we will use the results of \cite[Appendix A]{andreatta-phigamma}.
	Let us first note that the towers $\{R_{n,m}\}_{n \geqslant 0}$, $\{R_{n,p^n+1}\}_{n \geqslant 0}$, $\{S_{n,m}\}_{n \geqslant 0}$ and $\{S_{n,p^n+1}\}_{n \geqslant 0}$ satisfy \cite[A.1, Conditions (i)-(iv)]{andreatta-phigamma}, in particular, the results of \cite[Appendix A]{andreatta-phigamma} apply to these towers.
	Next, using the notation of Proposition \ref{prop:Rnm_Snm_normal}, let us set $\mathfrak{p}_{\infty,m} \coloneq \colim_n \mathfrak{p}_{n,m}$, $\mathfrak{p}_{\infty,\infty} \coloneq \colim_m \mathfrak{p}_{m,m}$, $\mathfrak{P}_{\infty,m} \coloneq \colim_n \mathfrak{P}_{n,m}$ and $\mathfrak{P}_{\infty,\infty} \coloneq \colim_m \mathfrak{P}_{m,m}$ respectively, as an ideal inside $R_{\infty,m}$, $R_{\infty,\infty}$, $S_{\infty,m}$ and $S_{\infty,\infty}$.
	Note that as filtered colimits preserve exact sequences (of abelian groups, see for example \cite[Theorem 2.6.15]{weibel}), we get that the natural maps in the following commutative diagram are injective:
	\begin{equation}\label{eq:Rinftym_mod_pinftym}
		\begin{tikzcd}
			R_{\infty,m}/\mathfrak{p}_{\infty,m} = \colim_n R_{n,m}/\mathfrak{p}_{n,m} & R_{\infty,\infty}/\mathfrak{p}_{\infty,\infty} = \colim_m R_{m,m}/\mathfrak{p}_{m,m} \\
			S_{\infty,m}/\mathfrak{P}_{\infty,m} = \colim_n S_{n,m}/\mathfrak{P}_{n,m} & S_{\infty,\infty}/\mathfrak{P}_{\infty,\infty} = \colim_m S_{m,m}/\mathfrak{P}_{m,m}.
			\arrow[hook, from=1-1, to=1-2]
			\arrow[hook, from=1-1, to=2-1]
			\arrow[hook, from=1-2, to=2-2]
			\arrow[hook, from=2-1, to=2-2]
		\end{tikzcd}
	\end{equation}
	It is easy to see that the top left corner is the intersection of the top right corner and the bottom left corner inside the bottom right corner.

	Let $A$ denote either of the rings $R_{\infty,m}$, $R_{\infty,\infty}$, $S_{\infty,m}$ or $S_{\infty,\infty}$, and let $\mathfrak{q}$ denote the respective ideal $\mathfrak{p}_{\infty,m}$, $\mathfrak{p}_{\infty,\infty}$, $\mathfrak{P}_{\infty,m}$ or $\mathfrak{P}_{\infty,\infty}$.
	Then, from \cite[Lemma A.4]{andreatta-phigamma}, there exists a unique multiplicative map $\textbf{w} \colon A/\mathfrak{q} \rightarrow \widehat{A}$ such that $\textbf{w}(a) = a \textrm{ mod } \mathfrak{q}$ and $\textbf{w}(a^p) = \textbf{w}(a)^p$ for any $a$ in $A/\mathfrak{q}$.
	Additionally, from \cite[Lemma A.5]{andreatta-phigamma}, we have that any non-zero element $a$ of $A$ may uniquely be written as
	\begin{equation*}
		a = \textstyle \sum_{n \geqslant 0} \textbf{w}(a_n)p^{j_n},
	\end{equation*}
	with $a_n$ in $A/\mathfrak{q}$ so that the sequence of rational numbers $\{j_n \textrm{ such that } a_n \neq 0\}$ is strictly increasing and it is either finite or it converges to infinity.

	Now, let $a$ be an element of the intersection $\widehat{R}_{\infty,\infty} \cap \widehat{S}_{\infty,m} \subset \widehat{S}_{\infty,\infty}$, and we claim that $a$ is in $\widehat{R}_{\infty,m}$.
	Indeed, from the previous paragraph we have a unique presentation $a = \sum_{n \geqslant 0} \textbf{w}(a_n)p^{j_n}$ with $a_n$ in $R_{\infty,\infty}/\mathfrak{p}_{\infty,\infty} \cap S_{\infty,m}/\mathfrak{P}_{\infty,m} \subset S_{\infty,\infty}/\mathfrak{P}_{\infty,\infty}$.
	From the commutative diagram in \eqref{eq:Rinftym_mod_pinftym}, and the computation of $R_{\infty,m}/\mathfrak{p}_{\infty,m}$ and $S_{\infty,m}/\mathfrak{P}_{\infty,m}$ in the proof of Proposition \ref{prop:Rnm_Snm_normal} it follows that $R_{\infty,m}/\mathfrak{p}_{\infty,m} = R_{\infty,\infty}/\mathfrak{p}_{\infty,\infty} \cap S_{\infty,m}/\mathfrak{P}_{\infty,m}$, i.e.\ $a_n$ is in $R_{\infty,m}/\mathfrak{p}_{\infty,m}$.
	Hence, $a$ is in $\widehat{R}_{\infty,m}$, as claimed.
	This allows us to conclude.
\end{proof}

From Lemma \ref{lem:Ax-Sen_R} and Proposition \ref{prop:Rinftym_in_Sinftym}, we deduce the following:
\begin{lem}\label{lem:Galinv_Rinftym}
	Let $A$ denote either of the rings $R_{0,m}$ or $S_{0,m}$, and correspondingly let $A_{\infty,\infty}$ (resp.\ $A_{\infty}$) denote $R_{\infty,\infty}$ (resp.\ $R_{\infty,m}$) or $S_{\infty,\infty}$ (resp.\ $S_{\infty,m}$).
	Then, the following claims are true:
	\begin{enumerate}
		\item[\textup{(1)}] We have that $\bigl.\widehat{\overline{A}}\,\bigr.^{H_A} = \widehat{A}_{\infty}$ and $\bigl.\widehat{\overline{A}}\,\bigr.^{G_A} = A$. 

		\item[\textup{(2)}] Let $\overline{A}^{\flat} \coloneq \lim_{\varphi} \overline{A}/p\overline{A}$ and note that $G_A$ acts naturally and continuously on $\overline{A}^{\flat}$.
			Let $A_{\infty,\infty}^{\flat} \coloneq \lim_{\varphi} A_{\infty,\infty}/pA_{\infty,\infty}$.
			Then, we have that $(\overline{A}^{\flat})^{H_{A,\infty}} = A_{\infty,\infty}^{\flat}$ and $(\overline{A}^{\flat})^{H_A} = A_{\infty}^{\flat}$.
	\end{enumerate}
\end{lem}
\begin{proof}
	It is a well-known fact that $\bigl.\widehat{\overline{S}}\,\bigr.^{H_{S,m}} = \widehat{S}_{\infty,m}$, and for $\Gamma_S \isomorphic \Gal(S_{\infty,m}[1/p]/S_{0,m}[1/p])$ we have that $\widehat{S}_{\infty,m}^{\,\Gamma_S} = S_{0,m}$.
	So, let us now consider the case of $R_{0,m}$.
	From Lemma \ref{lem:Ax-Sen_R} (3), we have that $\bigl.\widehat{\overline{R}}\,\bigr.^{H_{R,\infty}} = \widehat{R}_{\infty,\infty}$.
	Then, it remains to show that for $\Delta \coloneq \Gal(R_{\infty,\infty}[1/p]/R_{\infty,m}[1/p]) \isomorphic \Gal(S_{\infty,\infty}[1/p]/S_{\infty,m}[1/p])$ we have $\widehat{R}_{\infty,\infty}^{\Delta} = \widehat{R}_{\infty,m}$ and for $\Gamma \coloneq \Gal(R_{\infty,m}[1/p]/R_{0,m}[1/p]) \isomorphic \Gamma_S$, we have $\widehat{A}_{\infty,m}^{\,\Gamma} = R_{0,m}$.
	But, this follows from Proposition \ref{lem:Rinftym_in_Sinftym} because we have
	\begin{equation*}
		\begin{aligned}
			\widehat{R}_{\infty,\infty}^{\,\Delta} &= \widehat{R}_{\infty,\infty} \cap \widehat{S}_{\infty,\infty}^{\,\Delta} = \widehat{R}_{\infty,\infty} \cap \widehat{S}_{\infty,m} = \widehat{R}_{\infty,m},\\
			\widehat{R}_{\infty,m}^{\,\Gamma} &= \widehat{R}_{\infty,m} \cap \widehat{S}_{\infty,m}^{\,\Gamma} = \widehat{R}_{\infty,m} \cap S_{0,m} = R_{0,m},
		\end{aligned}
	\end{equation*}
	where the intersections in the first equation are taken inside $\widehat{S}_{\infty,\infty}$, and in the second equation inside $\widehat{S}_{\infty,m}$.

	For part (2), recall that from \cite[Section 1.2.2]{fontaine-corps-periodes}, we have an isomorphism of monoids
	\begin{equation}\label{eq:Abar_monoid_iso}
		\overline{R}^{\flat} \isomorphic \lim_{x \mapsto x^p} \widehat{\overline{R}},
	\end{equation}
	which may be upgraded to an isomorphism of rings by equipping the right hand term with an appropriate ring structure from loc.\ cit.
	Additionally, by equipping the right hand term with a component-wise action of $G_R$, we see that the isomorphism in \eqref{eq:Abar_monoid_iso} is $G_R\textrm{-equivariant}$.
	As taking $H_{R,m}\textrm{-invariant}$ elements commutes with taking limits, it follows that 
	\begin{equation*}
		\begin{aligned}
			(\overline{R}^{\flat})^{H_{R,\infty}} \isomorphic \lim_{x \mapsto x^p} \bigl.\widehat{\overline{R}}\,\bigr.^{H_{R,\infty}} = \lim_{x \mapsto x^p} \widehat{R}_{\infty,\infty} \lisomorphic R_{\infty,\infty}^{\flat},\\
			(R_{\infty,\infty}^{\flat})^{\Gamma} \isomorphic \lim_{x \mapsto x^p} \widehat{R}_{\infty,\infty}^{\,\Gamma} = \lim_{x \mapsto x^p} \widehat{R}_{\infty,m} \lisomorphic R_{\infty,m}^{\flat},
		\end{aligned}
	\end{equation*}
	where we used part (1) for the two equalitues, and \cite[Section 1.2.2]{fontaine-corps-periodes} for the isomorphisms.
	This allows us to conclude.
\end{proof}

We end this section with a standard fact.
\begin{lem}\label{lem:Rbarhat_perfectoid}
	The rings $\widehat{R}_{\infty,m}$, $\widehat{R}_{\infty,\infty}$, $\widehat{\overline{R}}$, $\widehat{S}_{\infty,m}$, $\widehat{S}_{\infty,\infty}$ and $\widehat{\overline{S}}$ are perfectoid in the sense of \cite[Definition 3.5]{bhatt-morrow-scholze-1}.
\end{lem}
\begin{proof}
	We shall only show the claim for $\widehat{R}_{\infty,m}$, $\widehat{R}_{\infty,\infty}$ and $\widehat{\overline{R}}$ (the claim for $\widehat{S}_{\infty,m}$, $\widehat{S}_{\infty,\infty}$ and $\widehat{\overline{S}}$ follow by analogous arguments).
	To check that $\widehat{R}_{\infty,m}$, $\widehat{R}_{\infty,\infty}$ and $\widehat{\overline{R}}$ satisfy the assumptions of \cite[Definition 3.5]{bhatt-morrow-scholze-1}, let us set $\pi \coloneq \zeta_{p^2}-1$, then clearly $\pi^p$ divides $p$.
	It is easy to check that the natural power $p$ map $R_{n+1,m}/pR_{n+1,m} \rightarrow R_{n,m}/pR_{n+1,m}$ is surjective for $n \geqslant 0$, and taking the colimit over $n$ shows that the natural Frobenius endomorphism on $R_{\infty,m}/pR_{\infty,m}$ is surjective.
	Additionally, from \cite[Corollary 3.6]{andreatta-iovita-semistable} (resp.\ \cite[Proposition 2.0.1]{brinon-relatif}), the natural Frobenius endomorphism on $R_{\infty,\infty}/pR_{\infty,\infty}$ (resp.\ $\overline{R}/p\overline{R}$) is surjective.
	Moreover, from \cite[Lemma 3.10]{andreatta-iovita-semistable}, the kernel of Fontaine's map $\theta \colon A_{\inf}(R_{\infty,\infty}) = W(R_{\infty,\infty}^{\flat}) \twoheadrightarrow \widehat{R}_{\infty,\infty}$ (resp.\ $\theta \colon A_{\inf}(\overline{R}) = W(\overline{R}^{\flat}) \twoheadrightarrow \widehat{\overline{R}}$) is principal.
	Therefore, it follows that $\widehat{R}_{\infty,\infty}$ (resp.\ $\widehat{\overline{R}}$) is perfectoid in the sense of \cite[Definition 3.5]{bhatt-morrow-scholze-1}.

	It remains to check that the kernel of the map $\theta \colon A_{\inf}(R_{\infty,m}) = W(R_{\infty,m}^{\flat}) \twoheadrightarrow \widehat{R}_{\infty,m}$ is principal.
	So, let $x$ be an element of $A_{\inf}(R_{\infty,m})$ such that $\theta(x) = 0$.
	Then, we may write $x = \xi y$, for some $y$ in $A_{\inf}(\overline{R})$ and $\xi$ a generator of $\textup{ker } \theta \subset A_{\inf}(O_{F_{\infty}})$.
	By the functoriality of the $p\textrm{-typical}$ Witt vector construction, $A_{\inf}(\overline{R})$ admits a natural action of $G_R$, and since $\xi$ is a nonzerodivisor in $A_{\inf}(\overline{R})$ (see \cite[Lemma 3.10]{bhatt-morrow-scholze-1}), it follows that $y$ is an element of $A_{\inf}(\overline{R})^{H_{R,m}}$.
	As taking Galois invariant elements commutes with left exact functors, in particular, the Witt vector functor, therefore, from Lemma \ref{lem:Galinv_Rinftym} (2) we conclude that $y$ is an element of $A_{\inf}(\overline{R})^{H_{R,m}} = W((\overline{R}^{\flat})^{H_{R,m}}) = W(R_{\infty,m}^{\flat}) = A_{\inf}(R_{\infty,m})$.
	Hence, it follows that the ideal $\textup{ker } \theta \subset A_{\inf}(R_{\infty,m})$ is principal, and $\widehat{R}_{\infty,m}$ is perfectoid.
	This completes our proof.
\end{proof}

\subsection{Localisation and its applications}\label{subsec:localisation_Rbar}\label{subsubsec:localisation_Rbar}

Let $\mathscr{P}(\overline{R})$ denote the set of height one prime ideals of $\overline{R}$ containing $p$.
The set $\mathscr{P}(\overline{R})$ is equipped with a transitive action of $G_R$ (see \cite[Theorem 9.3]{matsumura}).
For each prime $\mathfrak{p}$ in $\mathscr{P}(\overline{R})$, set $\GRp \coloneq \{g \in G_R \textrm{ such that } g(\mathfrak{p}) = \mathfrak{p}\}$, i.e.\ $\GRp$ is the decomposition subgroup of $G$ at $\mathfrak{p}$.

Let us set $O_L \coloneq (R_{(p)})^{\wedge}$ and $L \coloneq O_L[1/p]$.
For each $\mathfrak{p}$ in $\mathscr{P}(\overline{R})$, let $\Lbarp$ denote an algebraic closure of $L$ with ring of integers $\OLbarp$ containing $\Rbarp$.
Set $\GRhatp \coloneq \Gal(\Lbarp/L)$, so that we have a natural homomorphism $\GRhatp \rightarrow G_R$ which factors as $\GRhatp \twoheadrightarrow \GRp \subset G_R$.
Note that for each $\mathfrak{p}$ in $\mathscr{P}(\overline{R})$, we have a natural inclusion $\overline{R} \subset \OLbarp$, and hence we have a (non-canonical) isomorphism of Galois groups $\GRhatp \isomorphic G_L$.

Now, for each $\mathfrak{p}$ in $\mathscr{P}(\overline{R})$, let $\Cpplus$ denote the $p\textrm{-adic}$ completion of $\OLbarp$ and set $\Cp \coloneq \Fr(\Cpplus)$.
Then, $\Cp$ is a complete algebraically closed valuation field equipped with a continuous action of $\GRhatp$ and $(\Cpplus)^{\GRhatp} = O_L$ (see \cite[Theorem 1]{hyodo}).
Moreover, let $\Cplusp$ denote the $p\textrm{-adic}$ completion of $\Rbarp$ equipped with a continuous action of $\GRp$, and set $\mathbb{C}(\mathfrak{p}) \coloneq \Cplusp[1/p]$.
Then, an argument similar to \cite[Lemma 2.1]{abhinandan-relative-wach-ii} shows that for each $\mathfrak{p}$ in $\mathscr{P}(\overline{R})$, the ring $\Cplusp$ is $p\textrm{-torsion}$ free, the natural map $\Rbarp \rightarrow \Cplusp$ is injective, and its reduction modulo $p^n$, i.e. $\Rbarp/p^n\Rbarp \rightarrow \Cplusp/p^n\Cplusp$ is also injective.
Additionally, we have that the natural map $\Rbarp/p^n\Rbarp \rightarrow \OLbarp/p^n\OLbarp$ is injective, and thus the natural homomorphism of rings $\Cplusp \rightarrow \Cpplus$ is injective, in particular, inside $\Cp$ we have that $\Cplusp = \mathbb{C}(\mathfrak{p}) \cap \Cpplus$.
Note that the natural injective homomorphism of rings $\Cplusp \hookrightarrow \Cpplus$ is further compatible with the respective actions of $\GRhatp$, where the action of $\GRhatp$ on the left-hand term factors through $\GRhatp \twoheadrightarrow \GRp$.
In particular, we see that $\Cplusp^{\GRp} = O_L$.

Next, we see that reduction modulo $p^n$ of the natural injective homomorphisms of rings $\overline{R} \hookrightarrow \Rbarp \hookrightarrow O_{\Lbar(\mathfrak{p})}$, need not be injective.
However, taking the product over all $\mathfrak{p}$ in $\mathscr{P}(\overline{R})$ and arguing as in \cite[Lemma 2.2]{abhinandan-relative-wach-ii} shows that for each $n \geqslant 1$, we have the following natural $R\textrm{-linear}$ injective homomorphism:
\begin{equation}\label{eqn: injectivity of localization R/pR}
	\overline{R}/p^n\overline{R} \longhookrightarrow \textstyle\prod_{\pins} (\overline{R})_{\mathfrak{p}}/p^n (\overline{R})_{\mathfrak{p}}.
\end{equation}
Taking the limit over $n$, yields the following natural $R\textrm{-linear}$ injective homomorphisms:
\begin{equation}\label{eq:cplus_gequiv}
	\widehat{\overline{R}} \longhookrightarrow \textstyle\prod_{\pins} \Cplusp \longhookrightarrow \textstyle\prod_{\pins} \Cpplus.
\end{equation}
In particular, we see that $\widehat{\overline{R}} = \widehat{\overline{R}}[1/p] \cap \textstyle\prod_{\pins} \Cplusp \subset \big(\textstyle\prod_{\pins} \Cplusp\big)[1/p]$.

Note that in \eqref{eq:cplus_gequiv} the leftmost term admits a natural action of $G_R$, the middle term admits a natural action of $\prod_{\pins}\GRp$ and the rightmost term admits a natural action of $\prod_{\pins} \GRhatp$.
The two homomorphisms in \eqref{eq:cplus_gequiv} are compatible with these respective actions.
Moreover, a reasoning similar to \cite[Remarque 3.3.2]{brinon-relatif} shows that the middle term of \eqref{eq:cplus_gequiv} may be equipped with an action of $G_R$ and the left homomorphism in \eqref{eq:cplus_gequiv} is equivariant with respect to this action of $G_R$.
Furthermore, an argument similar to \cite[Lemma 2.4]{abhinandan-relative-wach-ii} using \cite[Corollary 3.6]{andreatta-iovita-semistable}, shows that the $O_L\textrm{-algebras}$ $\Cplusp$ and $\Cpplus$ are perfectoid in the sense of \cite[Definition 3.5]{bhatt-morrow-scholze-1}.

Let us now turn to some applications of the discussion above which will be used in studying some explicit periods in Section \ref{subsubsec:explicit_periods}.
In the next lemma, we consider $\overline{R}/p\overline{R}$ as an $R\textrm{-module}$.

\begin{lem}\label{lem:inject_mult_x_i}\label{lem: injectivity of multiplication by x_i}
	For each $1 \leqslant i \leqslant d$ and $n \geqslant 1$, the natural $\textrm{multiplication-by-}x_i$ map $\overline{R}/p^n\overline{R} \xrightarrow{x_i} \overline{R}/p^n\overline{R}$ is injective.
	Moreover, the natural $\textrm{multiplication-by-}x_i$ map $\widehat{\overline{R}} \xrightarrow{x_i} \widehat{\overline{R}}$ is also injective.
\end{lem}
\begin{proof}
	The statement is clear for $1 \leqslant i \leqslant a$ because $x_i$ is invertible in $R$.
	So, assume that $a+1 \leqslant i \leqslant d$.
	Let $\mathfrak{p}$ be an element of $\mathscr{P}(\overline{R})$, and note that $\mathfrak{p} \cap R \subset R$ is a prime ideal.
	As $\overline{R}$ is a normal extension of $R$, by the Cohen--Seidenberg theorem we get that $\mathrm{ht}(\mathfrak{p} \cap R) = 1$, in particular,  $\mathfrak{p} \cap R$ cannot contain the $R\textrm{-regular}$ sequence $\{p,x_i\}$.
	Therefore, we see that $x_i$ is not in $\mathfrak{p}$, i.e.\ $x_i$ is invertible in $(\overline{R})_{\mathfrak{p}}$.
	Now, consider the following commutative diagram:
	\begin{equation*}
		\begin{tikzcd}
			\overline{R}/p^n\overline{R} \arrow[r, "\eqref{eqn: injectivity of localization R/pR}"] \arrow[d, "x_i"] & \textstyle\prod_{\pins} (\overline{R})_{\mathfrak{p}}/p^n (\overline{R})_{\mathfrak{p}} \arrow[d, "x_i"]\\
			\overline{R}/p^n\overline{R} \arrow[r, "\eqref{eqn: injectivity of localization R/pR}"] & \textstyle\prod_{\pins} (\overline{R})_{\mathfrak{p}}/p^n (\overline{R})_{\mathfrak{p}},
		\end{tikzcd}
	\end{equation*}
	where the vertical arrows denote the natural multiplication map.
	As the horizontal arrows and the right vertical arrow in the diagram are injective, therefore, we conclude that the left vertical arrow is also injective.
	Moreover, taking the limit over $n$ yields that the natural multiplication map $\widehat{\overline{R}} \xrightarrow{x_i} \widehat{\overline{R}}$ is also injective.
	This allows us to conclude.
\end{proof}

\begin{lem}\label{lem:B_mult_xi_injective}
	Let $B \subset \overline{R}$ be a finite normal $R_{\infty}\textrm{-subalgebra}$.
	Then, for each $1 \leqslant i \leqslant d$ and $n \geqslant 1$, the natural $\textrm{multiplication-by-}x_i$ map $B/p^n B \xrightarrow{x_i} B/p^n B$ is injective.
\end{lem}
\begin{proof}
	The natural inclusion $B \subset \overline{R}$ induces the following commutative diagram:
	\begin{equation*}
		\begin{tikzcd}
			B/p^n B \arrow[r, hookrightarrow] \arrow[d, "x_i"] & \overline{R}/p^n\overline{R} \arrow[d, "x_i"]\\
			B/p^n B \arrow[r, hookrightarrow] & \overline{R}/p^n\overline{R},
		\end{tikzcd}
	\end{equation*}
	where the vertical arrows are multiplication by $x_i$ and the horizontal arrows are injective because we have $B = B[1/p] \cap \overline{R} \subset \overline{R}[1/p]$ since $B$ is a normal domain.
	As the right vertical arrow is injective by Lemma \ref{lem: injectivity of multiplication by x_i}, the claim follows.
\end{proof}

\begin{lem}\label{lem:rbarhat_noalmostzero}
	The rings $\widehat{\overline{R}}$ and $\overline{R}/p^n\overline{R}$ contain no almost zero elements.
\end{lem}
\begin{proof}
	From Proposition \ref{prop:properties_Rbarhat}, recall that $\widehat{\overline{R}}$ is $p\textrm{-torsion}$ free, therefore, it has no almost zero elements.
	To show that $\overline{R}/p^n\overline{R}$ has no almost zero elements, using the injective map in \eqref{eqn: injectivity of localization R/pR}, it is enough to show that $(\overline{R})_{\mathfrak{p}}/p^n (\overline{R})_{\mathfrak{p}}$ has no almost zero elements for each $\mathfrak{p}$ in $\mathscr{P}(\overline{R})$, where $\mathscr{P}(\overline{R})$ denotes the set of height one prime ideals of $\overline{R}$ containing $p$.

	Fix some $\mathfrak{p}$ in $\mathscr{P}(\overline{R})$, and note that $R_{(p)} \rightarrow (\overline{R})_{\mathfrak{p}}$ is an extension of valuation rings.
	In particular, the valuation of $R_{(p)}$ uniquely extends to $(\overline{R})_{\mathfrak{p}}$, which we denote by $v_{\mathfrak{p}}$ and normalise so that $v_{\mathfrak{p}}(p)$ equals $1$.
	Assume that $a$ is an element of $(\overline{R})_{\mathfrak{p}}$ such that its reduction modulo $p^n$ is almost zero.
	Then, for each $\varepsilon > 0$, one has that $p^{\varepsilon}a \in p^n (\overline{R})_{\mathfrak{p}}$.
	Therefore, $v_{\mathfrak{p}}(a) + \varepsilon \geqslant n$ for all $\varepsilon > 0$, and we conclude that $v_{\mathfrak{p}}(a) \geqslant n$.
	Hence, $a$ must be in $p^n (\overline{R})_{\mathfrak{p}}$, as claimed.
\end{proof}

\begin{prop}\label{X-divisible}
	For $d > a$ and each $a+1 \leqslant i \leqslant d$, we have that
	\begin{equation*}
		\cap_{m \geqslant 0} \hspace{1mm} x_i^m (\overline{R}/p^n\overline{R}) = 0 \quad \textrm{and} \quad \cap_{m \geqslant 0} (x_i^m \CRp) = 0.
	\end{equation*}
\end{prop}
\begin{proof} 
	Note that the second equality follows from the first, and it is enough to show that for each $a+1 \leqslant i \leqslant d$ we have $\cap_{m \geqslant 0} \hspace{1mm} x_i^m (\overline{R}/p\overline{R}) = 0$.
	Indeed, using the claim for $n = 1$ and an easy induction yields the claim for $n \geqslant 2$.
	
	To show the reduced claim, note that we have 
	\begin{equation*}
		R_{\infty}/pR_{\infty} = (R/pR)[\mu_{p^{\infty}}, x_1^{1/p^{\infty}}, \ldots, x_d^{1/p^{\infty}}],
	\end{equation*}
	and we directly observe that for $a+1 \leqslant i \leqslant d$, we have
	\begin{equation}\label{eqn: intersection for R_{infty}}
		\cap_{m \geqslant 0} \hspace{1mm} x_i^m  (R_\infty/pR_\infty) = 0.
	\end{equation} 
	
	Next, let us consider a finite normal $R\textrm{-subalgebra}$ $R' \subset \overline{R}$.
	Then, by Andr\'e's perfectoid Abhyankar Lemma \cite[Th\'eor\`eme~0.3.1]{andre-abhyankar}, note that for each $j \geqslant 0$ there exists a homomorphism $f_j \colon R_{\infty}'/pR_{\infty}' \rightarrow N_j$, where $N_j$ is a finite \'etale $R_{\infty}/pR_{\infty}\textrm{-algebra}$ and $\ker (f_j)$ is $\omega_j$-torsion, for $\omega_j = (\zeta_{p^j}-1)x_{a+1} \cdots x_d$.
	By \cite[Lemma~2.4.15]{gabber-ramero-almost-ring}, we may further assume that $N_j$ is finite free over $R_{\infty}/pR_{\infty}$.
	Combining this with \eqref{eqn: intersection for R_{infty}}, we thus obtain that 
	\begin{equation*}
		\cap_{m \geqslant 0} \hspace{1mm} x_i^m (N_j/pN_j) = 0. 
	\end{equation*}
	In particular, we see that any element $a$ in $\cap_{m \geqslant 0} \hspace{1mm} x_i^m(R_{\infty}'/pR_{\infty}')$ is killed by $\omega_j$ for each $j \geqslant 0$.
	As multiplication by $x_k$ on $R_{\infty}'/pR_{\infty}'$ is injective for $1 \leqslant k \leqslant d$ (see Lemma \ref{lem:B_mult_xi_injective}), therefore, we further obtain that $a$ is killed by $\pi_j$ for each $j \geqslant 0$.
	But this is only possible if $a = 0$ because $R_{\infty}'/pR_{\infty}'$ has no almost zero elements (see Lemma \ref{lem:rbarhat_noalmostzero}).
	Hence, we conclude that $\cap_{m \geqslant 0} \hspace{1mm} x_i^m (R_{\infty}'/pR_{\infty}') = 0$.
	
	Finally, let us show that $\cap_{m \geqslant 0} \hspace{1mm} x_i^m(\overline{R}/p\overline{R}) = 0$.
	Let $b$ be any element of $\cap_{m \geqslant 0} \hspace{1mm} x_i^m(\overline{R}/p\overline{R})$.
	Then, we may fix a finite normal $R\textrm{-subalgebra}$ $R' \subset \overline{R}$ such that $b$ is an element of $R'/pR'$.
	Let $b_m$ in $\overline{R}/p\overline{R}$ be such that $x_i^m b_m = a$.
	Again, as multiplication by $x_i$ is injective on $\overline{R}/p\overline{R}$ (see Lemma \ref{lem: injectivity of multiplication by x_i}), therefore, we obtain that $b_m$ is an element of $(\overline{R}/p\overline{R})^{H_{R'}}$ for each $n \geqslant 0$.
	Then, by Lemma~\ref{lem:Ax-Sen_R} (2), note that for any $j \geqslant 1$, one has that $\omega_j b_m$ is in $R_{\infty}'/pR_{\infty}'$ for all $m \geqslant 1$.
	Therefore, we see that $\omega_j b$ is in $\cap_{m \geqslant 0} \hspace{1mm} x_i^m (R_{\infty}'/p R_{\infty}')$, and the previous paragraph yields that $\omega_j b = 0$.
	Again, as multiplication by $x_k$ on $R_{\infty}'/pR_{\infty}'$ is injective for $1 \leqslant k \leqslant d$ (see Lemma \ref{lem:B_mult_xi_injective}), therefore, we further obtain that $\pi_j b = 0$ for each $j \geqslant 0$.
	But this is only possible if $b = 0$ because $\overline{R}/p\overline{R}$ has no almost zero elements (see Lemma \ref{lem:rbarhat_noalmostzero}).
	Hence, we conclude that $\cap_{m \geqslant 0} \hspace{1mm} x_i^m (\overline{R}/p\overline{R}) = 0$.
\end{proof}

\section{\texorpdfstring{$p\textrm{-adic}$}{-} representations and \texorpdfstring{$(\varphi, \Gamma)\textrm{-modules}$}{-}}\label{sec:phiGamma_modules}

The goal of this section is to study \'etale $(\varphi, \Gamma)\textrm{-modules}$ associated to $\mathbb{Z}_p\textrm{-representations}$ of $G_R$.
We shall keep the notation and convention of Section \ref{sec:prelims}.

Let us fix $m$ to be a positive integer coprime to $p$, and set $\calR \coloneq R_{0,m}$.
Recall that in Section \ref{subsubsec:galois_groups}, we defined rings $R_{\infty}$, $\calR_{\infty} \coloneq R_{\infty,m}$ and $R_{\infty,\infty}$ (resp.\ $S_{\infty}$, $S_{\infty,m}$ and $S_{\infty,\infty}$), whose respective $p\textrm{-adic}$ completions are perfectoid.

Let $\calR_{\infty}^{\etale}$ denote the union of finite normal $\calR_{\infty}\textrm{-subalgebras}$ $\calR' \subset \overline{R}$ such that $\calR'[1/p]/\calR_{\infty}[1/p]$ is \'etale.
From Section \ref{subsubsec:galois_groups}, recall that we have $G_R = \Gal(\overline{R}[1/p]/R[1/p])$, and we introduced Galois groups $H_{\calR} = \Gal(\overline{R}[1/p]/\calR_{\infty}[1/p])$ and $\Gamma_{\calR} = \Gal(\calR_{\infty}[1/p]/\calR[1/p]) \isomorphic \Gamma_R$.
Let us set
\begin{equation*}
	\begin{aligned}
		G_{\calR}^{\etale} &\coloneq \Gal(\calR_{\infty}^{\etale}[1/p]/\calR[1/p]),\\
		\HR^{\etale} &\coloneq \Gal(\calR_{\infty}^{\etale}[1/p]/\calR_{\infty}[1/p]).
	\end{aligned}
\end{equation*}
Let $\Delta_{\calR/R} \coloneq \Gal(\calR_{\infty}[1/p]/R_{\infty}[1/p]) \isomorphic \Gal(\calR[1/p]/R[1/p]) \isomorphic (\mathbb{Z}/m\mathbb{Z})^d \rtimes \Gal(F(\zeta_m)/F)$, and note that we have the following exact sequences:
\begin{equation}\label{eq:GR0met_ses}
	\begin{aligned}
		1 &\longrightarrow \HR^{\etale} \longrightarrow G_{\calR}^{\etale} \longrightarrow \Gamma_{\calR} \longrightarrow 1,\\
		1 &\longrightarrow \HR^{\etale} \longrightarrow \mathrm{Gal}(\calR_{\infty}^{\etale}[1/p]/R_{\infty}[1/p]) \longrightarrow \Delta_{\calR/R} \longrightarrow 1.
	\end{aligned}
\end{equation}
Let us also set $G_{\calR/R}^{\etale} \coloneq \Gal(\calR_{\infty}^{\etale}[1/p]/R[1/p])$, and note that we have the following exact sequences:
\begin{equation}\label{eq:GRet_ses}
	\begin{aligned}
		1 &\longrightarrow G_{\calR}^{\etale} \longrightarrow G_{\calR/R}^{\etale} \longrightarrow \Delta_{\calR/R} \longrightarrow 1,\\
		1 &\longrightarrow \HR^{\etale} \longrightarrow G_{\calR/R}^{\etale} \longrightarrow \Gamma_{\calR} \times \Delta_{\calR/R} \longrightarrow 1.
	\end{aligned}
\end{equation}

We shall also consider $R$ and $\calR$ as log-schemes equipped with a log-structure defined by the divisor $\{x_{a+1} \cdots x_d = 0\}$.
We equip the $R\textrm{-algebras}$ $\widehat{\calR}_{\infty}$, $\widehat{R}_{\infty,\infty}$ and $\widehat{\overline{R}}$ with a log-structure induced from the log-structure on $R$.
Set
\begin{equation*}
	\begin{aligned}
		\GR &\coloneq \Gal(\overline{R}[1/p]/\calR[1/p]),\\
		\GammaR &\coloneq \Gal(\calR_{\infty}[1/p]/\calR[1/p]) \isomorphic \Gamma_R,\\
		\HR &\coloneq \textup{ker}(\GR \twoheadrightarrow \GammaR) = \Gal(\overline{R}[1/p]/\calR_{\infty}[1/p]).
	\end{aligned}
\end{equation*}
Analogously, for $\calS \coloneq S_{0,m} \isomorphic \calR[1/(x_{a+1} \cdots x_d)]^{\wedge}$ (see \cite[Section 2.5]{abhinandan-relative-wach-ii}), we have the following Galois groups:
\begin{equation*}
	\begin{aligned}
		\GS &\coloneq \Gal(\overline{S}[1/p]/\calS[1/p]),\\
		\GammaS &\coloneq \Gal(\calS_{\infty}[1/p]/\calS[1/p]) = \Gal(S_{\infty,m}[1/p]/S_{0,m}[1/p]) \isomorphic \Gamma_S,\\
		\HS &\coloneq \textup{ker}(\GS \twoheadrightarrow \GammaS) = \Gal(\overline{S}[1/p]/\calS_{\infty}[1/p]).
	\end{aligned}
\end{equation*}
Note that $\Delta_{\calR/R} \isomorphic \Delta_{\calS/S} \coloneq \Gal(\calS[1/p]/S[1/p])$ and $\GammaS \isomorphic \GammaR$.

\subsection{Period rings}\label{subsec:period_rings}

We begin by recalling some standard construction of perfect and imperfect period rings in the theory of $(\varphi, \Gamma)\textrm{-modules}$.

\subsubsection{Perfect period rings}\label{subsubsec:perfect_period_rings}

Recall that $\widehat{\calR}_{\infty}$, $\widehat{R}_{\infty,\infty}$ and $\widehat{\overline{R}}$ are perfectoid algebras (see Lemma \ref{lem:Rbarhat_perfectoid}).
Additionally, we have the following:
\begin{lem}
	The ring $\widehat{\calR}_{\infty}^{\etale}$ is perfectoid in the sense of \cite[Definition 3.5]{bhatt-morrow-scholze-1}.
\end{lem}
\begin{proof}
	To check that $\widehat{\calR}_{\infty}^{\etale}$ satisfies the assumptions of \cite[Definition 3.5]{bhatt-morrow-scholze-1}, let us set $\pi \coloneq \zeta_{p^2}-1$, then clearly $\pi^p$ divides $p$.
	To check the surjectivity of the natural Frobenius endomorphism on $\calR_{\infty}^{\mathrm{ét}}/p\calR_{\infty}^{\etale}$, note that we may write $\calR_{\infty}^{\etale} = \colim_{\calR'/\calR_{\infty}} \calR'$ as described before the lemma, and it is enough to show that the natural Frobenius endomorphism on $\calR'/p\calR'$ is surjective.
	But this is clear because $\widehat{\calR}'$ is perfectoid in the sense of \cite[Definition 3.5]{bhatt-morrow-scholze-1}: indeed, since $\widehat{\calR}'[1/p]$ is finite \'etale over $\widehat{\calR}_{\infty}[1/p]$, it is perfectoid in the sense of Fontaine by \cite[Lemma 3.20]{bhatt-morrow-scholze-1} and \cite[Theorem 1.10]{scholze-perfectoid}, and $\widehat{\calR}' \subset \widehat{\calR}'[1/p]$ is a subring of integral elements.
	Finally, for the map $\theta \colon A_{\inf}(\calR_{\infty}^{\etale}) = W((\calR_{\infty}^{\etale})^{\flat}) \twoheadrightarrow \widehat{\calR}_{\infty}^{\etale}$, one may employ an argument similar to the proof of Lemma \ref{lem:Rbarhat_perfectoid} to obtain that the ideal $\textup{ker } \theta \subset A_{\inf}(\calR_{\infty}^{\etale})$ is principal.
	This completes our proof.
\end{proof}

Let $\varepsilon \coloneq (1, \zeta_p, \zeta_{p^2}, \ldots)$ and $\overline{\mu} \coloneq \varepsilon-1$ in $O_{F_{\infty}}^{\flat}$, and set $\mu \coloneq [\varepsilon] - 1$ in $A_{\inf}(O_{F_{\infty}}) \coloneq W(O_{F_{\infty}}^{\flat})$.
Let $\chi$ denote the $p\textrm{-adic}$ cyclotomic character, and note that for any $g$ in $G_R$, we have that $g(1+\mu) = (1+\mu)^{\chi(g)}$.
For each $1 \leqslant i \leqslant d$, let $x_i^{\flat} \coloneq (x_i, x_i^{1/p}, \ldots)$ denote an element of $R_{\infty}^{\flat}$ for a compatible system of $p\textrm{-power}$ roots $\{x_i^{1/p^n}\}_{n \geqslant 0}$ in $R_{\infty}$.

Following the standard notation in the theory of $(\varphi, \Gamma)\textrm{-modules}$, we have several perfect period rings in charateristic $p$.
Let us set $\tilde{E}^+(\overline{R}) \coloneq \overline{R}^{\flat}$ and $\tilde{E}(\overline{R}) \coloneq \tilde{E}^+(\overline{R})[1/\overline{\mu}]$ equipped with a natural and continuous action of $G_R$.
Additionally, set $\tilde{E}_{\calR}^+ \coloneq \calR_{\infty}^{\flat} = \tilde{E}^+(\overline{R})^{H_{\calR}}$ and $\tilde{E}_{\calR} \coloneq \tilde{E}_{\calR}^+[1/\overline{\mu}] = \tilde{E}(\overline{R})^{H_{\calR}}$ (see Lemma \ref{lem:Ax-Sen_R}) equipped with a natural and continuous action of $\Gal(\calR_{\infty}[1/p]/R[1/p]) \isomorphic \Gamma_{\calR} \times \Delta_{\calR/R}$.
Finally, let $H \coloneq \Gal(\overline{R}[1/p]/\calR_{\infty}^{\etale}[1/p])$ and set $\tilde{E}_{\calR}^{\etale,+} \coloneq \calR_{\infty}^{\etale,\flat} = \tilde{E}^+(\overline{R})^H$ and $\tilde{E}_{\calR}^{\etale} \coloneq \tilde{E}_{\calR}^{\etale,+}[1/\overline{\mu}] = \tilde{E}(\overline{R})^H$ (see Lemma \ref{lem:Ax-Sen_R}) equipped with a natural and continuous action of $G_{\calR/R}^{\etale} = \Gal(\calR_{\infty}^{\etale}[1/p]/R[1/p])$.

Applying the $p\textrm{-typical}$ Witt vector construction to the period rings described above, we obtain the following perfect period rings in mixed characteristic.
Let us set $A_{\inf}(\overline{R}) \coloneq W(\overline{R}^{\flat}) = W(\tilde{E}^+(\overline{R}))$ and $\tilde{A}(\overline{R}) \coloneq W(\tilde{E}(\overline{R}))$ equipped with a natural and continuous action of $(\varphi, G_R)$.
Additionally, set $A_{\inf}(\calR_{\infty}) \coloneq W(\calR_{\infty}^{\flat}) = W(\tilde{E}_{\calR}^+) = A_{\inf}(\overline{R})^{H_{\calR}}$ and $\tilde{A}_{\calR} \coloneq W(\tilde{E}_{\calR}) = \tilde{A}(\overline{R})^{H_{\calR}}$ (using the discussion above) equipped with a natural and continuous action of $(\varphi, \Gamma_{\calR} \times \Delta_{\calR/R})$.
Furthermore, set $A_{\inf}(\calR_{\infty}^{\etale}) \coloneq W(\calR_{\infty}^{\etale,\flat}) = W(\tilde{E}_{\calR}^{\etale,+}) = A_{\inf}(\overline{R})^H$ and $\tilde{A}_{\calR}^{\etale} \coloneq W(\tilde{E}_{\calR}^{\etale}) = \tilde{A}(\overline{R})^H$ (using the discussion above) equipped with a natural and continuous action of $(\varphi, G_{\calR/R}^{\etale})$.
Finally, let us also set $A_{\inf}(R_{\infty,\infty}) \coloneq W(R_{\infty,\infty}^{\flat})$, admitting the Frobenius on Witt vectors and a continuous action of $\Gamma_{R,\infty}$, and using Lemma \ref{lem:Ax-Sen_R} note that we have $A_{\inf}(R_{\infty,\infty}) = A_{\inf}(\overline{R})^{H_{R,\infty}}$.

Recall that we fixed Teichm\"uller lifts $[x_i^{\flat}]$ in $A_{\inf}(\calR_{\infty})$ for $1 \leqslant i \leqslant d$.
We equip $A_{\inf}(\calR_{\infty})$, $A_{\inf}(R_{\infty,\infty})$ and $A_{\inf}(\overline{R})$ with a log-structure defined by the divisor $\{[x_{a+1}^{\flat}] \cdots [x_d^{\flat}] = 0\}$.
Then, note that the $G_R\textrm{-equivariant}$ surjection $\theta \colon A_{\inf}(\overline{R}) \twoheadrightarrow \CRp$ is compatible with the respective log-structures, and $\textrm{Ker } \theta = \xi A_{\inf}(\overline{R})$.
The map $\theta$ further induces a $(\GammaR \times \Delta_{\calR/R})\textrm{-equivariant}$ (resp.\ $\Gamma_{R,\infty}\textrm{-equivariant}$) surjection $\theta \colon A_{\inf}(\calR_{\infty}) \twoheadrightarrow \widehat{\calR}_{\infty}$ (resp.\ $\theta \colon A_{\inf}(R_{\infty,\infty}) \twoheadrightarrow \widehat{R}_{\infty,\infty}$) compatible with the respective log-structures.

The preceding definitions may also be given for $S$ by replacing $\calR$ with $\calS$, $\calR_{\infty}$ with $\calS_{\infty}$, $R_{\infty,\infty}$ with $S_{\infty,\infty}$, $\overline{R}$ with $\overline{S}$, $G_R$ with $G_S$, and equipping $S$, $\calS$ and the corresponding perfectoid rings with the trivial log-structure.
We shall use those variations in the following.
Additionally, note that the natural $G_S\textrm{-equivariant}$ homomorphism of rings $\overline{R} \rightarrow \overline{S}$ induces a natural $(\varphi, G_S)\textrm{-equivariant}$ homomorphism of rings $A_{\inf}(\overline{R}) \rightarrow A_{\inf}(\overline{S})$ and $\tilde{A}(\overline{R}) \rightarrow \tilde{A}(\overline{S})$.

\begin{lem}\label{lem:Ainf_Rinftym_in_Sinftym}
	The homomorphism $A_{\inf}(\overline{R}) \rightarrow A_{\inf}(\overline{S})$ naturally induces the $(\varphi, \Gamma_{S,\infty})\textrm{-equivariant}$ (resp.\ $(\varphi, \GammaS \times \Delta_{\calR/R})\textrm{-equivariant}$) injective bottom (resp.\ top) horizontal arrow in the following commutative diagram: 
	\begin{equation*}
		\begin{tikzcd}
			A_{\inf}(\calR_{\infty}) \arrow[r, hookrightarrow] \arrow[d, hookrightarrow] & A_{\inf}(\calS_{\infty}) \arrow[d, hookrightarrow]\\
			A_{\inf}(R_{\infty,\infty}) \arrow[r, hookrightarrow] & A_{\inf}(S_{\infty,\infty}).
		\end{tikzcd}
	\end{equation*}
\end{lem}
\begin{proof}
	Use Lemma \ref{lem:Rinftym_in_Sinftym} (2).
\end{proof}

\begin{rem}
	The reader should be warned that the $(\varphi, G_S)\textrm{-equivariant}$ homomorphism of rings $A_{\inf}(\overline{R}) \rightarrow A_{\inf}(\overline{S})$ and $\tilde{A}(\overline{R}) \rightarrow \tilde{A}(\overline{S})$ may fail to be injective (see Example \ref{eg:gabber_example}).
\end{rem}

\subsubsection*{Localisation}

Let $\mathscr{P}(\overline{R})$ denote the set of minimal primes of $\overline{R}$ above $pR \subset R$, and for each prime $\mathfrak{p}$ in $\mathscr{P}(\overline{R})$, let $\Cp$ denote the complete valuation field described in Section \ref{subsubsec:localisation_Rbar} with its ring of integers being the perfectoid algebra $\Cpplus$.
Moreover, recall that we also have the perfectoid algebra $\Cplusp$.
So, we set $A_{\inf}(\Cpplus) \coloneq W(\Cpplusflat)$ (resp.\ $A_{\inf}(\Cplusp) \coloneq W(\Cplusp^{\flat})$) admitting the Frobenius on Witt vectors and continuous $\GRhatp\textrm{-action}$ (resp.\ $\GRp\textrm{-action}$).

Similar to above, we have a $\GRhatp\textrm{-equivariant}$ surjection $\theta \colon A_{\inf}(\Cpplus) \rightarrow \Cpplus$ with $\ker \theta = \xi A_{\inf}(\Cpplus)$ (resp.\ a $\GRp\textrm{-equivariant}$ surjection $\theta \colon A_{\inf}(\Cplusp) \rightarrow \Cplusp$ with $\ker \theta = \xi A_{\inf}(\Cplusp)$).
An argument similar to \cite[Lemma 2.5]{abhinandan-relative-wach-ii} then shows that for each prime $\mathfrak{p}$ in $\mathscr{P}(\overline{R})$, we have natural $(\varphi, \GRhatp)\textrm{-equivariant}$ injective ring homomorphisms $A_{\inf}(\Cplusp) \hookrightarrow A_{\inf}(\Cpplus)$ and $\tilde{A}(\mathbb{C}(\mathfrak{p})) \coloneq W(\mathbb{C}(\mathfrak{p})^{\flat}) \hookrightarrow W(\Cp^{\flat}) \eqcolon \tilde{A}(\Cp)$, where the action of $\GRhatp$ on left-hand terms factor through $\GRhatp \twoheadrightarrow \GRp$.
Moreover, we have a natural $(\varphi, \GRhatp)\textrm{-equivariant}$ identification $A_{\inf}(\Cplusp) = A_{\inf}(\Cpplus) \cap \tilde{A}(\mathbb{C}(\mathfrak{p}))$ as subrings of $\tilde{A}(\Cp)$.

By the functoriality of the tilting construction and the Witt vector construction, note that the action of $G_R$ on $\prod_{\pins} \Cplusp$ described after \eqref{eq:cplus_gequiv}, extends to respective natural actions of $G_R$ on $\prod_{\pins} A_{\inf}(\Cplusp)$ and $\prod_{\pins} \tilde{A}(\mathbb{C}(\mathfrak{p}))$.
Then, an argument similar to \cite[Lemma 2.8]{abhinandan-relative-wach-ii} shows that we have natural $(\varphi, G_R)\textrm{-equivariant}$ injective homomorphisms 
\begin{equation*}
	\begin{aligned}
		A_{\inf}(\overline{R}) &\longhookrightarrow \textstyle\prod_{\pins} A_{\inf}(\Cplusp),\\
		\tilde{A}(\overline{R}) &\longhookrightarrow \textstyle\prod_{\pins} \tilde{A}(\mathbb{C}(\mathfrak{p})),
	\end{aligned}
\end{equation*}
where the right-hand terms are equipped with a $G_R\textrm{-action}$ as described above.
Moreover, we have a natural $(\varphi, G_R)\textrm{-equivariant}$ identification 
\begin{equation*}
	A_{\inf}(\overline{R}) = \tilde{A}(\overline{R}) \cap \textstyle\prod_{\pins} A_{\inf}(\Cplusp) \subset \prod_{\pins} \tilde{A}(\mathbb{C}(\mathfrak{p})).
\end{equation*}

\subsubsection{Imperfect period rings in characteristic \texorpdfstring{$p$}{-}}\label{subsubsec:imperfect_rings_charp}

In this section, our goal is to define imperfect analogues of the characteristic $p$ period rings from Section \ref{subsubsec:perfect_period_rings}.

Recall that in Section \ref{subsubsec:perfect_period_rings} we fixed $\varepsilon \coloneq (1, \zeta_p, \zeta_{p^2}, \ldots)$ in $O_{F_{\infty}}^{\flat}$, and for each $1 \leqslant i \leqslant d$, we fixed $x_i^{\flat} \coloneq (x_i, x_i^{1/p}, \ldots)$ in $R_{\infty}^{\flat}$.
Additionally, we fixed topological generators $\{\gamma_0, \gamma_1, \ldots, \gamma_d\}$ of $\Gamma_R$ in Section \ref{subsubsec:galois_groups} such that we have $\gamma_j(x_i^{\flat}) = \varepsilon x_i^{\flat}$, if $j = i$, or $x_i^{\flat}$, if $j \neq i$.

Let us set $\calR_{\square} \coloneq R_{\square,0,m} = R_{\square}[\zeta_m, x_1^{1/m}, \ldots, x_d^{1/m}]$, and note that $\calR = R_{0,m}$ is a $p\textrm{-completed}$ \'etale algebra over $\calR_{\square}$.
Set
\begin{equation*}
	E_{\calR,\square}^+ \coloneq \kappa[\zeta_m, (x_1^{\flat})^{\pm 1/m}, \ldots, (x_a^{\flat})^{\pm 1/m}, (x_{a+1}^{\flat})^{1/m}, \ldots, (x_{d}^{\flat})^{1/m}]]\llbracket \overline{\mu} \rrbracket.
\end{equation*}
By definition, there exists a natural injective homomorphism $E_{\calR,\square}^+ \hookrightarrow \tilde{E}_{\calR}^+$, and its image is stable under the Frobenius endomorphism $\varphi$ and the action of $\Gamma_{\calR} \times \Delta_{\calR/R}$ on $\tilde{E}_{\calR}^+$; we equip $E_{\calR,\square}^+$ with the induced $(\varphi, \Gamma_{\calR} \times \Delta_{\calR/R})\textrm{-action}$.
Furthermore, note that we have an injective homomorphism $\overline{\iota} \colon \calR_{\square}/p\calR_{\square} \rightarrow E_{\calR,\square}^+$ defined by the $\kappa[\zeta_m]\textrm{-linear}$ map sending $x_i^{1/m} \mapsto (x_i^{\flat})^{1/m}$, and it is easy to see that $\overline{\iota}$ extends to an isomorphism of rings $(\calR_{\square}/p\calR_{\square})\llbracket \overline{\mu} \rrbracket \isomorphic E_{\calR,\square}^+$ (enough to check modulo $\overline{\mu}$ since both source and target are $\overline{\mu}\textrm{-adically}$ complete and $\overline{\mu}\textrm{-torsion-free}$).
Additionally, note that the Frobenius endomorphism of $(\calR_{\square}/p\calR_{\square})\llbracket \overline{\mu} \rrbracket \isomorphic E_{\calR,\square}^+$ is finite and faithfully flat of degree $p^{d+1}$.

Let $\ER^+$ denote the $\overline{\mu}\textrm{-adic}$ completion of the unique extension of the injective homomorphism $E_{\calR,\square}^+ \rightarrow \tilde{E}_{\calR}^+$ along the $p\textrm{-adically}$ completed \'etale map $\calR_{\square} \rightarrow \calR$.
Then, there exists a natural injective homomorphism $\ER^+ \hookrightarrow \tilde{E}_{\calR}^+$, and its image is stable under the Frobenius and the action of $\Gamma_{\calR} \times \Delta_{\calR/R}$ on $\tilde{E}_{\calR}^+$; we equip $\ER^+$ with the induced $(\varphi, \Gamma_{\calR} \times \Delta_{\calR/R})\textrm{-action}$.
Furthermore, the injective homomorphism $\overline{\iota} \colon \calR_{\square}/p\calR_{\square} \hookrightarrow E_{\calR,\square}^+ \hookrightarrow \ER^+$, and the isomorphism $(\calR_{\square}/p\calR_{\square})\llbracket \overline{\mu} \rrbracket \isomorphic E_{\calR,\square}^+ \hookrightarrow \ER^+$ naturally extend to a unique injective homomorphism $\overline{\iota} \colon (\calR/p\calR) \hookrightarrow \ER^+$ and an isomorphism of rings $(\calR/p\calR)\llbracket \overline{\mu} \rrbracket \isomorphic \ER^+$.
Additionally, note that the Frobenius on $(\calR/p\calR)\llbracket \overline{\mu} \rrbracket \isomorphic \ER^+$ is finite and faithfully flat of degree $p^{d+1}$.

Set $\ER \coloneq \ER^+[1/\mu] \hookrightarrow \tilde{E}_{\calR}$ and note that the Frobenius endomorphism $\varphi$ and the continuous action of $\Gamma_{\calR} \times \Delta_{\calR/R}$ on $\ER^+$ naturally extend to $\ER$.
Similar to above, the induced Frobenius endomorphism $\varphi$ on $\ER$ is finite and faithfully flat of degree $p^{d+1}$.
From the discussion in Appendix \ref{subsec:Kedlaya_Liu_ii}, we know that the ring $\ER$ is the decompletion and deperfection of the ring $\tilde{E}_{\calR}$ (see Definition \ref{defi:decompleted_rings}, Lemma \ref{lem:decompleting_tower} and Proposition \ref{prop:ar+_decomplete}).

Let $\pazocal{C} \coloneq \textup{F\'Et}(\ER)$ denote the category of finite \'etale extensions of $\ER$ inside $\tilde{E}_{\calR}^{\etale}$ (see Theorem \ref{thm:fet_equivalence}), and set
\begin{equation}\label{eq:decompleted_ring_modp}
	\ER^{\etale} \coloneq \colim_{\ER' \in \pazocal{C}} \ER' = \cup_{\ER' \in \pazocal{C}} \ER'.
\end{equation}
Then, $\ER^{\etale}$ is a subring of $\tilde{E}_{\calR}^{\etale}$, stable under the action of $(\varphi, G_{\calR/R}^{\etale})$ on the latter.
Moreover, $\tilde{E}_{\calR}^{\etale}$ may be identified with the completed perfection of $\ER^{\etale}$ (similar to Lemma \ref{lem:decompleting_tower}), in other words, $\ER^{\etale}$ is the decompletion of $\tilde{E}_{\calR}^{\etale}$.
Additionally, from Corollary \ref{cor:Gal_isoms} we have that $\ER^{\etale}$ is Galois over $\ER$ and $\Gal(\ER^{\etale}/\ER) \isomorphic \HR^{\etale}$.
In particular, we obtain a $(\varphi, \Gamma_{\calR} \times \Delta_{\calR/R})\textrm{-equivariant}$ isomorphism $\ER \isomorphic (\ER^{\etale})^{\HR^{\etale}}$.
Finally, let us set $\ER^{\etale,+} \coloneq \ER^{\etale} \cap \tilde{E}_{\calR}^{\etale,+} \subset \tilde{E}_{\calR}^{\etale}$, and note that it is stable under the Frobenius and $G_{\calR/R}^{\etale}\textrm{-action}$ on $\tilde{E}_{\calR}^{\etale}$.
Then, from the discussion above we have a $(\varphi, \Gamma_{\calR} \times \Delta_{\calR/R})\textrm{-equivariant}$ isomorphism $\ER^+ \isomorphic (\ER^{\etale,+})^{\HR^{\etale}}$.

\subsubsection{Imperfect period rings in mixed characteristic}\label{subsubsec:imperfect_rings_mixedchar}

In this section, our goal is to define imperfect analogues of the mixed characteristic period rings from Section \ref{subsubsec:perfect_period_rings}.
In particular, the discussion below is a mixed charactersitic lift of the discussion in Section \ref{subsubsec:imperfect_rings_charp}.

Recall that we set $\mu \coloneq [\varepsilon] - 1$ in $A_{\inf}(O_{F_{\infty}})$.
Let $\chi$ denote the $p\textrm{-adic}$ cyclotomic character, and note that for any $g$ in $G_R$, we have that $g(1+\mu) = (1+\mu)^{\chi(g)}$.
For each $1 \leqslant i \leqslant d$, let us fix Teichm\"uller lifts $[x_i^{\flat}]$ in $A_{\inf}(R_{\infty})$.
Then, for the topological generators $\{\gamma_0, \gamma_1, \ldots, \gamma_d\}$ of $\Gamma_R$, we have that $\gamma_j([x_i^{\flat}]) = (1+\mu)[x_i^{\flat}]$, if $j = i$, or $[x_i^{\flat}]$, if $j \neq i$.

Recall that we set $\calR_{\square} = R_{\square}[\zeta_m, x_1^{1/m}, \ldots, x_d^{1/m}]$, and $\calR$ is a $p\textrm{-completed}$ \'etale algebra over $\calR_{\square}$.
Set 
\begin{equation*}
	A_{\calR,\square}^+ \coloneq (p, \mu)\textrm{-adic completion of } O_F[\zeta_m][\mu, [x_1^{\flat}]^{\pm 1/m}, \ldots, [x_a^{\flat}]^{\pm 1/m}, [x_{a+1}^{\flat}]^{1/m}, \ldots, [x_{d}^{\flat}]^{1/m}].
\end{equation*}
By definition, there exists a natural injective homomorphism $A_{\calR,\square}^+ \hookrightarrow A_{\inf}(\calR_{\infty})$, and its image is stable under the Witt vector Frobenius endomorphism $\varphi$ and the action of $\Gamma_{\calR} \times \Delta_{\calR/R}$ on $\tilde{A}_{\calR}^+$; we equip $A_{\calR,\square}^+$ with the induced $(\varphi, \Gamma_{\calR} \times \Delta_{\calR/R})\textrm{-action}$.
Furthermore, note that we have an injective homomorphism $\iota \colon \calR_{\square} \rightarrow A_{\calR,\square}^+$ defined by the $O_F[\zeta_m]\textrm{-linear}$ map sending $x_i^{1/m} \mapsto [x_i^{\flat}]^{1/m}$, and it is easy to see that $\iota$ extends to an isomorphism of rings $\calR_{\square}\llbracket \mu \rrbracket \isomorphic A_{\calR,\square}^+$ (enough to check modulo $\mu$ since both source and target are $\mu\textrm{-adically}$ complete and $\mu\textrm{-torsion-free}$).
We extend the Frobenius endomorphism on $\calR_{\square}$ to a Frobenius endomorphism $\varphi$ on $\calR_{\square}\llbracket \mu \rrbracket$ by setting $\varphi(\mu) = (1+\mu)^p-1$.
Then, the Frobenius on $\calR_{\square}\llbracket \mu \rrbracket$ is finite and faithfully flat of degree $p^{d+1}$ .
Moreover, by the preceding discussion, it also follows that the injective homomorphism $\iota$ and the isomorphism $\calR_{\square}\llbracket \mu \rrbracket \isomorphic A_{\calR,\square}^+$ are Frobenius-equivariant, and respective lifts of the homomorphism $\overline{\iota}$ and the isomorphism $(\calR_{\square}/p\calR_{\square})\llbracket \mu \rrbracket \isomorphic E_{\calR,\square}^+$ from Section \ref{subsubsec:imperfect_rings_charp}.

Let $\AR^+$ denote the $(p, \mu)\textrm{-adic}$ completion of the unique extension of the injective homomorphism $A_{\calR,\square}^+ \rightarrow A_{\inf}(\calR_{\infty})$ along the $p\textrm{-adically}$ completed \'etale map $\calR_{\square} \rightarrow \calR$.
Then, there exists a natural injective homomorphism $\AR^+ \hookrightarrow A_{\inf}(\calR_{\infty})$, and its image is stable under the Witt vector Frobenius and the action of $\Gamma_{\calR} \times \Delta_{\calR/R}$ on $A_{\inf}(\calR_{\infty})$; we equip $\AR^+$ with the induced $(\varphi, \Gamma_{\calR} \times \Delta_{\calR/R})\textrm{-action}$.
Furthermore, the injective homomorphism $\iota \colon \calR_{\square} \hookrightarrow A_{\calR,\square}^+ \hookrightarrow \AR^+$ and the isomorphism $\calR_{\square}\llbracket \mu \rrbracket \isomorphic A_{\calR,\square}^+ \hookrightarrow \AR^+$, naturally extend to a unique injective homomorphism $\iota \colon \calR \rightarrow \AR^+$ and an isomorphism of rings $\calR\llbracket \mu \rrbracket \isomorphic \AR^+$.
Similar to above, the Frobenius endomorphism on $\calR$ extends to a Frobenius endomorphism $\varphi$ on $\calR\llbracket \mu \rrbracket$ so that $\varphi(\mu) = (1+\mu)^p-1$.
Then, the Frobenius on $\calR\llbracket \mu \rrbracket$ is finite and faithfully flat of degree $p^{d+1}$ .
Moreover, by the preceding discussion, it is easy to see that the injective homomorphism $\iota$ and the isomorphism $\calR\llbracket \mu \rrbracket \isomorphic \AR^+$ are Frobenius-equivariant.
In particular, the induced Frobenius endomorphism $\varphi$ on $\AR^+$ is finite and faithfully flat of degree $p^{d+1}$, and we have that $\varphi^*(\AR^+) \coloneq \AR^+ \otimes_{\varphi, \AR^+} \AR^+ \isomorphic \oplus_{\alpha} \varphi(\AR^+) u_{\alpha}$, where $u_{\alpha} \coloneq (1+\mu)^{\alpha_0} [x_1^{\flat}]^{\alpha_1/m} \cdots [x_d^{\flat}]^{\alpha_d/m}$ for $\alpha = (\alpha_0, \alpha_1, \ldots, \alpha_d)$ a $(d+1)\textrm{-tuple}$ with $\alpha_i$ in $\{0, 1, \ldots, p-1\}$, for $0 \leqslant i \leqslant d$.

Set $\AR \coloneq \AR^+[1/\mu]^{\wedge} \hookrightarrow \tilde{A}_{\calR}$ as the $p\textrm{-adic}$ completion, and note that the Frobenius endomorphism $\varphi$ and the continuous action of $\GammaR \times \Delta_{\calR/R}$ on $\AR^+$ naturally extend to $\AR$.
Similar to above, the induced Frobenius endomorphism $\varphi$ on $\AR$ is finite and faithfully flat of degree $p^{d+1}$ and $\varphi^*(\AR) \coloneq \AR \otimes_{\varphi, \AR} \AR \isomorphic \oplus_{\alpha} \varphi(\AR) u_{\alpha} = (\oplus_{\alpha} \varphi(\AR^+) u_{\alpha}) \otimes_{\varphi(\AR^+)} \varphi(\AR) \lisomorphic \AR^+ \otimes_{\varphi, \AR^+} \AR$.
It is clear that we have a $(\varphi, \GammaR \times \Delta_{\calR/R})\textrm{-equivariant}$ identification $\AR/p\AR \isomorphic \ER$.
Additionally, from the discussion in Appendix \ref{subsec:Kedlaya_Liu_ii}, we know that the ring $\AR$ is the decompletion and deperfection of the ring $\tilde{A}_{\calR}$ (see Definition \ref{defi:decompleted_rings}, Lemma \ref{lem:decompleting_tower} and Proposition \ref{prop:ar+_decomplete}).

The imperfect period rings analogous to the ones discussed above also make sense if we replace the $O_F[\zeta_m]\textrm{-algebra}$ $\calR$ with the $p\textrm{-adically}$ complete ring $\calS = \calR[1/(x_{a+1} \cdots x_d)]^{\wedge}$ (see \cite[Section 2.5]{abhinandan-relative-wach-ii}).
In the following, we shall freely use period rings for $\calS$ analogous to the ones defined for $\calR$ (see \cite[Section 2.5]{abhinandan-relative-wach-ii}).
As $\Delta_{\calR/R} \isomorphic \Delta_{\calS/S}$ and $\GammaS \isomorphic \GammaR$, therefore, we have a natural $(\varphi, \GammaR \times \Delta_{\calR/R})\textrm{-equivariant}$ homomorphism $\AR^+ \rightarrow \AS^+$, which extends to a $(\varphi, \GammaR \times \Delta_{\calR/R})\textrm{-equivariant}$ isomorphism 
\begin{equation*}
	\AR^+\big[1/([x_{a+1}^{\flat}] \cdots [x_d^{\flat}])\big]^{\wedge} \isomorphic \AS^+,
\end{equation*}
where the completion is $(p, \mu)\textrm{-adic}$.

Let $\pazocal{C} \coloneq \textup{F\'Et}(\AR) \isomorphic \textup{F\'Et}(\ER)$ denote the category of finite \'etale extensions of $\AR$ inside $\tilde{A}_{\calR}^{\etale}$ (see Theorem \ref{thm:fet_equivalence}), and set
\begin{equation}\label{eq:decompleted_ring_mixedchar}
	\AR^{\etale} \coloneq p\textrm{-adic completion of } \big(\colim_{\AR' \in \pazocal{C}} \AR'\big).
\end{equation}
Then, $\AR^{\etale}$ is a subring of $\tilde{A}_{\calR}^{\etale}$, stable under the action of $(\varphi, G_{\calR/R}^{\etale})$ on the latter.
Moreover, $\tilde{A}_{\calR}^{\etale}$ may be identified with the completed perfection of $\AR^{\etale}$ (similar to Lemma \ref{lem:decompleting_tower}), in other words, $\AR^{\etale}$ is the decompletion of $\tilde{A}_{\calR}^{\etale}$, and we have a natural $(\varphi, G_{\calR/R}^{\etale})\textrm{-equivariant}$ identification $\AR^{\etale}/p\AR^{\etale} \isomorphic \ER^{\etale}$.
Additionally, from Corollary \ref{cor:Gal_isoms} we have that $\AR^{\etale}$ is the $p\textrm{-adic}$ completion of a Galois cover of $\AR$ and $\textup{Aut}(\AR^{\etale}/\AR) \isomorphic \HR^{\etale}$.
In particular, we obtain a $(\varphi, \Gamma_{\calR} \times \Delta_{\calR/R})\textrm{-equivariant}$ isomorphism $\AR \isomorphic (\AR^{\etale})^{\HR^{\etale}}$.
Finally, let us set $\AR^{\etale,+} \coloneq \AR^{\etale} \cap A_{\inf}(\calR_{\infty}^{\etale}) \subset \tilde{A}_{\calR}^{\etale}$, and note that it is stable under the Frobenius and $G_{\calR/R}^{\etale}\textrm{-action}$ on $\tilde{A}_{\calR}^{\etale}$.
Then, from the discussion above we have a $(\varphi, \Gamma_{\calR} \times \Delta_{\calR/R})\textrm{-equivariant}$ isomorphism $\AR^+ \isomorphic (\AR^{\etale,+})^{\HR^{\etale}}$.

\begin{rem}
	For a mixed characteristic period ring $A_{\calR}^{\textrm{deco}}$ discussed above, where ``$\textrm{deco}$'' denotes any of the decorations mentioned above, we set $B_{\calR}^{\textrm{deco}} \coloneq A_{\calR}^{\textrm{deco}}[1/p]$ equipped with an induced action of $\varphi$ and $G_{\calR/R}^{\etale}$.
\end{rem}

\begin{lem}\label{lem:ARet_flat}
	The $(\varphi, G_{\calR/R}^{\etale})\textrm{-equivariant}$ ring homomorphism $\AR \rightarrow \AR^{\etale}$ is flat.
\end{lem}
\begin{proof}
	Note that the reduction modulo $p$ of the natural ring homomorphism $\AR \rightarrow \AR^{\etale}$ coincides with the natural ring homomorphism $\ER \rightarrow \ER^{\etale}$.
	As $\AR^{\etale}$ is $p\textrm{-adically}$ complete and $p\textrm{-torsion}$ free, therefore, by Lemma \ref{lem:local_criterion_flatness} it is enough to show that $\ER^{\etale}$ is flat over $\ER$.
	From \eqref{eq:decompleted_ring_modp}, we have that $\ER^{\etale}$ is the direct colimit of finite \'etale, in particular, flat $\ER\textrm{-algebras}$ $\ER'$, therefore, $\ER^{\etale}$ is flat over $\ER$.
\end{proof}

\subsection{\texorpdfstring{$p\textrm{-adic}$}{-} representations and \'etale \texorpdfstring{$(\varphi, \Gamma)\textrm{-modules}$}{-}}\label{subsec:padic_reps_phiGamma}

In this section, our main goal is to classify $p\textrm{-adic}$ representations of the group $G_{\calR/R}^{\etale}$, in terms of \'etale $(\varphi, \Gamma_{\calR} \times \Delta_{\calR/R})\textrm{-modules}$.

\subsubsection{Mod \texorpdfstring{$p$}{-} representations of \texorpdfstring{$G_{\calR/R}^{\mathrm{\'et}}$}{-}}\label{subsubsec:modp_reps_phiGamma}

We will first look at mod $p$ representations of $G_{\calR/R}^{\etale}$ and their associated $(\varphi, \Gamma_{\calR} \times \Delta_{\calR/R})\textrm{-modules}$ in characteristic $p$.
We begin with a general definition.

Let $A$ be a topological $\mathbb{Z}_p\textrm{-algebra}$ (resp.\ a discrete $\mathbb{Z}/p^n\mathbb{Z}\textrm{-module}$) equipped with a continuous action of a topological group $\Gamma$ and a continuous lifting $\varphi \colon A \rightarrow A$ of the absolute Frobenius modulo $p$ which commutes with the action of $\Gamma$, and such that $A^{\varphi=1} = \mathbb{Z}_p$ (resp.\ $A^{\varphi=1} = \mathbb{Z}/p^n\mathbb{Z}$).
\begin{defi}\label{defi:phiGamma_module}
	A $\varphi\textit{-module}$ over $A$ is a finitely generated $A\textrm{-module}$ equipped with a semilinear action of $\varphi$.
	A $\varphi\textrm{-module}$ $D$ over $A$ is called \textit{\'etale} if the natural $A\textrm{-linear}$ map $A\otimes_{\varphi,A} D \rightarrow D$ is an isomorphism.

	A $(\varphi, \Gamma)\textit{-module}$ over $A$ is a finitely generated $A\textrm{-module}$ equipped with continuous and semilinear actions of $\Gamma$ and $\varphi$, which commute with each other.
	A $(\varphi,\Gamma)\textrm{-module}$ $D$ is called \textit{\'etale} if the natural $A\textrm{-linear}$ map $A\otimes_{\varphi,A} D \rightarrow D$ is an isomorphism.

	A morphism between two \'etale $(\varphi, \Gamma)\textrm{-modules}$ over $A$ is an $A\textrm{-linear}$ map compatible with the respective $(\varphi, \Gamma)\textrm{-actions}$.
	Denote the category of \'etale $(\varphi, \Gamma)\textrm{-modules}$ (resp.\ finite projective) over $A$ as $(\varphi, \Gamma)\textup{-Mod}_{A}^{\etale}$ (resp.\ $(\varphi, \Gamma)\textup{-Mod}_{A}^{\etale,\textup{fproj}}$).
\end{defi}

\begin{lem}\label{lem:modp_phimod_proj}
	An \'etale $\varphi\textrm{-module}$ over $\ER$ is finite projective.
\end{lem}
\begin{proof}
	The claim may be shown by following the argument in \cite[Lemma 7.10]{andreatta-phigamma}.
	We reproduce it for completeness.

	Let $D$ be an \'etale $\varphi\textrm{-module}$ over $\ER$.
	As $\ER$ is noetherian, it suffices to show that $D$ is a flat module over $\ER$.
	Let $\mathfrak{m} \subset \ER$ be a maximal ideal, and let $\widehat{E}_{\calR}$ denote the $\mathfrak{m}\textrm{-adic}$ completion of $\ER$.
	Then, it is enough to show that the $\mathfrak{m}\textrm{-adic}$ completion $\widehat{D}$ of $D$ is a free module over $\widehat{E}_{\calR}$.
	Let $r$ denote the dimension of $D/\mathfrak{m}D$ as an $\ER/\mathfrak{m}\textrm{-vector space}$.
	Choosing a lift in $D$ of a basis of $D/\mathfrak{m}D$ defines a surjective map $f \colon \widehat{E}_{\calR}^{\oplus r} \twoheadrightarrow \widehat{D}$ by Nakayama's Lemma.
	As $D$ is an \'etale $\varphi\textrm{-module}$, we have isomorphisms $1 \otimes \varphi \colon \ER/\mathfrak{m}^{pn}\ER \otimes_{\ER} D/\mathfrak{m}^nD \isomorphic D/\mathfrak{m}^{pn}D$ for each $n \geqslant 1$.
	Using this and arguing by induction, we deduce that for each $n \geqslant 1$, the $\ER/\mathfrak{m}\ER\textrm{-modules}$ $(\ER/\mathfrak{m}^{pn}\ER)^{\oplus r}$ and $D/\mathfrak{m}^{pn}D$ have the same length.
	Hence, we conclude that $f$ must be an isomorphism, thus proving the claim.
\end{proof}

\begin{prop}\label{prop:modp_classify_decomp}
	Fix $m$ to be a positive integer coprime to $p$, and set $\calR \coloneq R_{0,m}$.
	\begin{enumerate}
		\item[\textup{(1)}] Let $T$ be a finite-dimensional $\mathbb{F}_p\textrm{-vector}$ space equipped with a linear action of $G_{\calR/R}^{\etale}$.
			Then,
			\begin{equation}\label{eq:phiGamma_modp}   
				D_{\calR}(T) \coloneq (E_{\calR}^{\etale} \otimes_{\mathbb{F}_p} T)^{H^{\etale}_{\calR}},
			\end{equation}
			has the natural structure of an \'etale $(\varphi, \Gamma_{\calR} \times \Delta_{\calR/R})\textrm{-module}$ over $E_{\calR}$.
			Additionally, the natural multiplication map $E_{\calR}^{\etale} \otimes_{E_{\calR}} E_{\calR}^{\etale} \rightarrow E_{\calR}^{\etale}$ induces a natural isomorphism of $E_{\calR}^{\etale}\textrm{-modules}$
			\begin{equation}\label{eq:phiGamma_comp_iso_modp_imperf}
				E_{\calR}^{\etale} \otimes_{E_{\calR}} D_{\calR}(T) \isomorphic E_{\calR}^{\etale} \otimes_{\mathbb{F}_p} T,
			\end{equation}
			compatible with the respective actions of $\varphi$ and $G_{\calR/R}^{\etale}$.
		
		\item[\textup{(2)}] The following natural functor induces an equivalence of categories:
			\begin{equation}\label{eq:modp_classify_decomp}
				D_{\calR} \colon \mathrm{Rep}_{\mathbb{F}_p}(G_{\calR/R}^{\etale}) \isomorphic (\varphi, \Gamma_{\calR} \times \Delta_{\calR/R})\textup{-Mod}_{E_{\calR}}^{\etale},
			\end{equation}
			with a quasi-inverse functor given as $T_{\calR}(D) \coloneq (E_{\calR}^{\etale} \otimes_{E_{\calR}} D)^{\varphi=1}$. 
	\end{enumerate}
\end{prop}
\begin{proof}
	The proof is analogous to the proof of \cite[Propositions 1.2.4 \& 1.2.6]{fontaine-phigamma} and \cite[Theorem 7.11]{andreatta-phigamma} and is based on \cite[Proposition 4.1.1]{katz73}.
	More precisely, given a finite \'etale morphism $A \rightarrow B$ which is Galois, with Galois group $G$, and $M$ a finitely generated $B\textrm{-module}$ equipped with a semilinear action of $G$, then the $G\textrm{-action}$ on $M$ determines a descent data $\mathrm{p}_{1}^*(M) \isomorphic \mathrm{p}_{2}^*(M)$ over $B \otimes_A B$ satisfying the cocycle condition over $B \otimes_A B \otimes_A B$.
	Here $\mathrm{p}_{1}^*(M)$ denotes the base change of $M$ along the map $B \rightarrow B \otimes_A B$ sending $b \mapsto b \otimes 1$, and similarly $\mathrm{p}_{2}^*(M)$ denotes the base change of $M$ along the map $B \rightarrow B \otimes_A B$ sending $b \mapsto 1 \otimes b$.
	Since an \'etale descent data is always effective (see, for example, \cite[Example B, p.\ 139]{bosch-lutkebohmert-raynaud}), therefore, there exists an $A\textrm{-module}$ $N$ such that $M = B \otimes_A N$, and $N$ is finitely generated over $A$.
	Moreover, it is clear that if $M$ is finite projective over $B$, then $N$ is finite projective over $A$.
	
	Let $T$ be an object of $\mathrm{Rep}_{\mathbb{F}_p}(G_{\calR/R}^{\etale})$.
	From \cite[Theorem 7.12]{scholze-perfectoid} and Corollary \ref{cor:Gal_isoms}, we may identify $\HR^{\etale}$ with the Galois group $\Gal(E_{\calR}^{\etale}/E_{\calR})$.
	As $T$ is finite and $H^{\etale}_{\calR}$ is profinite, there exists an open normal subgroup $H'$ of $H^{\etale}_{\calR}$, i.e.\ a subgroup of finite index, which acts trivially on $T$.
	Set $E_{\calR}' \coloneq (E_{\calR}^{\etale})^{H'}$ as a finite \'etale and galois extension of $E_{\calR}$ with Galois group $G \coloneq H^{\etale}_{\calR}/H'$.
	Then, we see that we have $D_{\calR}(T) = (E_{\calR}^{\etale} \otimes_{\mathbb{F}_p} T)^{H^{\etale}_{\calR}} = (E_{\calR}' \otimes_{\mathbb{F}_p} T)^G$. 
	Now, by applying the discussion of the previous paragraph to the finite \'etale extension $E_{\calR}'$ of $E_{\calR}$, we see that the finite free $E_{\calR}^{\etale}\textrm{-module}$ $E_{\calR}^{\etale} \otimes_{\mathbb{F}_p} T$ equipped with a semilinear (diagonal) action of the Galois group $H_{\calR,}^{\etale}$, descends to a finite projective $E_{\calR}\textrm{-module}$ $D_{\calR}(T) \coloneq (E_{\calR}^{\etale} \otimes_{\mathbb{F}_p} T)^{H^{\etale}_{\calR}}$.
	
	As $E_{\calR}^{\etale} \otimes_{\mathbb{F}_p} T$ is equipped with an action of $G_{\calR/R}^{\etale}$, and a Frobenius structure given by $\varphi \otimes 1$ which is $G_{\calR/R}^{\etale}\textrm{-equivariant}$, therefore, we see that $D_{\calR}(T)$ is equipped with a residual action of $(\varphi, \Gamma_{\calR} \times \Delta_{\calR/R})$ by \eqref{eq:GRet_ses}, and it is easy to check that the natural map $E_{\calR} \otimes_{\varphi, E_{\calR}} D_{\calR}(T) \rightarrow D_{\calR}(T)$ is an isomorphism.
	Additionally, we also have a natural isomorphism of $(\varphi, G_{\calR/R}^{\etale})\textrm{-modules}$ induced by the multiplication map:
	\begin{equation*}
		E_{\calR}^{\etale} \otimes_{E_{\calR}} D_{\calR}(T) \isomorphic E_{\calR}^{\etale} \otimes_{\mathbb{F}_p} T.
	\end{equation*}
	
	Conversely, the arguments of \cite{katz73} show that, for any \'etale $(\varphi, \Gamma_{\calR} \times \Delta_{\calR/R})\textrm{-module}$ $D$ over $\ER$ (finite projective by Lemma \ref{lem:modp_phimod_proj}), the $\mathbb{F}_p\textrm{-vector}$ space $T_{\calR}(D)$ is finite dimensional, and the following natural map of $(\varphi, G_{\calR/R}^{\etale})\textrm{-modules}$ over $E_{\calR}^{\etale}$ is an isomorphism:
	\begin{equation*}
		E_{\calR}^{\etale} \otimes_{\mathbb{F}_p} T_{\calR}(D) \isomorphic E_{\calR}^{\etale} \otimes_{E_{\calR}} D.
	\end{equation*}
	Then, from the preceding discussion, it is easy to check that the functors $D_{\calR}$ and $T_{\calR}$ are quasi inverse to each other.
	This allows us to conclude.
\end{proof}

The following is an easy consequence of Proposition \ref{prop:modp_classify_decomp}:
\begin{cor}\label{prop:modp_classify_perf}
	The following natural functor induces an equivalence of categories:
	\begin{equation}\label{eq:modp_classify_perf}
		\begin{aligned}
			\tilde{D}_{\calR} \colon \mathrm{Rep}_{\mathbb{F}_p}(G_{\calR/R}^{\etale}) &\isomorphic (\varphi, \Gamma_{\calR} \times \Delta_{\calR/R})\textup{-Mod}_{\tilde{E}_{\calR}}^{\etale,\textup{fproj}}\\
				T &\longmapsto (\tilde{E}_{\calR}^{\etale} \otimes_{\mathbb{F}_p} T)^{H^{\etale}_{\calR}},
		\end{aligned}
	\end{equation}
	with a quasi-inverse functor given as $\tilde{T}_{\calR}(\tilde{D}) \coloneq (\tilde{E}_{\calR}^{\etale} \otimes_{\tilde{E}_{\calR}} \tilde{D})^{\varphi=1}$.
	Additionally, for an $\mathbb{F}_p\textrm{-representation}$ $T$ of $G_{\calR/R}^{\etale}$, the natural multiplication map $\tilde{E}_{\calR}^{\etale} \otimes_{\tilde{E}_{\calR}} \tilde{E}_{\calR}^{\etale} \rightarrow \tilde{E}_{\calR}^{\etale}$ induces a natural isomorphism of $\tilde{E}_{\calR}^{\etale}\textrm{-modules}$
	\begin{equation}\label{eq:phiGamma_comp_iso_modp_perf}
		\tilde{E}_{\calR}^{\etale} \otimes_{\tilde{E}_{\calR}} \tilde{D}_{\calR}(T) \isomorphic \tilde{E}_{\calR}^{\etale} \otimes_{\mathbb{F}_p} T,
	\end{equation}
	compatible with the respective actions of $\varphi$ and $G_{\calR/R}^{\etale}$.
\end{cor}

In Proposition \ref{prop:modp_classify_decomp}, by restricting to $G_{\calR}^{\etale}\textrm{-action}$, instead of  $G_{\calR/R}^{\etale}\textrm{-action}$, we obtain the following variant:
\begin{cor}\label{cor:modp_classify_GR0m}
	The following natural functor induces an equivalence of categories:
	\begin{equation}\label{eq:modp_classify_GR0m}
		\begin{aligned}
			\DR \colon \mathrm{Rep}_{\mathbb{F}_p}(G_{\calR}^{\etale}) &\isomorphic (\varphi, \Gamma_{\calR})\textup{-Mod}_{\ER}^{\etale}\\
			T &\longmapsto (\ER^{\etale} \otimes_{\mathbb{F}_p} T)^{H^{\etale}_{\calR}},
		\end{aligned}
	\end{equation}
	with a quasi-inverse functor given as $\TR(D) \coloneq (\ER^{\etale} \otimes_{\ER} D)^{\varphi=1}$.  
	Additionally, for an $\mathbb{F}_p\textrm{-representation}$ $T$ of $G_{\calR}^{\etale}$, the natural multiplication map $\ER^{\etale} \otimes_{\ER} \ER^{\etale} \rightarrow \ER^{\etale}$ induces a natural isomorphism of $\ER^{\etale}\textrm{-modules}$
	\begin{equation}\label{eq:phiGamma_comp_iso_modp_GR0m}
		\ER^{\etale} \otimes_{\ER} \DR(T) \isomorphic \ER^{\etale} \otimes_{\mathbb{F}_p} T,
	\end{equation}
	compatible with the respective actions of $\varphi$ and $G_{\calR}^{\etale}$.
\end{cor}

\subsubsection{\texorpdfstring{$p\textrm{-adic}$}{-} representations of \texorpdfstring{$G_{\calR/R}^{\mathrm{\'et}}$}{-}}\label{subsubsec:padicreps_phiGamma}

In this section, we will lift the classification of mod $p$ representations from Proposition \ref{prop:modp_classify_decomp} to $p\textrm{-adic}$ representations of $G_{\calR/R}^{\etale}$.
We begin by noting some simple observations which will be useful in the main result Theorem \ref{thm:padic_classify_decomp} below.

\begin{lem}\label{lem:phimod_proj}
	A $p\textrm{-torsion free}$ \'etale $\varphi\textrm{-module}$ over $\AR$ is finite projective.
\end{lem}
\begin{proof}
	Let $D$ be a $p\textrm{-torsion free}$ \'etale $\varphi\textrm{-module}$ over $\AR$.
	As $D$ is finitely generated over the noetherian ring $\AR$, it suffices to show that $D$ is a flat $\AR\textrm{-module}$.
	Observe that $D$ is $p\textrm{-adically}$ complete, so by Lemma \ref{lem:local_criterion_flatness} it is enough to show that $D/pD$ is flat over $\AR/p\AR = \ER$.
	As $D/pD$ is an \'etale $\varphi\textrm{-module}$ over $\ER$, therefore, by Lemma \ref{lem:modp_phimod_proj} we have that $D/pD$ is finite projective over $\ER$, in particular, flat.
\end{proof}

\begin{lem}\label{lem:H_coh_Eet}
	We have that $H^1(\HR^{\etale}, \ER^{\etale}) = 0$ and $H^1(\HR^{\etale}, \AR^{\etale}) = 0$.
\end{lem}
\begin{proof}
	Let us note that we may write $H^1(\HR^{\etale}, \ER^{\etale}) = \colim_{\ER'/\ER}H^1(\Gal(\ER'/\ER), \ER')$, where the colimit is taken over all finite Galois $\ER\textrm{-subalgebras}$ $\ER' \subset \ER^{\etale}$.
	Note that we have $H^i(\Gal(\ER'/\ER), \ER') = 0$ for all $i \geqslant 1$ because the preceding Galois cohomology may be computed using the Amitsur complex of the finite \'etale, hence faithfully flat, map $\ER \rightarrow \ER'$ (for example, see \cite[Example 2.6]{milne-etale-cohomology}), which is well-known to have vanishing cohomology for $i \geqslant 1$.
	This shows the first claim.

	For the second claim, consider the following exact sequence:
	\begin{equation*}
		0 \longrightarrow \AR^{\etale} \xrightarrow{\hspace{1mm} p \hspace{1mm}} \AR^{\etale} \longrightarrow \ER^{\etale} \longrightarrow 0.
	\end{equation*}
	Taking the $\HR^{\etale}\textrm{-invariants}$ in the exact sequence above yields the following exact sequence
	\begin{equation*}
		0 \longrightarrow \AR \xrightarrow{\hspace{1mm} p \hspace{1mm}} \AR \longrightarrow \ER \xrightarrow{\hspace{1mm} 0 \hspace{1mm}} H^1(\HR^{\etale}, \AR^{\etale}) \xrightarrow{\hspace{1mm} p \hspace{1mm}} H^1(\HR^{\etale}, \AR^{\etale}) \longrightarrow 0,
	\end{equation*}
	where the third map is easily seen to be zero.
	We see that multiplication by $p$ is bijective on $H^1(\HR^{\etale}, \AR^{\etale})$, and using this we may write any cocyle $c \colon \HR^{\etale} \rightarrow \AR^{\etale}$ as a sum of coboundaries $\sum_{i \geqslant 0}p^ib_i$ which converges because $\AR^{\etale}$ is $p\textrm{-adically}$ complete.
	Hence, we conclude that $H^1(\HR^{\etale}, \AR^{\etale}) = 0$.
\end{proof}

\begin{thm}\label{thm:padic_classify_decomp}
	Fix $m$ to be a positive integer coprime to $p$, and set $\calR \coloneq R_{0,m}$.
	\begin{enumerate}
		\item[\textup{(1)}] Let $T$ be a $\mathbb{Z}_p\textrm{-representation}$ of $G_{\calR/R}^{\etale}$.
			Then, 
			\begin{equation}\label{eq:phiGamma}
				\DR(T) \coloneq (\AR^{\etale} \otimes_{\mathbb{Z}_p} T)^{H^{\etale}_{\calR}},
			\end{equation}
			has the natural structure of an \'etale $(\varphi, \GammaR \times \Delta_{\calR/R})\textrm{-module}$ over $\AR$.
			Additionally, the natural multiplication map $\AR^{\etale} \otimes_{\AR} \AR^{\etale} \rightarrow \AR^{\etale}$ induces an isomorphism
			\begin{equation}\label{eq:phiGamma_comp_iso}
				\AR^{\etale} \otimes_{\AR} \DR(T) \isomorphic \AR^{\etale} \otimes_{\mathbb{Z}_p} T,
			\end{equation}
			compatible with the respective actions of $\varphi$ and $G_{\calR/R}^{\etale}$.
		
		\item[\textup{(2)}] The following natural functor induces an equivalence of categories:
			\begin{equation}\label{eq:padic_classify_decomp}
				\DR \colon \textup{Rep}_{\mathbb{Z}_p}(G_{\calR/R}^{\etale}) \isomorphic (\varphi,\GammaR \times \Delta_{\calR/R})\textup{-Mod}_{\AR}^{\mathrm{ét}},
			\end{equation}
			with a quasi-inverse functor given as $\TR(D) \coloneq (\AR^{\etale} \otimes_{\AR} D)^{\varphi=1}$.
			The equivalence in \eqref{eq:padic_classify_decomp} restricts to a natural equivalence of categories:
			\begin{equation}\label{eq:free_padic_classify_decomp}
				\DR \colon \textup{Rep}_{\mathbb{Z}_p}^{\textup{free}}(G_{\calR/R}^{\etale}) \isomorphic (\varphi,\GammaR \times \Delta_{\calR/R})\textup{-Mod}_{\AR}^{\etale,\textup{fproj}},
			\end{equation}
			where note that any $p\textrm{-torsion}$ free \'etale $\varphi\textrm{-module}$ over $\AR$ is necessarily finite projective by Lemma \ref{lem:phimod_proj}.
	\end{enumerate}
\end{thm}
\begin{proof}
	Let $T$ be a $\mathbb{Z}_p\textrm{-representation}$ of $G_{\calR/R}^{\etale}$. 
	Note that the functor $\DR$ commutes with tensor products, so we may first assume that $p^nT = 0$ for some integer $n \geqslant 1$.
	Additionally, without loss of generality, we may further assume that $T$ is a finite free $\mathbb{Z}/p^n\mathbb{Z}\textrm{-module}$.

	For each $1 \leqslant i \leqslant n$, we have
	\begin{equation*}
		\DR(T/p^iT) = (\AR^{\etale} \otimes_{\mathbb{Z}_p} T/p^iT)^{\HR^{\etale}},
	\end{equation*}
	and the natural multiplication map $\AR^{\etale} \otimes_{\AR} \AR^{\etale} \rightarrow \AR^{\etale}$ induces the following $(\varphi, G_{\calR/R}^{\etale})\textrm{-equivariant}$ homomorphism for each $1 \leqslant i \leqslant n$:
	\begin{equation}\label{eq:phiGamma_comp}
		\AR^{\etale} \otimes_{\AR} \DR(T/p^iT) \longrightarrow \AR^{\etale} \otimes_{\mathbb{Z}_p} T/p^iT.
	\end{equation}
	For $i = 1$, from Proposition \ref{prop:modp_classify_decomp} it is clear that $\DR(T/pT)$ is an \'etale $(\varphi, \GammaR \times \Delta_{\calR/R})\textrm{-module}$ over $\AR$, and \eqref{eq:phiGamma_comp} is an isomorphism by \eqref{eq:phiGamma_comp_iso_modp_imperf}.

	For any $1 \leqslant i \leqslant n$, to show that $\DR(T/p^iT)$ is an \'etale $(\varphi, \GammaR \times \Delta_{\calR/R})\textrm{-module}$ over $\AR$ and \eqref{eq:phiGamma_comp} is an isomorphism, we shall proceed by induction.
	So, assume that this holds for some $1 \leqslant i \leqslant n$, and consider the following $G_{\calR/R}^{\etale}\textrm{-equivariant}$ exact sequence:
	\begin{equation}\label{eq:T_modpi_modpj}
		0 \longrightarrow p^iT/p^{i+1}T \longrightarrow T/p^{i+1}T \longrightarrow T/p^iT \longrightarrow 0.
	\end{equation}
	By extending scalars of the exact sequence \eqref{eq:T_modpi_modpj} along the $(\varphi, G_{\calR/R}^{\etale})\textrm{-equivariant}$ flat homomorphism $\mathbb{Z}_p \rightarrow \AR^{\etale}$, and taking the $\HR^{\etale}\textrm{-invariants}$, we obtain the following exact sequence:
	\begin{equation*}
		0 \longrightarrow \DR(p^iT/p^{i+1}T) \longrightarrow \DR(T/p^{i+1}T) \longrightarrow \DR(T/p^iT) \longrightarrow H^1(\HR^{\etale}, \AR^{\etale} \otimes_{\mathbb{Z}_p} p^iT/p^{i+1}T).
	\end{equation*}
	Using Lemma \ref{lem:H_coh_Eet} and the induction assumption for the $\mathbb{F}_p\textrm{-representation}$ $p^iT/p^{i+1}T$ of $G_{\calR/R}^{\etale}$, we deduce that $H^1(\HR^{\etale}, \AR^{\etale} \otimes_{\mathbb{Z}_p} p^iT/p^{i+1}T) = 0$, in particular, the following is a $(\varphi, \GammaR \times \Delta_{\calR/R})\textrm{-equivariant}$ exact sequence of $\AR\textrm{-modules}$:
	\begin{equation}\label{eq:DR_exact}
		0 \longrightarrow \DR(p^iT/p^{i+1}T) \longrightarrow \DR(T/p^{i+1}T) \longrightarrow \DR(T/p^iT) \longrightarrow 0.
	\end{equation}
	As $\AR$ is a noetherian ring, and by the induction assumption the first and the third term in \eqref{eq:DR_exact} are finitely presented over $\AR$, therefore, it follows that the middle term, i.e.\ $\DR(T/p^iT)$ is also a finitely presented $\AR\textrm{-module}$.

	Now, consider the following commutative diagram with exact rows:
	\begin{equation*}
		\begin{tikzcd}
			0 & \AR^{\etale} \otimes_{\AR} \DR(p^iT/p^{i+1}T) & \AR^{\etale} \otimes_{\AR} \DR(T/p^{i+1}T) & \AR^{\etale} \otimes_{\AR} \DR(T/p^iT) & 0 \\
			0 & \AR^{\etale} \otimes_{\mathbb{Z}_p} (p^iT/p^{i+1}T) & \AR^{\etale} \otimes_{\mathbb{Z}_p} T/p^{i+1}T & \AR^{\etale} \otimes_{\mathbb{Z}_p} T/p^iT & 0,
			\arrow[from=1-1, to=1-2]
			\arrow[from=1-2, to=1-3]
			\arrow[from=1-2, to=2-2, "\eqref{eq:phiGamma_comp}", "\wr"']
			\arrow[from=1-3, to=1-4]
			\arrow[from=1-3, to=2-3, "\eqref{eq:phiGamma_comp}"]
			\arrow[from=1-4, to=1-5]
			\arrow[from=1-4, to=2-4, "\eqref{eq:phiGamma_comp}", "\wr"']
			\arrow[from=2-1, to=2-2]
			\arrow[from=2-2, to=2-3]
			\arrow[from=2-3, to=2-4]
			\arrow[from=2-4, to=2-5]
		\end{tikzcd}
	\end{equation*}
	where the top row is the scalar extension of \eqref{eq:DR_exact} along the $(\varphi, G_{\calR/R}^{\etale})\textrm{-equivariant}$ flat homomorphism $\AR \rightarrow \AR^{\etale}$ (see Lemma \ref{lem:ARet_flat}), and the exactness of the bottom row was discussed above.
	The left and the right vertical arrow are isomorphisms by the induction assumption.
	Hence, it follows that the middle vertical arrow is also an isomorphism.
	Finally, it is easy to check that the Frobenius structure on $\DR(T)$ is \'etale, in particular, we conclude that $\DR(T)$ is an \'etale $(\varphi, \GammaR \times \Delta_{\calR/R})\textrm{-module}$ over $\AR$.
	
	Next, let us assume that $T$ is a finite free $\mathbb{Z}_p\textrm{-representation}$ of $G_{\calR/R}^{\etale}$.
	Then, $T = \lim_i T/p^iT$, and we have $(\varphi, \GammaR \times \Delta_{\calR/R})\textrm{-equivariant}$ identifications
	\begin{equation*}
		\begin{aligned}
			\DR(T) = (\AR^{\etale} \otimes_{\mathbb{Z}_p} T)^{\HR^{\etale}} &= \big(\AR^{\etale} \otimes_{\mathbb{Z}_p} \lim_i T/p^iT\big)^{\HR^{\etale}}\\
				&= \big(\lim_i(\AR^{\etale} \otimes_{\mathbb{Z}_p} T/p^iT)\big)^{\HR^{\etale}}\\
				&= \lim_i\big((\AR^{\etale} \otimes_{\mathbb{Z}_p} T/p^iT)^{\HR^{\etale}}\big) = \lim_i \DR(T/p^iT),
		\end{aligned}
	\end{equation*}
	where the third equality follows because $\AR^{\etale}$ is a flat $\mathbb{Z}_p\textrm{-algebra}$ and the fourth equality follows because taking $\HR^{\etale}\textrm{-invariants}$ preserves limits.
	Additionally, for each $i \geqslant 1$, we see that $T/p^iT$ is a finite free $\mathbb{Z}/p^i\mathbb{Z}\textrm{-representation}$ of $G_{\calR/R}^{\etale}$, and thus we have the $(\varphi, G_{\calR/R}^{\etale})\textrm{-equivariant}$ isomorphism from \eqref{eq:phiGamma_comp} for each $i$.
	Passing to the limit over $i \geqslant 1$ and using the discussion above yields the isomorphism in \eqref{eq:phiGamma_comp_iso} for a finite free $\mathbb{Z}_p\textrm{-representation}$ $T$ of $G_{\calR/R}^{\etale}$.

	To show that $\DR(T)$ is finitely generated over $\AR$ consider the following $G_{\calR/R}^{\etale}\textrm{-equivariant}$ exact sequence:
	\begin{equation*}
		\begin{tikzcd}
			0 \arrow[r] & T \arrow[r, "p^i"] & T \arrow[r] & T/p^iT \arrow[r] & 0.
		\end{tikzcd}
	\end{equation*}
	Upon tensoring the preceding exact sequence along the $(\varphi, G_{\calR/R}^{\etale})\textrm{-equivariant}$ flat homomorphism $\mathbb{Z}_p \rightarrow \AR^{\etale}$, taking the $\HR^{\etale}\textrm{-invariants}$, and using that $H^1(\HR^{\etale}, \AR^{\etale} \otimes_{\mathbb{Z}_p} T) = 0$ (follows from Lemma \ref{lem:H_coh_Eet} and the isomorphism \eqref{eq:phiGamma_comp_iso}), we obtain a natural $(\varphi, \GammaR \times \Delta_{\calR/R})\textrm{-equivariant}$ isomorphism $\DR(T)/p^i\DR(T) \isomorphic \DR(T/p^iT)$.
	For $i = 1$, note that $\DR(T/pT)$ is finitely generated over $\AR/p\AR \isomorphic \ER$, and the $\AR\textrm{-module}$ $\DR(T)$ is $p\textrm{-adically}$ separated because $\cap_{n \geqslant 0} p^n\DR(T) \subset \cap_{n \geqslant 0} p^n(\AR \otimes_{\mathbb{Z}_p} T) = 0$.
	Therefore, from \cite[\href{https://stacks.math.columbia.edu/tag/031D}{Tag 031D}]{stacks-project} it follows that $\DR(T)$ is finitely generated over $\AR$.
	Finally, it is easy to check that the Frobenius structure on $\DR(T)$ is \'etale, in particular, we conclude that $\DR(T)$ is an \'etale $(\varphi, \GammaR \times \Delta_{\calR/R})\textrm{-module}$ over $\AR$.
	As $\DR(T)$ is $p\textrm{-torsion}$ free, from Lemma \ref{lem:phimod_proj} it follows that $\DR(T)$ is finite projective over $\AR$, and of rank equal to $\textrm{rk}_{\mathbb{Z}_p} T$ using the isomorphism in \eqref{eq:phiGamma_comp_iso}.

	Conversely, let $D$ denote an \'etale $(\varphi, \GammaR \times \Delta_{\calR/R})\textrm{-module}$ over $\AR$.
	Let us first assume that $p^nD = 0$ for some $n \geqslant 1$.
	For each $1 \leqslant i \leqslant n$, we have
	\begin{equation*}
		\TR(D/p^iD) = (\AR^{\etale} \otimes_{\AR} D/p^iD)^{\varphi=1},
	\end{equation*}
	and the natural multiplication map $\AR^{\etale} \otimes_{\AR} \AR^{\etale} \rightarrow \AR^{\etale}$ induces the following $(\varphi, G_{\calR/R}^{\etale})\textrm{-equivariant}$ homomorphism for each $1 \leqslant i \leqslant n$:
	\begin{equation}\label{eq:phiGamma_comp_inv}
		\AR^{\etale} \otimes_{\mathbb{Z}_p} \TR(D/p^iD) \longrightarrow \AR^{\etale} \otimes_{\AR} D/p^iD.
	\end{equation}
	For $i = 1$, from Proposition \ref{prop:modp_classify_decomp} it is clear that $\TR(D/pD)$ is a finite $\mathbb{F}_p\textrm{-representation}$ of $G_{\calR/R}^{\etale}$, and \eqref{eq:phiGamma_comp_inv} is an isomorphism by \eqref{eq:phiGamma_comp_iso_modp_imperf}.

	For any $1 \leqslant i \leqslant n$, to show that $\TR(D/pD)$ is a finitely generated $\mathbb{Z}_p\textrm{-representation}$ of $G_{\calR/R}^{\etale}$ and \eqref{eq:phiGamma_comp_inv} is an isomorphism, we shall proceed by induction.
	So, assume that this holds for some $1 \leqslant i \leqslant n-1$, and consider the following exact sequence of \'etale $(\varphi, \GammaR \times \Delta_{\calR/R})\textrm{-modules}$ over $\AR$:
	\begin{equation}\label{eq:D_modpi}
		0 \longrightarrow p^iD/p^{i+1}D \longrightarrow D/p^{i+1}D \longrightarrow D/p^iD \longrightarrow 0.
	\end{equation}
	By extending scalars of the exact sequence \eqref{eq:D_modpi} along the $(\varphi, G_{\calR/R}^{\etale})\textrm{-equivariant}$ flat homomorphism $\AR \rightarrow \AR^{\etale}$ (see Lemma \ref{lem:ARet_flat}), and taking the $\varphi=1$ part, we obtain the following exact sequence:
	\begin{equation}\label{eq:TR_exact}
		0 \longrightarrow \TR(p^iD/p^{i+1}D) \longrightarrow \TR(D/p^{i+1}D) \longrightarrow \TR(D/p^iD) \longrightarrow 0,
	\end{equation}
	where the only nontrivial claim is surjectivity on the right, and it may be shown similar to the proof of \cite[Theorem 7.11]{andreatta-phigamma}.
	Indeed, let $x$ be an element of $\TR(D/p^iD)$, and let $y$ in $\AR' \otimes_{\AR} (D/p^{i+1}D)$ denote a lift of $x$, where $\AR' \subset \AR^{\etale}$ is a finite \'etale extension of $\AR$.
	We wish to find some $z$ in $\AR' \otimes_{\AR} (p^iD/p^{i+1}D)$ such that $\varphi(y+z) = y+z$.
	Note that $w = y-\varphi(y)$ belongs to $\AR' \otimes_{\AR} (p^iD/p^{i+1}D)$, so it remains to show that $\varphi(z)-z = w$ admits a solution in $\AR^{\etale} \otimes_{\AR} (p^iD/p^{i+1}D)$.
	As $p^iD/p^{i+1}D$ is killed by $p$, from the isomorphism \eqref{eq:phiGamma_comp_iso_modp_imperf} in Proposition \ref{prop:modp_classify_decomp}, we get that $(\varphi, G_{\calR/R}^{\etale})\textrm{-equivariant}$ isomorphism $\AR^{\etale} \otimes_{\AR} (p^iD/p^{i+1}D) = \ER^{\etale} \otimes_{\ER} (p^iD/p^{i+1}D) \isomorphic \ER^{\etale} \otimes_{\mathbb{F}_p} \TR(p^iD/p^{i+1}D)$, where the Frobenius on the last term is given as $\varphi \otimes 1$.
	Therefore, from the Artin--Schreier exact sequence for $\ER^{\etale}$ and the fact that $H^1_{\etale}(\ER^{\etale}, \mathbb{F}_p) = 0$, it follows that the equation $\varphi(z)-z = w$ admits a solution in $\AR^{\etale} \otimes_{\AR} (p^iD/p^{i+1}D)$, and therefore, \eqref{eq:TR_exact} is exact.
	Additionally, note that in the exact sequence \eqref{eq:TR_exact}, by the induction assumption its first and third term are finitely presented over $\mathbb{Z}_p$, therefore, it follows that the middle term, i.e.\ $\TR(D/p^iD)$ is also a finitely presented $\mathbb{Z}_p\textrm{-module}$.

	Now, consider the following commutative diagram with exact rows:
	\begin{equation*}
		\begin{tikzcd}
			0 & \AR^{\etale} \otimes_{\mathbb{Z}_p} \TR(p^iD/p^{i+1}D) & \AR^{\etale} \otimes_{\mathbb{Z}_p} \TR(D/p^{i+1}D) & \AR^{\etale} \otimes_{\mathbb{Z}_p} \TR(D/p^iD) & 0 \\
			0 & \AR^{\etale} \otimes_{\AR} (p^iD/p^{i+1}D) & \AR^{\etale} \otimes_{\AR} D/p^{i+1}D & \AR^{\etale} \otimes_{\AR} D/p^iD & 0,
			\arrow[from=1-1, to=1-2]
			\arrow[from=1-2, to=1-3]
			\arrow[from=1-2, to=2-2, "\eqref{eq:phiGamma_comp_inv}", "\wr"']
			\arrow[from=1-3, to=1-4]
			\arrow[from=1-3, to=2-3, "\eqref{eq:phiGamma_comp_inv}"]
			\arrow[from=1-4, to=1-5]
			\arrow[from=1-4, to=2-4, "\eqref{eq:phiGamma_comp_inv}", "\wr"']
			\arrow[from=2-1, to=2-2]
			\arrow[from=2-2, to=2-3]
			\arrow[from=2-3, to=2-4]
			\arrow[from=2-4, to=2-5]
		\end{tikzcd}
	\end{equation*}
	where the top row is the scalar extension of \eqref{eq:TR_exact} along the $(\varphi, G_{\calR/R}^{\etale})\textrm{-equivariant}$ flat homomorphism $\mathbb{Z}_p \rightarrow \AR^{\etale}$, and the exactness of the bottom row was discussed above.
	The left and the right vertical arrow are isomorphisms by the induction assumption.
	Hence, it follows that the middle vertical arrow is also an isomorphism.

	Next, let us assume that $D$ is a $p\textrm{-torsion}$ free \'etale $(\varphi, \GammaR \times \Delta_{\calR/R})\textrm{-module}$ over $\AR$.
	Then, $D$ is finite projective over $\AR$ by Lemma \ref{lem:phimod_proj}, in particular, $D = \lim_i D/p^iD$, and we have $G_{\calR/R}^{\etale}\textrm{-equivariant}$ identifications
	\begin{equation*}
		\begin{aligned}
			\TR(D) = (\AR^{\etale} \otimes_{\AR} D)^{\varphi=1} &= \big(\AR^{\etale} \otimes_{\AR} \lim_i D/p^iD\big)^{\varphi=1}\\
				&= \big(\lim_i(\AR^{\etale} \otimes_{\AR} D/p^iD)\big)^{\varphi=1}\\
				&= \lim_i\big((\AR^{\etale} \otimes_{\AR} D/p^iD)^{\varphi=1}\big) = \lim_i \TR(D/p^iD),
		\end{aligned}
	\end{equation*}
	where the third equality follows because $\AR^{\etale}$ is a flat $\AR\textrm{-algebra}$ (see Lemma \ref{lem:ARet_flat}), and the fourth equality follows because taking $\lim_i$ preserves kernels.
	Additionally, for each $i \geqslant 1$, we see that $D/p^iD$ is an \'etale $(\varphi, \GammaR \times \Delta_{\calR/R})\textrm{-module}$ over $\AR$, and thus we have the $(\varphi, G_{\calR/R}^{\etale})\textrm{-equivariant}$ isomorphism from \eqref{eq:phiGamma_comp_inv} for each $i$.
	Passing to the limit over $i \geqslant 1$ and using the discussion above yields the following $(\varphi, G_{\calR/R}^{\etale})\textrm{-equivariant}$ isomorphism:
	\begin{equation}\label{eq:phiGamma_comp_iso_inv}
		\AR^{\etale} \otimes_{\mathbb{Z}_p} \TR(D) \isomorphic \AR^{\etale} \otimes_{\AR} D.
	\end{equation}
	Finally, as $D$ is finite projective over $\AR$ and the homomorphism $\mathbb{Z}_p \rightarrow \AR^{\etale}$ is readily faithfully flat, it follows that $\TR(D)$ is finite projective, hence finite free, over $\mathbb{Z}_p$ of rank equal to $\textup{rk}_{\AR} D$ using \eqref{eq:phiGamma_comp_iso_inv}.

	From the discussion above, the functors $\DR$ and $\TR$ are well defined, and using the isomorphisms \eqref{eq:phiGamma_comp_iso} and \eqref{eq:phiGamma_comp_iso_inv}, it is easy to check that we have natural isomorphisms of functors $\TR \circ \DR \isomorphic \textrm{id}$ and $\DR \circ \TR \isomorphic \textrm{id}$.
	Hence, we conclude that \eqref{eq:padic_classify_decomp} is an equivalence, and from the discussion above it also follows that the preceding equivalence induces the equivalence in \eqref{eq:free_padic_classify_decomp}.
	This completes the proof.
\end{proof}

In Theorem \ref{thm:padic_classify_decomp}, by restricting to $G_{\calR}^{\etale}\textrm{-action}$, instead of  $G_{\calR/R}^{\etale}\textrm{-action}$, we obtain the following variant:
\begin{cor}\label{cor:padic_classify_GR0m}
	The following natural functor induces an equivalence of categories:
	\begin{equation}\label{eq:padic_classify_GR0m}
		\begin{aligned}
			\DR \colon \mathrm{Rep}_{\mathbb{Z}_p}(G_{\calR}^{\etale}) &\isomorphic (\varphi, \Gamma_{\calR})\textup{-Mod}_{\AR}^{\etale}\\
			T &\longmapsto (\AR^{\etale} \otimes_{\mathbb{Z}_p} T)^{H^{\etale}_{\calR}},
		\end{aligned}
	\end{equation}
	with a quasi-inverse functor given as $\TR(D) \coloneq (\AR^{\etale} \otimes_{\AR} D)^{\varphi=1}$.  
	The equivalence in \eqref{eq:padic_classify_GR0m} restricts to a natural equivalence of categories:
	\begin{equation}\label{eq:free_padic_classify_GR0m}
		\DR \colon \textup{Rep}_{\mathbb{Z}_p}^{\textup{free}}(G_{\calR}^{\etale}) \isomorphic (\varphi,\GammaR)\textup{-Mod}_{\AR}^{\etale,\textup{fproj}}.
	\end{equation}
	Additionally, for a $\mathbb{Z}_p\textrm{-representation}$ $T$ of $G_{\calR}^{\etale}$, the natural multiplication map $\AR^{\etale} \otimes_{\AR} \AR^{\etale} \rightarrow \AR^{\etale}$ induces a natural isomorphism of $\AR^{\etale}\textrm{-modules}$
	\begin{equation}\label{eq:phiGamma_comp_GR0m}
		\AR^{\etale} \otimes_{\AR} \DR(T) \isomorphic \AR^{\etale} \otimes_{\mathbb{Z}_p} T,
	\end{equation}
	compatible with the respective actions of $\varphi$ and $G_{\calR}^{\etale}$.
\end{cor}

We end this section with an interesting observation.
Let $T$ be a finite free $\mathbb{Z}_p\textrm{-representation}$ of $G_{\calR/R}^{\etale}$.
Let us set $\DR^+(T) \coloneq (A_{\calR}^{\etale,+} \otimes_{\mathbb{Z}_p} T)^{\HR^{\etale}} \subset \DR(T)$ to be the $(\varphi, \GammaR \times \Delta_{\calR/R})\textup{-module}$ over $\AR^+$ associated to $T$, and for $V \coloneq T[1/p]$ let $\DR^+(V) \coloneq \DR^+(T)[1/p]$ be the $(\varphi, \GammaR \times \Delta_{\calR/R})\textup{-module}$ over $\BR^+$ associated to $V$.

\begin{lem}\label{lem:regsec_D+}
	For each $1 \leqslant i \leqslant d$ the sequences $\{p, \mu, [x_i^{\flat}]\}$ and $\{p, \mu, [x_i^{\flat}]\}$ are regular on $\DR^+(T)$.
\end{lem}
\begin{proof} 
	Note that the sequences $\{p, \mu\}$ and $\{\mu, p\}$ are clearly regular on $\AR^+$.
	Additionally, we have that $\AR^+/(p,\mu)\AR^+ \isomorphic \calR/p\calR$ as rings, and it is also clear that $x_i$ is regular on $\calR/p\calR$ for each $1 \leqslant i \leqslant d$.

	Now, consider the following $(\varphi, G_{\calR/R}^{\etale})\textrm{-equivariant}$ exact sequence
	\begin{equation*}
		0 \longrightarrow p\AR^{\etale} \longrightarrow \AR^{\etale} \longrightarrow \ER^{\etale} \longrightarrow 0.
	\end{equation*}
	Tensoring the preceding exact sequence with $T$ (over $\mathbb{Z}_p$) and taking $\HR\textrm{-invariants}$, we obtain a natural $(\varphi, \GammaR \times \Delta_{\calR/R})\textrm{-equivariant}$ injective homomorphism $\DR^+(T)/p\DR^+(T) \hookrightarrow (\ER^{\etale} \otimes_{\mathbb{Z}_p} T)^{\HR^{\etale}}$.
	As we have a natural inclusion $\ER^{\etale,+} \hookrightarrow \tilde{E}^+(\overline{R})$, the ring $\ER^{\etale,+}$ is $\overline{\mu}\textrm{-torsion}$ free.
	So, it follows that $\overline{\mu}$ is regular on $\DR^+(T)/p\DR^+(T)$, i.e.\ the sequence $\{p, \mu\}$ is regular on $\DR^+(T)$.
	Additionally, $\DR^+(T) \subset \DR(T)$ is $\mu\textrm{-torsion}$ free, so from \cite[Lemma A.2]{abhinandan-relative-wach-ii} we conclude that $\{\mu, p\}$ is also regular on $\DR^+(T)$.

	It remains to show that $[x_i^{\flat}]$ is regular on $\DR^+(T)/(p,\mu)\DR^+(T)$.
	Consider the following $(\varphi, G_{\calR/R}^{\etale})\textrm{-equivariant}$ exact sequence
	\begin{equation*}
		0 \longrightarrow (p,\mu)\AR^{\etale} \longrightarrow \AR^{\etale} \longrightarrow \ER^{\etale}/\overline{\mu}\ER^{\etale} \longrightarrow 0.
	\end{equation*}
	Tensoring the preceding exact sequence with $T$ (over $\mathbb{Z}_p$) and taking $\HR\textrm{-invariants}$, we obtain a natural $(\varphi, \GammaR \times \Delta_{\calR/R})\textrm{-equivariant}$ injective homomorphism $\DR^+(T)/(p,\mu)\DR^+(T) \hookrightarrow (\ER^{\etale}\overline{\mu}\ER^{\etale} \otimes_{\mathbb{Z}_p} T)^{\HR^{\etale}}$.

	By definition, note that $\ER^{\etale,+} = \ER^{\etale} \cap \tilde{E}^+(\overline{R}) \subset \tilde{E}(\overline{R})$, in particular, $\ER^{\etale,+} = \ER^{\etale,+}[1/\overline{\mu}] \cap \tilde{E}^+(\overline{R})$ which implies that the natural map $\ER^{\etale,+}/\overline{\mu}\ER^{\etale,+} \rightarrow \tilde{E}^+(\overline{R})/\overline{\mu}\tilde{E}^+(\overline{R})$ is injective.
	So, for each $1 \leqslant i \leqslant d$, to get that $\textrm{multiplication-by-}[x_i^{\flat}]$ is injective on $\DR^+(T)/(p,\mu)\DR^+(T)$, it is enough to show that $\textrm{multiplication-by-}x_i^{\flat}$ is injective on $\tilde{E}^+(\overline{R})/\overline{\mu}\tilde{E}^+(\overline{R})$.
	From \cite[Proposition 5.17]{scholze-perfectoid}, we know that $\tilde{E}^+(\overline{R})/\overline{\mu}\tilde{E}^+(\overline{R}) = \overline{R}^{\flat}/\overline{\mu}\overline{R}^{\flat} \isomorphic \overline{R}/(\zeta_p-1)^p\overline{R}$.
	Therefore, we are reduced to showing that $\textrm{multiplication-by-}x_i$ is injective on $\overline{R}/(\zeta_p-1)^p\overline{R}$, or equivalently, the sequence $\{(\zeta_p-1)^p, x_i\}$ is regular on $\overline{R}$.
	From \cite[\href{https://stacks.math.columbia.edu/tag/07DV}{Tag 07DV}]{stacks-project}, the latter is equivalent to showing that the sequence $\{(\zeta_p-1)^{p(p-1)}, x_i\}$ is regular on $\overline{R}$, and since we have that $(\zeta_p-1)^{p(p-1)}\overline{R} = p^p\overline{R}$, we are reduced to showing that the sequence $\{p^p, x_i\}$ is regular on $\overline{R}$.
	But this follows from Lemma \ref{lem:inject_mult_x_i}, thus allowing us to conclude.
\end{proof}

\subsubsection{\texorpdfstring{$p\textrm{-adic}$}{-} representations of $G_R$}

In this section, our goal is to classify $p\textrm{-adic}$ representations of $G_R$ using the $(\varphi, \Gamma)\textrm{-modules}$ studied in the previous sections.
We being by noting a well-known fact: the $R[1/p]\textrm{-algebra}$ $\overline{R}[1/p]$ coincides with the maximal \'etale extension of $R_{\infty,\infty}[1/p]$ inside $\overline{\Fr(R)}$ by an application of Abhyankar's Lemma (for example, see \cite[Lemma 5.8]{colmez-niziol}).
Then, by an application of Elkik's approximation theorem (see the proof of loc.\ cit.), we obtain a natural isomorphism of groups 
\begin{equation*}
	H_{R,\infty} = \Gal(\overline{R}[1/p]/R_{\infty,\infty}[1/p]) \isomorphic \textup{Aut}(\widehat{\overline{R}}[1/p]/\widehat{R}_{\infty,\infty}[1/p]).
\end{equation*}
Combining this with the tilting equivalence of \cite[Theorems 5.25 and 7.9]{scholze-perfectoid} and \cite{faltings-almost} we obtain natural isomorphisms of Galois groups 
\begin{equation*}
	H_{R,\infty} \isomorphic \Gal\big(\widehat{\overline{R}}[1/p]/\widehat{R}_{\infty,\infty}[1/p]\big) \isomorphic \Gal\big(\overline{R}[1/p]^{\flat}/R_{\infty,\infty}[1/p]^{\flat}\big).
\end{equation*}

Let us set $\tilde{A}_{R,\infty} \coloneq W(R_{\infty,\infty}[1/p]^{\flat})$ equipped with the Witt vector Frobenius and a natural action of $\Gamma_{R,\infty}$.
Using Lemma \ref{lem:Ax-Sen_R} and the discussion above, it is easy to see that we have a natural $(\varphi, \Gamma_{R,\infty})\textrm{-equivariant}$ identification $\tilde{A}_{R,\infty} = \tilde{A}(\overline{R})^{H_{R,\infty}}$.
\begin{prop}\label{prop:padic_classify_Rinftyinfty}
	The following natural functor induces an equivalence of categories:
	\begin{equation}\label{eq:padic_classify_Rinftyinfty}
		\begin{aligned}
			\tilde{D}_{R,\infty} \colon \mathrm{Rep}_{\mathbb{Z}_p}^{\textup{free}}(G_R) &\isomorphic (\varphi, \Gamma_{R,\infty})\textup{-Mod}_{\tilde{A}_{R,\infty}}^{\etale,\textup{fproj}}\\
				T &\longmapsto (\tilde{A}(\overline{R}) \otimes_{\mathbb{Z}_p} T)^{H_{R,\infty}},
		\end{aligned}
	\end{equation}
	with a quasi-inverse functor given as $\tilde{T}_{R,\infty}(\tilde{D}) \coloneq (\tilde{A}(\overline{R}) \otimes_{\tilde{A}_{R,\infty}} \tilde{D})^{\varphi=1}$.
\end{prop}
\begin{proof}
	The equivalence in \eqref{eq:padic_classify_Rinftyinfty} follows by using the discussion above and \cite[Theorem 9.3.7]{kedlaya-liu1}.
\end{proof}

The following is a crucial observation for $p\textrm{-adic}$ representations of $G_R$.
\begin{thm}\label{thm:GR_rep_GR0m}
	Let $T$ be a finitely generated $\mathbb{Z}_p\textrm{-representation}$ of $G_R$.
	Then, there exists a positive integer $m$ coprime to $p$, depending on $T/pT$, and such that the action of $G_R$ on $T$ factors through $G_{\calR/R}^{\etale}$ for $\calR \coloneq R_{0,m}$.
\end{thm}
\begin{proof}
	Let us set $H \coloneq \ker(G_R \rightarrow \textrm{GL}(T/pT))$, and note that $B \coloneq \overline{R}^{H} \subset \overline{R}$ is a finite normal $R\textrm{-subalgebra}$.
	By Abhyankar's Lemma (see \cite[XIII, Proposition 5.2]{sga1}), there exists some $m \geqslant 1$ coprime to $p$ and $n \geqslant 0$ such that $B_{n,m}[1/p]$ is finite \'etale over $R_{n,m}[1/p]$.
	Now, take any $i \geqslant 2$ and set $H' \coloneq \ker(G_R \rightarrow \textrm{GL}(T/p^iT))$ and $C \coloneq \overline{R}^{H'}$.
	Then, we have natural injective homomorphisms $R[1/p] \hookrightarrow B[1/p] \hookrightarrow C[1/p]$, and the induced homomorphism $B_{n,m}[1/p] \rightarrow C_{n,m}[1/p]$ is unramified outside the divisor $\{x_{a+1} \cdots x_d = 0\}$.
	Without loss of generality we may further assume that the preceding homomorphism is log-\'etale and Galois (see Section \ref{subsubsec:fundamental_groups}).
	Note that the order of the Galois group $\Gal(C_{n,m}[1/p]/B_{n,m}[1/p])$ equals $p^k$ for some $k$ depending on $T$ and $i$ because $|G/H|$ (resp.\ $|G/H'|$) is a power of $p$ depending on $T$ (resp.\ $T$ and $i$).
	Therefore, it follows that the ramification index of $C_{n,m}[1/p]$, over each height one prime ideal $\mathfrak{p}$ of $\Spec(B_{n,m}[1/p])$, is a power of $p$.
	Thus, after replacing $n$ by some $n' \geqslant n$ if required, by the purity of ramification locus (see \cite[\href{https://stacks.math.columbia.edu/tag/0EA4}{Tag 0EA4}]{stacks-project}), it follows that $C_{n',m}[1/p]$ is finite \'etale over $B_{n',m}[1/p]$, and therefore $C_{\infty,m}[1/p]$ is finite \'etale over $B_{\infty,m}[1/p]$.
	Hence, we conclude that the action of $G_R$ on $T$ factors through $G_{\calR/R}^{\etale}$ for $\calR = R_{0,m}$, where $m$ depends on the mod $p$ representation $T/pT$ of $G_R$.
\end{proof}

\begin{cor}\label{cor:phiGammaR_descentR0m}
	Let $T$ be a $\mathbb{Z}_p\textrm{-representation}$ of $G_R$ and $\calR = R_{0,m}$ for some positive integer $m$ coprime to $p$ as in Theorem \ref{thm:GR_rep_GR0m}.
	Then, we have a natural $(\varphi, \Gamma_{R,\infty})\textrm{-equivariant}$ isomorphism of $\tilde{A}_{R,\infty}\textrm{-modules}$
	\begin{equation}\label{eq:phiGammaR_descentR0m}
		\tilde{D}_{R,\infty}(T) \isomorphic \tilde{A}_{R,\infty} \otimes_{\AR} \DR(T).
	\end{equation}
	Additionally, for a $\mathbb{Z}_p\textrm{-representation}$ $T$ of $G_R$, the natural multiplication map $\tilde{A}(\overline{R}) \otimes_{\tilde{A}_{R,\infty}} \tilde{A}(\overline{R}) \rightarrow \tilde{A}(\overline{R})$ induces a natural isomorphism of $\tilde{A}_{R,\infty}\textrm{-modules}$
	\begin{equation}\label{eq:phiGamma_comp_iso_padic_Rinftyinfty}
		\tilde{A}(\overline{R}) \otimes_{\tilde{A}_{R,\infty}} \tilde{D}_{R,\infty}(T) \isomorphic \tilde{A}(\overline{R}) \otimes_{\mathbb{Z}_p} T,
	\end{equation}
	compatible with the respective actions of $\varphi$ and $G_R$.
\end{cor}
\begin{proof}
	The isomorphism in \eqref{eq:phiGammaR_descentR0m} follows by using Theorem \ref{thm:GR_rep_GR0m} and taking $H_{R,\infty}\textrm{-invariants}$ of the extension of scalars of the isomorphism in \eqref{eq:phiGamma_comp_iso_modp_imperf} along the $(\varphi, G_R)\textrm{-equivariant}$ homomorphism $\AR^{\etale} \rightarrow \tilde{A}(\overline{R})$.
	The isomorphism in \eqref{eq:phiGamma_comp_iso_padic_Rinftyinfty} follows by combining \eqref{eq:phiGammaR_descentR0m} and extension along $\AR^{\etale} \rightarrow \tilde{A}(\overline{R})$ of \eqref{eq:phiGamma_comp_iso_modp_imperf}.
\end{proof}

\begin{rem}\label{rem:TRD_GRrep}
	For any \'etale $(\varphi, \GammaR \times \Delta)\textrm{-module}$ $D$ over $\AR$, note that by extending scalars of the natural $(\varphi, G_{\calR/R}^{\etale})\textrm{-equivariant}$ isomorphism \eqref{eq:phiGamma_comp_iso_inv} along the $(\varphi, G_R)\textrm{-equivariant}$ homomorphism $\AR^{\etale} \rightarrow \tilde{A}(\overline{R})$ and taking Frobenius-fixed part, we get that $\TR(D) \isomorphic (\tilde{A}(\overline{R}) \otimes_{A_R} D)^{\varphi=1}$ as $\mathbb{Z}_p\textrm{-representations}$ of $G_R$, and it is clear that the $G_R\textrm{-action}$ on $\TR(D)$ factors through $G_{\calR/R}^{\etale}$ (also see Theorem \ref{thm:GR_rep_GR0m}).
	Analogously, if $D$ is an \'etale $(\varphi, \GammaR)\textrm{-module}$ over $\AR$, then using the $(\varphi, G_{\calR}^{\etale})\textrm{-equivariant}$ isomorphisms \eqref{eq:phiGamma_comp_iso_inv} and \eqref{eq:phiGamma_comp_GR0m}, and applying an argument similar to above, we get that $\TR(D) \isomorphic (\tilde{A}(\overline{R}) \otimes_{A_R} D)^{\varphi=1}$ as $\mathbb{Z}_p\textrm{-representations}$ of $\GR$, and it is clear that the $\GR\textrm{-action}$ on $\TR(D)$ factors through $G_{\calR/R}^{\etale}$.
\end{rem}

The discussion above also applies if we replace the $O_F\textrm{-algebra}$ $R$ above with the $p\textrm{-adically}$ complete ring $S = R[1/(X_{a+1} \cdots X_d)]^{\wedge}$ (see \cite[Section 2.6]{abhinandan-relative-wach-ii}).
We will use the results of op.\ cit.\ in our constructions.

\begin{rem}\label{rem:phiGamma_comp_R0mS0m}
	Let $T$ be a finitely generated $\mathbb{Z}_p\textrm{-representation}$ of $G_R$, and $m$ a positive integer coprime to $p$ as in Theorem \ref{thm:GR_rep_GR0m}.
	We may also view $T$ as a representation of $G_S$ via the group homomorphism $G_S \rightarrow G_R$.
	Then, as described in \cite[Section 2.6]{abhinandan-relative-wach-ii}, there exists an \'etale $(\varphi, \Gamma_S)\textrm{-module}$ $D_S(T)$ associated to $T$.
	Additonally, by setting $\calR = R_{0,m}$ and using Theorems \ref{thm:padic_classify_decomp} and \ref{thm:GR_rep_GR0m}, we see that there exists an \'etale $(\varphi, \GammaR \times \Delta_{\calR/R})\textrm{-module}$ $D_{\calR}(T)$ over $\AR$.
	Then, extending scalars of the isomorphism \eqref{eq:phiGamma_comp_iso} from Theorem \ref{thm:padic_classify_decomp} along the $(\varphi, G_S)\textrm{-equivariant}$ composition $A_{\calR}^{\etale} \rightarrow \tilde{A}(\overline{R}) \rightarrow \tilde{A}(\overline{S})$, and using the classification of $p\textrm{-adic}$ representations of $G_{\calS}$ in terms of \'etale $(\varphi, \GammaS \times \Delta_{\calS/S})\textrm{-modules}$ over $A_{\calS}$ (see \cite{andreatta-phigamma}), yields a natural isomorphism of $(\varphi, \GammaS \times \Delta_{\calS/S})\textrm{-modules}$ over $\AS$ as $\AS \otimes_{\AR} D_{\calR}(T) \isomorphic \AS \otimes_{A_S} D_S(T)$.
\end{rem}

\section{Wach modules}\label{sec:wachmods}

In this section, we shall define and study the properties of Wach modules.
We will keep the setup and notation introduced in Sections \ref{sec:prelims} and \ref{sec:phiGamma_modules}.

\subsection{Wach modules}

In \cite[Section 3]{abhinandan-relative-wach-ii}, we defined and studied properties of Wach modules over the ring $\AS^+$.
In this section, our goal is to define and study analogous properties of Wach modules over $\AR^+$.
Let us set $[p]_q \coloneq \tfrac{q^p-1}{q-1} = \tfrac{\varphi(\mu)}{\mu}$ as an element of $\AR^+$.
\begin{defi}\label{defi:wach_mods}
	A \textit{Wach module} over $\AR^+$ is a finitely generated $\AR^+\textrm{-module}$ $N$ satisfying the following assumptions:
	\begin{enumerate}
		\item[(1)] The sequences $\{p, \mu\}$ and $\{\mu, p\}$ are regular on $N$.
	
		\item[(2)] The module $N$ is equipped with a semilinear action of $\GammaR$ such that the induced action of $\GammaR$ on $N/\mu N$ is trivial.
	
		\item[(3)] The module $N$ is equipped with an $\AR^+\textrm{-linear}$ and $\GammaR\textrm{-equivariant}$ Frobenius-structure, i.e.\ an isomorphism $\varphi_N \colon (\varphi^*N)[1/[p]_q] \isomorphic N[1/[p]_q]$.
	\end{enumerate}
	The module $N$ is said to be \textit{effective} if we have $\varphi_N(\varphi^*N) \subset N$.
	Denote by $(\varphi, \GammaR)\textup{-Mod}_{\AR^+}^{[p]_q}$, the category of Wach modules over $\AR^+$ with morphisms between objects being $\AR^+\textrm{-linear}$ and $(\varphi, \GammaR)\textrm{-equivariant}$.
\end{defi}

\begin{rem}\label{rem:automatic_continuity}
	In Definition \ref{defi:wach_mods}, note that from the triviality of the action of $\GammaR$ on $N/\mu N$ and by an argument similar to \cite[Lemma 3.4]{abhinandan-relative-wach-ii}, it follows that the action of $\GammaR$ on $N$ is continuous.
\end{rem}

\begin{defi}\label{defi:wach_mods_with_delta}
	A \textit{Wach module with $\Delta\textrm{-action}$} over $\AR^+$ (resp.\ $\AS^+$) is a Wach module $N$ over $\AR^+$ (resp.\ $\AS^+$ in the sense of \cite[Definition 3.8]{abhinandan-relative-wach-ii}) equipped with an additional semilinear action of $\Delta \coloneq \Delta_{\calR/R}$ commuting with the $(\varphi, \GammaR)\textrm{-action}$ (resp.\ $(\varphi, \GammaS)\textrm{-action}$) on $N$.
	Denote by $(\varphi, \GammaR \times \Delta)\textup{-Mod}_{\AR^+}^{[p]_q}$, the category of Wach modules with $\Delta\textrm{-action}$ over $\AR^+$ with morphisms between objects being $\AR^+\textrm{-linear}$ and $(\varphi, \GammaR \times \Delta)\textrm{-equivariant}$.
\end{defi}

\begin{lem}\label{lem:wachmod_RtoS}
	Let $\NR$ be a Wach module with $\Delta\textrm{-action}$ over $\AR^+$, then $\NS \coloneq \AS^+ \otimes_{\AR^+} \NR$ is a Wach module with $\Delta\textrm{-action}$ over $\AS^+$.
\end{lem}
\begin{proof}
	Note that we have $\GammaS \times \Delta \isomorphic \GammaR \times \Delta$ and the natural $(\varphi, \GammaR \times \Delta)\textrm{-equivariant}$ homomorphism $\AR^+ \rightarrow \AS^+$ is flat because $\AR^+\big[1/([X_{a+1}^{\flat}] \cdots [X_d^{\flat}])\big]^{\wedge} \isomorphic \AS^+$ from Section \ref{subsubsec:imperfect_rings_mixedchar}.
	Then, it is easy to see that the conditions analogous to (1), (2) and (3) of Definition \ref{defi:wach_mods} are true for $\NS$.
	Hence, $\NS$ is a Wach module with $\Delta\textrm{-action}$ over $\AS^+$.
\end{proof}

Next, we note some structural properties of Wach modules.

\begin{prop}\label{prop:wachmod_proj_pmu}
	Let $N$ be a Wach module with $\Delta\textrm{-action}$ over $\AR^+$.
	Then, $N[1/p]$ is finite projective over $\AR^+[1/p]$ and $N[1/\mu]$ is finite projective over $\AR^+[1/\mu]$.
\end{prop}
\begin{proof}
	The proof follows by an argument similar to the proof of \cite[Proposition 3.8]{abhinandan-relative-wach-ii}.
	\ifthenelse{\equal{\showdetailed}{1}}{
	For $r \in \mathbb{N}$ large enough, note that the Wach module $\mu^r N(-r)$ is always effective.
	So without loss of generality, we may assume that $N$ is effective.
	Then, the first claim follows by employing an argument similar to \cite[Proposition A.8]{abhinandan-relative-wach-ii} (also see \cite[Proposition 4.13]{dlms1}).
	For the second claim, note that $\AR^+ \rightarrow \AR$ is flat and $N$ is $p\textrm{-torsion}$ free, so we get that $\AR \otimes_{\AR^+} N$ is a $p\textrm{-torsion}$ free \'etale $\varphi\textrm{-module}$ over $\AR$, and therefore, finite projective by Lemma \ref{lem:phimod_proj}.
	Moreover, $\AR^+[1/\mu]$ is noetherian, so we see that we have a natural isomorphism $N[1/\mu]^{\wedge} \isomorphic \AR \otimes_{\AR^+[1/\mu]} N[1/\mu] = \AR \otimes_{\AR^+} N$, where ${}^{\wedge}$ denotes the $p\textrm{-adic}$ completion.
	Now, since the natural map $\Spec(\AR^+[1/\mu]^{\wedge}) \cup \Spec(\AR^+[1/\mu, 1/p]) \rightarrow \Spec(\AR^+[1/\mu])$ is a flat cover, therefore, by faithfully flat descent we obtain that $N[1/\mu]$ is a finite projective $\AR^+[1/\mu]\textrm{-module}$.}
\end{proof}

\begin{rem}\label{rem:wachmod_torsionfree}
	For a Wach module with $\Delta\textrm{-action}$ $N$ over $\AR^+$, note that $N$ is $p\textrm{-torsion}$ free, in particular, $N \subset N[1/p]$.
	As $N[1/p]$ is finite projective over $\AR^+[1/p]$ by Proposition \ref{prop:wachmod_proj_pmu}, so we conclude that $N$ is torsion free over $\AR^+$.
\end{rem}

\begin{lem}\label{lem:wach_intersection_lemma}
	Let $N$ be a Wach module with $\Delta\textrm{-action}$ over $\AR^+$.
	Then, we have that $N = N[1/p] \cap (\AR \otimes_{\AR^+} N) \subset \BR \otimes_{\AR^+} N$.
	Moreover, we have that $N = (\AS^+ \otimes_{\AR^+} N) \cap (\AR \otimes_{\AR^+} N) \subset \AS \otimes_{\AR^+} N$, as $(\varphi, \GammaR \times \Delta)\textrm{-modules}$ over $\AR^+$.
\end{lem}
\begin{proof}
	Let $\NR \coloneq N$, $\NS \coloneq \AS^+ \otimes_{\AR^+} N$ and $\DR \coloneq \AR \otimes_{\AR^+} N = \NR[1/\mu]^{\wedge}$.
	The first claim follows by an argument similar to \cite[Lemma 3.3]{abhinandan-relative-wach-ii}.
	From the definitions note that we have equalities $(\NR/p)[\mu] = (\NR/\mu)[p] = 0$ and $(\NR[1/\mu])/p = (\NR/p)[1/\mu]$.
	So, it follows that $\NR/p^n \hookrightarrow (\NR/p^n)[1/\mu]$ for all $n \geqslant 0$, and therefore, $\NR[1/p] \cap \NR[1/\mu] = \NR$ by \cite[Lemma A.1]{abhinandan-relative-wach-ii}.
	Furthermore, since $\DR/p^n = (\NR[1/\mu])/p^n = (\NR/p^n)[1/\mu]$ and $\NR$ is $p\textrm{-adically}$ complete, therefore, taking the $\lim_n$ of the inclusions $\NR/p^n \hookrightarrow (\NR/p^n)[1/\mu]$ gives that $\NR \hookrightarrow \DR$.
	So, similar to above, by using \cite[Lemma A.1]{abhinandan-relative-wach-ii}, we get that $\NR = \NR[1/p] \cap \DR$ as submodules of $\DR[1/p]$.

	The second claim follows by an argument similar to \cite[Lemma 3.11]{abhinandan-relative-wach-ii}.
	Note that $\NS[1/p] = \BS^+ \otimes_{\BR^+} \NR[1/p]$ and $\DR[1/p] = \BR \otimes_{\BR^+} \NR[1/p]$.
	So, we have that
	\begin{equation*}
		\NS[1/p] \cap \DR[1/p] = (\BS^+ \cap \BR) \otimes_{\BR^+} \NR[1/p] = \NR[1/p],
	\end{equation*}
	where the first equality holds because $\NR[1/p]$ is a finite projective $\BR^+\textrm{-module}$, and the second equality holds because one may easily show that $\BR^+ = \BS^+ \cap \BR \subset \BS$.
	Moreover, we have that $\NS \cap \DR \subset \NS[1/p] \cap \DS[1/p] = \NR[1/p]$, and using the claim from the previous paragraph we get that $\NS \cap \DR = \NS \cap \DR \cap \NR[1/p] = \NR$.
	This allows us to conclude.
\end{proof}

Next, note that the extension of scalars along $\AR^+ \rightarrow \AR$ induces a functor $(\varphi, \GammaR \times \Delta)\textup{-Mod}_{\AR^+}^{[p]_q} \rightarrow (\varphi, \GammaR \times \Delta)\textup{-Mod}_{\AR}^{\textrm{\'et}}$ (see the proof of Proposition \ref{prop:wachmod_proj_pmu}).
\begin{prop}\label{prop:wach_etale_ff_relative}
	The following natural functor is fully faithful:
	\begin{equation*}
		\begin{aligned}
			(\varphi, \GammaR \times \Delta)\textup{-Mod}_{\AR^+}^{[p]_q} &\longrightarrow (\varphi, \GammaR \times \Delta)\textup{-Mod}_{\AR}^{\textup{\'et}}\\
				N &\longmapsto \AR \otimes_{\AR^+} N.
		\end{aligned}
	\end{equation*}
	By forgetting the $\Delta\textrm{-action}$, the preceding functor restricts to a fully faithful functor $(\varphi, \GammaR)\textup{-Mod}_{\AR^+}^{[p]_q} \rightarrow (\varphi, \GammaR)\textup{-Mod}_{\AR}^{\etale}$.
\end{prop}
\begin{proof}
	The claim follows by an argument similar to the proof of \cite[Proposition 3.15]{abhinandan-relative-wach-ii}.
	Let $N, N'$ be two Wach modules with $\Delta\textrm{-action}$ over $\AR^+$.
	Write $\NR \coloneq N$, $\NS \coloneq \AS^+ \otimes_{\AR^+} N$, $\DR \coloneq \AR \otimes_{\AR^+} N$, $\DS \coloneq \AS \otimes_{\AR^+} N$.
	We need to show that for the Wach modules $\NR$ and $\NR'$, we have a natural bijection
	\begin{equation}\label{eqref:homset_bijection_relative}
		\textup{Hom}_{(\varphi, \GammaR \times \Delta)\textup{-Mod}_{\AR^+}^{[p]_q}}(\NR, \NR') \isomorphic \textup{Hom}_{(\varphi, \GammaR \times \Delta)\textup{-Mod}_{\AR}^{\textup{\'et}}}(\DR, \DR')
	\end{equation}
	As the homomorphism $\AR^+ \rightarrow \AR = \AR^+[1/\mu]^{\wedge}$ is injective, therefore, we see that the map in \eqref{eqref:homset_bijection_relative} is injective (also see Lemma \ref{lem:wach_intersection_lemma}).
	To check that \eqref{eqref:homset_bijection_relative} is surjective, take an $\AR\textrm{-linear}$ and $(\varphi, \GammaR \times \Delta)\textrm{-equivariant}$ map $f \colon \DR \rightarrow \DR'$.
	We need to show that $f(\NR) \subset \NR'$.
	Base changing $f$ along $\AR \rightarrow \AS$ and using the isomorphism $\GammaS \times \Delta \isomorphic \GammaR \times \Delta$, induces an $\AS\textrm{-linear}$ and $(\varphi, \GammaS \times \Delta)\textrm{-equivariant}$ map $f \colon \DS \rightarrow \DS'$.
	Then, from \cite[Proposition 3.12]{abhinandan-relative-wach-ii} we have that $f(\NS) \subset \NS'$.
	So, using Lemma \ref{lem:wach_intersection_lemma}, we get that inside $\DS'$ we have $f(\NR) = f(\NS \cap \DR) \subset f(\NS) \cap f(\DR) \subset \NS' \cap \DR' = \NR'$, thus concluding the proof.
\end{proof}

\subsection{\texorpdfstring{$G_R\textrm{-representations}$}{-} attached to Wach modules}\label{subsec:wach_mod_rep}

Composition of the functor in Proposition \ref{prop:wach_etale_ff_relative} with the categorical equivalence in Theorem \ref{thm:padic_classify_decomp} yields a fully faithful functor
\begin{equation}\label{eq:wach_reps_relative}
	\begin{aligned}
		\TR \colon (\varphi, \GammaR \times \Delta)\textup{-Mod}_{\AR^+}^{[p]_q} &\longrightarrow \textup{Rep}_{\mathbb{Z}_p}(G_{\calR/R}^{\etale})\\
			N &\longmapsto \big(\tilde{A}(\overline{R}) \otimes_{\AR^+} N\big)^{\varphi = 1}.
	\end{aligned}
\end{equation}
As mentioned in Remark \ref{rem:TRD_GRrep}, the $\mathbb{Z}_p\textrm{-module}$ $\TR(N)$ is in fact a representation of $G_R$ on which the action factors through the quotient $G_R \twoheadrightarrow G_{\calR/R}^{\etale}$, and in the following, we shall view $\TR(N)$ as a representation of $G_R$.

\begin{prop}\label{prop:wachmod_comp_relative}
	Let $N$ be a Wach module with $\Delta\textrm{-action}$ over $\AR^+$, and let $T \coloneq \TR(N)$ be the associated finite free $\mathbb{Z}_p\textrm{-representation}$ of $G_R$.
	Then, we have a natural $G_R\textrm{-equivariant}$ comparison isomorphism:
	\begin{equation}\label{eq:wachmod_comp_relative_ainf}
		A_{\inf}(\overline{R})[1/\mu] \otimes_{\AR^+} N \isomorphic A_{\inf}(\overline{R})[1/\mu] \otimes_{\mathbb{Z}_p} T.
	\end{equation}
	Additionally, \eqref{eq:wachmod_comp_relative_ainf} is compatible with the action of $(\varphi, G_R)$ after base change along $A_{\inf}(\overline{R})[1/\mu] \rightarrow \tilde{A}(\overline{R})$.
\end{prop}
\begin{proof}
	The proof of the claim follows by an argument similar to \cite[Proposition 3.14]{abhinandan-relative-wach-ii}.
	Note that for $T = \TR(N)$, we have that $\DR(T) \isomorphic \AR \otimes_{\AR^+} N$ as \'etale $(\varphi, \GammaR \times \Delta)\textrm{-modules}$ over $\AR$.
	Then, extending scalars of the isomorphism in \eqref{eq:phiGamma_comp_iso} along $\AR^{\etale} \rightarrow \tilde{A}(\overline{R})$ gives a $(\varphi, G_R)\textrm{-equivariant}$ isomorphism
	\begin{equation}\label{eq:phigamma_comp_tilde_relative}
		\tilde{A}(\overline{R}) \otimes_{\AR^+} N \isomorphic \tilde{A}(\overline{R}) \otimes_{\mathbb{Z}_p} T.
	\end{equation}
	Now, for $r \in \mathbb{N}$ large enough, the Wach module $\mu^r N (-r)$ is always effective and we have $\TR(\mu^rN(-r)) = T(-r)$ (the twist $(-r)$ denotes the Tate twist on which $\GammaR \times \Delta$ acts via the cyclotomic character).
	Therefore, we see that it is enough to show the claim for effective Wach modules (see Definition \ref{defi:wach_mods}), in particular, in the rest of the proof we will assume that $N$ is effective.

	Let $\mathscr{P}(\overline{R})$ denote the set of minimal primes of $\overline{R}$ above $pR \subset R$ (or equivalently, over $p\calR \subset \calR$).
	Then, $O_{\calL} \coloneq \calR_{(p)}^{\wedge}$ is a complete discrete valuation ring with uniformiser $p$ and imperfect residue field, and fraction field $\calL = O_{\calL}[1/p]$ which is a finite unramified Galois extension of $L$ with Galois group $\Gal(\calL/L) \isomorphic \Delta$.
	From Section \ref{subsubsec:localisation_Rbar}, recall that for each $\pins$, we have $\Lbar(\mathfrak{p}) \subset \Cp$, an algebraic closure of $L$ containing $\Rbarp$, and we set $\GRhatp \coloneq \Gal(\Lbar(\mathfrak{p})/L)$.
	Moreover, we have an isomorphism of Galois groups $\GammaL = \Gal(\calL_{\infty}/\calL) \isomorphic \GammaR$, and for each prime $\pins$, let $\ALplusp$ denote the base ring for Wach modules in the imperfect residue field case (see \cite[Section 2.1.2]{abhinandan-imperfect-wach}).
	To avoid confusion, let us write $\NR \coloneq N$ and $\NLp \coloneq \ALplusp \otimes_{\AR^+} N$, in particular, $\NLp$ is a Wach module with $\Delta\textrm{-action}$ over $\ALplusp$ finite free of rank $= \textrm{rk}_{\mathbb{Z}_p} T$.
	From \cite[Lemma 3.6]{abhinandan-imperfect-wach} note that for some $s \geqslant 0$, we have $\GRhatp\textrm{-equivariant}$ inclusions for each $\pins$,
	\begin{equation}\label{eq:wachmod_almost_comp_imperfect}
		\mu^s A_{\inf}(\Cpplus) \otimes_{\mathbb{Z}_p} T \subset A_{\inf}(\Cpplus) \otimes_{\ALplusp} \NLp \subset A_{\inf}(\Cpplus) \otimes_{\mathbb{Z}_p} T.
	\end{equation}
	
	Now, let $\GRp \coloneq \{g \in G_R \textrm{ such that } g(\mathfrak{p}) = \mathfrak{p}\}$ which is a subgroup of $\GRp$, and observe that the $(\varphi, \GRp)\textrm{-equivariant}$ composition $\AR^+ \rightarrow \tilde{A}(\overline{R}) \rightarrow \tilde{A}(\mathbb{C}(\mathfrak{p}))$ naturally factors as the $(\varphi, \GRp)\textrm{-equivariant}$ composition $\AR^+ \rightarrow \ALplusp \rightarrow \tilde{A}(\mathbb{C}(\mathfrak{p}))$.
	So, by base changing the $(\varphi, G_R)\textrm{-equivariant}$ isomorphism in \eqref{eq:phigamma_comp_tilde_relative} along the $(\varphi, \GRp)\textrm{-equivariant}$ homomorphism $\tilde{A}(\overline{R}) \rightarrow \tilde{A}(\mathbb{C}(\mathfrak{p}))$, we obtain a natural $(\varphi, \GRp)\textrm{-equivariant}$ isomorphism
	\begin{equation}\label{eq:phigamma_comp_tildep_relative}
		\tilde{A}(\mathbb{C}(\mathfrak{p})) \otimes_{\ALplusp} \NLp \isomorphic \tilde{A}(\mathbb{C}(\mathfrak{p})) \otimes_{\mathbb{Z}_p} T.
	\end{equation}
	All terms in \eqref{eq:wachmod_almost_comp_imperfect} and \eqref{eq:phigamma_comp_tildep_relative} admit $(\varphi, \GRhatp)\textrm{-equivariant}$ embedding into $\tilde{A}(\Cp) \otimes_{\ALplusp} \NLp \isomorphic \tilde{A}(\Cp) \otimes_{\mathbb{Z}_p} T$, where the action of $\GRhatp$ on \eqref{eq:phigamma_comp_tildep_relative} factors through $\GRhatp \twoheadrightarrow \GRp$.
	Therefore, taking the intersection of \eqref{eq:wachmod_almost_comp_imperfect} with \eqref{eq:phigamma_comp_tildep_relative} inside $\tilde{A}(\Cp) \otimes_{\ALplusp} \NLp \isomorphic \tilde{A}(\Cp) \otimes_{\mathbb{Z}_p} T$, and using the $(\varphi, \GRhatp)\textrm{-equivariant}$ identification $A_{\inf}(\Cplusp) = A_{\inf}(\Cpplus) \cap \tilde{A}(\mathbb{C}(\mathfrak{p})) \subset \tilde{A}(\Cp)$ for each $\pins$ (see the discussion on localisation in Section \ref{subsubsec:perfect_period_rings}), we obtain the following $(\varphi, \GRp)\textrm{-equivariant}$ inclusions:
	\begin{equation}\label{eq:wachmod_almost_comp_p}
		\mu^s A_{\inf}(\Cplusp) \otimes_{\mathbb{Z}_p} T \subset A_{\inf}(\Cplusp) \otimes_{\ALplusp} \NLp \subset A_{\inf}(\Cplusp) \otimes_{\mathbb{Z}_p} T,
	\end{equation}
	where the middle term may be written as $A_{\inf}(\Cplusp) \otimes_{\ALplusp} \NLp = A_{\inf}(\Cplusp) \otimes_{\AR^+} \NR$.

	Next, from the discussion on localisation in Section \ref{subsubsec:perfect_period_rings}, recall that $\prod_{\pins} A_{\inf}(\Cplusp)$ is equipped with a $G_R\textrm{-action}$ and we have a $(\varphi, G_R)\textrm{-equivariant}$ injective homomorphism $A_{\inf}(\overline{R}) \rightarrow \prod_{\pins} A_{\inf}(\Cplusp)$.
	Then, we may equip 
	\begin{equation*}
		\textstyle\prod_{\pins} (A_{\inf}(\Cplusp) \otimes_{\mathbb{Z}_p} T) = (\textstyle\prod_{\pins} A_{\inf}(\Cplusp)) \otimes_{\mathbb{Z}_p} T,
	\end{equation*}
	with the diagonal action of $(\varphi, G_R)$, and similarly for 
	\begin{equation*}
		\begin{aligned}
			\textstyle\prod_{\pins} \big(A_{\inf}(\Cplusp) \otimes_{\ALplusp} \NLp\big) &= \textstyle\prod_{\pins} \big(A_{\inf}(\Cplusp) \otimes_{\AR^+} \NR\big)\\
				&= \big(\textstyle\prod_{\pins} A_{\inf}(\Cplusp)\big) \otimes_{\AR^+} \NR,
		\end{aligned}
	\end{equation*}
	where the second equality follows from the fact that product is an exact functor on the category of $\AR^+\textrm{-modules}$ and $\NR$ is finitely presented over the noetherian ring $\AR^+$ (see \cite[\href{https://stacks.math.columbia.edu/tag/059K}{Tag 059K}]{stacks-project}).
	So, taking the product of \eqref{eq:wachmod_almost_comp_p} over all $\pins$ and using the discussion above, we obtain $(\varphi, G_R)\textrm{-equivariant}$ inclusions:
	\begin{equation}\label{eq:wachmod_prod_almost_comp}
		\begin{aligned}
			\mu^s \textstyle\prod_{\pins} \big(A_{\inf}(\Cplusp) \otimes_{\mathbb{Z}_p} T\big) &\subset \big(\textstyle\prod_{\pins} A_{\inf}(\Cplusp)\big) \otimes_{\AR^+} \NR\\
				&\subset \textstyle\prod_{\pins} \big(A_{\inf}(\Cplusp) \otimes_{\mathbb{Z}_p} T\big).
		\end{aligned}
	\end{equation}
	Inverting $\mu$ in \eqref{eq:wachmod_prod_almost_comp} yields the top horizontal isomorphism in the following $(\varphi, G_R)\textrm{-equivariant}$ commutative diagram:
	\begin{equation}\label{eq:wachmod_prod_muinverse_comp}
		\begin{tikzcd}[row sep=15pt]
			\big(\textstyle\prod_{\pins} A_{\inf}(\Cplusp)\big)[1/\mu] \otimes_{\AR^+[1/\mu]} \NR[1/\mu] & \big(\textstyle\prod_{\pins} A_{\inf}(\Cplusp)\big)[1/\mu] \otimes_{\mathbb{Z}_p} T\\
			\big(\textstyle\prod_{\pins} \tilde{A}(\mathbb{C}(\mathfrak{p}))\big) \otimes_{\AR^+} \NR & \big(\textstyle\prod_{\pins} \tilde{A}(\mathbb{C}(\mathfrak{p}))\big) \otimes_{\mathbb{Z}_p} T\\
			\tilde{A}(\overline{R}) \otimes_{\AR^+[1/\mu]} \NR[1/\mu] & \tilde{A}(\overline{R}) \otimes_{\mathbb{Z}_p} T,
			\arrow["\sim", from=1-1, to=1-2]
			\arrow[hook, from=1-1, to=2-1]
			\arrow[hook, from=1-2, to=2-2]
			\arrow["\sim", from=2-1, to=2-2]
			\arrow[hook', from=3-1, to=2-1]
			\arrow["\sim", "\eqref{eq:phigamma_comp_tilde_relative}"', from=3-1, to=3-2]
			\arrow[hook', from=3-2, to=2-2]
		\end{tikzcd}
	\end{equation}
	where the vertical arrows are natural inclusions by the discussion on localisation in Section \ref{subsubsec:perfect_period_rings}, and the horizontal isomorphism in the middle row is obtained by taking the product of \eqref{eq:phigamma_comp_tildep_relative} over all $\pins$ and using that (see \cite[\href{https://stacks.math.columbia.edu/tag/059K}{Tag 059K}]{stacks-project}):
	\begin{equation*}
		\begin{aligned}
			\textstyle\prod_{\pins} \big(\tilde{A}(\mathbb{C}(\mathfrak{p})) \otimes_{\ALplusp} \NLp\big) &= \textstyle\prod_{\pins} \big(\tilde{A}(\mathbb{C}(\mathfrak{p})) \otimes_{\AR^+} \NR\big)\\
			&= \big(\textstyle\prod_{\pins} \tilde{A}(\mathbb{C}(\mathfrak{p}))\big) \otimes_{\AR^+} \NR.
		\end{aligned}
	\end{equation*}
	Note that $\NR[1/\mu]$ is finite projective over $\AR^+[1/\mu]$ (see Proposition \ref{prop:wachmod_proj_pmu}), so in diagram \eqref{eq:wachmod_prod_muinverse_comp}, taking the intersection of the top left term and the bottom left term, inside the middle left term, gives
	\begin{equation*}
		\begin{aligned}
			\big(\tilde{A}(\overline{R}) \otimes_{\AR^+[1/\mu]} \NR[1/\mu]\big) &\cap \big(\big(\textstyle\prod_{\pins} A_{\inf}(\Cplusp)\big)[1/\mu] \otimes_{\AR^+[1/\mu]} \NR[1/\mu]\big)\\
			&= A_{\inf}(\overline{R})[1/\mu] \otimes_{\AR^+[1/\mu]} \NR[1/\mu] = A_{\inf}(\overline{R})[1/\mu] \otimes_{\AR^+} \NR,
		\end{aligned}
	\end{equation*}
	where the first equality follows from the discussion on localisation in Section \ref{subsubsec:perfect_period_rings}.
	Similarly, taking the intersection of the top right term and bottom right term, inside the middle right term, gives
	\begin{equation*}
		\big(\tilde{A}(\overline{R}) \otimes_{\mathbb{Z}_p} T\big) \cap \big(\big(\textstyle\prod_{\pins} A_{\inf}(\Cplusp)\big)[1/\mu] \otimes_{\mathbb{Z}_p} T\big) = A_{\inf}(\overline{R})[1/\mu] \otimes_{\mathbb{Z}_p} T,
	\end{equation*}
	where the equality again follows from the discussion on localisation in Section \ref{subsubsec:perfect_period_rings}.
	As the horizontal arrows in \eqref{eq:wachmod_prod_muinverse_comp} are bijective, therefore, we obtain that the intersection of the top and the bottom rows, inside the middle row, is naturally a $G_R\textrm{-equivariant}$ isomorphism, thus yielding the claimed isomorphism in \eqref{eq:wachmod_comp_relative_ainf}.
	From the preceding discussion, it also follows that the isomorphism in \eqref{eq:wachmod_comp_relative_ainf} is compatible with the respective Frobenii after base change along $A_{\inf}(\overline{R}) \rightarrow \tilde{A}(\overline{R})$.
\end{proof}

\begin{cor}\label{cor:wachmod_comp_relative}
	Let $N$ be a Wach module with $\Delta\textrm{-action}$ over $\AR^+$ and let $T \coloneq \TR(N)$ denote the associated finite free $\mathbb{Z}_p\textrm{-representation}$ of $G_R$.
	Then, we have a natural $G_R\textrm{-equivariant}$ comparison isomorphism:
	\begin{equation*}
		A_{\calR}^{\etale,+}[1/\mu] \otimes_{\AR^+} N \isomorphic A_{\calR}^{\etale,+}[1/\mu] \otimes_{\mathbb{Z}_p} T.
	\end{equation*}
	Additionally, the isomorphism above is compatible with the action of $(\varphi, G_R)$ after base change along the natural $(\varphi, G_R)\textrm{-equivariant}$ map $A_{\calR}^{\etale,+}[1/\mu] \rightarrow A_{\calR}^{\etale}$.
\end{cor}
\begin{proof}
	Note that $N[1/\mu]$ is finite projective over $\AR^+[1/\mu]$ (see Lemma \ref{prop:wachmod_proj_pmu}).
	Then, using the $(\varphi, G_R)\textrm{-equivariant}$ isomorphism in \eqref{eq:phigamma_comp_tilde_relative}, inside $\tilde{A}(\overline{R}) \otimes_{\mathbb{Z}_p} T$, we take the intersection of the $G_R\textrm{-equivariant}$ isomorphism \eqref{eq:wachmod_comp_relative_ainf} with the $(\varphi, \Gamma_R)\textrm{-equivariant}$ isomorphism \eqref{eq:phiGamma_comp_iso}, to obtain the following $G_R\textrm{-equivariant}$ isomorphism:
	\begin{equation*}
		A_{\calR}^{\etale,+}[1/\mu] \otimes_{\AR^+[1/\mu]} N[1/\mu] \isomorphic A_{\calR}^{\etale,+}[1/\mu] \otimes_{\mathbb{Z}_p} T,
	\end{equation*}
	where we have used the fact that $A_{\calR}^{\etale,+} = A_{\inf}(\overline{R}) \cap A_{\calR}^{\etale} \subset \tilde{A}(\overline{R})$.
	The last claim follows from the analogous claim in Proposition \ref{prop:wachmod_comp_relative}.
\end{proof}

\begin{prop}\label{prop:wachmod_almost_comp}
	Let $N$ be an effective Wach module with $\Delta\textrm{-action}$ over $\AR^+$ and $T \coloneq \TR(N)$ the associated finite free $\mathbb{Z}_p\textrm{-representation}$ of $G_R$.
	Then, we have $(\varphi, \GammaR \times \Delta)\textrm{-equivariant}$ inclusions $\mu^s \DR^+(T) \subset N \subset \DR^+(T)$ (see before Lemma \ref{lem:regsec_D+} for notation).
\end{prop}
\begin{proof}
	The proof follows in a manner similar to the proof of Proposition \ref{prop:wachmod_comp_relative}, so we will freely use the notations therein.
	By inverting $p$ in \eqref{eq:wachmod_prod_almost_comp} we have $(\varphi, G_R)\textrm{-equivariant}$ inclusions
	\begin{equation}\label{eq:wachmod_prod_pinverse_almost_comp}
		\begin{aligned}
			\mu^s \big(\textstyle\prod_{\pins} \big(A_{\inf}(\Cplusp) \otimes_{\mathbb{Z}_p} T\big)\big)[1/p] &\subset \big(\textstyle\prod_{\pins} A_{\inf}(\Cplusp)\big)[1/p] \otimes_{\BR^+} \NR[1/p]\\
			&\subset \big(\textstyle\prod_{\pins} \big(A_{\inf}(\Cplusp) \otimes_{\mathbb{Z}_p} T\big)\big)[1/p],
		\end{aligned}
	\end{equation}
	where the last term may also be written as $\big(\prod_{\pins} A_{\inf}(\Cplusp)\big)[1/p] \otimes_{\mathbb{Q}_p} V$, and similarly for the first term.
	Moreover, by inverting $p$ in \eqref{eq:phigamma_comp_tilde_relative}, we have the following $(\varphi, G_R)\textrm{-equivariant}$ comparison isomorphism:
	\begin{equation}\label{eq:phigamma_comp_pinverse}
		\tilde{A}(\overline{R})[1/p] \otimes_{\BR^+} \NR[1/p] \isomorphic \tilde{A}(\overline{R})[1/p] \otimes_{\mathbb{Q}_p} V.
	\end{equation}
	Using the discussion on localisation in Section \ref{subsubsec:perfect_period_rings} and diagram \eqref{eq:wachmod_prod_muinverse_comp} (after inverting $p$), we embed all terms of \eqref{eq:wachmod_prod_pinverse_almost_comp} and \eqref{eq:phigamma_comp_pinverse}, in a $(\varphi, G_R)\textrm{-equivariant}$ manner, inside
	\begin{equation}\label{eq:phigamma_comp_prod_pinverse}
		\big(\textstyle\prod_{\pins} \tilde{A}(\mathbb{C}(\mathfrak{p}))\big)[1/p] \otimes_{\BR^+} \NR[1/p] \isomorphic \big(\textstyle\prod_{\pins} \tilde{A}(\mathbb{C}(\mathfrak{p}))\big)[1/p] \otimes_{\mathbb{Q}_p} V,
	\end{equation}
	where the isomorphism above is obtained by inverting $p$ in the middle row of \eqref{eq:wachmod_prod_muinverse_comp}.
	Since $\NR[1/p]$ is a finite projective module over $\BR^+$, therefore, the intersection of the middle term of \eqref{eq:wachmod_prod_pinverse_almost_comp} and the left-hand term of \eqref{eq:phigamma_comp_pinverse}, inside the left-hand term of \eqref{eq:phigamma_comp_prod_pinverse}, gives
	\begin{equation*}
		\begin{aligned}
			\big(\tilde{A}(\overline{R})[1/p] \otimes_{\BR^+} \NR[1/p]\big) \cap \big(\big(\textstyle\prod_{\pins} A_{\inf}(\Cplusp)\big)[1/p] &\otimes_{\BR^+} \NR[1/p]\big)\\
			&= A_{\inf}(\overline{R})[1/p] \otimes_{\BR^+} \NR[1/p],
		\end{aligned}
	\end{equation*}
	where the equality follows from the discussion on localisation in Section \ref{subsubsec:perfect_period_rings}.
	Similarly, taking the intersection of the right-most terms of \eqref{eq:wachmod_prod_pinverse_almost_comp} and \eqref{eq:phigamma_comp_pinverse}, inside the right-hand term of \eqref{eq:phigamma_comp_prod_pinverse}, gives
	\begin{equation*}
		\big(\tilde{A}(\overline{R})[1/p] \otimes_{\mathbb{Q}_p} V\big) \cap \big(\big(\textstyle\prod_{\pins} A_{\inf}(\Cplusp)\big)[1/p] \otimes_{\mathbb{Q}_p} V\big) = A_{\inf}(\overline{R})[1/p] \otimes_{\mathbb{Q}_p} V,
	\end{equation*}
	where the equality again follows from the discussion on localisation in Section \ref{subsubsec:perfect_period_rings}.
	Therefore, from \eqref{eq:wachmod_prod_pinverse_almost_comp} and using the $(\varphi, G_R)\textrm{-equivariance}$ of \eqref{eq:phigamma_comp_pinverse}, we obtain the following $(\varphi, G_R)\textrm{-equivariant}$ inclusions:
	\begin{equation}\label{eq:wachmodrat_almost_comp}
		\mu^s \big(A_{\inf}(\overline{R})[1/p] \otimes_{\mathbb{Q}_p} V\big) \subset A_{\inf}(\overline{R})[1/p] \otimes_{\BR^+} \NR[1/p] \subset A_{\inf}(\overline{R})[1/p] \otimes_{\mathbb{Q}_p} V.
	\end{equation}
	Inverting $p$ in the isomorphism of Corollary \ref{cor:wachmod_comp_relative} and taking its intersection with \eqref{eq:wachmodrat_almost_comp}, inside \eqref{eq:phigamma_comp_pinverse}, we obtain the following $(\varphi, G_R)\textrm{-equivariant}$ inclusions:
	\begin{equation*}
		\mu^s \big(B_{\calR}^{\etale,+} \otimes_{\mathbb{Q}_p} V\big) \subset B_{\calR}^{\etale,+} \otimes_{\BR^+} \NR[1/p] \subset B_{\calR}^{\etale,+} \otimes_{\mathbb{Q}_p} V.
	\end{equation*}
	In the preceding equation, first we take the $\HR^{\etale}\textrm{-invariants}$ and then its intersection with $\DR(T) = \NR[1/\mu]^{\wedge}$, inside $\DR(V)$, to obtain $(\varphi, \GammaR \times \Delta)\textrm{-equivariant}$ inclusions $\mu^s \DR^+(T) \subset \NR \subset \DR^+(T)$, because we have $\NR = \NR[1/p] \cap \DR$ from Lemma \ref{lem:wach_intersection_lemma} and $\DR^+(T) = \DR(T) \cap \DR^+(V) \subset \DR(V)$, by definition.
	Hence, the proposition is proved.
\end{proof}

In light of Lemma \ref{lem:regsec_D+} and Proposition \ref{prop:wachmod_almost_comp}, we ask the following:
\begin{ques}
	For each $1 \leqslant i \leqslant d$, are the sequences $\{p, \mu, [x_i^{\flat}]\}$ and $\{\mu, p, [x_i^{\flat}]\}$ regular on a Wach module $N$ over $\AR^+$?
\end{ques}

Restricting the discussion above to $\GR\textrm{-equivariant}$ objects, instead of $G_R\textrm{-equivariant}$ objects, yields parallel statements for $p\textrm{-adic}$ representations of $\GR$.
Indeed, note that the composition of the functor in Proposition \ref{prop:wach_etale_ff_relative} with the categorical equivalence in Corollary \ref{cor:padic_classify_GR0m} yields a fully faithful functor
\begin{equation}\label{eq:wach_reps_GR0m}
	\begin{aligned}
		\TR \colon (\varphi, \GammaR)\textup{-Mod}_{\AR^+}^{[p]_q} &\longrightarrow \textup{Rep}_{\mathbb{Z}_p}(\GR^{\etale})\\
			N &\longmapsto \big(\tilde{A}(\overline{R}) \otimes_{\AR^+} N\big)^{\varphi = 1},
	\end{aligned}
\end{equation}
where we observe that the $\mathbb{Z}_p\textrm{-module}$ $\TR(N)$ is in fact a representation of $\GR$ on which the action factors through the quotient $\GR \twoheadrightarrow \GR^{\etale}$, and in the following, we shall view $\TR(N)$ as a representation of $\GR$.

\begin{cor}\label{cor:wachmod_comp_GR0m}
	Let $N$ be a Wach module over $\AR^+$, and let $T \coloneq \TR(N)$ be the associated finite free $\mathbb{Z}_p\textrm{-representation}$ of $\GR$.
	Then,
	\begin{enumerate}
		\item We have a natural $\GR\textrm{-equivariant}$ comparison isomorphism:
			\begin{equation}\label{eq:wachmod_comp_ainf_GR0m}
				A_{\inf}(\overline{R})[1/\mu] \otimes_{\AR^+} N \isomorphic A_{\inf}(\overline{R})[1/\mu] \otimes_{\mathbb{Z}_p} T,
			\end{equation}
			which is further compatible with the action of $(\varphi, \GR)$ after base change along $A_{\inf}(\overline{R})[1/\mu] \rightarrow \tilde{A}(\overline{R})$.
	
		\item The isomorphism in \eqref{eq:wachmod_comp_ainf_GR0m} restricts to a natural $\GR\textrm{-equivariant}$ comparison isomorphism:
			\begin{equation*}
				A_{\calR}^{\etale,+}[1/\mu] \otimes_{\AR^+} N \isomorphic A_{\calR}^{\etale,+}[1/\mu] \otimes_{\mathbb{Z}_p} T.
			\end{equation*}
			which is further compatible with the action of $(\varphi, \GR)$ after base change along the natural $(\varphi, \GR)\textrm{-equivariant}$ map $A_{\calR}^{\etale,+}[1/\mu] \rightarrow A_{\calR}^{\etale}$.

		\item We have $(\varphi, \GammaR)\textrm{-equivariant}$ inclusions $\mu^s \DR^+(T) \subset N \subset \DR^+(T)$.
	\end{enumerate}
\end{cor}
\begin{proof}
	By restricting to $\GR\textrm{-action}$, instead of $G_R\textrm{-action}$, the claim in (1) follows from Proposition \ref{prop:wachmod_comp_relative}, the claim in (2) follows from Corollary \ref{cor:wachmod_comp_relative}, and the claim in (3) follows from Proposition \ref{prop:wachmod_almost_comp}.
\end{proof}

\subsection{Constructing Wach modules}\label{subsec:constructing_wach}

The goal of this section is to prove the following claim:
\begin{thm}\label{thm:wachmod_existence}
	Let $\DR$ be an \'etale $(\varphi, \GammaR \times \Delta)\textrm{-module}$ over $\AR$ and set $\DS \coloneq \AS \otimes_{\AR} \DR$ as an \'etale $(\varphi, \GammaS \times \Delta)\textrm{-module}$ over $\AS$.
	Let $\NS \subset \DS$ be a Wach module with $\Delta\textrm{-action}$ over $\AS^+$.
	Then, $\NR \coloneq \DR \cap \NS \subset \DS$ is a Wach module with $\Delta\textrm{-action}$ over $\AR^+$, i.e.\ it satisfies all the axioms of Definition \ref{defi:wach_mods_with_delta}.
\end{thm}

The proof of Theorem \ref{thm:wachmod_existence} requires some preparations which we carry out next.
Let $O_{\calL} \coloneq \calR_{(p)}^{\wedge} = \calS_{(p)}^{\wedge}$ as in the proof of Proposition \ref{prop:wachmod_comp_relative}.
Recall that we have natural isomorphism of groups $\GammaR \isomorphic \GammaS \isomorphic \GammaL$, we shall identify these below.
Moreover, $\NL \coloneq \AL^+ \otimes_{\AS^+} \NS$ is a Wach module with $\Delta\textrm{-action}$ over $\AL^+$, and we have that $\NS = \DS \cap \NL \subset \DL = \AL \otimes_{\AR} \DR$ from \cite[Proposition 4.2]{abhinandan-relative-wach-ii}.
Consequently, we see that inside $\DL$, we have $(\varphi, \GammaR \times \Delta)\textrm{-equivariant}$ identifications $\NR = \DR \cap \NS = \DR \cap \NL \cap \DS = \DR \cap \NL$.

\begin{lem}\label{lem:fingen_finiteheight}
	The $\AR^+\textrm{-module}$ $\NR$ is finitely generated.
	Moreover, $\NR$ is of finite $[p]_q\textrm{-height}$, i.e.\ the cokernel of the injective map $1 \otimes \varphi \colon \varphi^*(\NR) \rightarrow \NR$ is killed by $[p]_q^s$, for some $s \in \mathbb{N}$.
\end{lem}
\begin{proof}
	The first claim of the lemma follows by the same argument as in \cite[Lemma 4.3]{abhinandan-relative-wach-ii} and the second claim follows by the same argument as in \cite[Lemma 4.5]{abhinandan-relative-wach-ii}.
\end{proof}

Next, we will show that $\AL^+ \otimes_{\AR^+} \NR \isomorphic \NL$ and $\AR \otimes_{\AR^+} \NR \isomorphic \DR$, using an approach parallel to \cite[Lemma 4.7 \& Proposition 4.8]{abhinandan-relative-wach-ii}.
For $n \in \mathbb{N}_{\geqslant 1}$, let $\NRn \coloneq \NR/p^n$, $\DRn \coloneq \DR/p^n$, $\NLn \coloneq \NL/p^n$, $\DLn \coloneq \DL/p^n$ and $M_n \coloneq \NLn \cap \DRn \subset \DLn$.
Then, have the following commutative diagram:
\begin{center}
	\begin{tikzcd}[row sep=15pt]
		M_n \arrow[r, "f_n"] \arrow[d] & M_1 \arrow[d]\\
		\DRn \arrow[r, twoheadrightarrow, "f_n"] & D_{\calR,1},
	\end{tikzcd}
\end{center}
where the vertical arrows are natural inclusions, the bottom horizontal arrow $f_n$ is the natural projection map and the top arrow is the induced map.
We have a similar diagram with the bottom row replaced by $\NLn \twoheadrightarrow N_{\calL, 1}$.

\begin{lem}\label{lem:Mn_large}
	The following statements are true:
	\begin{enumerate}
		\item[\textup{(1)}] $M_n$ is a finitely generated $\AR^+/p^n\textrm{-module}$ and $\NR \isomorphic \lim_n M_n$.
	
		\item[\textup{(2)}] $M_n$ is of finite $[p]_q\textrm{-height}$ $s$, for $s \in \mathbb{N}$ from Lemma \ref{lem:fingen_finiteheight}.
	
		\item[\textup{(3)}] $M_n[1/\mu] = \AR \otimes_{\AR^+} M_n \isomorphic \DRn$ and $\AL^+ \otimes_{\AR^+} M_n \isomorphic \NLn$.
	\end{enumerate}
\end{lem}
\begin{proof}
	The claim in (1) and (2) follow by an argument similar to Lemma \ref{lem:fingen_finiteheight} and the claim in (3) follows by an argument similar to \cite[Lemma 4.6]{abhinandan-relative-wach-ii}.
	To obtain the first isomorphism in (3), note that we have natural $\AR\textrm{-linear}$ inclusions $\AR \otimes_{\AR^+} M_n = M_n[1/\mu] \hookrightarrow \DRn \hookrightarrow \DLn$, because $M_n$ is $\mu\textrm{-torsion}$ free.
	Moreover, we have a natural $\AL\textrm{-linear}$ isomorphism $\NLn[1/\mu] \isomorphic \DLn$.
	Then, inside $\DLn$ we have that
	\begin{equation}\label{eq:mn_muinverse}
		M_n[1/\mu] \isomorphic \DRn \cap \NLn[1/\mu] \isomorphic \DRn \cap \DLn = \DRn,
	\end{equation}
	where the first isomorphism follows because we have a natural $\AR\textrm{-linear}$ inclusion $M_n[1/\mu] \hookrightarrow \DRn \cap \NLn[1/\mu]$, and to obtain its surjectivity let $x$ be an element of $\DRn \cap \NLn[1/\mu]$, then there exists some $k \in \mathbb{N}$ such that $\mu^kx$ is in $\NLn$, i.e.\ $\mu^k x$ is in $\DRn \cap \NLn = M_n$, as claimed.
	This gives us the first claim in (3).
	
	To obtain the second isomorphism in (3), note that from \eqref{eq:mn_muinverse} and \cite[Lemma A.1]{abhinandan-relative-wach-ii}, it follows that we have a natural $\AR^+/\mu\textrm{-linear}$ injective homomorphism $M_n/\mu \hookrightarrow \NLn/\mu$.
	As $\NLn/\mu$ is a free $O_{\calL}/p^n\textrm{-module}$, $\AR^+/\mu \isomorphic R$ and $\AL^+/\mu \isomorphic O_{\calL}$, and $R/p^n \hookrightarrow (R_{(p)})/p^n \isomorphic O_{\calL}/p^n$, therefore, by localising we obtain that $(R/p^n)_{(p)} \otimes_{R/p^n} M_n/\mu \hookrightarrow \NLn/\mu$.
	But, recall that localisation commutes with passing to quotient by ideals (see \cite[Theorem 4.2]{matsumura}), therefore, we have that $(R/p^n)_{(p)} \isomorphic (R_{(p)})/p^n \isomorphic O_{\calL}/p^n$, and hence $(\AL^+ \otimes_{\AR^+} M_n)/\mu \isomorphic O_{\calL}/p^n \otimes_{R/p^n} M_n/\mu \hookrightarrow \NLn/\mu$.
	In particular, we see that inside $\DLn$ we have,
	\begin{equation*}
		\AL^+ \otimes_{\AR^+} M_n = \NLn \cap \big(\AL^+ \otimes_{\AR^+} M_n[1/\mu]\big) \isomorphic \NLn \cap \big(\AL^+ \otimes_{\AR^+} \DRn\big) \isomorphic \NLn \cap \DLn = \NLn,
	\end{equation*}
	where we used \cite[Lemma A.1]{abhinandan-relative-wach-ii} for the equality and \eqref{eq:mn_muinverse} for the middle isomorphism.
	This completes the proof.
\end{proof}

Let $\mathscr{T}$ denote the set of $\AR^+\textrm{-submodules}$ $M' \subset M_1$ such that $M'$ is stable under the action of $\varphi$, it is of finite $[p]_q\textrm{-height}$ $s \in \mathbb{N}$ and $M'[1/\mu] = M_1[1/\mu] = D_{\calR,1} = \DR/p$.
Set $M^{\circ} \coloneq \cap_{M' \in \mathscr{T}} M' \subset M_1$.

\begin{lem}\label{lem:mcirc_props}
	The $\AR^+\textrm{-module}$ $M^{\circ}$ belongs to $\mathscr{T}$ and $f_n(M_n)$ is also in $\mathscr{T}$, for all $n \in \mathbb{N}_{\geqslant 1}$.
\end{lem}
\begin{proof}
	Follows by an argument similar to \cite[Lemma 4.7]{abhinandan-relative-wach-ii}.
	Let $M'$ be an element of $\mathscr{T}$.
	First, let us show that there exists some $r \in \mathbb{N}$ such that $\mu^r M_1 \subset M' \subset M_1$.
	Let $M'' \coloneq M_1/M'$ such that $M'' \neq 0$ and let $k = p(p-1)s \in \mathbb{N}$.
	Also, let $\varphi^*(M'') \coloneq \varphi^*(M_1)/\varphi^*(M')$ and let $1 \otimes \varphi_{M''} \colon \varphi^*(M'') \rightarrow M''$ denote the map induced from $1 \otimes \varphi_M$.
	Since $M_1$ (resp.\ $M'$) is of finite $[p]_q\textrm{-height}$ $k$ (since $s < k$), we define $\psi_M \colon M_1 \xrightarrow{\mu^k} \mu^k M_1 \rightarrow \varphi^*(M_1)$ (resp.\ $\psi_{M'} \colon M' \xrightarrow{\mu^k} \mu^k M' \rightarrow \varphi^*(M')$) to be the unique $\AR^+/p\textrm{-linear}$ map such that $\psi_M \circ (1 \otimes \varphi_M) = \mu^k \textup{Id}_{\varphi^*_M}$ (resp.\ $\psi_{M'} \circ (1 \otimes \varphi_{M'})  = \mu^k \textup{Id}_{\varphi^*_{M'}}$).
	Let $\psi_{M''} \colon M'' \rightarrow \varphi^*(M'')$ denote the map induced from $\psi_M$.
	Now, consider the following commutative diagram:
	\begin{center}
		\begin{tikzcd}
			0 \arrow[r] & \varphi^*(M') \arrow[r] \arrow[d, "1 \otimes \varphi_{M'}"] & \varphi^*(M_1) \arrow[r] \arrow[d, "1 \otimes \varphi_{M}"] & \varphi^*(M'') \arrow[r] \arrow[d, "1 \otimes \varphi_{M''}"] & 0\\
			0 \arrow[r] & M' \arrow[r] \arrow[d, "\psi_{M'}"] & M_1 \arrow[r] \arrow[d, "\psi_{M}"] & M'' \arrow[r] \arrow[d, "\psi_{M''}"] & 0\\
			0 \arrow[r] & \varphi^*(M') \arrow[r] & \varphi^*(M_1) \arrow[r] & \varphi^*(M'') \arrow[r] & 0.
		\end{tikzcd}
	\end{center}
	Note that $[p]_q = \mu^{p-1} \mod p$, $\varphi(\mu) = \mu^p \mod p$ and $\varphi([p]_q) = \mu^{p(p-1)} \mod p$.
	Since $M_1[1/\mu] = M'[1/\mu]$, let $i \in \mathbb{N}_{\geqslant 1}$ such that $\mu^{pi} M'' = 0$ and $\mu^{p(i-1)} M'' \neq 0$.
	Let $x$ be in $M''$ such that $\mu^{p(i-1)} x \neq 0$ and set $y = 1 \otimes x$ to be in $\varphi^*(M'')$.
	Then, $\varphi(\mu^{pi}) y = 1 \otimes \mu^{pi} x = 0$, but 
	\begin{equation*}
		\mu^{p^2(i-1)}y = \varphi(\mu^{p(i-1)}) y = 1 \otimes \mu^{p(i-1)} x \neq 0.
	\end{equation*}
	Let $z = (1 \otimes \varphi_{M''}) y$ be in $M''$, then $\mu^{pi} z = 0$.
	So, we have that
	\begin{equation*}
		0 = \psi_{M''}(\mu^{pi} z) = \mu^{pi} (\psi_{M''} \circ (1 \otimes \varphi_{M''}) y) = \mu^{pi+k} y.
	\end{equation*}
	Therefore, we get that $pi+k = pi+p(p-1)s > p^2 (i-1)$, i.e.\ $i < s + \frac{p}{p-1}$.
	Hence, it follows that $\mu^{p(s+1)} M'' = 0$.
	Since the constant $i$ obtained above is independent of $M'$, we also get that $\mu^{p(s+1)} M_1 \subset M^{\circ} \subset M_1$ and $M^{\circ}[1/\mu] = M_1[1/\mu]$.

	Next, we will show that $M^{\circ}$ is of finite height $s$.
	Let $x$ be in $M^{\circ}$, so that $x$ is in $M'$, for each $M'$ in $\mathscr{T}$, and there exists some $y$ in $\varphi^*(M') \subset \varphi^*(M_1)$ such that $(1 \otimes \varphi)y = [p]_q^s x$.
	Note that $y$ is unique in $\varphi^*(M_1)$ and since $\varphi \colon \AR^+ \rightarrow \AR^+$ is flat, we get that $y$ is in $\cap_{M' \in \mathscr{T}} (\AR^+ \otimes_{\varphi, \AR^+} M') = \AR^+ \otimes_{\varphi, \AR^+} (\cap_{M' \in \mathscr{T}} M') = \varphi^*(M^{\circ})$.
	Therefore, we conclude that $M^{\circ}$ is an element of $\mathscr{T}$.

	For the second part of the claim, note that $M_n[1/\mu] = \DRn$ and $f_n(\DRn) = \DRn/p = \DR/p$ (see Lemma \ref{lem:Mn_large}).
	So, we get that $f_n\big(M_n[1/\mu]\big) = \DR/p$ and we are left to show that $f_n(M_n)$ is of finite height $s$.
	Note that we have the following commutative diagram with exact rows:
	\begin{center}
		\begin{tikzcd}
			0 \arrow[r] & \varphi^*(\textrm{kernel}) \arrow[r] \arrow[d, "1 \otimes \varphi"] & \varphi^*(M_n) \arrow[r] \arrow[d, "1 \otimes \varphi"] & \varphi^*(f_n(M_n)) \arrow[r] \arrow[d, "1 \otimes \varphi"] & 0\\
			0 \arrow[r] & \textrm{kernel} \arrow[r] & M_n \arrow[r, "f_n"] & f_n(M_n) \arrow[r] & 0.
		\end{tikzcd}
	\end{center}
	The rightmost vertical arrow is injective because $f_n(M_n) \subset \DRn$ and the cokernel of the middle vertical arrow is killed by $[p]_q^s$ (see Lemma \ref{lem:Mn_large}).
	Hence, the cokernel of the rightmost vertical arrow is also killed by $[p]_q^s$.
	This concludes our proof.
\end{proof}

\begin{prop}\label{prop:wach_phigamm_comp}
	The natural inclusion $\NR \subset \DR$ extends to a natural $(\varphi, \GammaR \times \Delta)\textrm{-equivariant}$ isomorphism $\AR \otimes_{\AR^+} \NR \isomorphic \DR$.
\end{prop}
\begin{proof}
	Note that $\DR$ and $\NR$ are $p\textrm{-torsion}$ free and $p\textrm{-adically}$ complete, so it is enough to show the claim modulo $p$.
	From the proof of Lemma \ref{lem:fingen_finiteheight} (see \cite[Lemma 4.3]{abhinandan-relative-wach-ii}), recall that we have $\NR/p \subset M_1 = \DR/p \cap \NL/p \subset \DL/p$, and from Lemma \ref{lem:mcirc_props} we have that $M^{\circ} \subset \NR/p$.
	Therefore,
	\begin{equation*}
		\DR/p = M^{\circ}[1/\mu] \subset \AR/p \otimes_{\AR^+/p} \NR/p \subset M_1[1/\mu] = \DR/p.
	\end{equation*}
	Hence, $(\NR/p)[1/\mu] \isomorphic \DR/p$, thus proving the claim.
 \end{proof}

\begin{proof}[Proof of Theorem \ref{thm:wachmod_existence}]
	It is immediate that $\NR$ is $p\textrm{-torsion}$ free and $\mu\textrm{-torsion}$ free.
	From Lemma \ref{lem:fingen_finiteheight}, note that $\NR(T)$ is finitely generated over $\AR^+$, and we have that $\NR/p \subset (\NL/p) \cap (\DR/p) \subset \DL/p$, in particular, $\NR/p$ is $\mu\textrm{-torsion}$ free.
	Then, by using \cite[Lemma A.2]{abhinandan-relative-wach-ii}, we also get that $(\NR/\mu)[p] = (\NR/p)[\mu] = 0$, i.e.\ $\NR/\mu$ is $p\textrm{-torsion}$ free.
	Moreover, from Lemma \ref{lem:fingen_finiteheight}, we know that $\NR$ is of finite $[p]_q\textrm{-height}$, i.e.\ the cokernel of the injective map $1 \otimes \varphi \colon \varphi^*(\NR) \rightarrow \NR$ is killed by $[p]_q^s$, where $s$ is the height of $\NL$.
	Furthermore, from Proposition \ref{prop:wach_phigamm_comp} we have that $\AR \otimes_{\AR^+} \NR \isomorphic \DR$.
	Finally, recall that the action of $\GammaL$ is trivial on $\NL/\mu\NL$ and $\GammaL \isomorphic \GammaR$, so for any $g \in \GammaR$, we have that $(g-1)\NL \subset \mu \NL$.
	Therefore, we get that $(g-1)\NR \subset (\mu \NL) \cap \DR = \mu \NR$, so it follows that $\GammaR$ acts trivially on $\NR/\mu\NR$.
	This concludes our proof.
\end{proof}

\begin{cor}\label{cor:wachmod_base_change}
	The natural inclusion $\NR \subset \NS$ extends to a natural $(\varphi, \GammaS \times \Delta)\textrm{-equivariant}$ isomorphism $\AS^+ \otimes_{\AR^+} \NR \isomorphic \NS$.
	Similarly, the natural inclusion $\NR \subset \NL$ extends to a natural $(\varphi, \GammaL \times \Delta)\textrm{-equivariant}$ isomorphism $\AL^+ \otimes_{\AR^+} \NR \isomorphic \NL$.
\end{cor}
\begin{proof}
	Proof of both the claims follow by a similar argument, so we only show the first claim.
	Using Theorem \ref{thm:wachmod_existence}, observe that $\AS^+ \otimes_{\AR^+} \NR$ and $\NS$ are both Wach modules with $\Delta\textrm{-action}$ over $\AS^+$, and we have a natural $(\varphi, \GammaS \times \Delta)\textrm{-equivariant}$ homomorphism $\AS^+ \otimes_{\AR^+} \NR \rightarrow \NS$.
	From Proposition \ref{prop:wach_phigamm_comp} we see that the preceding homomorphism induces a natural isomorphism of \'etale $(\varphi, \GammaS \times \Delta)\textrm{-modules}$ $\AS \otimes_{\AR^+} \NR \rightarrow \AS \otimes_{\AS^+} \NS = \DS$ over $\AS$.
	Let $T \coloneq \TS(\AS \otimes_{\AR^+} \NR) \isomorphic \TS(\DS)$ denote the associated $\mathbb{Z}_p\textrm{-representation}$ of $G_S$, and note that $T[1/p]$ is a crystalline representation of $G_S$ from \cite[Theorem 3.35]{abhinandan-relative-wach-ii}.
	By the uniqueness of Wach modules associated to $T$ (see \cite[Theorem 1.5]{abhinandan-relative-wach-ii}), it follows that the natural $(\varphi, \GammaS \times \Delta)\textrm{-equivariant}$ homomorphism $\AS^+ \otimes_{\AR^+} \NR \rightarrow \NS$ is an isomorphism.
\end{proof}

By restricting to $\GR\textrm{-action}$, instead of $G_R\textrm{-action}$, in the discussion above we obtain the following:
\begin{rem}\label{rem:wachmod_existence}
	Let $\DR$ be an \'etale $(\varphi, \GammaR)\textrm{-module}$ over $\AR$ and set $\DS \coloneq \AS \otimes_{\AR} \DR$ as an \'etale $(\varphi, \GammaS)\textrm{-module}$ over $\AS$.
	Let $\NS \subset \DS$ be a Wach module over $\AS^+$.
	Then, $\NR \coloneq \DR \cap \NS \subset \DS$ is a Wach module over $\AR^+$, i.e.\ it satisfies all the axioms of Definition \ref{defi:wach_mods}.

	Moreover, the natural inclusion $\NR \subset \DR$ extends to a natural $(\varphi, \GammaR)\textrm{-equivariant}$ isomorphism $\AR \otimes_{\AR^+} \NR \isomorphic \DR$, the natural inclusion $\NR \subset \NS$ extends to a natural $(\varphi, \GammaS)\textrm{-equivariant}$ isomorphism $\AS^+ \otimes_{\AR^+} \NR \isomorphic \NS$, and the natural inclusion $\NR \subset \NL$ extends to a natural $(\varphi, \GammaL)\textrm{-equivariant}$ isomorphism $\AL^+ \otimes_{\AR^+} \NR \isomorphic \NL$.
\end{rem}

\section{Log-Crystalline Galois representations}\label{subsec:crystalline_rings}

The goal of this section is to define log-crystalline period rings, and study the properties of admissible $p\textrm{-adic}$ Galois representations with respect to these period rings.
We will keep the notation of previous sections.

\subsection{de Rham period rings}\label{subsec:deRham_period_rings}

In this section, we will define several logarithmic period rings and study their properties.
For a fixed positive integer $m$ coprime to $p$, let $\calR \coloneq R_{0,m}$.
We consider $R$ and $\calR$ as log-schemes equipped with a log-structure defined by the divisor $\{x_{a+1} \cdots x_d = 0\}$.
Moreover, we equip the $R\textrm{-algebras}$ $\widehat{\calR}_{\infty}$, $\widehat{R}_{\infty,\infty}$ and $\widehat{\overline{R}}$ with a log-structure induced from the log-structure on $R$.

Recall that in Section \ref{subsubsec:perfect_period_rings} we defined the infinitesimal period rings and studied their properties.
Our next goal is to define logarithmic de Rham period rings and study their properties.

\subsubsection{Logarithmic de Rham period rings}\label{subsubsec:BdR}

Definitions of various de Rham period rings below have been adapted to the current setting from \cite[Section 3.2]{andreatta-iovita-semistable}, \cite[Chapitre 5]{brinon-relatif}, \cite[Section 2]{brinon-imparfait} and \cite[Section 2.3]{abhinandan-relative-wach-ii}.

We start by noting that the $G_R\textrm{-equivariant}$ surjective homomorphism $\theta \colon A_{\inf}(\overline{R}) \twoheadrightarrow \widehat{\overline{R}}$ is compatible with log-structures described in Section \ref{subsubsec:perfect_period_rings} (see before Lemma \ref{lem:Ainf_Rinftym_in_Sinftym}), it naturally extends to a $G_R\textrm{-equivariant}$ surjective homomorphism $\theta \colon A_{\inf}(\overline{R})[1/p] \twoheadrightarrow \widehat{\overline{R}}[1/p]$ compatible with the induced log-structures, and its kernel is principal and generated by $\xi$.
The closed immersion of log-schemes induced by the map $\theta$ is exact, and the $n\textrm{-th}$ log-infinitesimal neighbourhood of $\theta$ (in the sense of \cite[Remark 5.8]{kato-fontaine-illusie-i}) coincides with the $n\textrm{-th}$ infinitesimal neighbourhood of $\theta$.
So we set
\begin{equation*}
	\BlogdR^+(\overline{R}) \coloneq \lim_n (A_{\inf}(\overline{R})[1/p])/(\ker \theta)^n,
\end{equation*}
equipped with a log-structure induced by the natural map $A_{\inf}(\overline{R}) \rightarrow \BlogdR^+(\overline{R})$.
Note that $t \coloneq \log(1+\mu)$ converges in $\BlogdR^+(\overline{R})$ and this ring is $t\textrm{-torsion free}$; we set $\BlogdR(\overline{R}) \coloneq \BlogdR^+(\overline{R})[1/t]$.

Analogously, we may $R[1/p]\textrm{-linearly}$ extend the $F\textrm{-linear}$ homomorphism $\theta \colon A_{\inf}(\overline{R})[1/p] \twoheadrightarrow \widehat{\overline{R}}[1/p]$ to obtain a $G_R\textrm{-equivariant}$ surjective homomorphism $\theta_R \colon (A_{\inf}(\overline{R}) \otimes_{O_F} R)[1/p] \twoheadrightarrow \widehat{\overline{R}}[1/p]$ compatible with the induced log-structures, where the log-structure on the source of $\theta_R$ is given as the product of the respective log-structure on the two components.
We define $\pazocal{O}B^+_{\textup{dR},\log,n}(\overline{R})$ to be the $n\textrm{-th}$ log-infinitesimal neighbourhood of the closed immersion of log-schemes induced by the map $\theta_R$, in the sense of \cite[Remark 5.8]{kato-fontaine-illusie-i}.
The rings $\pazocal{O}B^+_{\textup{dR},\log,n}(\overline{R})$ naturally form an inverse system for $n \geqslant 1$, and we set
\begin{equation*}
	\OBlogdR^+(\overline{R}) \coloneq \lim_n \pazocal{O}B^+_{\textup{dR},\log,n}(\overline{R}).
\end{equation*}
Note that $\OBlogdR^+(\overline{R})$ is $t\textrm{-torsion free}$, and we set $\OBlogdR(\overline{R}) \coloneq \OBlogdR^+(\overline{R})[1/t]$.

The rings defined above are naturally equipped with a $G_R\textrm{-action}$, a natural extension of the map $\theta$ (before inverting $t$) and a $G_R\textrm{-stable}$ decreasing, separated and exhaustive filtration.
Additionally, we have identifications $\OBlogdR^+(\overline{R})^{G_R} = \OBlogdR(\overline{R})^{G_R} = R[1/p]$.
Moreover, the ring $\OBlogdR(\overline{R})$ is further equipped with a $\BlogdR(\overline{R})\textrm{-linear}$ and $G_R\textrm{-equivariant}$ integrable logarithmic connection 
\begin{equation*}
	\partial \colon \OBlogdR(\overline{R}) \rightarrow \OBlogdR(\overline{R}) \otimes_{R[1/p]} \omega^1_{R/O_F},
\end{equation*}
satisfying Griffiths transversality with respect to the filtration; we equip $\OBlogdR^+(\overline{R})$ with the induced logarithmic connection.

Furthermore, similar to \cite[Proposition 3.15]{andreatta-iovita-semistable}, we see that the natural $G_R\textrm{-equivariant}$ injective homomorphism of rings $\BlogdR^+(\overline{R}) \hookrightarrow \OBlogdR^+(\overline{R})$ extends to a natural isomorphism of $\BlogdR^+(\overline{R})\textrm{-algebras}$:
\begin{equation}\label{eq:OBdRlog+_explicit}
	\begin{aligned}
		\BlogdR^+(\overline{R})\llbracket y_1, \ldots, y_d \rrbracket &\isomorphic \OBlogdR^+(\overline{R})\\
		y_i &\longmapsto 1-\tfrac{[x_i^{\flat}]}{{x_i}}.
	\end{aligned}
\end{equation}
Using the isomorphism in \eqref{eq:OBdRlog+_explicit}, we may describe the logarithmic connection on $\OBlogdR(\overline{R})$ more explicitly.
Let $N_i \colon \OBlogdR(\overline{R}) \rightarrow \OBlogdR(\overline{R})$ denote the unique $\BlogdR(\overline{R})\textrm{-linear}$ continuous logarithmic differential operator satisfying $N_i(y_j) = \delta_{ij}(1-y_j)$, where $\delta_{ij}$ denotes the Kronecker's $\delta\textrm{-symbol}$, and we have that
\begin{equation}\label{eq:OBlogdR_connection_explicit}
	\begin{aligned}
		\partial \colon \OBlogdR(\overline{R}) &\longrightarrow \OBlogdR(\overline{R}) \otimes_R \omega^1_{R/O_F}\\
			f &\longmapsto \textstyle\sum_{i=1}^d N_i(f) \otimes \dlog x_i.
	\end{aligned}
\end{equation}
Finally, let us note that the natural homomorphism of rings $R[1/p] \rightarrow \OBlogdR(\overline{R})$ is faithfully flat (see \cite[Proposition 3.18]{andreatta-iovita-semistable}).

\begin{rem}[\textbf{Classical de Rham period rings}]
	By equipping $R$, and thus $\widehat{\overline{R}}$ and $A_{\inf}(\overline{R})$, with the trivial log-structure in the discussion above yields the usual de Rham period rings, analogous to \cite[Chapitre 5]{brinon-relatif}.
	Indeed, using the map $\theta$ with trivial log-structures, we have that
	\begin{equation*}
		\BdR^+(\overline{R}) \coloneq \lim_n (A_{\inf}(\overline{R})[1/p])/(\ker \theta)^n.
	\end{equation*}
	Note that the ring $\BdR^+(\overline{R})$ coincides with the underlying ring of $\BlogdR^+(\overline{R})$.
	
	We have that $t \coloneq \log(1+\mu)$ converges in $\BdR^+(\overline{R})$ and this ring is $t\textrm{-torsion free}$, so we set $\BdR(\overline{R}) \coloneq \BdR^+(\overline{R})[1/t]$.
	Furthermore, by using the map $\theta_R$ with trivial log-structures, we obtain big period rings $\OBdR^+(\overline{R})$ and $\OBdR(\overline{R})$.
	These rings are equipped with a $G_R\textrm{-action}$, a natural extension of the map $\theta$ (before inverting $t$) and a $G_R\textrm{-stable}$ decreasing, separated and exhaustive filtration.
	Additionally, we have identifications $\OBdR^+(\overline{R})^{G_R} = \OBdR(\overline{R})^{G_R} = R[1/p]$, and by employing an argument similar to \cite[Th\'eor\`eme 5.4.1]{brinon-relatif} it follows that the natural homomorphism of rings $R[1/p] \rightarrow \OBdR(\overline{R})$ is faithfully flat.
	Moreover, the ring $\OBdR(\overline{R})$ is further equipped with a $\BdR(\overline{R})\textrm{-linear}$ and $G_R\textrm{-equivariant}$ integrable connection $\partial$ satisfying Griffiths transversality with respect to the filtration; we equip $\OBdR^+(\overline{R})$ with the induced connection.
	
	Furthermore, similar to \cite[Proposition 5.2.2]{brinon-relatif}, we see that the natural $G_R\textrm{-equivariant}$ injective homomorphism of rings $\BdR^+(\overline{R}) \hookrightarrow \OBdR^+(\overline{R})$ extends to a natural isomorphism of $\BdR^+(\overline{R})\textrm{-algebras}$:
	\begin{equation}\label{eq:OBdR+_explicit}
		\begin{aligned}
			\BdR^+(\overline{R})\llbracket z_1, \ldots, z_d \rrbracket &\isomorphic \OBdR^+(\overline{R})\\
			z_i &\longmapsto x_i-[x_i^{\flat}].
		\end{aligned}
	\end{equation}
	Using the isomorphism in \eqref{eq:OBdR+_explicit}, we may describe the connection on $\OBdR(\overline{R})$ more explicitly.
	Let $\partial_i \colon \OBdR(\overline{R}) \rightarrow \OBdR(\overline{R})$ denote the unique $\BdR(\overline{R})\textrm{-linear}$ continuous differential operator satisfying $\partial_i(z_j) = \delta_{ij}$, where $\delta_{ij}$ denotes the Kronecker's $\delta\textrm{-symbol}$, and we have that
	\begin{equation*}
		\begin{aligned}
			\partial \colon \OBdR(\overline{R}) &\longrightarrow \OBdR(\overline{R}) \otimes_R \Omega^1_{R/O_F}\\
				f &\longmapsto \textstyle\sum_{i=1}^d \partial_i(f) \otimes dx_i.
		\end{aligned}
	\end{equation*}

	By this discussion it is clear that the de Rham period rings described above are naturally a subring of their logarithmic counterparts.
	In particular, we have the following natural $G_R\textrm{-equivariant}$ inclusion of rings:
	\begin{equation*}
		\BdR^+(\overline{R}) \longhookrightarrow \OBdR^+(\overline{R}) \longhookrightarrow \OBlogdR^+(\overline{R}).
	\end{equation*}
	compatible with the respective filtrations and connections.
\end{rem}

Note that we have natural variations of these constructions over $\calR_{\infty}$, $R_{\infty,\infty}$, $\calS_{\infty}$, $S_{\infty,\infty}$ and $\overline{S}$, where $\calR_{\infty}$ and $R_{\infty, \infty}$ are equipped with the log-structure induced from the log-structure on $\calR$, and $\calS_{\infty}$, $S_{\infty,\infty}$ and $\overline{S}$ are equipped with the trivial log-structure.
Then, the natural $(\varphi, G_S)\textrm{-equivariant}$ homomorphism $A_{\inf}(\overline{R}) \rightarrow A_{\inf}(\overline{S})$ from Section \ref{subsubsec:perfect_period_rings} is compatible with log-structures, and it induces natural $G_S\textrm{-equivariant}$ homomorphisms of rings $\BlogdR^+(\overline{R}) \rightarrow \BdR^+(\overline{S})$ and $\OBdR^+(\overline{R}) \hookrightarrow \OBlogdR^+(\overline{R}) \rightarrow \OBdR^+(\overline{S})$, compatible with the respective filtrations and connections.
\begin{lem}\label{lem:BdR_Rinftym_in_Sinftym}
	The homomorphism $\BlogdR^+(\overline{R}) \rightarrow \BdR^+(\overline{S})$ naturally restricts to $\Gamma_{S,\infty}\textrm{-equivariant}$ (resp.\ $(\GammaS \times \Delta)\textrm{-equivariant}$) injective bottom (resp.\ top) horizontal homomorphism in the following commutative diagram of rings:
	\begin{equation*}
		\begin{tikzcd}
			\BlogdR^+(\calR_{\infty}) & \BdR^+(\calS_{\infty}) \\
			\BlogdR^+(R_{\infty,\infty}) & \BdR^+(S_{\infty,\infty}).
			\arrow[hook, from=1-1, to=1-2]
			\arrow[hook, from=1-1, to=2-1]
			\arrow[hook, from=1-2, to=2-2]
			\arrow[hook, from=2-1, to=2-2]
		\end{tikzcd}
	\end{equation*}   
	Similarly, the homomorphisms $\OBdR^+(\overline{R}) \hookrightarrow \OBlogdR^+(\overline{R}) \rightarrow \OBdR^+(\overline{S})$ naturally restrict to the $\Gamma_{S,\infty}\textrm{-equivariant}$ (resp.\ $(\GammaS \times \Delta)\textrm{-equivariant}$) injective bottom (resp.\ top) horizontal homomorphisms in the following commutative diagram of rings:
	\begin{equation*}
		\begin{tikzcd}
			\OBdR^+(\calR_{\infty}) & \OBlogdR^+(\calR_{\infty}) & \OBdR^+(\calS_{\infty}) \\
			\OBdR^+(R_{\infty,\infty}) & \OBlogdR^+(R_{\infty,\infty}) & \OBdR^+(S_{\infty,\infty}).
			\arrow[hook, from=1-1, to=1-2]
			\arrow[hook, from=1-1, to=2-1]
			\arrow[hook, from=1-2, to=1-3]
			\arrow[hook, from=1-2, to=2-2]
			\arrow[hook, from=1-3, to=2-3]
			\arrow[hook, from=2-1, to=2-2]
			\arrow[hook, from=2-2, to=2-3]
		\end{tikzcd}
	\end{equation*}  
\end{lem}
\begin{proof}
	Use Lemma \ref{lem:Ainf_Rinftym_in_Sinftym}.
\end{proof}

\subsubsection{Localisation of de Rham period rings}\label{subsubsec:localisation_dR}

Next, we look at the de Rham period rings for the localisation of $\overline{R}$ at primes $\pins$ (also see \cite[Section 3.4.5]{andreatta-iovita}).
Similar to above, for each $\pins$, we set 
\begin{equation*}
	\begin{aligned}
		\BdR^+(\Cpplus) &\coloneq \lim_n (A_{\inf}(\Cpplus)[1/p])/(\ker \theta)^n\\
		(\textrm{resp. } \BdR^+(\Cplusp) &\coloneq \lim_n (A_{\inf}(\Cplusp)[1/p])/(\ker \theta)^n),
	\end{aligned}
\end{equation*}
and set $\BdR(\Cpplus) \coloneq \BdR^+(\Cpplus)[1/t]$ (resp.\ $\BdR(\Cplusp) \coloneq \BdR^+(\Cplusp)[1/t]$) equipped with a $\GRhatp\textrm{-action}$ (resp.\ $\GRp\textrm{-action}$), a natural extension of the map $\theta$ (before inverting $t$) and a $\GRhatp\textrm{-stable}$ (resp.\ $\GRp\textrm{-stable}$) decreasing, exhaustive and separated filtration given by the formula $\Fil^k \BdR(\Cpplus) \coloneq t^k \BdR(\Cpplus)$ (resp.\ $\Fil^k \BdR(\Cplusp) \coloneq t^k \BdR(\Cplusp)$), for each $k \in \mathbb{Z}$.

Moreover, for each $\pins$, note that we have the big de Rham period rings $\OBdR^+(\Cpplus)$ and $\OBdR(\Cpplus)$ equipped with an $L\textrm{-linear}$ $\GRhatp\textrm{-action}$, a natural extension of the map $\theta$ (before inverting $t$), a $\GRhatp\textrm{-stable}$ decreasing, separated and exhaustive filtration and a $\GRhatp\textrm{-equivariant}$ integrable connection satisfying Griffiths transversality with respect to the filtration (see \cite[Section 2]{brinon-imparfait}).
Analogously, following the construction of big de Rham period rings in \cite[Chapitre 5]{brinon-relatif}, for each $\pins$, we can construct period rings for the perfectoid $O_L\textrm{-algebra}$ $\Cplusp$.
In particular, we have big de Rham period rings $\OBdR^+(\Cplusp)$ and $\OBdR(\Cplusp)$ equipped with an $L\textrm{-linear}$ $\GRp\textrm{-action}$, a natural extension of the map $\theta$ (before inverting $t$ and denoted $\theta_L$), a $\GRp\textrm{-stable}$ decreasing, separated and exhaustive filtration and a $\GRp\textrm{-equivariant}$ integrable connection satisfying Griffiths transversality with respect to the filtration.
Additionally, note that we have the natural $\GRp\textrm{-equivariant}$ injective homomorphism of rings $\BdR^+(\Cplusp) \hookrightarrow \OBdR^+(\Cplusp)$.

Now, an argument similar to \cite[Lemma 2.9]{abhinandan-relative-wach-ii} shows that the natural $\GRhatp\textrm{-equivariant}$ injective homomorphism of rings $A_{\inf}(\Cplusp) \hookrightarrow A_{\inf}(\Cpplus)$ (see the discussion on localisation in Section \ref{subsubsec:perfect_period_rings}) extends to a natural $\GRhatp\textrm{-equivariant}$ injective homomorphism of rings $\BdR(\Cplusp) \hookrightarrow \BdR(\Cpplus)$ compatible with the respective filtrations.
Moreover, the preceding natural $\GRhatp\textrm{-equivariant}$ homomorphism further extends to a natural $L\textrm{-linear}$ and $\GRhatp\textrm{-equivariant}$ injective homomorphism of rings $\OBdR(\Cplusp) \hookrightarrow \OBdR(\Cpplus)$ compatible with the respective filtrations and connections.

Let us note that we have natural $\GRhatp\textrm{-equivariant}$ injective homomorphisms of rings $A_{\inf}(\Cplusp) \hookrightarrow \BdR(\Cplusp) \hookrightarrow \OBdR(\Cplusp)$, for each $\pins$.
As product is an exact functor on the category of abelian groups, therefore, the preceding natural homomorphisms, extend to natural injective homomorphisms
\begin{equation}\label{eq:gr_act_prodbdr}
	\textstyle\prod_{\pins} A_{\inf}(\Cplusp) \longhookrightarrow \textstyle\prod_{\pins} \BdR(\Cplusp) \longhookrightarrow \textstyle\prod_{\pins} \OBdR(\Cplusp).
\end{equation}
Moreover, an argument similar to \cite[Remarque 3.3.2]{brinon-relatif} shows that the products $\prod_{\pins} \BdR(\Cplusp)$ and $\prod_{\pins} \OBdR(\Cplusp)$ may, respectively, be equipped with an action of $G_R$, extending the $G_R\textrm{-action}$ on $\prod_{\pins} A_{\inf}(\Cplusp)$.
In particular, the natural injective homomorphisms in \eqref{eq:gr_act_prodbdr} are $G_R\textrm{-equivariant}$.
Using the observations above and employing an argument similar to \cite[Lemma 2.11]{abhinandan-relative-wach-ii} shows that we have a natural $R[1/p]\textrm{-linear}$ $G_R\textrm{-equivariant}$ injective homomorphisms 
\begin{equation*}
	\OBdR(\overline{R}) \longhookrightarrow \OBlogdR(\overline{R}) \longhookrightarrow \textstyle\prod_{\pins} \OBdR(\Cplusp).
\end{equation*}

Furthermore, using the description of the grading of the filtration on $\OBlogdR(\overline{R})$ from \cite[Proposition 3.15]{andreatta-iovita-semistable}, and employing arguments similar to \cite[Remarks 2.3 and 2.12]{abhinandan-relative-wach-ii} (also see \cite[Proposition 6.2.6]{brinon-relatif}), we observe that the preceding injective homomorphisms naturally extend to $L\textrm{-linear}$ and $G_R\textrm{-equivariant}$ injective homomorphisms
\begin{equation*}
	L \otimes_{R[1/p]} \OBdR(\overline{R}) \longhookrightarrow L \otimes_{R[1/p]}\OBlogdR(\overline{R}) \longhookrightarrow \textstyle\prod_{\pins} \OBdR(\Cplusp),
\end{equation*}
where $L = (R_{(p)})^{\wedge}[1/p]$ (see Section \ref{subsubsec:localisation_Rbar}).

\subsection{Log-Crystalline period rings}

In this section we recall period rings in the log-crystalline setting and study their properties.

\subsubsection{Crystalline period rings}\label{subsubsec:Bcris}

Definitions of various crystalline period rings below have been adapted to the current setting from \cite[Chapitre 6]{brinon-relatif}.

Recall that the $G_R\textrm{-equivariant}$ surjective homomorphism $\theta \colon A_{\inf}(\overline{R}) \twoheadrightarrow \widehat{\overline{R}}$ is compatible with log-structures and its kernel is principal and generated by $\xi$.
The closed immersion of log-schemes induced by the map $\theta$ is exact, and the $n\textrm{-th}$ log-PD envelope of $\theta$ (in the sense of \cite[Remark 5.5.1]{kato-fontaine-illusie-i}) coincides with the $n\textrm{-th}$ PD envelope of $\theta$.
So we set
\begin{equation*}
	\Alogcrys(\overline{R}) \coloneq A_{\inf}(\overline{R})[\{\xi^{[k]}\}_{k \in \mathbb{N}}]_{p}^{\wedge},
\end{equation*}
equipped with a log-structure induced by the natural map $A_{\inf}(\overline{R}) \rightarrow \Alogcrys(\overline{R})$.
Note that $t \coloneq \log(1+\mu)$ converges in $\Alogcrys(\overline{R})$, and this ring is $p\textrm{-torsion free}$ and $t\textrm{-torsion free}$; we set $\Blogcrys^+(\overline{R}) \coloneq \Alogcrys(\overline{R})[1/p]$ and $\Blogcrys(\overline{R}) \coloneq \Alogcrys(\overline{R})[1/t]$.
These rings are equipped with a natural action of $G_R$ and a Frobenius endomorphism $\varphi$.

Analogously, we may $R\textrm{-linearly}$ extend the $O_F\textrm{-linear}$ homomorphism $\theta \colon A_{\inf}(\overline{R}) \twoheadrightarrow \widehat{\overline{R}}$ to obtain a $G_R\textrm{-equivariant}$ surjective homomorphism $\theta_R \colon A_{\inf}(\overline{R}) \otimes_{O_F} R \twoheadrightarrow \widehat{\overline{R}}$ compatible with the induced log-structures.
We define $\pazocal{O}A_{\textup{cris},\log,n}(\overline{R})$ to be the $n\textrm{-th}$ log-PD envelope of the closed immersion of log-schemes induced by the map $\theta_R$, in the sense of \cite[Definition 5.4]{kato-fontaine-illusie-i}.
The rings $\pazocal{O}A_{\textup{cris},\log,n}(\overline{R})$ naturally form an inverse system for $n \geqslant 1$, and we set
\begin{equation*}
	\OAlogcrys(\overline{R}) \coloneq \lim_n \pazocal{O}A_{\textup{cris},\log,n}(\overline{R}).
\end{equation*}
Note that the $\Alogcrys(\overline{R})\textrm{-algebra}$ $\OAlogcrys(\overline{R})$ is $p\textrm{-torsion free}$ and $t\textrm{-torsion free}$, and we set $\OBlogcrys^+(\overline{R}) \coloneq \OAlogcrys(\overline{R})[1/p]$ and $\OBlogcrys(\overline{R}) = \OAlogcrys(\overline{R})[1/t]$.

The rings defined above are naturally equipped with a $G_R\textrm{-action}$ and a Frobenius endomorphism $\varphi$, and the map $\theta$ naturally extends to $\OAlogcrys(\overline{R})$ and $\OBlogcrys^+(\overline{R})$.
Moreover, we have $G_R\textrm{-equivariant}$ natural injective homomorphisms
\begin{equation*}
	\begin{aligned}
		\Alogcrys(\overline{R}) &\longhookrightarrow \Blogcrys^+(\overline{R}) \longhookrightarrow \Blogcrys(\overline{R}) \longhookrightarrow \BlogdR(\overline{R}),\\
		\OAlogcrys(\overline{R}) &\longhookrightarrow \OBlogcrys^+(\overline{R}) \longhookrightarrow \OBlogcrys(\overline{R}) \longhookrightarrow \OBlogdR(\overline{R}),
	\end{aligned}
\end{equation*}
and we equip all the log-crystalline period rings above with a $G_R\textrm{-stable}$ decreasing, separated and exhaustive filtration induced from the filtration on the log-de Rham period rings.
Additionally, we equip $\OBlogcrys(\overline{R})$ with a $\Blogcrys(\overline{R})\textrm{-linear}$ and $G_R\textrm{-equivariant}$ integrable logarithmic connection induced from the logarithmic connection $\partial$ on $\OBlogdR(\overline{R})$, which satisfies Griffiths transversality with respect to the filtration because the same is true for the logarithmic connection on $\OBlogdR(\overline{R})$; we also equip $\OBlogcrys^+(\overline{R})$ and $\OAlogcrys(\overline{R})$ with induced logarithmic connections.
Furthermore, we have that $\OBlogcrys^+(\overline{R})^{G_R} = \OBlogcrys(\overline{R})^{G_R} = R[1/p]$ and $\OAlogcrys(\overline{R})^{G_R} = R$, and it is easy to see that $(\Fil^0 \OBlogcrys(\overline{R}))^{\varphi=1, \partial=0} = R[1/p]$.
In addition, note that we have the following natural isomorphism of $\Alogcrys(\overline{R})\textrm{-algebras}$ (also see \cite[Section 3.4.1]{andreatta-iovita}):
\begin{equation}\label{eq:OAcryslog_explicit}
	\begin{aligned}
		\Alogcrys(\overline{R})[y_1, \ldots, y_d]_{\textup{PD}}^{\wedge} &\isomorphic \OAlogcrys(\overline{R})\\
		y_i^{[k]} &\longmapsto (1-\tfrac{[x_i^{\flat}]}{x_i})^{[k]},
	\end{aligned}
\end{equation}
where the left-hand-term denotes the $p\textrm{-adic}$ completion of the PD-polynomial algebra over $\Acrys(\overline{R})$ in variables $\{y_1, \ldots, y_d\}$.

\begin{prop}\label{prop:OBlogcrys_ff}
	The natural homomorphism of rings $R[1/p] \rightarrow \OBlogcrys(\overline{R})$ is faithfully flat.
\end{prop}
\begin{proof}
	Follows by employing an argument similar to \cite[Th\'eor\`eme 6.3.8]{brinon-relatif} in our setting.
\end{proof}

\begin{rem}[\textbf{Classical crystalline period rings}]
	We set $\Acrys(\overline{R}) \coloneq A_{\inf}(\overline{R})[\{\xi^{[k]}\}_{k \in \mathbb{N}}]_{p}^{\wedge}$ and note that it coincides with the underlying ring of $\Alogcrys(\overline{R})$.
	We set $\Bcrys^+(\overline{R}) \coloneq \Acrys(\overline{R})[1/p]$ and $\Bcrys(\overline{R}) \coloneq \Bcrys^+(\overline{R})[1/t]$.
	Furthermore, one may define big crystalline period rings over $\overline{R}$ such as $\OAcrys(\overline{R})$, $\OBcrys^+(\overline{R})$ and $\OBcrys(\overline{R})$.
	The rings defined above are naturally equipped with a continuous action of $G_R$ and a Frobenius endomorphism $\varphi$, and the map $\theta$ naturally extends to $\OAcrys(\overline{R})$ and $\OBcrys^+(\overline{R})$.
	Moreover, we have $G_R\textrm{-equivariant}$ natural injective homomorphisms
	\begin{equation*}
		\begin{aligned}
			\Acrys(\overline{R}) &\longhookrightarrow \Bcrys^+(\overline{R}) \longhookrightarrow \Bcrys(\overline{R}) \longhookrightarrow \BdR(\overline{R}),\\
			\OAcrys(\overline{R}) &\longhookrightarrow \OBcrys^+(\overline{R}) \longhookrightarrow \OBcrys(\overline{R}) \longhookrightarrow \OBdR(\overline{R}),
		\end{aligned}
	\end{equation*}
	and we equip all the crystalline period rings above with a $G_R\textrm{-stable}$ decreasing, separated and exhaustive filtration induced from the filtration on the de Rham period rings.
	Additionally, we equip $\OBcrys(\overline{R})$ with a $\Bcrys(\overline{R})\textrm{-linear}$ and $G_R\textrm{-equivariant}$ integrable connection induced from the connection $\partial$ on $\OBdR(\overline{R})$, which satisfies Griffiths transversality with respect to the filtration because the same is true for the connection on $\OBdR(\overline{R})$; we also equip $\OBcrys^+(\overline{R})$ and $\OAcrys(\overline{R})$ with induced connections.
	Furthermore, we have that $\OBcrys^+(\overline{R})^{G_R} = \OBcrys(\overline{R})^{G_R} = R[1/p]$ and $\OAcrys(\overline{R})^{G_R} = R$, and by employing an argument similar to \cite[Th\'eor\`eme 6.3.8]{brinon-relatif} it follows that the natural homomorphism of rings $R[1/p] \rightarrow \OBcrys(\overline{R})$ is faithfully flat.
	In addition, note that we have the following natural isomorphism of $\Acrys(\overline{R})\textrm{-algebras}$:
	\begin{equation}\label{eq:OAcrys_explicit}
		\begin{aligned}
			\Acrys(\overline{R})[z_1, \ldots, z_d]_{\textup{PD}}^{\wedge} &\isomorphic \OAcrys(\overline{R})\\
			z_i^{[k]} &\longmapsto (x_i-[x_i^{\flat}])^{[k]},
		\end{aligned}
	\end{equation}
	where the left-hand-term denotes the $p\textrm{-adic}$ completion of the PD-polynomial algebra over $\Acrys(\overline{R})$ in variables $\{z_1, \ldots, z_d\}$.
	Note that we have natural variations of these constructions over $R_{\infty}$ (resp.\ $R_{\infty,\infty}$) as well, and an argument similar to \cite[Corollary 4.34]{morrow-tsuji} shows that we have natural $(\varphi, \Gamma_R)\textrm{-equivariant}$ (resp.\ $(\varphi, \Gamma_{R,\infty})\textrm{-equivariant}$) isomorphism $\OAcrys(R_{\infty}) \isomorphic \OAcrys(\overline{R})^{H_R}$ (resp.\ $\OAcrys(R_{\infty,\infty}) \isomorphic \OAcrys(\overline{R})^{H_{R,\infty}}$).
		
	By this discussion it is clear that the crystalline period rings described above are naturally a subring of their logarithmic counterparts.
	In particular, we have the following natural $G_R\textrm{-equivariant}$ inclusion of rings:
	\begin{equation*}
		\Bcrys^+(\overline{R}) \longhookrightarrow \OBcrys^+(\overline{R}) \longhookrightarrow \OBlogcrys^+(\overline{R}).
	\end{equation*}
	compatible with the respective filtrations and connections.
\end{rem}

Note that we have natural variations of these constructions over $\calR_{\infty}$, $R_{\infty,\infty}$, $\calS_{\infty}$, $S_{\infty,\infty}$ and $\overline{S}$, where $\calR_{\infty}$ and $R_{\infty, \infty}$ are equipped with the log-structure induced from the log-structure on $R$, and $\calS_{\infty}$, $S_{\infty,\infty}$ and $\overline{S}$ are equipped with the trivial log-structure.
Then, an argument similar to \cite[Corollary 4.34]{morrow-tsuji} shows that we have a natural $(\varphi, \GammaR \times \Delta)\textrm{-equivariant}$ (resp.\ $(\varphi, \Gamma_{R,\infty})\textrm{-equivariant}$) isomorphism $\OAlogcrys(\calR_{\infty}) \isomorphic \OAlogcrys(\overline{R})^{H_{\calR}}$ (resp.\ $\OAlogcrys(R_{\infty,\infty}) \isomorphic \OAlogcrys(\overline{R})^{H_{R,\infty}}$).
Analogous statements hold for the corresponding crystalline period rings for $S$, i.e.\ we have a natural $(\varphi, \GammaS \times \Delta)\textrm{-equivariant}$ (resp.\ $(\varphi, \Gamma_{S,\infty})\textrm{-equivariant}$) isomorphism $\OAcrys(\calS_{\infty}) \isomorphic \OAcrys(\overline{S})^{H_{\calS}}$ (resp.\ $\OAcrys(S_{\infty,\infty}) \isomorphic \OAcrys(\overline{S})^{H_{S,\infty}}$).
Additionally, we note that the natural $(\varphi, G_S)\textrm{-equivariant}$ homomorphism $A_{\inf}(\overline{R}) \rightarrow A_{\inf}(\overline{S})$ from Section \ref{subsubsec:perfect_period_rings} is compatible with log-structures, and it induces natural $(\varphi, G_S)\textrm{-equivariant}$ homomorphisms of rings $\Blogcrys^+(\overline{R}) \rightarrow \Bcrys^+(\overline{S})$ and $\OBcrys^+(\overline{R}) \hookrightarrow \OBlogcrys^+(\overline{R}) \rightarrow \OBcrys^+(\overline{S})$, compatible with the respective filtrations and connections.

\begin{lem}\label{lem:Acrys_Rinftym_in_Sinftym}
	The homomorphism $\Acrys(\overline{R}) \rightarrow \Acrys(\overline{S})$ naturally restricts to a $\Gamma_{S,\infty}\textrm{-equivariant}$ bottom (resp.\ $(\GammaS \times \Delta)\textrm{-equivariant}$ top) injective homomorphism of rings 
	\begin{equation*}
		\begin{tikzcd}
			\Alogcrys(\calR_{\infty}) & \Acrys(\calS_{\infty}) \\
			\Alogcrys(R_{\infty,\infty}) & \Acrys(S_{\infty,\infty}).
			\arrow[hook, from=1-1, to=1-2]
			\arrow[hook, from=1-1, to=2-1]
			\arrow[hook, from=1-2, to=2-2]
			\arrow[hook, from=2-1, to=2-2]
		\end{tikzcd}
	\end{equation*}
	Similarly, the homomorphism $\OAcrys(\overline{R}) \rightarrow \OAcrys(\overline{S})$ naturally restricts to $\Gamma_{S,\infty}\textrm{-equivariant}$ bottom (resp.\ $(\GammaS \times \Delta)\textrm{-equivariant}$ top) injective homomorphisms in the following commutative diagram of rings:
	\begin{equation*}
		\begin{tikzcd}
			\OAcrys(\calR_{\infty}) & \OAlogcrys(\calR_{\infty}) & \OAcrys(\calS_{\infty}) \\
			\OAcrys(R_{\infty,\infty}) & \OAlogcrys(R_{\infty,\infty}) & \OAcrys(S_{\infty,\infty}).
			\arrow[hook, from=1-1, to=1-2]
			\arrow[hook, from=1-1, to=2-1]
			\arrow[hook, from=1-2, to=1-3]
			\arrow[hook, from=1-2, to=2-2]
			\arrow[hook, from=1-3, to=2-3]
			\arrow[hook, from=2-1, to=2-2]
			\arrow[hook, from=2-2, to=2-3]
		\end{tikzcd}
	\end{equation*} 
\end{lem}
\begin{proof}
	Use Lemma \ref{lem:BdR_Rinftym_in_Sinftym}.
\end{proof}

\subsubsection{Max period rings}

We begin by setting
\begin{equation*}
	\Alogmax(\overline{R}) \coloneq A_{\inf}(\overline{R})[\xi/p]_p^{\wedge},
\end{equation*}
equipped with a log-structure induced by the natural map $A_{\inf}(\overline{R}) \rightarrow \Alogmax(\overline{R})$.
Note that $t \coloneq \log(1+\mu)$ converges in $\Alogmax(\overline{R})$, and this ring is $p\textrm{-torsion free}$ and $t\textrm{-torsion free}$; we set $\Blogmax^+(\overline{R}) \coloneq \Alogmax(\overline{R})[1/p]$ and $\Blogmax(\overline{R}) \coloneq \Alogmax(\overline{R})[1/t]$.
These rings are equipped with a natural action of $G_R$ and a Frobenius endomorphism $\varphi$, and the map $\theta$ naturally extends to $\Alogmax(\overline{R})$ and $\Blogmax^+(\overline{R})$.
Moreover, we have $G_R\textrm{-equivariant}$ natural injective homomorphisms
\begin{equation*}
	\Alogmax(\overline{R}) \longhookrightarrow \Blogmax^+(\overline{R}) \longhookrightarrow \Blogmax(\overline{R}) \longhookrightarrow \BlogdR(\overline{R}),
\end{equation*}
and we equip the log-crystalline period rings above with a $G_R\textrm{-stable}$ decreasing, separated and exhaustive filtration induced from the filtration on the log-de Rham period ring.

Let $\Alogmax(\overline{R})[\tfrac{y_1}{p}, \ldots, \tfrac{y_d}{p}]_p^{\wedge}$ denote a subring of $\BlogdR^+(\overline{R})\llbracket y_1, \ldots, y_d\rrbracket$.
Under the isomorphism \eqref{eq:OBdRlog+_explicit}, let $\OAlogmax(\overline{R}) \subset \OBlogdR^+(\overline{R})$ denote the inverse image of $\Alogmax(\overline{R})[\tfrac{y_1}{p}, \ldots, \tfrac{y_d}{p}]_p^{\wedge}$.
In particular, we have the following natural isomorphism of $\Alogmax(\overline{R})\textrm{-algebras}$:
\begin{equation}\label{eq:OAmaxlog_explicit}
	\begin{aligned}
		\Alogmax(\overline{R})[\tfrac{y_1}{p}, \ldots, \tfrac{y_d}{p}]_p^{\wedge} &\isomorphic \OAlogmax(\overline{R})\\
		\tfrac{y_i}{p} &\longmapsto \tfrac{1}{p}(1-\tfrac{[x_i^{\flat}]}{x_i}).
	\end{aligned}
\end{equation}
Note that the $\Alogmax(\overline{R})\textrm{-algebra}$ $\OAlogmax(\overline{R})$ is $p\textrm{-torsion free}$ and $t\textrm{-torsion free}$, and we set $\OBlogmax^+(\overline{R}) \coloneq \OAlogmax(\overline{R})[1/p]$ and $\OBlogmax(\overline{R}) = \OAlogmax(\overline{R})[1/t]$.
The rings defined above may naturally be equipped with induced structures, in particular, a log-structure, a natural action of $G_R$, and a $G_R\textrm{-stable}$ decreasing, separated and exhaustive filtration.
We further equip $\OBlogmax(\overline{R})$ with a $\Blogmax(\overline{R})\textrm{-linear}$ and $G_R\textrm{-equivariant}$ integrable logarithmic connection induced from the logarithmic connection $\partial$ on $\OBlogdR(\overline{R})$, which satisfies Griffiths transversality with respect to the filtration; we also equip $\OBlogmax^+(\overline{R})$ and $\OAlogmax(\overline{R})$ with induced logarithmic connections.
In addition, we have $G_R\textrm{-equivariant}$ and filtration compatible natural injective homomorphisms of rings $\Alogcrys(\overline{R}) \hookrightarrow \Alogmax(\overline{R}) \hookrightarrow \BlogdR(\overline{R})$ and $\OAlogcrys(\overline{R}) \hookrightarrow \OAlogmax(\overline{R}) \hookrightarrow \OBlogdR(\overline{R})$.
So, we easily get that $\OBlogmax^+(\overline{R})^{G_R} = \OBlogmax(\overline{R})^{G_R} = R[1/p]$ and $\OAlogmax(\overline{R})^{G_R} = R$.

\begin{rem}[\textbf{Classical max period rings}]
	Let us set $\Amax(\overline{R}) \coloneq A_{\inf}(\overline{R})[\xi/p]_p^{\wedge}$, and we have that $t \coloneq \log(1+\mu)$ converges in $\Amax(O_{F_{\infty}})$.
	Moreover, $\Amax(\overline{R})$ is $p\textrm{-torsion free}$ and $t\textrm{-torsion free}$, and so we set $\Bmax^+(\overline{R}) \coloneq \Amax(\overline{R})[1/p]$ and $\Bmax(\overline{R}) \coloneq \Bmax^+(\overline{R})[1/t]$.
	Furthermore, one may define ``big'' period rings over $\overline{R}$ such as $\OAmax(\overline{R})$, $\OBmax^+(\overline{R})$ and $\OBmax(\overline{R})$.
	These rings are equipped with a continuous action of $G_R$ and a Frobenius endomorphism $\varphi$, and the map $\theta$ naturally extends to $\OAmax(\overline{R})$ and $\OBmax^+(\overline{R})$.
	Then, similar to above, we have $G_R\textrm{-equivariant}$ natural injective homomorphisms
	\begin{equation*}
		\begin{aligned}
			\Amax(\overline{R}) &\longhookrightarrow \Bmax^+(\overline{R}) \longhookrightarrow \Bmax(\overline{R}) \longhookrightarrow \BdR(\overline{R}),\\
			\OAmax(\overline{R}) &\longhookrightarrow \OBmax^+(\overline{R}) \longhookrightarrow \OBmax(\overline{R}) \longhookrightarrow \OBdR(\overline{R}),
		\end{aligned}
	\end{equation*}
	and we equip all the ``max'' period rings above with a $G_R\textrm{-stable}$ decreasing, separated and exhaustive filtration induced from the filtration on the de Rham period rings.
	Additionally, we equip $\OBmax(\overline{R})$ with a $\Bmax(\overline{R})\textrm{-linear}$ and $G_R\textrm{-equivariant}$ integrable connection induced from the connection $\partial$ on $\OBdR(\overline{R})$, which satisfies Griffiths transversality with respect to the filtration because the same is true for the connection on $\OBdR(\overline{R})$; we also equip $\OBmax^+(\overline{R})$ and $\OAmax(\overline{R})$ with induced connections.
	Furthermore, we have that $\OBmax^+(\overline{R})^{G_R} = \OBmax(\overline{R})^{G_R} = R[1/p]$ and $\OAmax(\overline{R})^{G_R} = R$.
	Note that we have natural variations of these constructions over $R_{\infty}$ (resp.\ $R_{\infty,\infty}$) as well, and an argument similar to \cite[Corollary 4.34]{morrow-tsuji} shows that we have a natural $(\varphi, \Gamma_R)\textrm{-equivariant}$ (resp.\ $(\varphi, \Gamma_{R,\infty})\textrm{-equivariant}$) isomorphism $\OAmax(R_{\infty}) \isomorphic \OAmax(\overline{R})^{H_R}$ (resp.\ $\OAmax(R_{\infty,\infty}) \isomorphic \OAmax(\overline{R})^{H_{R,\infty}}$).
	Recall that we have the following natural isomorphism of $\Amax(\overline{R})\textrm{-algebras}$:
	\begin{equation}\label{eq:OAmax_explicit}
		\begin{aligned}
			\Amax(\overline{R})\big[\tfrac{z_1}{p}, \ldots \tfrac{z_d}{p}\big]_p^{\wedge} &\isomorphic \OAmax(\overline{R})\\
			\tfrac{z_i}{p} &\longmapsto \tfrac{x_i-[x_i^{\flat}]}{p},
		\end{aligned}
	\end{equation}
	where the left-hand-term denotes the $p\textrm{-adic}$ completion of the polynomial algebra over $\Amax(\overline{R})$ in variables $\{\tfrac{z_1}{p}, \ldots, \tfrac{z_d}{p}\}$.
	In addition, note that we have $G_R\textrm{-equivariant}$ and filtration compatible natural injective homomorphisms of rings $\Acrys(\overline{R}) \hookrightarrow \Amax(\overline{R}) \hookrightarrow \BdR(\overline{R})$ and $\OAcrys(\overline{R}) \hookrightarrow \OAmax(\overline{R}) \hookrightarrow \OBdR(\overline{R})$.
	So, we see that we have $\OBmax^+(\overline{R})^{G_R} = \OBmax(\overline{R})^{G_R} = R[1/p]$ and $\OAmax(\overline{R})^{G_R} = R$.
\end{rem}

Note that we have natural variations of these constructions over $\calR_{\infty}$, $R_{\infty,\infty}$, $\calS_{\infty}$, $S_{\infty,\infty}$ and $\overline{S}$, where $\calR_{\infty}$ and $R_{\infty, \infty}$ are equipped with the log-structure induced from the log-structure on $R$, and $\calS_{\infty}$, $S_{\infty,\infty}$ and $\overline{S}$ are equipped with the trivial log-structure.
Then, an argument similar to \cite[Corollary 4.34]{morrow-tsuji} shows that we have a natural $(\varphi, \GammaR \times \Delta)\textrm{-equivariant}$ (resp.\ $(\varphi, \Gamma_{R,\infty})\textrm{-equivariant}$) isomorphism $\OAlogmax(R_{\infty}) \isomorphic \OAlogmax(\overline{R})^{H_{\calR}}$ (resp.\ $\OAlogmax(R_{\infty,\infty}) \isomorphic \OAlogmax(\overline{R})^{H_{R,\infty}}$).
Additionally, note that the natural $(\varphi, G_S)\textrm{-equivariant}$ homomorphism $A_{\inf}(\overline{R}) \rightarrow A_{\inf}(\overline{S})$ from Section \ref{subsubsec:perfect_period_rings} is compatible with log-structures, and it induces natural $(\varphi, G_S)\textrm{-equivariant}$ homomorphisms of rings $\Blogmax^+(\overline{R}) \rightarrow \Bmax^+(\overline{S})$ and $\OBmax^+(\overline{R}) \hookrightarrow \OBlogmax^+(\overline{R}) \rightarrow \OBmax^+(\overline{S})$, compatible with the respective filtrations and connections.

\begin{lem}\label{lem:Amax_Rinftym_in_Sinftym}
	Statements analogous to Lemma \ref{lem:Acrys_Rinftym_in_Sinftym} also hold after replacing $\Acrys$ and $\OAcrys$ with $\Amax$ and $\OAmax$, respectively.
\end{lem}
\begin{proof}
	Use Lemma \ref{lem:BdR_Rinftym_in_Sinftym}.
\end{proof}

\subsubsection{Localisation for crystalline period rings}\label{subsubsec:localisation_crys}

From \cite[Sections 2.3 \& 2.4]{brinon-imparfait}, for each $\pins$ we have rings $\Acrys(\Cpplus)$, $\Bcrys(\Cpplus)$, $\OAcrys(\Cpplus)$, $\OBcrys(\Cpplus)$, and $\Amax(\Cpplus)$, $\Bmax(\Cpplus)$, $\OAmax(\Cpplus)$, $\OBmax(\Cpplus)$ equipped with a $\GRhatp\textrm{-action}$, a natural extension of the map $\theta$ (before inverting $t$) and a Frobenius endomorphism $\varphi$.
Moreover, we have a natural $L\textrm{-linear}$ and $\GRhatp\textrm{-equivariant}$ injective homomorphism of rings $\OAcrys(\Cpplus) \hookrightarrow \OAmax(\Cpplus) \hookrightarrow \OBdR^+(\Cpplus)$.
Using the preceding inclusion, we equip all crystalline period rings for $\Cpplus$ described above with an induced $\GRhatp\textrm{-stable}$ decreasing, separated and exhaustive filtration and an induced $\GRhatp\textrm{-equivariant}$ integrable connection (on rings with a prefix ``$\pazocal{O}$'') satisfying Griffiths transversality with respect to the filtration.

Analogously, following the construction of crystalline period rings in \cite[Chapitre 6]{brinon-relatif} for each $\pins$, we can construct period rings for the perfectoid $O_L\textrm{-algebra}$ $\Cplusp$ (similar to the de Rham period rings for $\Cplusp$ described in Section \ref{subsubsec:BdR} using \cite[Chapitre 5]{brinon-relatif}).
Consequently, we have period rings $\Acrys(\Cplusp)$, $\Bcrys(\Cplusp)$, $\OAcrys(\Cplusp)$, $\OBcrys(\Cplusp)$ and $\Amax(\Cplusp)$, $\Bmax(\Cplusp)$, $\OAmax(\Cplusp)$ and $\OBmax(\Cplusp)$ equipped with a $\GRp\textrm{-action}$, a natural extension of the map $\theta$ (before inverting $t$) and a Frobenius endomorphism $\varphi$.
Then, one can adapt the proof of \cite[Proposition 6.2.1]{brinon-relatif} to obtain a natural $O_L\textrm{-linear}$ and $\GRp\textrm{-equivariant}$ injective homomorphism of rings $\OAcrys(\Cplusp) \hookrightarrow \OAmax(\Cplusp) \hookrightarrow \OBdR^+(\Cplusp)$.
Using the preceding inclusion, we equip all the crystalline period rings for $\Cplusp$ described above with an induced $\GRp\textrm{-stable}$ decreasing, separated and exhaustive filtration and an induced $\GRp\textrm{-equivariant}$ integrable connection (on rings with a prefix ``$\pazocal{O}$'') satisfying Griffiths transversality with respect to the filtration.

Now, by employing an argument similar to \cite[Lemma 2.13]{abhinandan-relative-wach-ii} we see that for each $\pins$, the natural $(\varphi, \GRhatp)\textrm{-equivariant}$ injective homomorphism $A_{\inf}(\Cplusp) \hookrightarrow A_{\inf}(\Cpplus)$ (see the discussion on localisation in Section \ref{subsubsec:perfect_period_rings}) extends to natural $(\varphi, \GRhatp)\textrm{-equivariant}$ and filtration compatible injective homomorphisms of rings $\Acrys(\Cplusp) \hookrightarrow \Acrys(\Cpplus)$, $\Amax(\Cplusp) \hookrightarrow \Amax(\Cpplus)$, $\OAcrys(\Cplusp) \hookrightarrow \OAcrys(\Cpplus)$ and $\OAmax(\Cplusp) \hookrightarrow \OAmax(\Cpplus)$, where the latter two are $O_L\textrm{-linear}$ and also compatible with the respective connections.

From the discussion above, it is easy to see that we have the following diagram of injective homomorphisms:
\begin{equation}\label{eq:gr_act_prodbcrys}
	\begin{tikzcd}
		\textstyle\prod_{\pins} A_{\inf}(\Cplusp) \arrow[r, hookrightarrow] & \textstyle\prod_{\pins} \Acrys(\Cplusp) \arrow[r, hookrightarrow] \arrow[d, hookrightarrow] & \textstyle\prod_{\pins} \OAcrys(\Cplusp) \arrow[d, hookrightarrow]\\
		& \textstyle\prod_{\pins} \Amax(\Cplusp) \arrow[r, hookrightarrow] & \textstyle\prod_{\pins} \OAmax(\Cplusp) \arrow[d, hookrightarrow]\\
		& & \textstyle\prod_{\pins} \OBdR(\Cplusp),
	\end{tikzcd}
\end{equation}
where the all homomorphisms (except the vertical arrow from the second row to the third row) are compatible with the respective Frobenii.
Moreover, by employing an argument similar to \cite[Remarque 3.3.2]{brinon-relatif}, we see that the products $\prod_{\pins} \Acrys(\Cplusp)$, $\prod_{\pins} \Amax(\Cplusp)$, $\prod_{\pins} \OAcrys(\Cplusp)$ and $\prod_{\pins} \OAmax(\Cplusp)$ are stable under the natural $G_R\textrm{-action}$ on $\prod_{\pins} \OBdR(\Cplusp)$ (see Section \ref{subsubsec:BdR}), and we equip them with the induced action.
Then, it follows that the injective homomorphisms in \eqref{eq:gr_act_prodbcrys} are $G_R\textrm{-equivariant}$ as well.
Using the preceding discussion, we get a natural $R\textrm{-linear}$ and $(\varphi, G_R)\textrm{-equivariant}$ injective homomorphisms 
\begin{equation*}
	\begin{aligned}
		\OAlogcrys(\overline{R}) &\longhookrightarrow \textstyle\prod_{\pins} \OAcrys(\Cplusp)\\
		\OAlogmax(\overline{R}) &\longhookrightarrow \textstyle\prod_{\pins} \OAmax(\Cplusp).
	\end{aligned}
\end{equation*}

\begin{lem}\label{lem:OAcrys_modp_inject}
	For each $n \geqslant 1$, the following natural homomorphisms are injective:
	\begin{equation}\label{eq:OAcrys_modp_inject}
		\begin{aligned}
			\OAlogcrys(\overline{R})/p^n\OAlogcrys(\overline{R}) &\longrightarrow \textstyle\prod_{\pins} \OAcrys(\Cplusp)/p^n\OAcrys(\Cplusp)\\
			\OAlogmax(\overline{R})/p^n\OAlogmax(\overline{R}) &\longrightarrow \textstyle\prod_{\pins} \OAmax(\Cplusp)/p^n\OAmax(\Cplusp).
		\end{aligned}
	\end{equation}
\end{lem}
\begin{proof}
	As $\OAlogcrys(\overline{R})$ and $\OAlogmax(\overline{R})$ are $p\textrm{-adically}$ complete and $p\textrm{-torsion}$ free, it is enough to show the injectivity of the map in \eqref{eq:OAcrys_modp_inject} for $n=1$, and the injectivity for $n \geqslant 2$ would follow from an easy induction.
	In the following, we only present an argument for $\OAlogmax(\overline{R})$, the case of $\OAlogmax(\overline{R})$ may be argued similarly using \eqref{eq:OAcrys_explicit}.

	Using the description of $\OAmax(\overline{R})$ from \eqref{eq:OAmax_explicit}, we may further reduce ourselves to showing that the following $(\varphi, G_R)\textrm{-equivariant}$ homomorphism is injective:
	\begin{equation*}
		\Alogmax(\overline{R})/p \longrightarrow \textstyle\prod_{\pins} \Amax(\Cplusp)/p.
	\end{equation*}
	Note that we have 
	\begin{equation*}
		\Alogmax(\overline{R})/p = (A_{\inf}(\overline{R})[Y]/(pY-\xi))/p = (A_{\inf}(\overline{R})[Y]/(p, \xi) = (\overline{R}/p)[Y],
	\end{equation*}
	where $Y$ is an indeterminate.
	Similarly, we have that $\Alogmax(\Cplusp)/p = ((\overline{R})_{\mathfrak{p}}/p)[Y]$.
	So, we are reduced to showing that the following homomorphism is injective:
	\begin{equation*}
		(\overline{R}/p)[Y] \longrightarrow \textstyle\prod_{\pins} ((\overline{R})_{\mathfrak{p}}/p)[Y].
	\end{equation*}
	But the injectivity of the preceding map follows from \eqref{eqn: injectivity of localization R/pR}.
	This allows us to conclude.
\end{proof}

\subsubsection{Some explicit periods}\label{subsubsec:explicit_periods}

In this section, we will record some important technical observations for some special $p\textrm{-adic}$ periods.
We will work under the assumption that $d > a$, which implies that the elements $x_{a+1}, \ldots, x_d$ are not invertible in $R$.
For each $1 \leqslant i \leqslant d$, let us set 
\begin{equation*}
	\beta_i \coloneq \log\big(\tfrac{x_i}{[x_i^{\flat}]}\big) = \textstyle\sum_{k=0}^{\infty} \tfrac{(-1)^k}{k+1} \big(\tfrac{x_i-[x_i^{\flat}]}{[x_i^{\flat}]}\big)^{k+1} \in \OBlogdR(\overline{R}).
\end{equation*}
For $1 \leqslant i \leqslant a$, from its definition it is clear that $\beta_i$ is an element of $\OAcrys(\overline{R})$, and for $a+1 \leqslant i \leqslant d$, we have that $\beta_i$ is an element of $\OAlogcrys(\overline{R})$.
Our first goal in this section is to show the following claim:
\begin{lem}\label{lem:betanotrational}
	For any non-zero element $f$ of $\OAmax(\overline{R})$ and $a+1 \leqslant i \leqslant d$, the product $\beta_if$ does not belong to $\OAmax(\overline{R})$.
	Analogous claim also holds for $\OAcrys(\overline{R})$.
\end{lem}

To prove Lemma \ref{lem:betanotrational}, we need some preparations.

\begin{lem}\label{lem: X-divisible for Amax}
	For each $a+1 \leqslant i \leqslant d$, we have that 
	\begin{equation*}
		\underset{k=1}{\overset{\infty}\cap} [x_i^{\flat}]^k \OAmax(\overline{R}) = 0.
	\end{equation*}
\end{lem}
\begin{proof}
	From \eqref{eq:OAmax_explicit}, recall that each element of $\OAmax(\overline{R})$ may be written uniquely as a power series in variables $\{z_1/p, \ldots, z_d/p\}$ with coefficients in $\Amax(\overline{R})$.
	Therefore, to get the claim it suffices to show that
	\begin{equation*}
		\underset{k=1}{\overset{\infty}\cap} [x_i^{\flat}]^k \Amax(\overline{R}) = 0.
	\end{equation*}
	
	Let $a$ be a non-zero element in the left hand side of the equation above.
	Dividing $a$ by a sufficiently large power of $\xi$ (if necessary), we may further assume that $a$ is not in $\xi \Amax(\overline{R})$, i.e.\ $a$ is not in the kernel of the map $\theta$.
	But, then we observe that
	\begin{equation*}
		\theta(a) \in \underset{k=1}{\overset{\infty}\cap} x_i^k \CRp = 0,
	\end{equation*}
	where the equality follows from Proposition~\ref{X-divisible}, contradicting the assumption above.
	Hence, it follows that $a$ must be zero, and the lemma is proved.
\end{proof}

\begin{lem}\label{lem: sublemma betanotrational}
	For each $1 \leqslant i \leqslant d$ and $n \geqslant 1$, the following natural $\textrm{multiplication-by-}[x_i^{\flat}]$ map is injective:
	\begin{equation*}
		\OAmax(\overline{R})/p^n\OAmax(\overline{R}) \xrightarrow{[x_i^{\flat}]} \OAmax(\overline{R})/p^n\OAmax(\overline{R}).
	\end{equation*}
\end{lem}
\begin{proof}
	Consider the following commutative diagram:
	\begin{equation*}
		\begin{tikzcd}
			\OAmax(\overline{R})/p^n\OAmax(\overline{R}) \arrow[r] \arrow[d, "{[x_i^{\flat}]}"] & \textstyle\prod_{\pins} \OAmax(\Cplusp)/p^n \OAmax(\Cplusp) \arrow[d, "{[x_i^{\flat}]}"]\\
			\OAmax(\overline{R})/p^n\OAmax(\overline{R}) \arrow[r] & \textstyle\prod_{\pins} \OAmax(\Cplusp)/p^n \OAmax(\Cplusp),
		\end{tikzcd}
	\end{equation*}
	where the horizontal arrows are injective by Lemma \ref{lem:OAcrys_modp_inject}.
	The right vertical arrow is clearly injective, therefore, it follows that the left vertical arrow is also injective.
\end{proof}

Now we are ready to prove Lemma \ref{lem:betanotrational}.

\begin{proof}[Proof of Lemma \ref{lem:betanotrational}]
	It is enough to show the claim for $\OAmax(\overline{R})$.
	So, let $f$ be any element of $\OAmax(\overline{R})$, fix some $a+1 \leqslant i \leqslant d$, and set $g \coloneq \beta_if$ in $\OBlogdR(\overline{R})$.
	We shall show that if $g$ belongs $\OAmax(\overline{R})$, then $f$ must be zero.

	Let $\partial_i$ denote the differential operator on $\OAmax(\overline{R})$ induced from $\OBdR(\overline{R})$ (see Sections \ref{subsubsec:BdR} and \ref{subsubsec:Bcris}).
	Let $\OAmax^{(i)} (\overline{R})$ denote the kernel of the operator $\partial_i$ on $\OAmax(\overline{R})$.
	Then, we may write $f$ and $g$ uniquely as
	\begin{equation*}
		f = \textstyle\sum_{n=0}^{\infty} F_n \tfrac{z_i^n}{p^n}, \qquad g = \textstyle\sum_{n=0}^{\infty} G_n \tfrac{z_i^n}{p^n},
	\end{equation*}
	where $F_n$ and $G_n$ belong to $\OAmax^{(i)}(\overline{R})$, for each $n \geqslant 0$.
	Additionally, we may write $\beta_i = \sum_{m=1}^{\infty} (-1)^{m-1}\tfrac{z_i^m}{m [x_i^{\flat}]^m}$ in $\OBlogdR(\overline{R})$.
	Then, an easy computation inside $\OBlogdR(\overline{R})$ shows that
	\begin{equation*}
		G_n = \textstyle\sum_{k+m=n} (-1)^{m-1}\tfrac{p^{m}}{m [x_i^{\flat}]^m} \cdot F_k.
	\end{equation*}
	Consequently, we obtain an infinite system of linear equations in $F_k$, for $n\geqslant 1$:
	\begin{equation}\label{eqn: main relation in lem:betanotrational}
	\textstyle\sum_{k=0}^{n-1}(-1)^{n-k+1}\frac{p^{n-k}[x_i^{\flat}]^k }{n-k} \cdot F_k = G_n [x_i^{\flat}]^n.
	\end{equation}
	By induction on $n$, we shall show that for each $0 \leqslant k \leqslant n-1$, we have
	\begin{equation}\label{eqn: induction in betanotrational}
		F_k \in [x_i^{\flat}]^{n-k} \OAmax^{(i)}(\overline{R}).
	\end{equation}
	This will imply that $F_k$ belongs to $\cap_{n = k+1}^{\infty} [x_i^{\flat}]^{n-k} \OAmax^{(i)} (\overline{R})$, and therefore by Lemma \ref{lem: X-divisible for Amax} it will follow that $F_k = 0$, for all $k \geqslant 0$.
	
	To prove the induction claim, let $n=1$.
	Then equation \eqref{eqn: main relation in lem:betanotrational} for $n=1$ gives us that $pF_0 = [x_i^{\flat}]G_1 $ belongs to $[x_i^{\flat}] \OAmax^{(i)}(\overline{R})$.
	An application of Lemma \ref{lem: sublemma betanotrational} then shows that $F_0$ is in $[x_i^{\flat}]\OAmax^{(i)}(\overline{R})$.
	
	Next, assume that \eqref{eqn: induction in betanotrational} holds for some $n = N \geqslant 1$.
	Then, for each $0 \leqslant k \leqslant N$, we may write $F_k = [x_i^{\flat}]^{N-k} H_k$, for some $H_k$ in $\OAmax^{(i)}(\overline{R})$.
	We claim that $H_k$ belongs to $[x_i^{\flat}] \OAmax^{(i)}(\overline{R})$.
	Indeed, fix a prime number $q \geqslant N+1$ such that $q \neq p$, and consider the equations \eqref{eqn: main relation in lem:betanotrational} for $q \leqslant n \leqslant q+N$.
	Since multiplication by $[x_i^{\flat}]$ is injective on $\OAmax^{(i)}(\overline{R})$ (see Lemma \ref{lem: sublemma betanotrational}), therefore, we obtain a system of linear equations in $\{H_0, \ldots, H_N\}$, for $q \leqslant n \leqslant q+N$:
	\begin{equation*}
		\textstyle\sum_{k=0}^N (-1)^{n-k+1}\frac{p^{n-k}}{n-k} \cdot H_k = L_n [x_i^{\flat}],
	\end{equation*}
	where $L_n$ belongs to $\OAmax^{(i)}(\overline{R})$.
	It is easy to see that the determinant of the preceding system of linear equations is equal to
	\begin{equation*}
		\pm p^{q(N+1)} \left \vert 
		\begin{matrix} 
			\frac{p^q}{q} & -\frac{p^{q-1}}{q-1} & \ldots & (-1)^N \frac{p^{q-N}}{q-N}\\
			\frac{p^{q+1}}{q+1} & -\frac{p^{q}}{q} & \ldots & (-1)^N \frac{p^{q-N+1}}{q-N+1}\\
			\vdots & \vdots & \ddots & \vdots\\
			\frac{p^{q+N}}{q+N} & -\frac{p^{q+N-1}}{q+N-1} & \ldots & (-1)^N \frac{p^{q}}{q}
		\end{matrix}
		\right \vert.
	\end{equation*}
	An elementary computation shows that the product of diagonal elements gives the term of the lowest $q\textrm{-adic}$ valuation in the standard expression of this determinant as a sum of signed products.
	Consequently, this determinant is a non-zero rational number, and from Cramer's formulas we get that $p^{s_k}H_k$ belongs to $[x_i^{\flat}] \OAmax^{(i)}(\overline{R})$, for some $s_k \geqslant 0$ and each $0 \leqslant k \leqslant N$. 
	Then, by applying Lemma \ref{lem: sublemma betanotrational} repeatedly, we conclude that $H_k$ belongs to $[x_i^{\flat}] \OAmax^{(i)}(\overline{R})$, for each $0 \leqslant k \leqslant N$.
	In particular, it follows that $F_k$ belongs to $[x_i^{\flat}]^{N+1-k} \OAmax^{(i)}(\overline{R})$, thus proving the induction claim.
	Hence, it follows that $F_k = 0$, for all $k \geqslant 0$.
\end{proof}

\begin{prop}\label{prop: beta_transcendental}
	For $a+1 \leqslant i \leqslant d$, the element $\beta_i$ is not algebraic over $\OAcrys(\overline{R})$.
\end{prop}
\begin{proof}
	We shall prove the claim by contradiction, i.e.\ for a fixed $a+1 \leqslant i \leqslant d$ assume that $\beta_i$ is algebraic over $\OAcrys(\overline{R})$, and under this assumption we shall produce a contradiction.
	So, let the following be an algebraic relation of minimal degree over $\OAcrys(\overline{R})$:
	\begin{equation*}
		a_n\beta_i^{n}+a_{n-1}\beta_i^{n-1}+\cdots +a_0=0, \qquad a_n\neq 0.
	\end{equation*}
	Let $N_i \colon \OAlogcrys(\overline{R}) \rightarrow \OAlogcrys(\overline{R})$ denote the differential operator induced from \eqref{eq:OBlogdR_connection_explicit}.
	Then, we see that $N_i(\beta_i^k) = k\beta_i^{k-1}$, and applying $N_i$ to the algebraic relation mentioned above, we see that
	\begin{equation*}
		N_i(a_n) \beta_i^{n} + \big(N_i(a_{n-1}) + na_n\big) \beta_i^{n-1} + \big(N_i(a_{n-2}) + (n-1)a_{n-1}\big) \beta_i^{n-2}+ \cdots + \big(N_i(a_0) + a_1\big) = 0.
	\end{equation*}
	From the minimality of the chosen relation we obtain that
	\begin{equation*}
		N_i(a_n)a_k = a_n \big(N_i(a_k)+(k+1)a_{k+1}\big), \qquad 0 \leqslant k \leqslant n-1.
	\end{equation*}
	In particular, 
	\begin{equation*}
		a_n N_i(a_{n-1}) - N_i(a_n)a_{n-1}  = -n a_n^2. 
	\end{equation*}
	Let $\mathfrak{p}$ be a fixed minimal prime ideal of $\overline{R}$ containing $p$, and such that the image of $a_n$ under the natural map $r_{\mathfrak{p}} \colon \OAlogcrys(\overline{R}) \rightarrow \OAcrys(\Cplusp)$ is non-zero (see the equation before Lemma \ref{lem:OAcrys_modp_inject}).
	From \cite[Proposition 2.9]{brinon-imparfait} we know that the ring $\OAcrys(\Cplusp)$ is integral, and in its field of fractions we have that
	\begin{equation*}
		N_i\left(\tfrac{r_{\mathfrak{p}}(a_{n-1})}{r_{\mathfrak{p}}(a_n)}\right ) = -n,
	\end{equation*}
	where we identify $x_i$ with $r_{\mathfrak{p}}(x_i)$ and, by abusing notations, write $N_i$ for the analogous derivation over $\OAcrys(\Cplusp)$.
	Since we have that $N_i(r_{\mathfrak{p}}(\beta_i)) = 1$, therefore, we obtain that 
	\begin{equation*}
		nr_{\mathfrak{p}}(\beta_i) = -\tfrac{r_{\mathfrak{p}}(a_{n-1})}{r_{\mathfrak{p}}(a_n)}+c_{\mathfrak{p}}, 
	\end{equation*}
	where the element $c_\mathfrak{p}$ belongs to the field of fractions of the integral domain 
	\begin{equation*}
		\OAcrys^{(i)}(\Cplusp)) = \ker(\partial_i \colon \OAcrys(\Cplusp) \longrightarrow \OAcrys(\Cplusp)).
	\end{equation*}
	Hence, we get that
	\begin{equation}\label{eq:an_betai_relation}
		r_{\mathfrak{p}}(a_{n-1} + na_n\beta_i)=c_{\mathfrak{p}} r_{\mathfrak{p}}(a_n).
	\end{equation}
	Using the explicit description of $\OAcrys(\overline{R})$ (see \eqref{eq:OAcrys_explicit} and \cite[Section 2.4]{abhinandan-relative-wach-ii}), let us write $a_n$ and $a_{n-1}$ in the following form:
	\begin{equation*}
		a_n = \textstyle\sum_{m=0}^{\infty} f_m \tfrac{z_i^m}{m!}, \qquad a_{n-1} = \sum_{m=0}^{\infty} g_m \tfrac{z_i^m}{m!},
	\end{equation*}
	where $f_m$ and $g_m$ belong to $\OAcrys^{(i)}(\overline{R})$.
	Let $m_0$ be the smallest integer such that $f_{m_0} \neq 0$.
	Comparing the terms of degree $m_0$ with respect to $z_i$ in equation \eqref{eq:an_betai_relation}, we find that $r_{\mathfrak{p}}(g_{m_0}) = c_{\mathfrak{p}} r_{\mathfrak{p}}(f_{m_0})$.
	Therefore, we get that
	\begin{equation}\label{eq:fm0_an_betai_relation}
		r_{\mathfrak{p}}(f_{m_0} a_{n-1} + nf_{m_0}a_n\beta_i) = r_{\mathfrak{p}}(g_{m_0}a_n),
	\end{equation}
	for any $\mathfrak{p}$ such that $r_{\mathfrak{p}}(a_n) \neq 0$.
	But, if $r_{\mathfrak{p}}(a_n) = 0$, then we see that both sides of the equation \eqref{eq:fm0_an_betai_relation} is zero, and therefore, we conclude that the equality in \eqref{eq:fm0_an_betai_relation} holds for all minimal primes $\mathfrak{p}$ of $\overline{R}$ lying above $pR$.
	Now, from the injectivity of the map $\underset{\mathfrak{p}}\prod r_{\mathfrak{p}}$ (see Section \ref{subsubsec:localisation_crys}) we get that 
	\begin{equation*}
		f_{m_0}a_{n-1} + nf_{m_0}a_n\beta_i = g_{m_0}a_n,
	\end{equation*}
	in $\OAlogcrys(\overline{R})$.
	But then it is clear that $nf_{m_0}a_n\beta_i$, in fact, belongs to $\OAcrys(\overline{R})$, which contradicts Lemma \ref{lem:betanotrational} above.
	Hence, $\beta_i$ is not algebraic over $\OAcrys(\overline{R})$, thus concluding our proof.
\end{proof}

\subsubsection{A logarithmic period ring}\label{subsubsec:Blog}

In the previous sections we looked at several log-crystalline period rings.
In particular, we defined the period ring $\OBlogcrys(\overline{R})$ and studied some of its properties.
In this section, we will define yet another logarithmic period ring using the explicit periods studied in Section \ref{subsubsec:explicit_periods}.

\begin{defi}\label{defi:ring_OBst}
	Define the logarithmic period ring
	\begin{equation*}
		\OAlog(\overline{R}) \coloneq \OAcrys(\overline{R})[\beta_{a+1}, \ldots, \beta_d],
	\end{equation*} 
	to be the $\OAcrys(\overline{R})\textrm{-subalgebra}$ of $\OAlogcrys(\overline{R})$ generated by $\beta_{a+1}, \ldots, \beta_d$.
	We also define $\OBlog^+(\overline{R}) \coloneq \OAlog(\overline{R})[1/p]$ and $\OBlog(\overline{R}) \defeq \OBlog^+(\overline{R})[1/t]$.
\end{defi}

We summarise the basic properties of the ring $\OBlog(\overline{R})$ in the following, which follow from easy computations and the corresponding property for $\OBcrys(\overline{R})$ and $\OBlogcrys(\overline{R})$.

\begin{prop}\label{prop:OBlog_properties}
	The following statements hold true:
	\begin{enumerate}
		\item[\textup{(1)}] The natural homomorphism of $\OBcrys(\overline{R})\textrm{-algebras}$
			\begin{equation*}
				\OBcrys(\overline{R})[T_{a+1}, \ldots T_d] \longrightarrow \OBlog(\overline{R}),
			\end{equation*}
			sending $T_i$ to $\beta_i$, is an isomorphism.
		
		\item[\textup{(2)}] The subring $\OBlog(\overline{R}) \subset \OBlogcrys(\overline{R})$ is stable under the action of the Galois group $G_R$ and the Frobenius $\varphi$ on the latter.
			Additionally, we have that $\varphi(\beta_i) = p\beta_i$, the $G_R\textrm{-invariants}$ are computed as
			\begin{equation*}
				\OBlog(\overline{R})^{G_R} = R[1/p],
			\end{equation*}
			and the homomorphism of rings $R[1/p] \rightarrow \OBlog(\overline{R})$ is faithfully flat.
		
		\item[\textup{(3)}] Equip $\OBlog(\overline{R}) \subset \OBlogcrys(\overline{R})$ with the induced filtration.
			Additionally, the logarithmic connection $\partial$ on $\OBlogcrys(\overline{R})$ induces a logarithmic connection
			\begin{equation*}
				\OBlog(\overline{R}) \longrightarrow \OBlog(\overline{R}) \otimes_{R} \omega^1_{R/O_F},
			\end{equation*}
			satisfying Griffiths transversality with respect to the filtration.
			Moreover, we have that 
			\begin{equation*}
				(\Fil^0\OBlog(\overline{R}))^{\varphi=1,\partial=0} = \mathbb{Q}_p.
			\end{equation*}
	\end{enumerate}
\end{prop}

\subsection{Log-crystallline Galois representations}\label{subsec:logcrys_reps}

In this section we shall define and study properties of the main objects of this paper, namely, log-crystalline representations of $G_R$.

\subsubsection{\texorpdfstring{$G_R\textrm{-regularity}$}{-}}

Let $G$ be a topological group and let $\textrm{Rep}_{\mathbb{Q}_p}(G)$ denote the category of $\mathbb{Q}_p\textrm{-vector}$ spaces admitting a linear and continuous action of $G$, and morphisms being $\mathbb{Q}_p\textrm{-linear}$ and $G\textrm{-equivariant}$ homomorphisms of vector spaces.

\begin{defi}
	A \textit{Tannakian subcategory} of $\textrm{Rep}_{\mathbb{Q}_p}(G)$ is a full subcategory stable under sub-objects, quotients, direct sums, tensor products, dual and contains the unit object, i.e.\ $\mathbb{Q}_p$ equipped with the trivial action of $G$.
	If $\mathcal{C}$ is such a category and $\Lambda$ a commutative ring, then any $\otimes\textrm{-functor}$ $F$ from $\mathcal{C}$ to the category of $\Lambda\textrm{-modules}$ is said to be a $\Lambda\textit{-fibre}$ functor if it is faithful, exact and takes values in finite projective $\Lambda\textrm{-modules}$.
\end{defi}

Let $B$ be a reduced $\mathbb{Q}_p\textrm{-algebra}$ equipped with a continuous and $\mathbb{Q}_p\textrm{-linear}$ action of $G$.
For a $\mathbb{Q}_p\textrm{-representation}$ $V$ of $G$, set
\begin{equation*}
	D_B(V) \coloneq (B \otimes_{\mathbb{Q}_p} V)^G,
\end{equation*}
as a module over $B^G$.
Then, we have a natural homomorphism of $B\textrm{-modules}$:
\begin{equation*}
	\begin{aligned}
		\alpha_B(V) \colon B \otimes_{B^G} D_B(V) &\longrightarrow B \otimes_{\mathbb{Q}_p} V\\
					b \otimes d &\longmapsto bd.
	\end{aligned}
\end{equation*}

\begin{defi}
	A $\mathbb{Q}_p\textrm{-representation}$ $V$ of $G$ is said to be $B\textit{-admissible}$ if $\alpha_B(V)$ is an isomorphism.
\end{defi}

\begin{defi}[{\cite[Définition 8.0.1]{brinon-relatif}}]\label{defi:GR_regular}
	The ring $B$ is said to be \textit{$G\textrm{-regular}$} if it satisfies the following properties:
	\begin{enumerate}
		\item[\textup{(1)}] The ring $B^G$ is noetherian.

		\item[\textup{(2)}] The ring $B$ is faithully flat over $B^G$.
		
		\item[\textup{(3)}] For all $\mathbb{Q}_p\textrm{-representations}$ $V$ of $G$, the homomorphism $\alpha_B(V)$ is injective.

		\item[\textup{(4)}] If $V$ is a one-dimensional $B\textrm{-admissible}$ $F\textrm{-representation}$ of $G$, then its $\mathbb{Q}_p\textrm{-linear}$ dual is also $B\textrm{-admissible}$.
	\end{enumerate}
	The category of $B\textrm{-admissible}$ $\mathbb{Q}_p\textrm{-representations}$ of $G$ shall be denoted as $\textrm{Rep}_{B\textrm{-adm}}(G)$ and it is a Tannakian subcategory of $\textrm{Rep}_{\mathbb{Q}_p}(G)$.
	Moreover, the functor $D_B$ is a $B^G\textrm{-fibre}$ functor.
\end{defi}

\subsubsection{Admissible representations for \texorpdfstring{$\OBlogdR(\overline{R})$}{-} and \texorpdfstring{$\OBlogcrys(\overline{R})$}{-}}\label{subsubsec:log_crys_dR_admis_rep}

For any $p\textrm{-adic}$ Galois representation $V$ of $G_R$, set
\begin{equation*}
	\ODlogdR(V) \coloneq \big(\OBlogdR(\overline{R}) \otimes_{\mathbb{Q}_p} V\big)^{G_R}.
\end{equation*}

\begin{defi}
	A $p\textrm{-adic}$ representation $V$ of $G_R$ is said to be \textit{log-de Rham} if the following natural homomorphism
	\begin{equation}\label{eq:deRham_admis}
		\alpha_{\textrm{log-dR},R}(V) \colon \OBlogdR(\overline{R}) \otimes_{R[1/p]} \ODlogdR(V) \longrightarrow \OBlogdR(\overline{R}) \otimes_{\mathbb{Q}_p} V,
	\end{equation}
	is an isomorphism.
\end{defi}

For a log-de Rham representation $V$ of $G_R$, from \cite[Section 3.3]{andreatta-iovita-semistable}, we have that $\ODlogdR(V)$ is a finite projective $R[1/p]\textrm{-module}$ equipped with a decreasing, separated and exhaustive filtration by finite projective $R[1/p]\textrm{-modules}$ and a $p\textrm{-adically}$ quasi-nilpotent logarithmic connection
\begin{equation*}
	\partial \colon \ODlogdR(V) \longrightarrow \ODlogdR(V) \otimes_R \omega^1_{R/O_F},
\end{equation*}
satisfying Griffiths transversality with respect to the filtration.
Additionally, the isomorphism $\alpha_{\textrm{log-dR},R}(V)$ is compatible with the respective filtrations and $G_R\textrm{-actions}$.

For any $p\textrm{-adic}$ Galois representation $V$ of $G_R$, set
\begin{equation*}
	\ODlogcrys(V) \coloneq \big(\OBlogcrys (\overline{R}) \otimes_{\mathbb{Q}_p} V\big)^{G_R}.
\end{equation*}

\begin{defi}\label{defi:logcrys_rep}
	A $p\textrm{-adic}$ representation $V$ of $G_R$ is said to be \textit{log-crystalline} if the following natural homomorphism
	\begin{equation}\label{eq:crys_admis}
		\alpha_{\textrm{log-cris},R}(V) \colon \OBlogcrys(\overline{R}) \otimes_{R[1/p]} \ODlogcrys(V) \longrightarrow \OBlogcrys(\overline{R}) \otimes_{\mathbb{Q}_p} V,
	\end{equation}
	is an isomorphism.
\end{defi}

For a log-crystalline representation $V$ of $G_R$, it is clear that we have a natural identification $\ODlogcrys(V) \isomorphic \ODlogdR(V)$, in particular, $\ODlogcrys(V)$ is a finite projective $R[1/p]\textrm{-module}$, and by transport of structure, it is equipped with a decreasing, separated and exhaustive filtration by finite projective $R[1/p]\textrm{-modules}$ and a $p\textrm{-adically}$ quasi-nilpotent logarithmic connection
\begin{equation*}
	\partial \colon \ODlogcrys(V) \longrightarrow \ODlogcrys(V) \otimes_R \omega^1_{R/O_F},
\end{equation*}
satisfying Griffiths transversality with respect to the filtration.
Additionally, the natural Frobenius map $\varphi^*\ODlogcrys(V) \rightarrow \ODlogcrys(V)$ is bijective, and the isomorphism $\alpha_{\textrm{log-cris},R}(V)$ is compatible with the respective Frobenii, filtrations and $G_R\textrm{-actions}$.

\begin{defi}
	A $p\textrm{-adic}$ representation $V$ of $G_R$ is said to be \textit{strongly log-crystalline} if the following natural homomorphism
	\begin{equation}\label{eq:log_admis}
		\alpha_{\textrm{log,R}}(V) \colon \OBlog(\overline{R}) \otimes_{R[1/p]} \ODlog(V) \longrightarrow \OBlog(\overline{R}) \otimes_{\mathbb{Q}_p} V,
	\end{equation}
	is an isomorphism.
\end{defi}

For a strongly log-crystalline representation $V$ of $G_R$, it is clear that we have a natural identification $\ODlog(V) \isomorphic \ODlogcrys(V)$, in particular, $\ODlog(V)$ is a finite projective $R[1/p]\textrm{-module}$, and by transport of structure, it is equipped with a decreasing, separated and exhaustive filtration by finite projective $R[1/p]\textrm{-modules}$ and a $p\textrm{-adically}$ quasi-nilpotent logarithmic connection
\begin{equation*}
	\partial \colon \ODlog(V) \longrightarrow \ODlog(V) \otimes_R \omega^1_{R/O_F},
\end{equation*}
satisfying Griffiths transversality with respect to the filtration.
Additionally, the natural Frobenius map $\varphi^*\ODlog(V) \rightarrow \ODlog(V)$ is bijective, and the isomorphism $\alpha_{\textrm{log},R}(V)$ is compatible with the respective Frobenii, filtrations and $G_R\textrm{-actions}$.

\begin{prop}\label{prop:GR_regularity}
	The rings $\OBlogdR(\overline{R})$, $\OBlogcrys(\overline{R})$ and $\OBlog(\overline{R})$ are $G_R\textrm{-regular}$.
\end{prop}
\begin{proof}
	Let $B$ denote one of the rings $\OBlogdR(\overline{R})$, $\OBlogcrys(\overline{R})$ and $\OBlog(\overline{R})$.
	We first note that $B^{G_R} = R[1/p]$ (see the discussions in Sections \ref{subsubsec:BdR}, \ref{subsubsec:Bcris} and \ref{subsubsec:Blog}), and the natural homomorphism $R[1/p] \rightarrow B$ is faithfully flat (see \cite[Proposition 3.18]{andreatta-iovita-semistable}, Proposition \ref{prop:OBlogcrys_ff} and Proposition \ref{prop:OBlog_properties}).

	Next, let $V$ be a $\mathbb{Q}_p\textrm{-representation}$ of $G_R$ and from Lemma \ref{lem:alpha_injective} below we have that $\alpha_B(V)$ is injective.
	So it remains to show that if $V$ is $B\textrm{-admissible}$, then the finite dimensional $\mathbb{Q}_p\textrm{-vector}$ space $V^* \coloneq \textup{Hom}_{\mathbb{Q}_p}(V, \mathbb{Q}_p)$ is also $B\textrm{-admissible}$.
	Let us note that by the faithful flatness of the map $R[1/p] \rightarrow B$ and the $B\textrm{-admissibility}$ of $V$, we know that $D_B(V)$ is a finite projective $R[1/p]\textrm{-module}$.
	Then, we have the following $G_R\textrm{-equivariant}$ isomorphism of $B\textrm{-modules}$:
	\begin{equation*}
		\begin{aligned}
			\textup{Hom}_{R[1/p]}(D_B(V), R[1/p]) \otimes_{R[1/p]} B &\isomorphic \textup{Hom}_{R[1/p]}(D_B(V), B)\\
			&\isomorphic \textup{Hom}_B(D_B(V) \otimes_{R[1/p]} B, B)\\
			&\isomorphic \textup{Hom}_B(V \otimes_{\mathbb{Q}_p} B, B)\\
			&\isomorphic \textup{Hom}_{\mathbb{Q}_p}(V, B) \\
			&\isomorphic \textup{Hom}_{\mathbb{Q}_p}(V, \mathbb{Q}_p) \otimes_{\mathbb{Q}_p} B = V^* \otimes_{\mathbb{Q}_p} B,
		\end{aligned}
	\end{equation*}
	where the first isomorphism follows because $D_B(V)$ is a finite projective $R[1/p]\textrm{-module}$, the second and the fourth isomorphisms follow from the $\otimes\textrm{-Hom}$ adjunction, the third isomorphism follows from the $B\textrm{-admissibility}$ of $V$, and the fifth isomorphism follows because $V$ is a finite dimensional $\mathbb{Q}_p\textrm{-vector}$ space.
	Taking $G_R\textrm{-invariants}$ in the diagram above yields a natural $R[1/p]\textrm{-linear}$ isomorphism $\textup{Hom}_{R[1/p]}(D_B(V), R[1/p]) \isomorphic D_B(V^*)$, and then the diagram further implies that $\alpha_{B}(V^*)$ is bijective, i.e.\ $V^*$ is $B\textrm{-admissible}$.
	This concludes our proof.
\end{proof}

The following claim was used in the proof of Proposition \ref{prop:GR_regularity}:

\begin{lem}\label{lem:alpha_injective}
	Let $V$ be a $p\textrm{-adic}$ representation of $G_R$.
	Then, the homomorphisms $\alpha_{\textup{log-dR},R}(V)$, $\alpha_{\textup{log-cris},R}(V)$ and $\alpha_{\textup{log},R}(V)$ are injective.
\end{lem}
\begin{proof}
	The idea of the proof comes from the proof of \cite[Theorem 4.11]{abhinandan-relative-wach-ii} and \cite[Lemme 8.2.2]{brinon-relatif}.
	Let $B$ denote one of the rings $\OBlogdR(\overline{R})$, $\OBlogcrys(\overline{R})$ and $\OBlog(\overline{R})$, and note that $B^{G_R} = R[1/p]$ (see the discussions in Sections \ref{subsubsec:BdR}, \ref{subsubsec:Bcris} and \ref{subsubsec:Blog}).
	Then, by definition we have that $D_B(V)$ is a torsion-free $R[1/p]\textrm{-module}$.
	From Section \ref{subsubsec:localisation_Rbar}, recall that $O_L = (R_{(p)})^{\wedge}$ is a complete discrete valuation ring with uniformiser $p$ and an imperfect residue field, and we set $L = O_L[1/p]$ as its fraction field.
	Then, we have natural injective $R[1/p]\textrm{-linear}$ homomorphisms 
	\begin{equation}\label{eq:dbv_in_ldbv}
		D_B(V) \longhookrightarrow R_{(p)} \otimes_R D_B(V) \longhookrightarrow O_L \otimes_R D_B(V) = L \otimes_{R[1/p]} D_B(V),
	\end{equation}
	compatible with the respective filtrations and connections, and where the first arrow is injective because $D_B(V)$ is a torsion-free $R[1/p]\textrm{-module}$, and the second arrow is injective because the homomorphism $R_{(p)} \rightarrow O_L$ is faithfully flat as $R_{(p)}$ is noetherian and $p$ is in the Jacobson radical of $R_{(p)}$ (see \cite[Theorem 8.14]{matsumura}).

	Next, from the discussion in Sections \ref{subsubsec:localisation_dR}, \ref{subsubsec:localisation_crys} and \ref{subsubsec:Blog}, we have a natural $L\textrm{-linear}$ and $G_R\textrm{-equivariant}$ injective homomorphism
	\begin{equation}\label{eq:lb_in_BdRprod}
		L \otimes_{R[1/p]} B \longhookrightarrow \textstyle\prod_{\pins} \OBdR(\Cplusp) \longhookrightarrow \textstyle\prod_{\pins} \OBdR(\Cpplus).
	\end{equation}
	Tensoring the preceding injective homomorphism with $V$ (over $\mathbb{Q}_p$) and considering the diagonal action of $G_R$, we obtain a $(\varphi, G_R)\textrm{-equivariant}$ injective map
	\begin{equation}\label{eq:lbv_embed}
		L \otimes_{R[1/p]} B \otimes_{\mathbb{Q}_p} V \longrightarrow \big(\textstyle\prod_{\pins} \OBdR(\Cplusp)\big) \otimes_{\mathbb{Q}_p} V = \textstyle\prod_{\pins} (\OBdR(\Cplusp) \otimes_{\mathbb{Q}_p} V).
	\end{equation}
	The composition in \eqref{eq:lbv_embed} further induces a natural $G_R\textrm{-equivariant}$ homomorphism
	\begin{equation*}
		L \otimes_{R[1/p]} B \otimes_{\mathbb{Q}_p} V \longrightarrow \OBdR(\Cplusp) \otimes_{\mathbb{Q}_p} V,
	\end{equation*}
	compatible with the respective filtrations and connections.
	Now, we take the $G_R\textrm{-invariant}$ part of \eqref{eq:lbv_embed} and note that taking $G_R\textrm{-invariants}$ commutes with product.
	So, we obtain the following $\varphi\textrm{-equivariant}$ $L\textrm{-linear}$ injective homomorphisms:
	\begin{equation}\label{eq:ldbv_embed}
		\begin{aligned}
			L \otimes_{R[1/p]} D_B(V) &\longhookrightarrow \big(\textstyle\prod_{\pins} \OBdR(\Cplusp) \otimes_{\mathbb{Q}_p} V\big)^{G_R}\\
			&\xrightarrow{\hspace{1mm} = \hspace{1mm}} \textstyle\prod_{\pins} (\OBdR(\Cplusp) \otimes_{\mathbb{Q}_p} V)^{G_R}\\
			&\longhookrightarrow \textstyle\prod_{\pins} (\OBdR(\Cplusp) \otimes_{\mathbb{Q}_p} V)^{\GRp},
		\end{aligned}
	\end{equation}
	where we note that the last arrow is injective because $\GRp \subset G_R$ is a subgroup.
	Moreover, since $G_R$ acts transitively on $\mathscr{P}(\overline{R})$, therefore it transitively permutes the components of the product $\prod_{\pins} (\OBdR(\Cplusp) \otimes_{\mathbb{Q}_p} V)^{\GRp}$.
	So, if $x \neq 0$ is an element of $L \otimes_{R[1/p]} D_B(V)$, then its image $(x_{\mathfrak{p}})_{\pins}$ under the composition \eqref{eq:ldbv_embed} satisfies that $x_{\mathfrak{p}} \neq 0$, for all $\pins$.
	Therefore, for each $\pins$, composing \eqref{eq:ldbv_embed} with the natural $\varphi\textrm{-equivariant}$ $L\textrm{-linear}$ projection
	\begin{equation*}
		\textstyle\prod_{\pins} (\OBdR(\Cplusp) \otimes_{\mathbb{Q}_p} V)^{\GRp} \longrightarrow (\OBdR(\Cplusp) \otimes_{\mathbb{Q}_p} V)^{\GRp},
	\end{equation*}
	gives a natural $\varphi\textrm{-equivariant}$ $L\textrm{-linear}$ injective homomorphism
	\begin{equation}\label{eq:ldbv_embed_p}
		L \otimes_{R[1/p]} D_B(V) \longhookrightarrow (\OBdR(\Cplusp) \otimes_{\mathbb{Q}_p} V)^{\GRp},
	\end{equation}
	compatible with the respective filtrations and connections, where the left-hand term is equipped with the $L\textrm{-linear}$ extension of the filtration on $D_B(V)$ and an integrable tensor product connection ($\partial \otimes 1 + 1 \otimes \partial$).

	Next, from Section \ref{subsubsec:localisation_dR}, recall that we have a natural $\GRhatp\textrm{-equivariant}$ injective homomorphism of $L\textrm{-algebras}$ $\OBdR(\Cplusp) \rightarrow \OBdR(\Cpplus)$ compatible with the respective filtrations and connections, and where the $\GRhatp\textrm{-action}$ on the left-hand term factors through $\GRhatp \twoheadrightarrow \GRp$.
	Tensoring the preceding injective homomorphism with $V$ (over $\mathbb{Q}_p$), equipping each term with the diagonal action of $\GRhatp$ and taking the $\GRhatp\textrm{-invariants}$ yields a natural $L\textrm{-linear}$ injective homomorphism
	\begin{equation*}
		(\OBdR(\Cplusp) \otimes_{\mathbb{Q}_p} V)^{\GRp} \longhookrightarrow \ODdRL(V),
	\end{equation*}
	compatible with the respective filtrations and connections.
	Composing \eqref{eq:dbv_in_ldbv} and \eqref{eq:ldbv_embed_p} with the preceding $L\textrm{-linear}$ homomorphism gives a natural $L\textrm{-linear}$ injective homomorphism
	\begin{equation}\label{eq:dbv_functoriality}
		D_B(V) \longhookrightarrow L \otimes_{R[1/p]} D_B(V) \longhookrightarrow \ODdRL(V),
	\end{equation}
	compatible with the respective filtrations and connections.
	
	Now, using the discussion above, we consider the following $G_R\textrm{-equivariant}$ commutative diagram:
	\begin{equation*}
		\begin{tikzcd}[column sep=large]
			B \otimes_{R[1/p]} D_B(V) &  B \otimes_{\mathbb{Q}_p} V\\
			B \otimes_{R[1/p]} \ODdRL(V) \\
			B \otimes_{R[1/p]} L \otimes_L \ODdRL(V) \\
			\textstyle\prod_{\pins} \big(\OBdR(\Cpplus) \otimes_L \ODdRL(V)\big) & \textstyle\prod_{\pins} \big(\OBdR(\Cpplus) \otimes_{\mathbb{Q}_p} V\big),
			\arrow[from=1-1, to=1-2, "\alpha_B(V)"]
			\arrow[from=1-1, to=2-1, hookrightarrow, "\eqref{eq:dbv_functoriality}"]
			\arrow[from=1-2, to=4-2, hookrightarrow, "\eqref{eq:lb_in_BdRprod}"]
			\arrow[from=2-1, to=3-1, equal]
			\arrow[from=3-1, to=4-1, hookrightarrow, "\eqref{eq:lb_in_BdRprod}"]
			\arrow[from=4-1, to=4-2, hookrightarrow, "\alpha_{\textup{dR},L}(V)"]
		\end{tikzcd}
	\end{equation*}
	where the top left vertical arrow is scalar extension of the injective map \eqref{eq:dbv_functoriality} along the flat homomorphism $R[1/p] \rightarrow B$, the bottom left arrow is the tensor product (over $L$) of the injective map \eqref{eq:lb_in_BdRprod} with the finite dimensional $L\textrm{-vector}$ space $\ODdRL(V)$ (see \cite{brinon-imparfait}), the bottom horizontal map is the product of natural injective $\GRhatp\textrm{-equivariant}$ homomorphisms analogous to \eqref{eq:deRham_admis} (see \cite[Proposition 3.22]{brinon-imparfait}), and the right vertical arrow is the tensor product with $V$ (over $\mathbb{Q}_p$) of the injective homomorphism \eqref{eq:lb_in_BdRprod} precomposed with the $G_R\textrm{-equivariant}$ injective homomorphism $B \hookrightarrow L \otimes_{R[1/p]} B$.
	From the diagram, it clearly follows that $\alpha_B(V)$ is injective, thus proving the claim.
\end{proof}

Finally, let us examine the relationship between log-de Rham (resp.\ log-crystalline) representations of $G_R$ and de Rham (resp.\ crystalline) representations of $G_S$ and $G_L$.
Recall that we have a homomorphism of groups $G_L \rightarrow G_S \rightarrow G_R$ which defines a natural action of $G_S$ and $G_L$ on $\OBlogcrys(\overline{R})$. 
Let $V$ be a $p\textrm{-adic}$ representation of $G_R$.
Then, the homomorphism $G_L \rightarrow G_S \rightarrow G_R$ equips $V$ with the structure of a $G_S\textrm{-representation}$ and $G_L\textrm{-representation}$.

\begin{prop}\label{prop:deRham_base_change}
	Let $V$ be a log-de Rham representation of $G_R$.
	Then $V$ is a de Rham representation of $G_S$, and we have the following natural isomorphism of $S[1/p]\textrm{-modules}$
	\begin{equation}\label{eq:deRham_base_change}
		S[1/p] \otimes_{R[1/p]} \ODlogdR(V) \isomorphic \ODdRS(V),
	\end{equation}
	compatible with the respective filtrations and connections, where the left-hand-term is equipped with the tensor-product connection, and the filtration on $S[1/p]$ is trivial.
	Analogous claims hold after replacing $S[1/p]$ above with $L$.
\end{prop}
\begin{proof}
	We only prove the claim for $S[1/p]$, and the claim for $L$ follows by employing a similar argument.
	From Section \ref{subsubsec:BdR} recall that we have a natural $G_S\textrm{-equivariant}$ homomorphism $\OBlogdR(\overline{R}) \rightarrow \OBdR(\overline{S})$ compatible with the respective filtrations and connections.
	Base changing the isomorphism in \eqref{eq:deRham_admis} along the preceding homomorphism yields the following $G_S\textrm{-equivariant}$ isomorphism of $\OBdR(\overline{S})\textrm{-modules}$:
	\begin{equation*}
		\alpha_{\textrm{dR,S}}(V) \colon \OBdR(\overline{S}) \otimes_{R[1/p]} \ODlogdR(V) \isomorphic \OBdR(\overline{S}) \otimes_{\mathbb{Q}_p} V,
	\end{equation*}
	compatible with the respective connections.
	Taking the $G_S\textrm{-fixed}$ points of the isomorphism $\alpha_{\textrm{dR,S}}(V)$ yields the isomorphism of $S[1/p]\textrm{-modules}$ in \eqref{eq:deRham_base_change} compatible with the respective connections.
	In particular, from the isomorphism $\alpha_{\textrm{dR,S}}(V)$, we conclude that $V$ is a de Rham representation of $G_S$.

	To check the compatibility of \eqref{eq:deRham_base_change} with filtrations, we need to introduce some notations.
	Let us set $\OBlogHT(\overline{R}) \coloneq \textrm{gr}^{\bullet} \OBlogdR(\overline{R})$ and $\OBHT(\overline{S}) \coloneq \textrm{gr}^{\bullet} \OBdR(\overline{S})$ as graded rings.
	Similarly, set $\ODlogHT(V) \coloneq \textrm{gr}^{\bullet} \ODlogdR(V) = (\OBlogHT(\overline{R}) \otimes_{\mathbb{Q}_p} V)^{G_R}$ as a graded $R[1/p]\textrm{-module}$ and $\ODHTS(V) \coloneq \textrm{gr}^{\bullet} \ODHTS(V) = (\OBHT(\overline{S}) \otimes_{\mathbb{Q}_p} V)^{G_S}$ as a graded $S[1/p]\textrm{-module}$.
	Now, passing to the associated graded in the isomorphism $\alpha_{\textrm{dR,S}}(V)$ yields the following $G_S\textrm{-equivariant}$ isomorphism of graded $\OBHT(\overline{S})\textrm{-modules}$:
	\begin{equation*}
		\alpha_{\textrm{HT,S}}(V) \colon \OBHT(\overline{S}) \otimes_{R[1/p]} \ODlogHT(V) \longrightarrow \OBHT(\overline{S}) \otimes_{\mathbb{Q}_p} V.
	\end{equation*}
	Taking the $G_S\textrm{-fixed}$ points of the isomorphism $\alpha_{\textrm{HT,S}}(V)$ yields the following isomorphism of graded $S[1/p]\textrm{-modules}$ (see \cite[Sections 3.2 and 3.3]{andreatta-iovita}):
	\begin{equation*}
		S[1/p] \otimes_{R[1/p]} \ODlogHT(V) \isomorphic \ODHTS(V).
	\end{equation*}
	As the respective filtrations on $\ODlogdR(V)$ and $\ODdRS(V)$ are exhaustive and separated, an easy induction shows that the isomorphism in \eqref{eq:deRham_base_change} is compatible with filtrations.
\end{proof}

\begin{cor}\label{cor:crys_base_change}
	Let $V$ be a log-crystalline representation of $G_R$.
	Then, $V$ is a crystalline representation of $G_S$ and we have the following natural isomorphism of $S[1/p]\textrm{-modules}$
	\begin{equation}\label{eq:crys_base_change}
		S[1/p] \otimes_{R[1/p]} \ODlogcrys(V) \isomorphic \ODcrysS(V),
	\end{equation}
	compatible with the respective Frobenii, filtrations and connections, where the left-hand-term is equipped with the tensor-product Frobenius and connection, and the filtration on $S[1/p]$ is trivial.
	Analogous claims hold after replacing $S[1/p]$ above with $L$.
\end{cor}
\begin{proof}
	We only prove the claim for $S[1/p]$, and the claim for $L$ follows by employing a similar argument.
	From Section \ref{subsubsec:Bcris} recall that we have a natural $G_S\textrm{-equivariant}$ homomorphism $\OBlogcrys(\overline{R}) \rightarrow \OBcrys(\overline{S})$ compatible with the respective Frobenii, filtrations and connections.
	Base changing the isomorphism in \eqref{eq:crys_admis} along the preceding homomorphism yields the following $G_S\textrm{-equivariant}$ isomorphism of $\OBcrys(\overline{S})\textrm{-modules}$:
	\begin{equation*}
		\alpha_{\textrm{cris,S}}(V) \colon \OBcrys(\overline{S}) \otimes_{R[1/p]} \ODlogcrys(V) \longrightarrow \OBcrys(\overline{S}) \otimes_{\mathbb{Q}_p} V,
	\end{equation*}
	compatible with the respective Frobenii and connections.
	Taking the $G_S\textrm{-fixed}$ points of the isomorphism $\alpha_{\textrm{cris,S}}(V)$ yields the isomorphism of $S[1/p]\textrm{-modules}$ in \eqref{eq:crys_base_change} compatible with the respective connections.
	In particular, from the isomorphism $\alpha_{\textrm{cris,S}}(V)$, we conclude that $V$ is a crystalline representation of $G_S$.
	Finally, from Proposition \ref{prop:deRham_base_change} it follows that the isomorphism in \eqref{eq:crys_base_change} is compatible with the respective filtrations.
	This concludes our proof.
\end{proof}

\subsection{Some examples}

In this section our goal is to provide two examples of Galois representations to emphasise the novelty of log-crystalline representations described in the previous sections.

\begin{exam}
	We first provide a prototypical example of a log-crystalline representation of $G_R$.
	For simplicity, let us assume that $R = \mathbb{Z}_p\langle x \rangle$.
	Let $c \colon G_R \rightarrow \mathbb{Z}_p$ denote a $1\textrm{-cocycle}$, and for any $g$ in $G_R$ set $g([x^{\flat}]) = [\varepsilon]^{c(g)}[x^{\flat}]$.
	Then, for $\beta = \log(\tfrac{[x^{\flat}]}{x})$ in $\OBlogcrys(\overline{R})$, we have that $g(\beta) = \beta + c(g)t$.

	Let $V = \mathbb{Q}_p e_1 + \mathbb{Q}_p e_2$ denote a two-dimensional $\mathbb{Q}_p\textrm{-representation}$ of $G_R$, with the action of any $g$ in $G_R$ given as $g(e_1) = \chi(g) e_1$ and $g(e_2) = c(g)e_1 + e_2$.
	Then, it is easy to see that we have $\ODlogcrys(V) = R[1/p] f_1 + R[1/p] f_2$ with $f_1 = t^{-1} \otimes e_1$ and $f_2 = 1 \otimes e_2 - t^{-1}\beta \otimes e_1$.
	From this description it is straightforward to verify that $V$ is a log-crystalline representation of $G_R$, i.e.\ it is $\OBlogcrys(\overline{R})\textrm{-admissible}$, but it is \textit{not} crystalline, i.e.\ $V$ is not $\OBcrys(\overline{R})\textrm{-admissible}$.
\end{exam}

\begin{exam}
	We will now provide an example of a $p\textrm{-adic}$ representation of $G_R$ which is \textit{not} log-crystalline but it is a crystalline representation of $G_S$.
	For simplicity again, let us set $R = \mathbb{Z}_p\langle x \rangle$.
	Let us first note that for each $(m, p)=1$, we have that $x^{1/m}$ belongs to $\OBlogcrys(\overline{R})$.
	Indeed, by formally writing $x^{1/m} = [x^{\flat}]^{1/m} (1-y)^{1/m}$ with $y = 1-\tfrac{x}{[x]^{\flat}}$, we see that the series $(1-y)^{1/m} = \sum_{k \geqslant 0} (-1)^k \binom{1/m}{k} y^k$ converges in $\OBlogcrys(\overline{R})$.
	An analogous argument shows that for each $m \geqslant 0$, we have that $x^{1/m}$ belongs to $\OBlogdR(\overline{R})$.

	Let $V = \mathbb{Q}_p(\eta)$ denote the one dimensional representation of $G_R$ given by the character $\eta$ described as $g(x^{1/(p-1)}) = \eta(g)x^{1/(p-1)}$ for any $g$ in $G_R$, i.e.\ $\eta(g)$ is in $\mu_{p-1} \subset \mathbb{Z}_p^{\times}$.
	From the previous paragraph, we have that $x^{1/(p-1)}$ is in $\OBlogcrys(\overline{R})$, and therefore $x^{-1/(p-1)}$ is in $\OBcrys(\overline{S})$.
	Let $\{e\}$ denote a $\mathbb{Q}_p\textrm{-basis}$ of $V$.
	Then, we see that $f_S = ex^{-1/(p-1)}$ is in $\OBcrys(\overline{S})$, satisfying $g(f_S) = g(e)g(x^{-1/(p-1)}) = f_S$.
	Therefore, we get that $\ODcrysS(V) = S[1/p] f_S$.

	Let $\calR = R[x^{1/(p-1)}]$ and $\calS = S[x^{1/(p-1)}]$, and note that we have $\pazocal{O}D_{\textup{log-crys},\calR}(V) = \calR[1/p]e$ and $\pazocal{O}D_{\textup{crys},\calS}(V) = \calS[1/p]e = \calS[1/p]d_S$.
	Then, it is easy to compute that $\ODlogcrys(V) = \pazocal{O}D_{\textup{log-crys},\calR}(V) \cap \ODcrysS(V) = R[1/p] f_R$, where $f_R = x f_S$ and the intersection takes place inside $\pazocal{O}D_{\textup{crys},\calS}(V)$.
	In particular, we see that the natural injective map $\OBlogcrys(\overline{R}) \otimes_{R[1/p]} \ODlogcrys(V) \rightarrow \OBlogcrys(\overline{R}) \otimes_{\mathbb{Q}_p} V$ is not surjective.
	Hence, we conclude that $V = \mathbb{Q}_p(\eta)$ is a crystalline representation of $G_S$, but it is \textit{not} a log-crystalline representation of $G_R$.
\end{exam}

\section{Wach modules associated to log-crystalline representations}\label{sec:logcris_wachmod}

In this section, we shall study the relationship between $(\varphi, \Gamma)\textrm{-modules}$ associated to a log-crystalline representations.
We will keep the notation of previous sections.
In particular, $R$ denotes a $p\textrm{-adically}$ smooth affine algebra, and $G_R$ its log-\'etale fundamental group.

\subsection{Some logarithmic divided power rings}\label{subsec:logPD_rings}

In this section, we will define some log-crystalline period rings similar to the crystalline period rings studied in \cite[Section 4.3.1]{abhinandan-relative-wach-i} and \cite[Section 3.6]{abhinandan-relative-wach-ii}.
These rings will be used to carry out some computations in later sections.

Let $m \geqslant 1$ be an integer, let $\calR \coloneq R_{0,m}$ and set $\Delta \coloneq \Delta_{\calR/R}$.
Let $\varpi \coloneq \zeta_{p^r}-1$, where $r = 1$, for $p \geqslant 3$, and $r = 2$, for $p = 2$.
Let $q = 1+\mu$ and set
\begin{equation*}
	\ARpi^+ \coloneq \AR^+[Y]/(Y^{p^r}-q) \isomorphic \AR^+[q^{1/p^r}] = \AR^+[\varphi^{-r}(q)] \subset A_{\inf}(\calR_{\infty}).
\end{equation*}
By definition, it is clear that $\ARpi^+$ is finite free as a module over $\AR^+$, with a basis given as $\{1, q^{1/p^r}, \ldots, q^{(p^r-1)/p^r}\}$.
Moreover, it is also clear that $\ARpi^+$ is stable under the Frobenius on $A_{\inf}(\calR_{\infty})$, and as the action of $\GammaR \times \Delta$ on $A_{\inf}(\calR_{\infty})$ is Frobenius-equivariant, therefore, it follows that $\ARpi^+$ is stable under the action of $\GammaR \times \Delta$; we equip it with the induced structures.
We further equip $\ARpi^+$ with a log structure defined by the divisor $\{[x_{a+1}^{\flat}] \cdots [x_d^{\flat}] = 0\}$.

Next, we note that the restriction to $\ARpi^+$ of the map $\theta$ on $A_{\inf}(\calR_{\infty})$ (see Section \ref{subsubsec:perfect_period_rings}), gives a surjective ring homomorphism $\theta \colon \ARpi^+ \twoheadrightarrow \calR[\varpi]$.
We define $\ARpi^{\PD}$ to be the $p\textrm{-adic}$ completion of the divided power envelope of $\ARpi^+$, with respect to $\textrm{ker } \theta$.
Furthermore, the map $\theta$ extends $R\textrm{-linearly}$ to a surjective ring homomorphism $\theta_R \colon R \otimes_{\mathbb{Z}} \ARpi^+ \twoheadrightarrow \calR[\varpi]$, given as $x \otimes y \mapsto x\theta(y)$.
Similar to the definition of $\OAlogcrys(\overline{R})$ in Section \ref{subsubsec:Bcris}, we define $\OARpi^{\logPD}$ to be the $p\textrm{-adic}$ completion of the log-PD envelope of $(R \otimes_{\mathbb{Z}} \ARpi^+)$, with respect to $\textrm{ker } \theta_R$.
The morphisms $\theta$ and $\theta_R$ naturally extend to respective surjections $\theta \colon \ARpi^{\PD} \twoheadrightarrow \calR[\varpi]$ and $\theta_R \colon \OARpi^{\logPD} \twoheadrightarrow \calR[\varpi]$.
In addition, note that we have the following natural isomorphism of $\ARpi^{\PD}\textrm{-algebras}$:
\begin{equation}\label{eq:OARlogPD_explicit}
	\begin{aligned}
		\ARpi^{\PD}[y_1, \ldots, y_d]_{\textup{PD}}^{\wedge} &\isomorphic \OARpi^{\logPD}\\
		y_i^{[k]} &\longmapsto (1-\tfrac{[x_i^{\flat}]}{x_i})^{[k]},
	\end{aligned}
\end{equation}
where the left-hand-term denotes the $p\textrm{-adic}$ completion of the PD-polynomial algebra over $\ARpi^{\PD}$ in variables $\{y_1, \ldots, y_d\}$.
Then, by employing an argument similar to \cite[Lemma 3.33]{abhinandan-relative-wach-ii}, we obtain the following:
\begin{lem}\label{lem:OARlogpd_OAlogcrys}
	The natural $(\varphi, \GammaR \times \Delta)\textrm{-equivariant}$ injective homomorphism $\ARpi^+ \rightarrow A_{\inf}(\calR_{\infty})$ extends to natural injective homomorphisms $\ARpi^{\PD} \rightarrow \Alogcrys(\calR_{\infty})$ and $\OARpi^{\logPD} \rightarrow \OAlogcrys(\calR_{\infty})$.
	Moreover, the source rings are stable under the $(\varphi, \GammaR \times \Delta)\textrm{-action}$ on the target rings.
\end{lem}

Using the injectivity of the map $\ARpi^{\PD} \hookrightarrow \Alogcrys(\calR_{\infty})$ (resp.\ $\OARpi^{\logPD} \hookrightarrow \OAlogcrys(\calR_{\infty})$) from Lemma \ref{lem:OARlogpd_OAlogcrys}, we equip the source rings with structures induced from the target rings, in particular, a Frobenius endomorphism $\varphi$, a continuous action of $\GammaR \times \Delta$, and an induced $(\GammaR \times \Delta)\textrm{-stable}$ decreasing, separated and exhaustive filtration.
In particular, the natural homomorphism $\ARpi^{\PD} \rightarrow \OARpi^{\logPD}$ is injective and compatible with all the structures.
Moreover, as the filtration on $\Alogcrys(\calR_{\infty})$ (resp.\ $\OAlogcrys(\calR_{\infty})$) is given by the divided powers of the ideal $\textup{ker }\theta \subset \Alogcrys(\calR_{\infty})$ (resp.\ $\textrm{ker } \theta_R \subset \OAlogcrys(\calR_{\infty})$), therefore, the induced filtration on $\ARpi^{\PD}$ (resp.\ $\OARpi^{\logPD}$) is also given by the divided powers of the ideal $\textrm{ker } \theta \subset \ARpi^{\PD}$ (resp.\ $\textrm{ker } \theta_{\calR} \subset \OARpi^{\logPD}$).
Furthermore, the $\Alogcrys(\calR_{\infty})\textrm{-linear}$ integrable log-connection on $\OAlogcrys(\calR_{\infty})$, denoted as $\partial$, induces a $(\GammaR \times \Delta)\textrm{-equivariant}$ $\ARpi^{\PD}\textrm{-linear}$ integrable log-connection $\partial_A \colon \OARpi^{\logPD} \rightarrow \OARpi^{\logPD} \otimes \Omega^1_{\calR}$.
Note that the connection $\partial_A$ satisfies Griffiths transversality with respect to the induced filtration: indeed, from definitions we have that
\begin{equation*}
	\begin{aligned}
		\partial_A(\Fil^k \OARpi^{\logPD}) &\subset (\OARpi^{\logPD} \otimes \Omega^1_{\calR}) \cap \partial(\Fil^k \OAlogcrys(\calR_{\infty})) \\
			&\subset (\OARpi^{\logPD} \cap \Fil^{k-1} \OAlogcrys(\calR_{\infty})) \otimes \Omega^1_{\calR} = \Fil^{k-1} \OARpi^{\logPD} \otimes \Omega^1_{\calR},
	\end{aligned}
\end{equation*}
where the second inclusion follows because the connection on $\OAlogcrys(\calR_{\infty})$ satisfies Griffiths transversality with respect to the filtration (see Section \ref{subsubsec:Bcris}).
Additionally, we have that $\OAlogcrys(\calR_{\infty})^{\partial = 0} = \Alogcrys(\calR_{\infty})$, and so it follows that $(\OARpi^{\logPD})^{\partial_A=0} = \ARpi^{\PD}$.

\begin{rem}\label{rem:fil1ar+_intersect}
	Note that for the composition $\AR^+ \hookrightarrow A_{\inf}(\calR_{\infty}) \xrightarrow{\hspace{1mm}\theta\hspace{1mm}} \widehat{\calR}_{\infty}$, we have that $\theta(\AR^+) = \calR$, and therefore, $\textrm{ker } \theta = \mu \AR^+ \subset \AR^+$.
	So, it follows that $\Fil^1 \OARpi^{\logPD} \cap \AR^+ = \Fil^1 \ARpi^{\logPD} \cap \AR^+ = \mu \AR^+$ inside $\OARpi^{\logPD}$.
\end{rem}

Analogously, by replacing $\calR$ with $\calS$ above and equipping it with the trivial log structure, we obtain variants of all the period rings described (see \cite[Section 3.6]{abhinandan-relative-wach-ii}).
In particular, we have period rings $\ASpi^+ \hookrightarrow  A_{\inf}(\calS_{\infty})$, $\ASpi^{\PD} \hookrightarrow \Acrys(\calS_{\infty})$ and $\OASpi^{\PD} \hookrightarrow \OAcrys(\calS_{\infty})$, and the source rings are stable under the $(\varphi, \GammaS \times \Delta)\textrm{-action}$ on the target rings; we equip these rings with induced structures.
Note that the $(\varphi, \GammaS \times \Delta)\textrm{-equivariant}$ injective homomorphism of rings $A_{\inf}(\calR_{\infty}) \hookrightarrow A_{\inf}(\calS_{\infty})$ from Lemma \ref{lem:Ainf_Rinftym_in_Sinftym} restricts to a $(\varphi, \GammaS \times \Delta)\textrm{-equivariant}$ injective homomorphism of rings $\ARpi^+ \hookrightarrow \ASpi^+$.
Similarly, the $(\varphi, \GammaS \times \Delta)\textrm{-equivariant}$ injective homomorphism of rings $\Alogcrys(\calR_{\infty}) \hookrightarrow \Acrys(\calS_{\infty})$ (resp.\ $\OAlogcrys(\calR_{\infty}) \hookrightarrow \OAcrys(\calS_{\infty})$) from Lemma \ref{lem:Acrys_Rinftym_in_Sinftym} restricts to a $(\varphi, \GammaS \times \Delta)\textrm{-equivariant}$ injective homomorphism of rings $\ARpi^{\PD} \hookrightarrow \ASpi^{\PD}$ (resp.\ $\OARpi^{\logPD} \hookrightarrow \OASpi^{\PD}$).

\begin{lem}\label{lem:OASpiPD_OAlogcrys_intersect}
	The following natural $(\varphi, \GammaS \times \Delta)\textrm{-equivariant}$ inclusion of rings
	\begin{equation}\label{eq:OASpiPD_OAlogcrys_intersect}
		\OARpi^{\logPD} \longhookrightarrow \OASpi^{\PD} \cap \OAlogcrys(\calR_{\infty}),
	\end{equation}
	where the intersection takes place inside $\OAcrys(\calS_{\infty})$, is bijective.
\end{lem}
\begin{proof}
	For the PD rings appearing in \eqref{eq:OASpiPD_OAlogcrys_intersect}, we have their explicit description from \eqref{eq:OAcryslog_explicit}, \eqref{eq:OARlogPD_explicit} and \cite[Lemma 4.20]{abhinandan-relative-wach-i}, using which we easily see that it is enough to show that the following natural $(\varphi, \GammaS \times \Delta)\textrm{-equivariant}$ inclusion of rings
	\begin{equation}\label{eq:ASpiPD_OAlogcrys_intersect}
		\ARpi^{\logPD} \longhookrightarrow \ASpi^{\PD} \cap \Acrys(\calR_{\infty}),
	\end{equation}
	where the intersection takes place inside $\Acrys(\calS_{\infty})$, is bijective.
	Now, similar to \cite[Corollaire 6.1.2]{brinon-relatif} and \cite[Lemma 3.33]{abhinandan-relative-wach-i} we compute that 
	\begin{equation*}
		\begin{tikzcd}
			\ARpi^{\PD}/p \isomorphic (\calR_1[\zeta_{p^{r+1}}]/p)[Y_0, Y_1, \ldots]/(Y_k^p)_{k \geqslant 0} & (\calR_{\infty}/p)[Y_0, Y_1, \ldots]/(Y_k^p)_{k \geqslant 0} \lisomorphic \Acrys(\calR_{\infty})/p \\
			\ASpi^{\PD}/p \isomorphic (\calS_1[\zeta_{p^{r+1}}]/p)[Y_0, Y_1, \ldots]/(Y_k^p)_{k \geqslant 0} & (\calS_{\infty}/p)[Y_0, Y_1, \ldots]/(Y_k^p)_{k \geqslant 0} \lisomorphic \Acrys(\calS_{\infty})/p,
			\arrow[hook, from=1-1, to=1-2]
			\arrow[hook, from=1-1, to=2-1]
			\arrow[hook, from=1-2, to=2-2]
			\arrow[hook, from=2-1, to=2-2]
		\end{tikzcd}
	\end{equation*}
	where the top left horizontal isomorphism sends $q^{1/p^r} \mapsto \zeta_{p^{r+1}}$ and $[x_i^{\flat}] \mapsto x_i^{1/p}$ for $1 \leqslant i \leqslant d$, and similarly for the other isomorphisms.
	Then, an easy and explicit computation shows that we have $\calR_1[\zeta_{p^{r+1}}]/p = \calR_{\infty}/p \cap \calS_1[\zeta_{p^{r+1}}]/p \subset \calS_{\infty}/p$.
	Therefore, from the description of the rings in \eqref{eq:ASpiPD_OAlogcrys_intersect} modulo $p$ above, it follows that inside $\Acrys(\calS_{\infty})/p$ we have a $(\varphi, \GammaS \times \Delta)\textrm{-equivariant}$ identification of rings
	\begin{equation}
		\ARpi^{\logPD}/p = \ASpi^{\PD}/p \cap \Acrys(\calR_{\infty})/p.
	\end{equation}
	As both the source and the target of \eqref{eq:ASpiPD_OAlogcrys_intersect} are $p\textrm{-adically}$ complete and $p\textrm{-torsion}$ free, therefore, using the preceding identification, we conclude that \eqref{eq:ASpiPD_OAlogcrys_intersect} is bijective.
\end{proof}

\begin{lem}\label{lem:OARpilogPD_AS+_intersect}
	The following natural $(\varphi, \GammaR \times \Delta)\textrm{-equivariant}$ inclusion of rings
	\begin{equation}\label{eq:OARpilogPD_AS+_intersect}
		\AR^+ \longhookrightarrow \OARpi^{\logPD} \cap \AS^+,
	\end{equation}
	where the intersection takes place inside $\OASpi^{\PD}$, is bijective.
\end{lem}
\begin{proof}
	It is clear that the homomorphism in \eqref{eq:OARpilogPD_AS+_intersect} is $(\varphi, \GammaR \times \Delta)\textrm{-equivariant}$ and injective.
	To show that it is surjective, let us recall that we have isomorphisms of rings $\AR^+ \isomorphic \calR\llbracket \mu \rrbracket$, $\AS^+ \isomorphic \calS\llbracket \mu \rrbracket$ and $\ARpi^+ \isomorphic \calR\llbracket \varphi^{-r}(\mu)\rrbracket$ (it is enough to check these modulo $p$, whence the result follows from the discussion in Section \ref{subsec:period_rings} and at the beginning of Section \ref{subsec:logPD_rings}).
	So, using the isomorphism in \eqref{eq:OARlogPD_explicit} (resp.\ \cite[Lemma 4.20]{abhinandan-relative-wach-i}) we obtain an inclusion of rings $\OARpi^{\logPD} \hookrightarrow \calR[1/p]\llbracket\varphi^{-r}(\mu), y_1, \ldots, y_d\rrbracket \eqcolon C_{\calR}$ (resp.\ $\OASpi^{\PD} \hookrightarrow \calS[1/p]\llbracket\varphi^{-r}(\mu), y_1, \ldots, y_d\rrbracket \eqcolon C_{\calS}$), sending $\varphi^{-r}(\mu) \mapsto \varphi^{-r}(\mu)$, and for each $1 \leqslant i \leqslant d$, we have that $[x_i^{\flat}] \mapsto x_i$ and $(1-\tfrac{[x_i^{\flat}]}{x_i})^{[k]} \mapsto \tfrac{y_i^k}{k!}$.
	Then, it suffices to show that $\calR\llbracket \mu \rrbracket \isomorphic C_{\calR} \cap \calS\llbracket \mu \rrbracket \hookrightarrow C_{\calS}$, as rings.
	But it is clear that we have 
	\begin{equation*}
		C_{\calR} \cap \calS\llbracket \mu \rrbracket = \calR[1/p]\llbracket \varphi^{-r}(\mu) \rrbracket \cap \calS\llbracket \mu \rrbracket = \calR[1/p]\llbracket \mu \rrbracket \cap \calS\llbracket \mu \rrbracket = \calR\llbracket \mu \rrbracket,
	\end{equation*}
	because we have $\calR[1/p] \cap \calS = \calR$, or equivalently, we have $\calR/p\calR \hookrightarrow (\calR/p\calR)[1/(x_{a+1}\cdots x_d)] = \calS/p\calS$.
	Hence, it follows that \eqref{eq:OARpilogPD_AS+_intersect} is bijective.
\end{proof}

Note that we have a $(\varphi, \GammaR \times \Delta)\textrm{-equivariant}$ isomorphism of rings $A_{R,\varpi}^+[\zeta_m, [x_1^{\flat}]^{1/m}, \ldots, [x_d^{\flat}]^{1/m}] \isomorphic \ARpi^+$.
In particular, $\ARpi^+$ is finite free as a module over $A_{R,\varpi}^+$ with a basis given by the elements $\{\zeta_m^{i_0} [x_1^{\flat}]^{i_1} \cdots [x_d^{\flat}]^{i_d}\}_{\mathfrak{i} \in I}$ for $I = \{\mathfrak{i} = (i_0, i_1, \ldots, i_d) \textrm{ such that } 0 \leqslant i_j \leqslant m-1 \textrm{ for all } j\}$.
So, we obtain the following:
\begin{lem}\label{lem:OARpiPD_descent}
	There exists a natural $(\varphi, \GammaR \times \Delta)\textrm{-equivariant}$ isomorphism of rings
	\begin{equation}\label{eq:OARpiPD_descent}
		\pazocal{O}A_{R,\varpi}^{\logPD}[[x_1^{\flat}]^{1/m}, \ldots, [x_d^{\flat}]^{1/m}] \isomorphic \OARpi^{\logPD}.
	\end{equation}
	In particular, we have a $(\varphi, \GammaR)\textrm{-equivariant}$ isomorphism of rings $\pazocal{O}A_{R,\varpi}^{\logPD} \isomorphic (\OARpi^{\logPD})^{\Delta}$.
\end{lem}

\subsection{Associating Wach modules to log-crystalline represetations}

The main goal of this section is to show the following:
\begin{thm}\label{prop:descent Wach modules}\label{thm:Delta_descent_wach_mods}
	Let $T$ be a finite free $\mathbb{Z}_p\textrm{-representation}$ of $G_R$ such that $T[1/p]$ is log-crystalline.
	Then, there exists a Wach module $N_R(T)$ over $A_R^+$, functorially associated to $T$.
\end{thm}

We begin by comparing $\ODlogcrys(T[1/p])$ and Wach modules associated to $T$.

\subsubsection{A comparison isomorphism for log-crystalline representations}\label{subsubsec:OA_comp_iso}

Let $T$ be a $\mathbb{Z}_p\textrm{-representation}$ of $G_R$ such that $V \coloneq T[1/p]$ is log-crystalline in the sense of Definition \ref{defi:logcrys_rep}.
Let $\ODlogcrys(V)$ denote the associated filtered $\varphi\textrm{-module}$ over $R[1/p]$ equipped with a quasi-nilpotent integrable log-connection satisfying Griffiths transversality with respect to the filtration from Section \ref{subsec:logcrys_reps}.
From Corollary \ref{cor:crys_base_change} we know that $V$ is a crystalline representation of $G_S$, and we have a natural isomorphism of $S[1/p]\textrm{-modules}$ (see \eqref{eq:crys_base_change})
\begin{equation*}
	S[1/p] \otimes_{R[1/p]} \ODlogcrys(V) \isomorphic \ODcrysS(V),
\end{equation*}
compatible with the respective Frobenii, filtrations and connections, where the left-hand-term is equipped with the tensor-product Frobenius and connection, and the filtration on $S[1/p]$ is trivial.

Let $m \geqslant 1$ be an integer coprime to $p$ such that the action of $G_R$ on $T$ factors through $G_{\calR/R}^{\etale}$ for $\calR \coloneq R_{0,m}$ (see Theorem \ref{thm:GR_rep_GR0m}).
Then, in the notation of Section \ref{sec:wachmods}, we have the associated \'etale $(\varphi, \GammaR \times \Delta)\textrm{-module}$ $\DR(T)$ over $\AR$.
Additionally, as $V = T[1/p]$ is a crystalline representation of $G_S$, we have the Wach module $N_S(T)$ over $A_S^+$ associated to $T$.

Set $\NS(T) \coloneq \AS^+ \otimes_{A_S^+} N_S(T)$, and note that it is a Wach module with $\Delta\textrm{-action}$ over $\AS^+$ associated to $T$.
Moreover, by the uniqueness of the \'etale $(\varphi, \GammaS \times \Delta)\textrm{-module}$ over $\AS$ associated to $T$, it follows that $\DS(T) \coloneq \AS \otimes_{\AS^+} \NS(T) \isomorphic \AS \otimes_{\AR} \DR(T)$ as \'etale $(\varphi, \GammaS \times \Delta)\textrm{-modules}$ over $\AS$.
Here we have used the identification $\GammaR \times \Delta \isomorphic \GammaS \times \Delta$.
Consequently, from Theorem \ref{thm:wachmod_existence}, we obtain that 
\begin{equation*}
	\NR(T) \coloneq \NS(T) \cap \DR(T) \subset \DS(T)
\end{equation*}
is the Wach module with $\Delta\textrm{-action}$ over $\AR^+$ associated to $T$.
Set $\NR(V) \coloneq \NR(T)[1/p]$ as a module over $\BR^+ = \AR^+[1/p]$ equipped with the induced action of $\varphi$ and $\GammaR \times \Delta$.
Similarly, set $N_S(V) \coloneq N_S(T)[1/p]$ (resp.\ $\NS(V) \coloneq \NS(T)[1/p]$) as a module over $B_S^+ = A_S^+[1/p]$ (resp.\ $\BS^+ = \AS^+[1/p]$) equipped with the induced action of $\varphi$ and $\GammaS \times \Delta$.

The goal of this section is to prove the following claim:
\begin{prop}\label{prop:OA_comp_iso}
	Let $T$ be a $\mathbb{Z}_p\textrm{-representation}$ of $G_R$ such that $V \coloneq T[1/p]$ is log-crystalline, and the associated Wach module $\NS(T)$ over $\AS^+$ is effective.
	Let $\ODR \coloneq (\OARpi^{\logPD} \otimes_{\AR^+} \NR(V))^{\GammaR}$ as a $(\varphi, \Delta)\textrm{-module}$ over $\calR[1/p]$.
	Then, we have a natural $(\varphi, \Delta)\textrm{-equivariant}$ isomorphism of $\calR[1/p]\textrm{-modules}$
	\begin{equation}\label{eq:Dlogcrys_ODR_comp}
		\calR[1/p] \otimes_{R[1/p]} \ODlogcrys(V) \isomorphic \ODR.
	\end{equation}
	Additionally, we have the following $\OARpi^{\logPD}\textrm{-linear}$ and $(\varphi, \GammaR \times \Delta)\textrm{-equivariant}$ natural isomorphisms 
	\begin{equation}\label{eq:OA_comp_iso}
		\OARpi^{\logPD} \otimes_R \ODlogcrys(V) \isomorphic \OARpi^{\logPD} \otimes_{\calR} \ODR \isomorphic \OARpi^{\logPD} \otimes_{\AR^+} \NR(V).
	\end{equation}
\end{prop}
\begin{proof}
	Let us first consider the following $(\varphi, G_R)\textrm{-equivariant}$ and $\OBlogcrys(\overline{R})\textrm{-linear}$ isomorphism
	\begin{equation*}
		\OBlogcrys(\overline{R}) \otimes_{R[1/p]} \ODlogcrys(V) \isomorphic \OBlogcrys(\overline{R}) \otimes_{\mathbb{Q}_p} V \lisomorphic \OBlogcrys(\overline{R}) \otimes_{\BR^+} \NR(V),
	\end{equation*}
	where the first isomorphism holds because $V$ is a log-crystalline representation of $G_R$ (see \eqref{eq:crys_admis} in Definition \ref{defi:logcrys_rep}), and the second isomorphism is the extension of scalars along the $(\varphi, G_R)\textrm{-equivariant}$ map $A_{\inf}(\overline{R})[1/\mu] \hookrightarrow \OBlogcrys(\overline{R})$ of the isomorphism \eqref{eq:wachmod_comp_relative_ainf} from Proposition \ref{prop:wachmod_comp_relative}.
	By taking the $H_{R,\infty}\textrm{-invariants}$ of the composition above, and using that $\OBlogcrys(R_{\infty,\infty}) = \OBlogcrys(\overline{R})^{H_{R,\infty}}$ (see Section \ref{subsubsec:Bcris}), we obtain the following $(\varphi, \Gamma_{R,\infty})\textrm{-equivariant}$ and $\OBlogcrys(R_{\infty,\infty})\textrm{-linear}$ isomorphism:
	\begin{equation}\label{eq:Dlogcrys_NR_comp}
		\OBlogcrys(R_{\infty,\infty}) \otimes_{R[1/p]} \ODlogcrys(V) \isomorphic \OBlogcrys(R_{\infty,\infty}) \otimes_{\BR^+} \NR(V).
	\end{equation}

	Similarly, we have the following $(\varphi, G_S)\textrm{-equivariant}$ and $\OBcrys(\overline{S})\textrm{-linear}$ isomorphism:
	\begin{equation*}
		\begin{aligned}
			\OBcrys(\overline{S}) \otimes_{R[1/p]} \ODlogcrys(V) &\isomorphic \OBcrys(\overline{S}) \otimes_{S[1/p]} \ODcrysS(V)\\
				&\isomorphic \OBcrys(\overline{S}) \otimes_{\mathbb{Q}_p} V\\
				&\lisomorphic \OBcrys(\overline{S}) \otimes_{B_S^+} N_S(V)\\
				&\isomorphic \OBcrys(\overline{S}) \otimes_{\BS^+} \NS(V) \lisomorphic \OBcrys(\overline{S}) \otimes_{\BR^+} \NR(V),
		\end{aligned}
	\end{equation*}
	where the first isomorphism is the extension of scalars, along the $(\varphi, G_S)\textrm{-equivariant}$ map $S[1/p] \hookrightarrow \OBcrys(\overline{S})$, of the isomorphism \eqref{eq:crys_base_change} from Corollary \ref{cor:crys_base_change}; the second isomorphism holds because $V$ is a crystalline representation of $G_S$ (see \cite[Section 8.2]{brinon-relatif}); the third isomorphism is the extension of scalars, along the $(\varphi, G_S)\textrm{-equivariant}$ map $A_{\inf}(\overline{S})[1/\mu] \hookrightarrow \OBcrys(\overline{S})$, of the isomorphism in \cite[Proposition 3.14]{abhinandan-relative-wach-ii}; the fourth isomorphism holds because we have $\NS(V) = \BS^+ \otimes_{B_S^+} N_S(V)$; the last isomorphism is the extension of scalars, along the $(\varphi, G_S)\textrm{-equivariant}$ map $\AS^+ \hookrightarrow \OBcrys(\overline{S})$, of the $(\varphi, \GammaS \times \Delta)\textrm{-equivariant}$ and $\AS^+\textrm{-linear}$ isomorphism $\AS^+ \otimes_{\AR^+} \NR(T) \isomorphic \NS(T)$ in Corollary \ref{cor:wachmod_base_change}.
	By taking the $H_{S,\infty}\textrm{-invariants}$ of the composition above, and using that $\OBcrys(S_{\infty,\infty}) = \OBcrys(\overline{S})^{H_{S,\infty}}$ (see Section \ref{subsubsec:Bcris}), we obtain the following $(\varphi, \Gamma_{S,\infty})\textrm{-equivariant}$ and $\OBcrys(S_{\infty,\infty})\textrm{-linear}$ isomorphism:
	\begin{equation}\label{eq:DcrysS_NS_comp}
		\OBcrys(S_{\infty,\infty}) \otimes_{R[1/p]} \ODlogcrys(V) \isomorphic \OBcrys(S_{\infty,\infty}) \otimes_{\BR^+} \NR(V).
	\end{equation}

	Next, let us recall that from \cite[Theorem 3.35 and Proposition 3.36]{abhinandan-relative-wach-ii}, we have the following $(\varphi, \Gamma_R)\textrm{-equivariant}$ and $\pazocal{O}A_{S,\varpi}^{\PD}\textrm{-linear}$ isomorphism:
	\begin{equation*}
		\pazocal{O}A_{S,\varpi}^{\PD} \otimes_S \ODcrysS(V) \isomorphic \pazocal{O}A_{S,\varpi}^{\PD} \otimes_{A_S^+} N_S(V).
	\end{equation*}
	Extending scalars of the preceding isomorphism along the $(\varphi, \GammaR \times \Delta)\textrm{-equivariant}$ map $\pazocal{O}A_{S,\varpi}^{\PD} \hookrightarrow \OASpi^{\PD}$, we obtain the second isomorphism in the following:
	\begin{equation}\label{eq:Dlogcrys_NR_Scomp}
		\begin{aligned}
			\OASpi^{\PD} \otimes_R \ODlogcrys(V) &\isomorphic \OASpi^{\PD} \otimes_S \ODcrysS(V)\\
				&\isomorphic \OASpi^{\PD} \otimes_{A_S^+} N_S(V)\\
				&\isomorphic \OASpi^{\PD} \otimes_{\AR^+} \NR(V),
		\end{aligned}
	\end{equation}
	and where the first isomorphism follows from the isomorphism \eqref{eq:crys_base_change} from Corollary \ref{cor:crys_base_change}, and the last isomorphism holds because because we have a $(\varphi, \GammaS \times \Delta)\textrm{-equivariant}$ and $\BS^+\textrm{-linear}$ isomorphism $\BS^+ \otimes_{B_S^+} N_S(V) = \NS(V) \lisomorphic \BS^+ \otimes_{\BR^+} \NR(V)$ (see Corollary \ref{cor:wachmod_base_change}).
	It may easily be checked that \eqref{eq:DcrysS_NS_comp} is the base change of \eqref{eq:Dlogcrys_NR_Scomp} along the $(\varphi, \Gamma_{R,\infty})\textrm{-equivariant}$ map $\OASpi^{\PD} \hookrightarrow \OBcrys(S_{\infty,\infty})$ (see the discussion before Lemma \ref{lem:OASpiPD_OAlogcrys_intersect}).

	Now, consider the following diagram, where the vertical arrows are natural $(\varphi, \Gamma_{R,\infty})\textrm{-equivariant}$ maps:
	\begin{equation}\label{eq:Dlogcrys_NR_comps}
		\begin{tikzcd}
			\OBlogcrys(R_{\infty,\infty}) \otimes_{R[1/p]} \ODlogcrys(V) & \OBlogcrys(R_{\infty,\infty}) \otimes_{\BR^+} \NR(V)\\
			\OBcrys(S_{\infty,\infty}) \otimes_{R[1/p]} \ODlogcrys(V) & \OBcrys(S_{\infty,\infty}) \otimes_{\BR^+} \NR(V)\\
			\OASpi^{\PD} \otimes_R \ODlogcrys(V) & \OASpi^{\PD} \otimes_{\AR^+} \NR(V).
			\arrow["\sim", "\eqref{eq:Dlogcrys_NR_comp}"', from=1-1, to=1-2]
			\arrow[hook, from=1-1, to=2-1]
			\arrow[hook, from=1-2, to=2-2]
			\arrow["\sim", "\eqref{eq:DcrysS_NS_comp}"', from=2-1, to=2-2]
			\arrow[hook', from=3-1, to=2-1]
			\arrow["\sim", "\eqref{eq:Dlogcrys_NR_Scomp}"', from=3-1, to=3-2]
			\arrow[hook', from=3-2, to=2-2]
		\end{tikzcd}
	\end{equation}
	The injectivity of the top vertical arrows follow from Lemma \ref{lem:Acrys_Rinftym_in_Sinftym}, and the fact that $\ODlogcrys(V)$ is a finite projective $R[1/p]\textrm{-module}$ (see after Definition \ref{defi:logcrys_rep}) and $\NR(V)$ is a finite projective $\BR^+\textrm{-module}$ (see Proposition \ref{prop:wachmod_proj_pmu}).
	Similarly, the injectivity of the bottom vertical arrows follow from the discussion before Lemma \ref{lem:OASpiPD_OAlogcrys_intersect}, and the fact that $\ODlogcrys(V)$ is a finite projective $R[1/p]\textrm{-module}$ and $\NR(V)$ is a finite projective $\BR^+\textrm{-module}$.
	From the discussion preceding the diagram, it is easy to see that the top and the bottom squares commute because the constructions of this paper are compatible with that of \cite{brinon-relatif} and \cite{abhinandan-relative-wach-ii}.

	In the commutative diagram \eqref{eq:Dlogcrys_NR_comps}, taking the intersection of the top row with the bottom row inside the middle row yields the following $\OARpi^{\logPD}\textrm{-linear}$ and $(\varphi, \GammaR \times \Delta)\textrm{-equivariant}$ natural isomorphism
	\begin{equation}\label{eq:Dlogcrys_NR_Rcomp}
		\OARpi^{\logPD} \otimes_R \ODlogcrys(V) \isomorphic \OARpi^{\logPD} \otimes_{\AR^+} \NR(V),
	\end{equation}
	where we used Lemma \ref{lem:OASpiPD_OAlogcrys_intersect} to obtain that $\OASpi^{\PD} \cap \OBcrys(R_{\infty,\infty}) = \OASpi^{\PD} \cap \OAcrys(R_{\infty,\infty}) = \OARpi^{\logPD}$ inside $\OBcrys(S_{\infty,\infty})$.
	Taking the $\GammaR\textrm{-invariant}$ elements in the isomorphism \eqref{eq:Dlogcrys_NR_Rcomp}, we get the natural $(\varphi, \Delta)\textrm{-equivariant}$ and $\calR[1/p]\textrm{-linear}$ isomorphism claimed in \eqref{eq:Dlogcrys_ODR_comp}.
	Moreover, by extending scalars along the $(\varphi, \GammaR \times \Delta)\textrm{-equivariant}$ map $\calR \hookrightarrow \OARpi^{\PD}$, of the isomorphism in \eqref{eq:Dlogcrys_ODR_comp}, and using \eqref{eq:Dlogcrys_NR_Rcomp}, we obtain the second isomorphism in \eqref{eq:OA_comp_iso}.
\end{proof}

\begin{cor}\label{cor:NR_modmu}
	There exists natural $(\varphi, \Delta)\textrm{-equivariant}$ isomorphisms of $\calR[1/p]\textrm{-modules}$:
	\begin{equation}\label{eq:NR_modmu}
		\calR[1/p] \otimes_{R[1/p]} \ODlogcrys(V) \isomorphic \ODR \isomorphic \NR(V)/\mu\NR(V).
	\end{equation}
\end{cor}
\begin{proof}
	Let us first note that we have $(\Fil^1 \OARpi^{\logPD} \otimes_{\AR^+} \NR(V)) \cap \NR(V) = (\Fil^1 \OARpi^{\logPD} \cap \AR^+) \otimes_{\AR^+} \NR(V) = \mu \NR(V)$, where the first equality follows because $\NR(V)$ is finite projective over $\BR^+$, in particular flat over $\AR^+$, and the second equality follows from Remark \ref{rem:fil1ar+_intersect}.
	So, let us consider the following diagram with exact rows:
	\begin{equation*}
		\begin{tikzcd}
			0 & \mu\NR(V) & \NR(V) & \NR(V)/\mu\NR(V) & 0 \\
			0 & (\Fil^1 \OARpi^{\logPD}) \otimes_{\AR^+} \NR(V) & \OARpi^{\PD} \otimes_{\AR^+} \NR(V) & \calR[\varpi] \otimes_{\calR} (\NR(V)/\mu \NR(V)) & 0 \\
			0 & (\Fil^1 \OARpi^{\logPD}) \otimes_{\calR} \ODR & \OARpi^{\PD} \otimes_{\calR} \ODR & \calR[\varpi] \otimes_{\calR} \ODR & 0 
			\arrow[from=1-1, to=1-2]
			\arrow[from=1-2, to=1-3]
			\arrow[from=1-2, to=2-2]
			\arrow[from=1-3, to=1-4]
			\arrow[from=1-3, to=2-3]
			\arrow[from=1-4, to=1-5]
			\arrow[from=1-4, to=2-4]
			\arrow[from=2-1, to=2-2]
			\arrow[from=2-2, to=2-3]
			\arrow[from=2-3, to=2-4]
			\arrow[from=2-4, to=2-5]
			\arrow[from=3-1, to=3-2]
			\arrow["\wr"', from=3-2, to=2-2]
			\arrow[from=3-2, to=3-3]
			\arrow["\wr"', "\eqref{eq:OA_comp_iso}", from=3-3, to=2-3]
			\arrow[from=3-3, to=3-4]
			\arrow["\wr"', from=3-4, to=2-4]
			\arrow[from=3-4, to=3-5]
		\end{tikzcd}
	\end{equation*}
	where from the exactness of the second row and the discussion above, it follows that the vertical maps from the first to the second row are natural inclusions.
	Moreover, the middle vertical arrow from the third to the second row is the isomorphism \eqref{eq:OA_comp_iso} in Proposition \ref{prop:OA_comp_iso}, and the left vertical arrow is the tensor product of the $\OARpi^{\logPD}\textrm{-linear}$ isomorphism in the middle vertical arrow with the $\OARpi^{\logPD}\textrm{-module}$ $\Fil^1 \OARpi^{\logPD}$, in particular, the left vertical arrow is bijective as well.
	So, we conclude that the right vertical arrow is also an isomorphism.
	Taking the $\Gal(\calR[\varpi][1/p]/\calR[1/p]) = \Gal(F(\zeta_{p^i})/F)\textrm{-invariants}$ (recall that $i = 1$, for $p \geqslant 3$, and $i = 2$, for $p = 2$) of the right vertical arrows gives a natural $\calR[1/p]\textrm{-linear}$ and $(\varphi, \Delta)\textrm{-equivariant}$ isomorphism
	\begin{equation*}
		\ODR \isomorphic \NR(V)/\mu\NR(V).
	\end{equation*}
	Combining the preceding isomorphism with the isomorphism in \eqref{eq:Dlogcrys_ODR_comp} yields the claimed isomorphisms in \eqref{eq:NR_modmu}.
\end{proof}

\begin{cor}\label{cor:OA_comp_iso_Delta}
	Let $\pazocal{O}D_R \coloneq (\OARpi^{\logPD} \otimes_{\AR^+} \NR(V))^{\GammaR \times \Delta}$ as a $\varphi\textrm{-module}$ over $R[1/p]$.
	Then, we have a natural $\varphi\textrm{-equivariant}$ isomorphism of $R[1/p]\textrm{-modules}$
	\begin{equation}\label{eq:Dlogcrys_ODR_comp_Delta}
		\ODlogcrys(V) \isomorphic \pazocal{O}D_R.
	\end{equation}
	Additionally, let $N_R(V) \coloneq \NR(V)^{\Delta}$ as a $(\varphi, \Gamma_R)\textrm{-module}$ over $B_R^+$.
	Then, we have the following $\pazocal{O}A_{R,\varpi}^{\logPD}\textrm{-linear}$ and $(\varphi, \GammaR)\textrm{-equivariant}$ natural isomorphisms 
	\begin{equation}\label{eq:OA_comp_iso_Delta}
		\pazocal{O}A_{R,\varpi}^{\logPD} \otimes_R \ODlogcrys(V) \isomorphic \pazocal{O}A_{R,\varpi}^{\logPD} \otimes_R \pazocal{O}D_R \isomorphic \pazocal{O}A_{R,\varpi}^{\logPD} \otimes_{A_R^+} N_R(V).
	\end{equation}
	Furthermore, there exists a natural $\varphi\textrm{-equivariant}$ isomorphism of $R[1/p]\textrm{-modules}$:
	\begin{equation}\label{eq:NR_modmu_Delta}
		\ODlogcrys(V) \isomorphic \pazocal{O}D_R \isomorphic (N_R(V)/\mu N_R(V)).
	\end{equation}
\end{cor}
\begin{proof}
	The isomorphism in \eqref{eq:Dlogcrys_ODR_comp_Delta} follows by taking the $\Delta\textrm{-invariants}$ of the isomorphism \eqref{eq:Dlogcrys_ODR_comp} from Proposition \ref{prop:OA_comp_iso}.
	The isomorphism in \eqref{eq:OA_comp_iso_Delta} follows by taking the $\Delta\textrm{-invariants}$ of \eqref{eq:OA_comp_iso_Delta} from Proposition \ref{prop:OA_comp_iso} and using the $(\varphi, \GammaR \times \Delta)\textrm{-equivariant}$ isomorphism \eqref{eq:OARpiPD_descent} from Lemma \ref{lem:OARpiPD_descent}.
	The isomorphism in \eqref{eq:NR_modmu_Delta} follows by taking the $\Delta\textrm{-invariants}$ of the isomorphism \eqref{eq:NR_modmu} from Corollary \ref{cor:NR_modmu}.
\end{proof}

\subsubsection{Descent of Wach modules for log-crystalline representations}

Our goal in this section is to show that for log-crystalline representations of $G_R$ Wach modules descend over to the ring $A_R^+$.

Let $T$ be a finite free $\mathbb{Z}_p\textrm{-representation}$ of $G_R$ such that $T[1/p]$ is log-crystalline.
For a fixed integer $r$ observe that $V = T[1/p]$ is log-crystalline if and only if $V(-r)$ is log-crystalline.
Additionally, for large enough integers $r$, we have that $N_S(T(-r)) = \mu^r N_S(T)(-r)$ is effective.
Therefore, in the following we shall assume that $N_S(T)$ is effective.

Let $m \geqslant 1$ be an integer coprime to $p$ such that the action of $G_R$ on $T$ factors through $G_{\calR/R}^{\etale}$ for $\calR \coloneq R_{0,m}$ (see Theorem \ref{thm:GR_rep_GR0m}).
Then, from the discussion in Section \ref{subsubsec:OA_comp_iso}, we have an associated Wach module with $\Delta\textrm{-action}$ $\NR \coloneq \NR(T)$ over $\AR^+$, where $\Delta \coloneq \Delta_{\calR/R}$.
As indicated above, without loss of generality, we shall assume that $\NR$ is effective.

Let us set $N_R \coloneq N_R(T) \coloneq \NR(T)^{\Delta}$ as an $A_R^+\textrm{-module}$, equipped with the induced action of $(\varphi, \Gamma_R)$.
We begin with an observation.
\begin{lem}\label{lem:ratWachmod_proj_Delta}
	The $B_R^+\textrm{-module}$ $N_R[1/p] = \NR[1/p]^{\Delta}$ is finite projective.
\end{lem}
\begin{proof}
	From Proposition \ref{prop:wachmod_proj_pmu}, recall that $\NR[1/p]$ is a finite projective $\BR^+\textrm{-module}$.
	Moreover, $\AR^+$ is finite free as a module over $A_R^+$.
	So, it follows that $\NR[1/p]$ is finite projective as a module over $B_R^+$.
	Additionally, from Lemma \ref{lem:Gfixed_trace}, we know that $N_R[1/p]$ is a direct summand of $\NR[1/p]$ as a $B_R^+\textrm{-module}$.
	Therefore, we conclude that $N_R[1/p]$ is a finite projective module over $B_R^+$.
\end{proof}

We are now ready to prove Theorem \ref{thm:Delta_descent_wach_mods}.
\begin{proof}[Proof of Theorem \ref{thm:Delta_descent_wach_mods}]
	To show that $N_R$ is a Wach module over $A_R^+$, let us first note that $A_R^+ = (\AR^+)^{\Delta}$, and $\AR^+$ is finite as a module over $A_R^+$.
	As $\NR$ is a finitely generated $\AR^+\textrm{-module}$, we get that it is also finitely generated over $A_R^+$.
	Moreover, we have that $N_R \subset \NR$ and $A_R^+$ is noetherian, therefore, $N_R$ is also finitely generated over $A_R^+$, and hence $(p, \mu)\textrm{-adically}$ complete.
	It is clear that $N_R$ is $p\textrm{-torsion}$ free and $\mu\textrm{-torsion}$ free.
	Therefore, from Lemma \ref{lem:descent_Delta}, we obtain that $N_R/p^nN_R \isomorphic (\NR/p^n\NR)^{\Delta}$ for each $n \geqslant 1$.
	In particular, we have that $N_R/pN_R$ is $\mu\textrm{-torsion}$ free, and thus the sequences $\{p, \mu\}$ and $\{\mu, p\}$ are regular on $N_R$.
	Similarly, using Lemma \ref{lem:descent_Delta}, it also follows that we have a $\varphi\textrm{-equivariant}$ isomorphism of $R\textrm{-modules}$ $N_R/\mu N_R \isomorphic (\NR/\mu\NR)^{\Delta}$, in particular, the action of $\Gamma_R \isomorphic \GammaR$ is trivial modulo $\mu$.

	Next, we will show that the following natural $(\varphi, \GammaR \times \Delta)\textrm{-equivariant}$ homomorphism of $\AR^+\textrm{-modules}$ is bijective:
	\begin{equation}\label{eq:fR+}
		\begin{aligned}
			f_{\calR}^+ \colon \AR^+ \otimes_{A_R^+} N_R &\longrightarrow \NR\\
				a \otimes n &\longmapsto an.
		\end{aligned}
	\end{equation}
	Let $Q$ denote the cokernel of $f_{\calR}^+$, and note that it is a finitely generated $\AR^+\textrm{-module}$.
	From Section \ref{subsec:constructing_wach} recall that we have a natural $(\varphi, \GammaR \times \Delta)\textrm{-equivariant}$ and $\AR^+/p\AR^+\textrm{-linear}$ inclusion $\NR^+/p\NR^+ \hookrightarrow \NS^+/p\NS^+$.
	In particular, we observe that $[x_i^{\flat}]$ is regular on $\NR/p\NR$.
	Then, from Proposition \ref{prop:descent_Delta} it follows that if the cokernel of $f_{\calR}^+$, i.e.\ $Q$ is nonzero, then it must be $p\textrm{-torsion}$ free.
	So, to show that $Q = 0$ it is enough to show that $Q$ is killed by some power of $p$.
	
	To that end, using that $\GammaR \isomorphic \GammaS$, let us consider the base change of the map $f_{\calR}^+$ along the $(\varphi, \GammaS \times \Delta)\textrm{-equivariant}$ flat homomorphism of rings $\AR^+ \rightarrow \AS$, and note that it yields the following $(\varphi, \GammaS \times \Delta)\textrm{-equivariant}$ injective homomorphism of $\AS^+\textrm{-modules}$:
	\begin{equation}\label{eq:fS+}
		f_{\calS}^+ \colon \AS^+ \otimes_{A_R^+} N_R \longrightarrow \AS^+ \otimes_{\AR^+} \NR \isomorphic \NS \coloneq \NS(T),
	\end{equation}
	where we used Corollary \ref{cor:wachmod_base_change} for the isomorphism.
	Then, from Lemma \ref{lem:NS_descent} below, we have that $f_{\calS}^+$ is bijective.

	Now, let us consider the following diagram:
	\begin{equation}\label{eq:fR+_basechange}
		\begin{tikzcd}[column sep=5mm]
			& \OARpi^{\logPD} \otimes_{\AR^+} \NR[1/p] & \OASpi^{\PD} \otimes_{\AR^+} \NR[1/p] & \AS^+ \otimes_{\AR^+} \NR[1/p] \\
			\OARpi^{^\logPD} \otimes_R \ODlogcrys(V) \\
			& \OARpi^{\logPD} \otimes_{A_R^+} N_R[1/p] & \OASpi^{\PD} \otimes_{A_R^+} N_R[1/p] & \AS^+ \otimes_{A_R^+} N_R[1/p],
			\arrow[hook, from=1-2, to=1-3]
			\arrow[hook', from=1-4, to=1-3]
			\arrow["\eqref{eq:OA_comp_iso}"', "\sim", from=2-1, to=1-2]
			\arrow["\eqref{eq:OA_comp_iso_Delta}", "\sim"', from=2-1, to=3-2]
			\arrow["\eqref{eq:fR+}"', "\wr", from=3-2, to=1-2]
			\arrow[hook, from=3-2, to=3-3]
			\arrow["\wr", from=3-3, to=1-3]
			\arrow["\eqref{eq:fS+}"', "\wr", from=3-4, to=1-4]
			\arrow[hook', from=3-4, to=3-3]
		\end{tikzcd}
	\end{equation}
	where the top (resp.\ bottom) left horizontal arrow is obtained by tensoring the natural $(\varphi, \GammaR \times \Delta)\textrm{-equivariant}$ inclusion $\OARpi^{\logPD} \hookrightarrow \OASpi^{\PD}$ (see before Lemma \ref{lem:OASpiPD_OAlogcrys_intersect}) with the finite projective $\BR^+\textrm{-module}$ $\NR[1/p]$, see Proposition \ref{prop:wachmod_proj_pmu} (resp.\ $B_R^+\textrm{-module}$ $N_R[1/p]$, see Lemma \ref{lem:ratWachmod_proj_Delta}), and the top (resp.\ bottom) right horizontal arrow is obtained by tensoring the natural $(\varphi, \GammaR \times \Delta)\textrm{-equivariant}$ inclusion $\AS^+ \hookrightarrow \OASpi^{\PD}$ with the finite projective $\BR^+\textrm{-module}$ $\NR[1/p]$ (resp.\ $B_R^+\textrm{-module}$ $N_R[1/p]$).
	In diagram \eqref{eq:fR+_basechange}, the top (resp.\ bottom) slanted arrow is the isomorphism \eqref{eq:OA_comp_iso} (resp.\ \eqref{eq:OA_comp_iso_Delta} extended along the $(\varphi, \GammaR \times \Delta)\textrm{-equivariant}$ homomorphism $\pazocal{O}A_{R,\varpi}^{\logPD} \hookrightarrow \OARpi^{\logPD}$) from Proposition \ref{prop:OA_comp_iso} (resp.\ Corollary \ref{cor:OA_comp_iso_Delta}).
	The left vertical arrow is the scalar extension of $f_{\calR}^+$ in \eqref{eq:fR+} along the $(\varphi, \GammaR \times \Delta)\textrm{-equivariant}$ homomorphism of rings $\AR^+ \rightarrow \OARpi^{\logPD}$, and the left triangle commutes by definition which implies that the left vertical arrow is bijective.
	The middle vertical arrow is the base change of the left vertical arrow along the $(\varphi, \GammaR \times \Delta)\textrm{-equivariant}$ homomorphism of rings $\OARpi^{\logPD} \rightarrow \OASpi^{\PD}$, and the right square commutes by definition.

	In diagram \eqref{eq:fR+_basechange}, taking the intersection of the left vertical isomorphism with the right vertical isomorphism inside the middle vertical isomorphism, we obtain the following isomorphism of $\BR^+\textrm{-modules}$:
	\begin{equation*}
		f_{\calR}^+[1/p] \colon \AR^+ \otimes_{A_R^+} N_R[1/p] \isomorphic \NR[1/p],
	\end{equation*}
	where we have used that inside $\OASpi^{\PD}$ we have a $(\varphi, \GammaR \times \Delta)\textrm{-equivariant}$ identification of rings $\OARpi^{\logPD} \cap \AS^+ \isomorphic \AR^+$ (see Lemma \ref{lem:OARpilogPD_AS+_intersect}).
	In particular, it follows that the cokernel $Q$ of the map $f_{\calR}^+$ from \eqref{eq:fR+} is killed by a power of $p$.
	Hence, we conclude that $Q$ must be zero, i.e.\ \eqref{eq:fR+} is bijective.

	Finally, let us show that $N_R$ admits a Frobenius structure.
	Recall that the Frobenius structure on $\NR$ is given as a $(\GammaR \times \Delta)\textrm{-equivariant}$ and $\AR^+\textrm{-linear}$ isomorphism $\varphi_N \colon (\varphi^*\NR)[1/[p]_q] \isomorphic \NR[1/[p]_q]]$.
	Using the isomorphism in \eqref{eq:fR+}, noting that the natural homomorphism $A_R^+ \rightarrow \AR^+$ is Frobenius-equivariant, and using the Frobenius structure on $\NR$, we obtain the following $(\GammaR \times \Delta)\textrm{-equivariant}$ and $\AR^+\textrm{-linear}$ isomorphisms:
	\begin{equation*}
		(\AR^+ \otimes_{A_R^+} \varphi^* N_R)[1/[p]_q] \isomorphic (\varphi^*\NR)[1/[p]_q] \xrightarrow[\varphi_N]{\sim} \NR[1/[p]_q] \isomorphic \AR^+ \otimes_{A_R^+} N_R[1/[p]_q],
	\end{equation*}
	where $\varphi^* N_R = A_R^+ \otimes_{\varphi, A_R^+} N_R$.
	Taking the $\Delta\textrm{-invariants}$ in the isomorphism above and using Lemma \ref{lem:tensorprod_fixedpoints}, we obtain a $\GammaR\textrm{-equivariant}$ and $A_R^+\textrm{-linear}$ isomorphism:
	\begin{equation*}
		\varphi^* N_R[1/[p]_q] \isomorphic N_R[1/[p]_q].
	\end{equation*}
	Hence, we conclude that $N_R$ is a Wach module over $A_R^+$.
\end{proof}

The following observation was used above:
\begin{lem}\label{lem:NS_descent}
	Let $N_S \coloneq N_S(T)$ be the Wach module over $A_S^+$ associated to $T$.
	Then, we have a natural $(\varphi, \Gamma_R)\textrm{-equivariant}$ isomorphism of $A_S^+\textrm{-modules}$ $A_S^+ \otimes_{A_R^+} N_R \isomorphic N_S$.
	Consequently, the map $f_{\calS}^+$ from \eqref{eq:fS+} is bijective.
\end{lem}
\begin{proof}
	Using that $\NS = \AS^+ \otimes_{A_S^+} N_S$, it is clear that the second claim easily follows from the first claim.
	For the first claim, note that we have a natural $(\varphi, \Gamma_R)\textrm{-equivariant}$ identification of $\AR^+\textrm{-modules}$ $N_R = \NR^{\Delta} = \NR \cap \NS^{\Delta} = \NR \cap N_S$, where the intersection takes place inside $\NS$.
	Therefore, the $A_S^+\textrm{-linearisation}$ of the $(\varphi, \Gamma_R)\textrm{-equivariant}$ and $A_R^+\textrm{-linear}$ injective map $N_R \hookrightarrow N_S$, yields a natural $(\varphi, \Gamma_R)\textrm{-equivariant}$ map 
	\begin{equation*}
		j \colon A_S^+ \otimes_{A_R^+} N_R \longrightarrow N_S,
	\end{equation*}
	compatible with the natural $(\varphi, \Gamma_R)\textrm{-equivariant}$ inclusion of the source and the target into $\NS$.
	As both the source and the target of the preceding map are $p\textrm{-adically}$ complete and $p\textrm{-torsion}$ free, it is enough to show that the modulo $p$ map, i.e.\
	\begin{equation*}
		j \textrm{ mod } p \colon A_S^+/pA_S^+ \otimes_{A_R^+/pA_R^+} N_R/pN_R \longrightarrow N_S/pN_S,
	\end{equation*}
	is bijective.

	From the disucssion above (see before the definition of the map $f$), recall that we have a natural $(\varphi, \Gamma_R)\textrm{-equivariant}$ identification $N_R/pN_R = (\NR/p\NR)^{\Delta}$.
	Moreover, we have a $(\varphi, \Gamma_R \times \Delta)\textrm{-equivariant}$ injective homomorphism
	\begin{equation*}
		\NR/p\NR \hookrightarrow \AS^+/p\AS^+ \otimes_{\AR^+/p\AR^+} \NR/p\NR \isomorphic \NS/p\NS,
	\end{equation*}
	where the isomorphism is reduction modulo $p$ of the natural isomorphism $\AS^+ \otimes_{\AR^+} \NR \isomorphic \NS$ from Corollary \ref{cor:wachmod_base_change}.
	Therefore, it follows that inside $\NS/p\NS$, we have natural $(\varphi, \Gamma_R)\textrm{-equivariant}$ identifications
	\begin{equation*}
		N_R/pN_R = (\NR/p\NR)^{\Delta} \isomorphic \NR/p\NR \cap (\NS/p\NS)^{\Delta} = \NR/p\NR \cap N_S/pN_S,
	\end{equation*}
	where we have used that $(\NS/p\NS)^{\Delta} = (\AS^+/p\AS^+ \otimes_{A_S^+/pA_S^+} N_S/pN_S)^{\Delta} = N_S/pN_S$.
	In particular, the map $j \textrm{ mod }p$ is injective.

	Now, consider the following $(\varphi, \Gamma_R \times \Delta)\textrm{-equivariant}$ commutative diagram
	\begin{equation*}
		\begin{tikzcd}[column sep=large]
			A_S^+/pA_S^+ \otimes_{A_R^+/pA_R^+} N_R/pN_R & N_S/pN_S \\
			A_S^+/pA_S^+ \otimes_{A_R^+/pA_R^+} \NR/p\NR & \NS/p\NS,
			\arrow[hookrightarrow, "j \textrm{ mod } p", from=1-1, to=1-2]
			\arrow[hookrightarrow, from=1-1, to=2-1]
			\arrow[hookrightarrow, from=1-2, to=2-2]
			\arrow["\sim", from=2-1, to=2-2]
		\end{tikzcd}
	\end{equation*}
	where the right vertical arrow is the natural inclusion, the left vertical arrow is the extension along the flat homomorphism $\AR^+ \rightarrow \AS^+$ of the natural inclusion $N_R/pN_R \hookrightarrow \NR/p\NR$, and the bottom horizontal arrow is reduction modulo $p$ of the natural isomorphism $A_S^+ \otimes_{A_R^+} \NR = \AS^+ \otimes_{\AR^+} \NR \isomorphic \NS$ from Corollary \ref{cor:wachmod_base_change}.
	From the diagram above, it follows that inside $\NS/p\NS$, we have natural $(\varphi, \Gamma_R)\textrm{-equivariant}$ identifications
	\begin{equation*}
		\begin{aligned}
			A_S^+/pA_S^+ \otimes_{A_R^+/pA_R^+} N_R/pN_R &\xrightarrow{\hspace{0.5mm}=\hspace{0.5mm}} A_S^+/pA_S^+ \otimes_{A_R^+/pA_R^+} (\NR/p\NR)^{\Delta}\\
				&\xrightarrow{\hspace{0.5mm}=\hspace{0.5mm}} (A_S^+/pA_S^+ \otimes_{A_R^+/pA_R^+} \NR/p\NR)^{\Delta}\\
				&\isomorphic (\NS/p\NS)^{\Delta} = N_S/pN_S,
		\end{aligned}
	\end{equation*}
	and the preceding composition coincides with $j \textrm{ mod } p$.
	Hence, the map $j$ is bijective.
\end{proof}

\subsection{A categorical equivalence}

The main goal of this section is to show the following:
\begin{thm}\label{thm:logcrys_wach_equiv}
	There exists a natural equivalence of categories:
	\begin{equation}\label{eq:logcrys_wach_equiv}
		\begin{aligned}
			\textup{Rep}_{\mathbb{Z}_p}^{\textup{log-cris}}(G_R) &\isomorphic (\varphi, \Gamma_R)\textup{-Mod}_{A_R^+}^{[p]_q}\\
				T &\longmapsto N_R(T),
		\end{aligned}
	\end{equation}
	with a quasi-inverse functor given as $N \mapsto T_R(N) \coloneq (\tilde{A}(\overline{R}) \otimes_{A_R^+} N)^{\varphi=1}$.
\end{thm}
\begin{proof}
	The functor in \eqref{eq:logcrys_wach_equiv} is well defined by Theorem \ref{thm:Delta_descent_wach_mods}.
	For the quasi-inverse functor, from Proposition \ref{prop:wach_is_logcrys} below it follows that for a Wach module $N$ over $A_R^+$, the associated $p\textrm{-adic}$ representation $T_R(N)[1/p]$ of $G_R$ is log-crystalline.
	Then, using the isomorphism \eqref{eq:wachmod_comp_relative_ainf} in Proposition \ref{prop:wachmod_comp_relative}, it is easy to see that we have $T_R \circ N_R \isomorphic id$ and $N_R \circ T_R \isomorphic id$.
\end{proof}

\begin{prop}\label{prop:wach_is_logcrys}
	Let $N$ be a Wach module over $A_R^+$, and set $T_R(N) \coloneq (\tilde{A}(\overline{R}) \otimes_{A_R^+} N)^{\varphi=1}$ as a finite free $\mathbb{Z}_p\textrm{-representation}$ of $G_R$.
	Then, $T_R(N)[1/p]$ is log-crystalline in the sense of Definition \ref{defi:logcrys_rep}.
\end{prop}
\begin{proof}
	The proof works by adapting the proof of crystalline toric setting from \cite[Theorem 3.35]{abhinandan-relative-wach-ii} to the log-crystalline non-toric setting.
	Let us briefly sketch the argument.

	Let us first note that for $r$ in $\mathbb{N}$ large enough, the Wach module $\mu^rN(-r)$ is always effective and we have that $T_R(\mu^rN(-r)) = T_R(N)(-r)$.
	Therefore, it is enough to show the claim for effective Wach modules, and we shall assume such in the rest of the proof.

	Let us set $\pazocal{O}D_R \coloneq (\OARpi^{\logPD} \otimes_{A_R^+} N[1/p])^{\Gamma_R}$ as an $R[1/p]\textrm{-module}$ equipped with an induced tensor-product Frobenius and a logarithmic connection induced from the logarithmic connection on $\OARpi^{\logPD}$.
	Then, we claim that the following natural homomomorphism
	\begin{equation}\label{eq:OARpilogPD_N_rat_iso}
		\OARpi^{\logPD} \otimes_R \pazocal{O}D_R \longrightarrow \OARpi^{\logPD} \otimes_{A_R^+} N[1/p],
	\end{equation}
	is an isomorphism compatible with respective Frobenii, logarithmic connections and actions of $\Gamma_R$.
	Indeed, to obtain the isomorphism above, one may proceed as in the proof of \cite[Proposition 3.36]{abhinandan-relative-wach-ii} by first showing that the analogue of \cite[Lemma 3.39]{abhinandan-relative-wach-ii} holds in our setting, where one replaces the ring $\pazocal{O}S_m^{\PD}$ of loc.\ cit.\ with the ring $\pazocal{O}A_k^{\logPD}$ (for $k = 1$ if $p \neq 2$, and $k = 2$ if $p = 2$) as follows:
	define $A_k^{\PD}$ to be the $p\textrm{-adic}$ completion of the ring $A_R^+[\tfrac{\mu}{p^k}, \tfrac{\mu^2}{2!p^{2k}}, \cdots, \tfrac{\mu^i}{i!p^{ik}}, \cdots]$ equipped with a Frobenius, a continuous action of $\Gamma_R$, and a log-structure induced from the log-structure on $A_R^+$, i.e.\ induced by the divisor $\{[x_{a+1}^{\flat}] \cdots [x_d^{\flat}] = 0\}$.
	Then, we set $\pazocal{O}A_k^{\logPD}$ to be the $p\textrm{-adic}$ completion of the logarithmic divided power envelope of the surjection $R \otimes_{O_F} A_k^{\logPD} \twoheadrightarrow A_k^{\logPD}$.

	Similar to \cite[Lemma 3.38]{abhinandan-relative-wach-ii}, after replacing the de Rham period ring in loc.\ cit.\ with the logarithmic de Rham period ring $\OBlogdR(\overline{R})$), we see that the $(\varphi, \Gamma_R)\textrm{-action}$ on $A_k^{\logPD}$ naturally extends to a $(\varphi, \Gamma_R)\textrm{-action}$ on $\pazocal{O}A_k^{\logPD}$ , and the latter is naturally equipped with an integrable logarithmic connection.
	Proceeding exactly as in the proof of \cite[Lemma 3.39]{abhinandan-relative-wach-ii} (after replacing the divided power rings of loc.\ cit.\ with their respective logarithmic divided power analogues), one may obtain the following natural $\Gamma_R\textrm{-equivariant}$ isomorphism
	\begin{equation*}
		\pazocal{O}A_k^{\logPD} \otimes_R D_k \isomorphic \pazocal{O}A_k^{\logPD} \otimes_{A_R^+} N[1/p],
	\end{equation*}
	where $D_k \coloneq (\pazocal{O}A_k^{\logPD} \otimes_{A_R^+} N[1/p])^{\Gamma_R}$ is equipped with an induced Frobenius structure and an integrable logarithmic connection.
	Note that the surjectivity of the preceding map is obtained by showing that, in our setting, the formula \cite[Equation (3.35)]{abhinandan-relative-wach-ii} converges in $D_k$.
	Using the preceding isomorphism and the Frobenius structure on $D_k$, and arguing as in the proof of \cite[Proposition 3.36]{abhinandan-relative-wach-ii}, one obtains that the homomorphism in \eqref{eq:OARpilogPD_N_rat_iso} is bijective.

	Now, let $V \coloneq T_R(N)[1/p]$ and consider the following diagram:
	\begin{equation}\label{eq:logcrys_admis_wach}
		\begin{tikzcd}
			\OBlogcrys(\overline{R}) \otimes_{R[1/p]} \pazocal{O}D_R & \OBlogcrys(\overline{R}) \otimes_{B_R^+} N[1/p] \\
			\OBlogcrys(\overline{R}) \otimes_{R[1/p]} \ODlogcrys(V) & \OBlogcrys(\overline{R}) \otimes_{\mathbb{Q}_p} V,
			\arrow["\sim", "\eqref{eq:OARpilogPD_N_rat_iso}"', from=1-1, to=1-2]
			\arrow[from=1-1, to=2-1]
			\arrow["\wr", "\eqref{eq:wachmod_comp_relative_ainf}"', from=1-2, to=2-2]
			\arrow["\eqref{eq:crys_admis}", from=2-1, to=2-2]
		\end{tikzcd}
	\end{equation}
	where the top horizontal arrow is the extension of scalars of \eqref{eq:OARpilogPD_N_rat_iso} along the natural $(\varphi, G_R)\textrm{-equivariant}$ injective homomorphism $\pazocal{O}A_k^{\logPD} \hookrightarrow \OBlogcrys(\overline{R})$; the left vertical arrow is induced by tensoring the preceding map with the finite projective $B_R^+\textrm{-module}$ $N[1/p]$ (see Proposition \ref{prop:wachmod_proj_pmu}), taking $G_R\textrm{-invariants}$ and extending scalars along the $(\varphi, G_R)\textrm{-equivariant}$ faithfully flat map $R[1/p] \rightarrow \OBlogcrys(\overline{R})$ (see Proposition \ref{prop:OBlogcrys_ff}); the right vertical arrow is the extension of scalars of $\eqref{eq:wachmod_comp_relative_ainf}$ along the $(\varphi, G_R)\textrm{-equivariant}$ homomorphism $A_{\inf}(\overline{R})[1/\mu] \rightarrow \OBlogcrys(\overline{R})$.
	An easy diagram chase shows that the diagram in \eqref{eq:logcrys_admis_wach} is commutative, and from the preceding discussion it is clear that \eqref{eq:logcrys_admis_wach} is compatible with the respective $(\varphi, G_R)\textrm{-actions}$ and logarithmic connections.
	Taking the invariants of the diagram under the $G_R\textrm{-action}$ yields a Frobenius-equivariat isomorphism of $R[1/p]\textrm{-modules}$
	\begin{equation*}
		\pazocal{O}D_R \isomorphic \ODlogcrys(V),
	\end{equation*}
	compatible with the respective logarithmic connections.
	In particular, it follows that the left vertical arrow in \eqref{eq:logcrys_admis_wach} is bijective, and hence the bottom arrow is bijective as well.
	Hence, we conclude that $V = T_R(N)[1/p]$ is a log-crystalline representation of $G_R$.
\end{proof}

\appendix

\section{Appendix}\label{sec:appendix}

\subsection{Some commutative algebra}\label{subsec:commutative_algebra}

\begin{lem}\label{lem:local_criterion_flatness}
	Let $A$ be a noetherian $a\textrm{-adically}$ complete ring for a nonzerodivisor $a$ in $A$.
	Let $M$ be an $a\textrm{-torsion}$ free and $a\textrm{-adically}$ complete $A\textrm{-module}$.
	Then, $M/aM$ is flat over $A/aA$ if and only if $M$ is flat over $A$.
\end{lem}
\begin{proof}
	Assume that $M/aM$ is flat over $A/aA$.
	Note that $\textup{Tor}_1^{A}(A/aA, M) = 0$ because $M$ is $a\textrm{-torsion}$ free.
	So, by the local criterion of flatness (see \cite[\href{https://stacks.math.columbia.edu/tag/051C}{Tag 051C}]{stacks-project}), it follows that $M/a^nM$ is flat over $A/a^nA$.
	Consequently, $M = \lim_n M/a^nM$ is $a\textrm{-completely}$ flat over $M$, hence flat by \cite[\href{https://stacks.math.columbia.edu/tag/0912}{Tag 0912}]{stacks-project}.
	The converse is obvious.
\end{proof}

Let $A$ be a $\mathbb{Z}_p\textrm{-algebra}$, and $G$ a finite group of order prime to $p$ acting trivially on $A$.
Let $M$ be an $A\textrm{-module}$ admitting an $A\textrm{-linear}$ action of $G$.
\begin{lem}\label{lem:invariants_modulo}
	We have that $H^1(G, M) = 0$.
	If $a$ in $A$ be a nonzerodivisor such that it is regular on $M$, then we have a natural identification $M^G/aM^G \isomorphic (M/aM)^G$ of modules over $A/aA$.
\end{lem}
\begin{proof}
	Note that $H^1(G, M) = 0$ because $|G|$ is invertible on $M$ (see \cite[Proposition 6.1.10]{weibel}).
	Now, consider the following $A\textrm{-linear}$ and $G\textrm{-equivariant}$ exact sequence:
	\begin{equation*}
		0 \longrightarrow M \xrightarrow{\hspace{1mm} a \hspace{1mm}} M \longrightarrow M/aM \longrightarrow 0.
	\end{equation*}
	Taking the $G\textrm{-invariants}$ in the sequence above and using that $H^1(G, M) = 0$, we obtain that $M^G/aM^G \isomorphic (M/aM)^G$.
\end{proof}

Let $\textrm{tr} \colon M \rightarrow M$ denote the map $x \mapsto \textrm{tr}(x) = \sum_{g \in G} g(x)$, and note that $\textrm{tr}(M)$ is an $A\textrm{-submodule}$ of $M$.
\begin{lem}\label{lem:Gfixed_trace}
	The following claims are true:
	\begin{enumerate}
		\item[\textup{(1)}] We have an $A\textrm{-linear}$ identification $M^G = \textup{tr}(M)$.

		\item[\textup{(2)}] The module $M^G$ is a direct summand of $M$ as an $A\textrm{-module}$.
	\end{enumerate}
\end{lem}
\begin{proof}
	The first claim follows by a direct computation.
	For the second claim, consider the following $A\textrm{-linear}$ exact sequence:
	\begin{equation*}
		0 \longrightarrow \textup{ker(tr)} \longrightarrow M \xrightarrow{\hspace{1mm}|G|^{-1}\textrm{tr}\hspace{1mm}} M^G \longrightarrow 0,
	\end{equation*}
	and observe that the natural inclusion $M^G \hookrightarrow M$ provides an $A\textrm{-linear}$ splitting of the surjective arrow.
\end{proof}

\begin{lem}\label{lem:tensorprod_fixedpoints}
	Let $N$ be an $A\textrm{-module}$ equipped with the trivial action of $G$.
	Then, we have a natural identification of $A\textrm{-modules}$ $(M \otimes_A N)^G \isomorphic M^G \otimes_A N$.
\end{lem}
\begin{proof}
	Note that the natural homomorphism of $A\textrm{-modules}$ $M^G \otimes_A N \rightarrow M \otimes_A N$ is injective because $M^G$ is an $A\textrm{-linear}$ direct summand of $M$ (see Lemma \ref{lem:Gfixed_trace}).
	To get that it is surjective, note that for any $z = \sum_{i \in I} x_i \otimes y_i$ in $(M \otimes_A N)^G$, we have that $z = |G|^{-1}\textrm{tr}(z) = \sum_{i \in I} |G|^{-1}\textrm{tr}(x_i) \otimes y_i$ is in $M^G \otimes_A N$.
\end{proof}

In this rest of this section, we will prove some technical lemmas on Galois descent.
We will use the notation of Section \ref{subsubsec:imperfect_rings_mixedchar}.

Let $m \geqslant 1$ be an integer coprime to $p$.
Let us set $A \coloneq R\llbracket \mu \rrbracket$ and $B \coloneq A[\zeta_m]$, where $\zeta_m$ is a primitive $m^{\textrm{th}}$ root of unity.
Let us also set $C \coloneq \calR\llbracket \mu \rrbracket$.
Let $G \coloneq \Gal(\Fr(B)/\Fr(A)) \isomorphic \Gal(\Fr(R[\zeta_m])/\Fr(R)) \isomorphic \Gal(\mathbb{Q}_p(\zeta_m)/\mathbb{Q}_p)$ and $\Delta \coloneq \Gal(\Fr(C)/\Fr(A)) \isomorphic \Gal(\Fr(\calR)/\Fr(R))$.
For $1 \leqslant i \leqslant d$, let us set $J_i \coloneq \Gal(\Fr(B[x_i^{1/m}])/\Fr(B)) \isomorphic \Gal(\Fr(R[\zeta_m, x_i^{1/m}])/\Fr(R[\zeta_m]))$, and note it is a cyclic group of order $m$.
Set $J \coloneq \Gal(\Fr(C)/\Fr(B)) \isomorphic \Gal(\Fr(\calR)/\Fr(R[\zeta_m]))$, and note that it is naturally isomorphic to the direct product of groups $J_1 \times \cdots \times J_d$.
The groups discussed above sit in the following exact sequence:
\begin{equation*}
	1 \longrightarrow J \longrightarrow \Delta \longrightarrow G \longrightarrow 1.
\end{equation*}

\begin{lem}\label{lem:descent_G}
	Let $M$ be a $p\textrm{-torsion}$ free and $p\textrm{-adically}$ complete $\mathbb{Z}_p[\zeta_m]\textrm{-module}$ equipped with a semilinear action of $G$.
	Then, the natural $G\textrm{-equivariant}$ homomorphism $\mathbb{Z}_p[\zeta_m] \otimes_{\mathbb{Z}_p} M^G \rightarrow M$ is an isomorphism.
	Additionally, we have that $H^1(G,M)=0$, and we have a natural identification $M^G/p^nM^G \isomorphic (M/p^nM)^G$ of modules over $\mathbb{Z}/p^n\mathbb{Z}$.
\end{lem}
\begin{proof}
	Note that $M$ is a torsion-free module over a complete discrete valuation ring $\mathbb{Z}_p[\zeta_m]$, in particular, $M$ is flat over $\mathbb{Z}_p[\zeta_m]$.
	So we may write $M = \colim_{i \in I} M_i$, a filtered colimit of finite free modules over $\mathbb{Z}_p[\zeta_m]$ over an indexing set $I$.
	Upto replacing $M_i$ by its image inside $M$ via the natural map $M_i \rightarrow M$, and taking the stabilisation in $M$ of the image under the action of finite group $G$, we may assume that $M_i$ is a finite free $G\textrm{-stable}$ $\mathbb{Z}_p[\zeta_m]\textrm{-submodule}$ of $M$, for each $i$ in $I$.
	As the colimit is filtered, we may further write $M = \cup_{i \in I} M_i$, and thus we have that $M^G = \cup_{i \in I} M_i^G$.
	Since tensor product commutes with colimit, we are reduced to showing that the natural $G\textrm{-equivariant}$ homomorphism $\mathbb{Z}_p[\zeta_m] \otimes_{\mathbb{Z}_p} M_i^G \rightarrow M_i$ is an isomorphism for each $i$ in $I$.
	But the reduced claim holds by Galois descent (or equivalently, by the triviality of the non-abelian $H^1(G, \textup{GL}_k(\mathbb{Z}_p[\zeta_m]))$).
	Hence, we obtain the claimed $G\textrm{-equivariant}$ isomorphism $\mathbb{Z}_p[\zeta_m] \otimes_{\mathbb{Z}_p} M^G \isomorphic M$.
	Additionally, the preceding isomorphism implies that $H^1(G, M) = H^1(G, \mathbb{Z}_p[\zeta_m] \otimes_{\mathbb{Z}_p} M^G) = H^1(G, \mathbb{Z}_p[\zeta_m]) \otimes_{\mathbb{Z}_p} M^G = 0$.
	Finally, using that $H^1(G,M) = 0$, it is easy to verify that we have $M^G/p^nM^G \isomorphic (M/p^nM)^G$ for each $n \geqslant 1$ (see the proof of Lemma \ref{lem:invariants_modulo}).
\end{proof}

\begin{lem}\label{lem:descent_Delta}
	Let $M$ be a $p\textrm{-adically}$ complete and $p\textrm{-torsion}$ free $C\textrm{-module}$ equipped with a semilinear action of $\Delta$.
	Then, $H^1(\Delta,M)=0$ and we have a natural identification $M^G/p^nM^G \isomorphic (M/p^nM)^G$ of modules over $A/p^nA$.
\end{lem}
\begin{proof}
	Note that we have $\Delta = J \rtimes G$.
	Then, the triviality of $H^1(G,M)$ follows from Lemmas \ref{lem:invariants_modulo} and \ref{lem:descent_G}, and the inflation-restriction exact sequence.
	Using that $H^1(G,M) = 0$, it is easy to verify that we have $M^G/p^nM^G \isomorphic (M/p^nM)^G$ for each $n \geqslant 1$ (see the proof of Lemma \ref{lem:invariants_modulo}).
\end{proof}

\begin{lem}\label{lem: descent lemma 2}\label{lem:descent_J}
	Let $a$ in $A$ be either $p$ or $\mu$, and let $N$ be any torsion-free module over $C$ equipped with a semilinear action of $J$.
	Assume that for each $1 \leqslant i \leqslant d$, the sequence $\{a, x_i\}$ is regular on $N$.
	Then, the following natural $J\textrm{-equivariant}$ homomorphism of $C\textrm{-modules}$ is injective:
	\begin{equation*}
		\begin{aligned}
			f \colon C \otimes_B N^J &\longrightarrow N \\
				\textstyle\sum_{i=1}^d c_i \otimes n_i &\longmapsto \textstyle\sum_{i=1}^d c_in_i,
		\end{aligned}
	\end{equation*}
	and if its cokernel is nonzero, then it is $a\textrm{-torsion}$ free.
\end{lem}
\begin{proof}
	To get the injectivity of $f$, set $x = x_{a+1} \cdots x_d$ and consider the following $J\textrm{-equivariant}$ diagram
	\begin{equation}\label{eq:invx_Jdescent}
		\begin{tikzcd}
			{C \otimes_B N^J} & {C \otimes_B N^J[1/x]} \\
			N & {N[1/x]},
			\arrow[hookrightarrow, from=1-1, to=1-2]
			\arrow["f"', from=1-1, to=2-1]
			\arrow["\wr", from=1-2, to=2-2]
			\arrow[hookrightarrow, from=2-1, to=2-2]
		\end{tikzcd}
	\end{equation}
	where the horizontal arrows are naturally injective, and the right vertical arrow is an isomorphism by observing that $C \otimes_B N^J[1/x] = C[1/x] \otimes_{B[1/x]} (N[1/x])^J$ and using Galois descent.
	
	Let us now assume that the cokernel of $f$ is nonzero and we claim that it is $a\textrm{-torsion}$ free.
	Note that $C$ is free as a module over $B$, and let $\{c_i\}_{i \in I}$ denote a $B\textrm{-basis}$ of $C$ for an indexing set $I$ of order equal to $\textrm{rk}_B C$.
	Let $y$ in $C \otimes_B N^J$ such that $f(y) = az$, and observe that it is enough to show that $y$ is in $a(C \otimes_B N^J)$.
	Writing $y = \sum_{i \in I} c_i \otimes n_i$ for some unique $n_i$ in $N^J$, and using diagram \eqref{eq:invx_Jdescent}, we observe that $y = \sum_{i \in I} c_i \otimes n_i$ is in $a(C \otimes_B N^{\Delta})[1/x]$.
	Therefore, it follows that $n_i$ is in $aN^J[1/x]$ for each $i \in I$.
	But, note that the sequence $\{a, x_i\}$ is regular on $N^J$ for each $1 \leqslant i \leqslant d$.
	Indeed, we have $H^1(J,N) = 0$ by \cite[Proposition 6.1.10]{weibel} because $m$ is invertible in $C$, and using this fact, it is easy to verify that we have $N^J/aN^J = (N/aN)^J$ implying that $\{a, x_i\}$ is regular on $N^J$ because $\{a, x_i\}$ is regular on $N$.
	Hence, we conclude that $n_i$ is in $aN^J[1/x] \cap N^J = aN^J$, where the intersection takes place inside $N^J[1/x]$.
	This allows us to conclude.
\end{proof}

\begin{prop}\label{prop:descent_Delta}
	Let $a$ in $A$ be either $p$ or $\mu$, and let $N$ be a $p\textrm{-adically}$ complete and torsion-free module over $C$ equipped with a semilinear action of $\Delta$.
	Assume that for each $1 \leqslant i \leqslant d$, the sequence $\{a, x_i\}$ is regular on $N$.
	Then, the following natural $\Delta\textrm{-equivariant}$ homomorphism of $C\textrm{-modules}$ is injective:
	\begin{equation*}
		\begin{aligned}
			C \otimes_A N^{\Delta} &\longrightarrow N \\
			\textstyle\sum_{i=1}^d c_i \otimes n_i &\longmapsto \textstyle\sum_{i=1}^d c_in_i,
		\end{aligned}
	\end{equation*}
	and if its cokernel is nonzero, then it is $a\textrm{-torsion}$ free.
\end{prop}
\begin{proof}
	Note that $H^1(J,N) = 0$ by \cite[Proposition 6.1.10]{weibel} because $m$ is invertible in $C$.
	Using this fact, it is easy to verify that we have $N^J/p^nN^J = (N/p^nN)^J$ for each $n \geqslant 1$.
	Thus, it follows that $N^J$ is $p\textrm{-adically}$ complete, and it is clearly $p\textrm{-torsion}$ free.
	By applying Lemma \ref{lem:descent_G} to $N^J$, we obtain that the natural $G\textrm{-equivariant}$ homomorphism $B \otimes_A N^{\Delta} = \mathbb{Z}_p[\zeta_m] \otimes_{\mathbb{Z}_p} N^{\Delta} \rightarrow N^J$ is an isomorphism.
	Then, from Lemma \ref{lem:descent_J}, it follows that the natural $\Delta\textrm{-equivariant}$ homomorphism of $C\textrm{-modules}$ $C \otimes_A N^{\Delta} \isomorphic C \otimes_B N^J \rightarrow N$ is injective, and if its cokernel is nonzero, then it is $a\textrm{-torsion}$ free.
\end{proof}

\subsection{Period rings of Kedlaya--Liu}\label{subsec:Kedlaya_Liu_ii}

In this section, our goal is to apply the decompletion techniques of Kedlaya--Liu \cite{kedlaya-liu1, kedlaya-liu2} to the perfect period rings studied in Section \ref{subsubsec:perfect_period_rings}.
We will keep the notation and conventions of Section \ref{sec:prelims} and \ref{subsubsec:perfect_period_rings}.
In particular, we have a $p\textrm{-adically}$ complete ring $R$, and for any $m \geqslant 0$ coprime to $p$ we set $\calR = R_{0,m} = R[\zeta_m, x_1^{1/m}, \ldots, x_d^{1/m}]$ and $\calR_n = R_{n,m} = R[\zeta_{mp^n}, x_1^{1/mp^n}, \ldots, x_d^{1/mp^n}]$ for $n \geqslant 0$.

\subsubsection{Decompletion}

To apply the decompletion results of \cite{kedlaya-liu1, kedlaya-liu2}, let us first recall some standard notation. 
For any element $a$ in $\tilde{E}^+(\overline{R}) = \overline{R}^{\flat}$ we may write $a = (a_0, a_1, \ldots)$, with $a_m$ in $\overline{R}/p\overline{R}$ for all $m \geqslant 0$, and we set $a^{(0)} \coloneq \lim_{m \geqslant 0} b_m^{p^m}$, where $b_m$ in $\overline{R}$ is any lift of $a_m$; note that $a^{(0)}$ is well defined and does not depend on the choice of the lifts $b_m$ (see \cite[Section 1.2.2]{fontaine-corps-periodes}).
Choose a compatible system $\varepsilon \coloneq (1, \zeta_p, \zeta_{p^2}, \ldots)$ of primitive roots of unity in $O_{F_{\infty}}^{\flat}$, which we shall also consider as an element of $\tilde{E}^+(\overline{R})$, and set $\overline{\mu} \coloneq \varepsilon-1$.
Then $\tilde{E}(\overline{R}) = \tilde{E}^+(\overline{R})[1/\overline{\mu}]$, and for any $a/\overline{\mu}^n$ in $\tilde{E}(\overline{R})$, we set $(a/\overline{\mu}^n)^{(0)} \coloneq a^{(0)}/(\overline{\mu}^{(0)})^n$.

Following \cite[Section 4.1]{andreatta-brinon}, we define a map $\tilde{v} \colon \tilde{E}(\overline{R}) \rightarrow \mathbb{Q} \cup \{+\infty\}$ by setting
\begin{equation*}
	\tilde{v}(a) \coloneq \tfrac{p}{p-1} \max \{n \in \mathbb{Q} \textrm{ such that } a \textup{ is in } \overline{\mu}^{n}\tilde{E}^+(\overline{R})\}.
\end{equation*}
Note that $\tilde{v}(\overline{\mu}^{(0)}) = \tfrac{p}{p-1}$, and the map above satisfies the following properties:
\begin{itemize}
	\item $\tilde{v}(a+b) \geqslant \min\{\tilde{v}(a), \tilde{v}(b)\}$.
	\item $\widetilde{v}(ab) \geqslant \tilde{v}(a) + \tilde{v}(b)$.
	\item $\widetilde{v}(a\overline{\mu}) = \widetilde{v}(a) + \tfrac{p}{p-1}$.
\end{itemize}
The map $\tilde{v}$ naturally restricts to a map $\tilde{v} \colon \tilde{E}(\calR_{\infty}) \rightarrow \mathbb{Q} \cup \{+\infty\}$ satisfying the analogous properties.

\begin{defi}
	For any $r$ in $\mathbb{Q}_{>0}$, let us define the ring
	\begin{equation*}
		\tilde{A}^{\dagger,r}_{\calR} \coloneq \{a = \textstyle\sum_{k \geqslant 0} p^k[a_k] \in \tilde{A}_{\calR} \textrm{ such that } \tilde{v}(a_k)+\tfrac{pr}{p-1}k \rightarrow +\infty \textrm{ as } k \to +\infty\}.
	\end{equation*}
\end{defi}

\begin{rem}
	Let $a$ be an element of $\tilde{A}^{\dagger,r}_{\calR}$.
	As we have that $\tilde{v}(\overline{\mu}^{(0)}) = \frac{p}{p-1}$, therefore the series $\theta(a) \coloneq \sum_{k \geqslant 0} p^k a_k^{(0)}$ converges in $\widehat{\overline{R}}[1/p]$ if $r \leqslant \tfrac{p-1}{p}$.
	More generally, it is easy to see that $\theta(\varphi^{-n}(a))$ converges for all $n \geqslant \log_p(\tfrac{pr}{p-1})$.
	Consequenty, the resulting map
	\begin{equation*}
		\theta \circ \varphi^{-n} \colon \tilde{A}^{\dagger,r}_{\calR} \longrightarrow \widehat{\overline{R}},
	\end{equation*}
	is a homomorphism of rings for $n \geqslant \log_p(\tfrac{pr}{p-1})$.
\end{rem}

\begin{defi}\label{defi:decompleted_rings}
	For any $r$ in $\mathbb{Q}_{>0}$, let us define the ring
	\begin{equation*}
		\AR^{\dagger,r} \coloneq \{x \in \tilde{A}_{\calR}^{\dagger,r} \textrm{ such that } \theta(\varphi^{-n}(x)) \textrm{ is in } \calR_n[1/p] \textrm{ for each } n \geqslant \log_p(\tfrac{pr}{p-1})\}.
	\end{equation*}
	Set $\AR^{\dagger} \coloneq \cup_{r > 0} \AR^{\dagger,r}$.
	Let $\mathscr{A}_{\calR}$ denote the $p\textrm{-adic}$ completion of $\AR^{\dagger}$, and set $\mathscr{E}_{\calR} \coloneq \mathscr{A}_{\calR}/p\mathscr{A}_{\calR}$.
\end{defi}

\begin{lem}\label{lem:decompleting_tower}
	The following claims hold true:
	\begin{enumerate}
		\item[\textup{(1)}] There exists some $r > 0$ such that the projection $\AR^{\dagger,r} \rightarrow \mathscr{E}_{\calR}$ is strict surjective.

		\item[\textup{(2)}] We have that $\mathscr{E}_{\calR}^{\textup{rad}} \coloneq \cup_{n \geqslant 0} \varphi^{-n}(\mathscr{E}_{\calR})$ is dense in $\tilde{E}_{\calR}$.
	\end{enumerate}
	In other words, the tower $\{\calR_n\}_{n \geqslant 0}$ is weakly decompleting in the sense of \cite[Definition 5.2.3]{kedlaya-liu2}.
\end{lem}
\begin{proof}
	The claim follows from \cite[Lemma 7.4.3]{kedlaya-liu2}.
\end{proof}

\subsubsection{Decompleted integral period rings}

In this section, we shall explicitly describe the rings $\mathscr{E}_{\calR}$ and $\mathscr{A}_{\calR}$ by relating them to the rings studied in Sections \ref{subsubsec:imperfect_rings_charp} and \ref{subsubsec:imperfect_rings_mixedchar}.
We begin by defining the following rings:
\begin{equation*}
	\begin{aligned}
		\mathscr{A}_{\calR}^+ &\coloneq \mathscr{A}_{\calR} \cap A_{\inf}(\calR_{\infty}) \subset \tilde{A}_{\calR},\\
		\mathscr{E}_{\calR}^+ &\coloneq \mathscr{E}_{\calR} \cap \tilde{E}_{\calR}^+ \subset \tilde{E}_{\calR}.
	\end{aligned}
\end{equation*}
It is easy to check that the ring $\mathscr{A}_{\calR}^+$ (resp.\ $\mathscr{E}_{\calR}^+$) is $(p, \mu)\textrm{-adically}$ (resp.\ $\overline{\mu}\textrm{-adically}$) complete, and we have that $\mathscr{A}_{\calR} = \mathscr{A}_{\calR}^+[1/\mu]_p^{\wedge}$ using that $\mathscr{E}_{\calR}^+ = \mathscr{A}_{\calR}^+/p\mathscr{A}_{\calR}^+$.

\begin{rem}\label{rem:ar+_alt_desc}
	Note that $\varphi \colon A_{\inf}(\calR_{\infty}) \isomorphic A_{\inf}(\calR_{\infty})$, and therefore for any $a$ in $A_{\inf}(\calR_{\infty})$ and $n \geqslant 0$ we have that $\theta(\varphi^{-n}(a))$ belongs to $\widehat{R}_{\infty}$.
	In particular, the natural injective homomorphism of rings $A_{\inf}(\calR_{\infty}) \hookrightarrow \tilde{A}_{\calR}$ factors through $A_{\inf}(\calR_{\infty}) \hookrightarrow \tilde{A}^{\dagger,r}_{\calR}$ for all $r > 0$.
	Then, from Definition \ref{defi:decompleted_rings}, we obtain the following description:
	\begin{equation*}
		\mathscr{A}_{\calR}^+ = \{x \in A_{\inf}(\calR_{\infty}) \textrm{ such that } \theta(\varphi^{-n}(x)) \textrm{ is in } \calR_n \textrm{ for each } n \geqslant 0\}.
	\end{equation*}
\end{rem}

The goal of this section is to show the following claim:
\begin{prop}\label{prop:ar+_decomplete}
	The subrings $\AR^+$ (see Section \ref{subsubsec:imperfect_rings_mixedchar}) and $\mathscr{A}_{\calR}^+$ of $A_{\inf}(\calR_{\infty})$ are naturally identified, i.e.\ there exists a natural isomorphism
	\begin{equation*}
		\AR^+ \isomorphic \mathscr{A}_{\calR}^+,
	\end{equation*}
	as subrings of $A_{\inf}(\calR_{\infty})$.
	In particular, we also get that $\ER^+ \isomorphic \mathscr{E}_{\calR}^+$ as subrings of $\tilde{E}_{\calR}^+$.
\end{prop}

Let us first show that there exists a natural homomorphism of rings
\begin{equation*}
	f \colon \AR^+ \longrightarrow \mathscr{A}_{\calR}^+,
\end{equation*}
compatible with their respective inclusions in $A_{\inf}(\calR_{\infty})$.
To construct the map $f$, we begin with an observation:
\begin{lem}\label{lem:lim_ring_of_norms}
	Let $\{B_n\}_{n \geqslant 0}$ denote a family of $\mathbb{F}_p\textrm{-algebras}$ satisfying the following conditions:
	\begin{itemize}
		\item For each $n \geqslant 0$, there exists a natural inclusion $B_n \hookrightarrow B_{n+1}$.
		\item For each $n \geqslant 0$ the Frobenius endomorphism $\varphi \colon B_{n+1} \rightarrow B_{n+1}$ factors through an isomorphism $B_{n+1} \xrightarrow[\phi_B]{\hspace{1mm}\sim\hspace{1mm}} B_n \hookrightarrow B_{n+1}$.
	\end{itemize}
	Let $C_0$ be an \'etale algebra over $B_0$.
	Set $C_n \coloneq B_n \otimes_{B_0} C_0$ and $D_n \coloneq C_n \otimes_{\mathbb{F}_p} \mathbb{Z}[\zeta_{p^n}]/p\mathbb{Z}[\zeta_{p^n}]$ for each $n \geqslant 0$.
	\begin{enumerate}
		\item[\textup{(1)}] For each $n \geqslant 0$ the Frobenius endomorphism $\varphi \colon C_{n+1} \rightarrow C_{n+1}$ factors through an isomorphism $C_{n+1} \xrightarrow[\phi_C]{\hspace{1mm}\sim\hspace{1mm}} C_n \hookrightarrow C_{n+1}$.
			In particular, we have a natural isomorphism of rings
			\begin{equation*}
				\textstyle\lim_{\phi_C} C_n \isomorphic C_0.
			\end{equation*}

		\item[\textup{(2)}] For each $n \geqslant 0$ the Frobenius endomorphism $\varphi \colon D_{n+1} \rightarrow D_{n+1}$ factors through a homomorphism $D_{n+1} \xrightarrow{\hspace{1mm}\phi_D\hspace{1mm}} D_n \hookrightarrow D_{n+1}$.
			Additionally, we have a natural isomorphism of rings
			\begin{equation*}
				\textstyle\lim_{\phi_D} D_n \isomorphic D_0\llbracket \overline{\mu} \rrbracket = C_0\llbracket \overline{\mu} \rrbracket
			\end{equation*}
	\end{enumerate}
\end{lem}
\begin{proof}
	It is straightforward to verify that for each $n \geqslant 0$ the Frobenius endomorphism $\varphi \colon C_{n+1} \rightarrow C_{n+1}$ (resp.\ $\varphi \colon D_{n+1} \rightarrow D_{n+1}$) factors through a homomorphism $C_{n+1} \xrightarrow{\hspace{1mm}\phi_C\hspace{1mm}} C_n \hookrightarrow C_{n+1}$ (resp.\ $D_{n+1} \xrightarrow{\hspace{1mm}\phi_D\hspace{1mm}} D_n \hookrightarrow D_{n+1}$).

	To show part (1), let us first note that the relative Frobenius map $\varphi^* C_0 = B_0 \otimes_{\varphi,B_0} C_0 \rightarrow C_0$, given by $b \otimes c \mapsto b\varphi(c)$, is a $B_0\textrm{-linear}$ isomorphism because $C_0$ is \'etale over $B_0$ (for example, see \cite[\href{https://stacks.math.columbia.edu/tag/0EBS}{Tag 0EBS}]{stacks-project}).
	Then, by extending scalars of the preceding isomorphism along the natural inclusion $B_0 \hookrightarrow B_n$ and using the commutative diagram
	\begin{center}
		\begin{tikzcd}
			{B_0} & {B_{n+1}} & {B_n} \\
			{B_0} && {B_n},
			\arrow[hook, from=1-1, to=1-2]
			\arrow["\varphi"', from=1-1, to=2-1]
			\arrow["{\phi_B}", from=1-2, to=1-3]
			\arrow[equal, from=1-3, to=2-3]
			\arrow[hook, from=2-1, to=2-3]
		\end{tikzcd}
	\end{center}
	we obtain $B_n\textrm{-linear}$ isomorphisms
	\begin{equation*}
		\phi_B^* C_{n+1} = B_n \otimes_{\phi_B,B_{n+1}} B_{n+1} \otimes_{B_0} C_0 \isomorphic B_n \otimes_{B_0} B_0 \otimes_{\varphi,B_0} C_0 = B_n \otimes_{B_0} \varphi^* C_0 \isomorphic B_n \otimes_{B_0} C_0 = C_n,
	\end{equation*}
	sending $b \otimes c \mapsto b\phi_C(c)$.
	Precomposing the preceding isomorphism with the isomorphism $C_{n+1} \isomorphic B_n \otimes_{\phi_B,B_{n+1}} C_{n+1} = \phi_B^* C_{n+1}$ linear over $\phi_B \colon B_{n+1} \isomorphic B_n$, yields an isomorphism $C_{n+1} \isomorphic \phi_B^* C_{n+1} \isomorphic C_n$.
	Then, from the discussion above we see that the composition $C_{n+1} \isomorphic \phi_B^* C_{n+1} \isomorphic C_n$ coincides with the homomorphism $\phi_C \colon C_{n+1} \rightarrow C_n$, hence the latter is bijective.
	Passing to the limit, we obtain that $\lim_{\phi_C} C_n \isomorphic C_0$.

	For part (2), note that we have the following commutative diagram:
	\begin{equation*}
		\begin{tikzcd}
			D_{n+1} = C_{n+1} \otimes_{\mathbb{F}_p} \mathbb{Z}[\zeta_{p^{n+1}}]/p\mathbb{Z}[\zeta_{p^{n+1}}] \arrow[r, "\phi \otimes \phi"] \arrow[d, "\wr", "\phi^{n+1} \otimes 1"'] & C_n \otimes_{\mathbb{F}_p} \mathbb{Z}[\zeta_{p^n}]/p\mathbb{Z}[\zeta_{p^n}] = D_n \arrow[d, "\phi^n \otimes 1", "\wr"']\\
			C_0 \otimes_{\mathbb{F}_p} \mathbb{Z}[\zeta_{p^{n+1}}]/p\mathbb{Z}[\zeta_{p^{n+1}}] \arrow[r, "1 \otimes \phi"] & C_0 \otimes_{\mathbb{F}_p} \mathbb{Z}[\zeta_{p^n}]/p\mathbb{Z}[\zeta_{p^n}],
		\end{tikzcd}
	\end{equation*}
	Therefore, upon passing to the limit we get that $\lim_{\phi_D} D_n \isomorphic \lim_{1 \otimes \phi} C_0[\zeta_{p^{n+1}}] = C_0\llbracket \overline{\mu} \rrbracket = D_0\llbracket \overline{\mu} \rrbracket$.
\end{proof}

Let $\calR_{\square} = R_{\square,0,m}$.
For each $n \geqslant 0$, in the notation of Lemma \ref{lem:lim_ring_of_norms}, set $B_n = \calR_{\square}[x_1^{1/mp^n}, \ldots, x_d^{1/mp^n}]/p$, $C_n = \calR[x_1^{1/mp^n}, \ldots, x_d^{1/mp^n}]/p$ and $D_n = C_n[\zeta_{p^n}] = \calR_n/p\calR_n$.
Then, from Lemma \ref{lem:lim_ring_of_norms}, it follows that we have the following natural isomorphism of rings:
\begin{equation}\label{eq:lim_Rmn_phi}
	\textstyle\lim_{\phi} \calR_n/p\calR_n \isomorphic (\calR/p\calR)\llbracket \overline{\mu} \rrbracket = \ER^+.
\end{equation}

\begin{lem}\label{lem:ar+_decomp_map}
	The natural injective homomorphism $\AR^+ \hookrightarrow A_{\inf}(R_{\infty})$ factors through $\AR^+ \xrightarrow{\hspace{1mm} f \hspace{1mm}} \mathscr{A}_{\calR}^+ \hookrightarrow A_{\inf}(R_{\infty})$.
	In particular, $f$ is injective.
\end{lem}
\begin{proof}
	Let $a$ be an element of $\AR^+$, and using Remark \ref{rem:ar+_alt_desc} note that to obtain the claim it suffices to show that $\theta(\varphi^{-n}(a))$ belongs to $\calR_n$ for each $n \geqslant 0$.
	Using the isomorphism $\iota \colon \calR\llbracket \mu \rrbracket \isomorphic \AR^+$ from Section \ref{subsubsec:imperfect_rings_mixedchar}, we may write $a = \sum_{i \geqslant 0} \iota(a_i)\mu^i$, for some $a_i$ in $\calR$.
	Then, using Lemma \ref{lem:lim_ring_of_norms}, it is easy to check that $\theta(\varphi^{-n}(\iota(a_i)))$ is in $\calR_n$ for each $i, n \geqslant 0$.
	Since $\theta(\varphi^{-n}(\mu)) = \zeta_{p^n}-1$ for $n \geqslant 0$, therefore, we conclude that $\theta(\varphi^{-n}(a))$ belongs to $\calR_n$ for each $n \geqslant 0$.
\end{proof}

\begin{proof}[Proof of Proposition \ref{prop:ar+_decomplete}]
	From Lemma \ref{lem:ar+_decomp_map}, we have an injective homomorphism of rings $f \colon \AR^+ \rightarrow \mathscr{A}_{\calR}^+$ compatible with their respective inclusions into $A_{\inf}(\calR_{\infty})$.
	As the source and the target of $f$ are $p\textrm{-adically}$ complete and $p\textrm{-torsion}$ free, it suffices to show that $f$ is bijective modulo $p$.
	Since we have that $\ER^+ \xrightarrow{\hspace{1mm} f \hspace{1mm}} \mathscr{E}_{\calR}^+ \hookrightarrow \tilde{E}_{\calR}^+$, and the composition is also injective, therefore, it follows that $f$ is injective modulo $p$.

	To show the surjectivity of $f$ modulo $p$, let $b$ be an element of $\mathscr{E}_{\calR}^+$, and write $b = (b_0, b_1, \ldots)$ in $\tilde{E}_{\calR}^+$ with $b_n$ in $\calR_{\infty}/p\calR_{\infty}$ satisfying $b_{n+1}^p = b_n$ for each $n \geqslant 0$.
	Let $a$ in $\mathscr{A}_{\calR}^+$ denote a lift of $b$.
	From Remark \ref{rem:ar+_alt_desc} we have that $\theta(\varphi^{-n}(a))$ belongs to $\calR_n$ for each $n \geqslant 0$.
	As the equality $b_n = \theta(\varphi^{-n}(a)) \textrm{ mod } p$ holds in $\calR_{\infty}/p\calR_{\infty}$, therefore, we conclude that $b_n$ belongs to $\calR_n/p\calR_n$ for each $n \geqslant 0$.
	Hence, $b = (b_0, b_1, \ldots)$ belongs to $\lim_{\phi} \calR_n/p\calR_n \isomorphic \ER^+$, where the isomorphism is from \eqref{eq:lim_Rmn_phi}.
	This completes our proof.
\end{proof}

\subsubsection{Galois groups}

In this section, we will relate the Galois groups of various period rings considered in Section \ref{subsec:period_rings}.

From Proposition \ref{prop:ar+_decomplete}, let us first note that the ring $\mathscr{A}_{\calR}$ (resp.\ $\mathscr{E}_{\calR}$) coincides with the ring $\AR$ (resp.\ $\ER$) described in Section \ref{subsubsec:imperfect_rings_mixedchar} (resp.\ Section \ref{subsubsec:imperfect_rings_charp}).
We shall confuse the two notation in the following.

\begin{lem}\label{lem:normal_domain_check}
	The rings $\AR$, $\ER$, $\tilde{E}_{\calR}$, $\widehat{\calR}_{\infty}[1/p]$ and $\calR_{\infty}[1/p]$ are normal domains.
\end{lem}
\begin{proof}
	From the discussion in Section \ref{subsubsec:imperfect_rings_charp} (resp.\ Section \ref{subsubsec:imperfect_rings_mixedchar}), it is clear that $\ER$ (resp.\ $\AR$) is a normal domain.
	In the notation of Lemma \ref{lem:decompleting_tower}, note that $\mathscr{E}_{\calR}^{\textup{rad}} \coloneq \cup_{n \geqslant 0} \varphi^{-n}(\mathscr{E}_{\calR})$ is a normal domain satisfying \cite[A.1, Conditions (i)-(iv)]{andreatta-phigamma}, and $\tilde{E}_{\calR}$ is its completion, hence a normal domain by \cite[Lemma A.3 and Proposition A.6]{andreatta-phigamma}.
	Similarly, we have that $\calR_{\infty} = \colim_n \calR_n$ is a normal domain satisfying \cite[A.1, Conditions (i)-(iv)]{andreatta-phigamma}, and $\widehat{\calR}_{\infty}$ is its completion, hence a normal domain by \cite[Lemma A.3 and Proposition A.6]{andreatta-phigamma}.
\end{proof}

\begin{nota}
	For a ring $A$, let $\textup{F\'Et}(A)$ denote the category of finite \'etale algebras over $A$.
\end{nota}

\begin{thm}\label{thm:fet_equivalence}
	The following categories are naturally equivalent: 
	\begin{equation}\label{eq:fet_equivalence}
		\textup{F\'Et}(A_{\calR}) \isomorphic \textup{F\'Et}(E_{\calR}) \isomorphic \textup{F\'Et}(\tilde{E}_{\calR}) \lisomorphic \textup{F\'Et}(\widehat{\calR}_{\infty}[1/p]) \lisomorphic \textup{F\'Et}(\calR_{\infty}[1/p]).
	\end{equation}
\end{thm}
\begin{proof}
	In \eqref{eq:fet_equivalence}, the first equivalence follows from \cite[Theorem 1.2.8]{kedlaya-liu1}, the second equivalence follows from \cite[Theorem 3.1.15]{kedlaya-liu1}, the third equivalence follows from \cite[Theorems 5.25 and 7.9]{scholze-perfectoid}, and the fourth equivalence follows from \cite[Theorem 2.6.8]{kedlaya-liu1}.
\end{proof}

As all the rings appearing in the statement of Theorem \ref{thm:fet_equivalence} are normal domains by Lemma \ref{lem:normal_domain_check}, therefore, we may relate their absolute Galois groups.
Let $\pazocal{C}$ denote the full subcategory of $\textup{F\'Et}(\calR_{\infty}[1/p])$ whose objects are contained in $\overline{R}[1/p]$; by abusing notation we also let $\pazocal{C}$ denote the corresponding equivalent subcategory of each of the equivalent categories in \eqref{eq:fet_equivalence}.
Let us set 
\begin{equation*}
	\begin{aligned}
		A_{\calR}^{\etale,\circ} &\coloneq \colim_{A_{\calR}' \in \pazocal{C}} A_{\calR}', \qquad E_{\calR}^{\etale} \coloneq \colim_{E_{\calR}' \in \pazocal{C}} E_{\calR}',\\
		\tilde{E}_{\calR}^{\etale,\circ} &\coloneq \colim_{\tilde{E}_{\calR}' \in \pazocal{C}} \tilde{E}_{\calR}', \qquad \widehat{\calR}_{\infty}^{\etale,\circ} \coloneq \colim_{\widehat{\calR}_{\infty}'[1/p] \in \pazocal{C}} \widehat{\calR}_{\infty}'.
	\end{aligned}
\end{equation*}
Note that $A_{\calR}^{\etale}$ denotes the $p\textrm{-adic}$ completion of $A_{\calR}^{\etale,\circ}$, and by continuity we have that $\textup{Aut}(A_{\calR}^{\etale}/\AR) \isomorphic \Gal(A_{\calR}^{\etale,\circ}/\AR)$.
Similarly, the ring $\tilde{E}_{\calR}^{\etale}$ denotes the completion for the $\overline{\mu}\textrm{-adic}$ topology on $\tilde{E}_{\calR}^{\etale,\circ}$, and by continuity we obtain that $\textup{Aut}(\tilde{E}_{\calR}^{\etale}/\tilde{E}_{\calR}) \isomorphic \Gal(\tilde{E}_{\calR}^{\etale,\circ}/\tilde{E}_{\calR})$.
Furthermore, the ring $\widehat{\calR}_{\infty}^{\etale}$ denotes the $p\textrm{-adic}$ completion of $\widehat{\calR}_{\infty}^{\etale,\circ}$, so again by continuity we get that $\textup{Aut}(\widehat{\calR}_{\infty}^{\etale}[1/p]/\widehat{\calR}_{\infty}[1/p]) \isomorphic \Gal(\widehat{\calR}_{\infty}^{\etale,\circ}[1/p]/\widehat{\calR}_{\infty}[1/p])$.

From the discussion above and the categorical equivalence in \eqref{eq:fet_equivalence} of Theorem \ref{thm:fet_equivalence}, we obtain the following:
\begin{cor}\label{cor:Gal_isoms}
	There exist natural isomorphisms of Galois groups:
	\begin{equation*}
		\begin{tikzcd}[column sep=4mm]
			\Gal(A_{\calR}^{\etale,\circ}/A_{\calR}) & \Gal(E_{\calR}^{\etale}/E_{\calR}) & \Gal(\tilde{E}_{\calR}^{\etale,\circ}/\tilde{E}_{\calR}) & \Gal(\widehat{\calR}_{\infty}^{\etale,\circ}[1/p]/\widehat{\calR}_{\infty}[1/p]) & \Gal(\calR_{\infty}^{\etale}[1/p]/\calR_{\infty}[1/p]) \\
			\textup{Aut}(A_{\calR}^{\etale}/\AR) && \textup{Aut}(\tilde{E}_{\calR}^{\etale}/\tilde{E}_{\calR}) & \textup{Aut}(\widehat{\calR}_{\infty}^{\etale}[1/p]/\widehat{\calR}_{\infty}[1/p]).
			\arrow["\sim"', from=1-2, to=1-1]
			\arrow["\sim"', from=1-3, to=1-2]
			\arrow["\sim", from=1-3, to=1-4]
			\arrow["\sim", from=1-4, to=1-5]
			\arrow["\wr", from=2-1, to=1-1]
			\arrow["\wr", from=2-3, to=1-3]
			\arrow["\wr", from=2-4, to=1-4]
		\end{tikzcd}
	\end{equation*}
\end{cor}

\phantomsection
\printbibliography[heading=bibintoc, title={References}]
\Addresses

\end{document}